\documentclass[11pt,reqno,tbtags,a4paper]{amsart}

\title[Fringe trees,  branching processes and $m$-ary search trees]
{Fringe trees, Crump--Mode--Jagers branching processes and $m$-ary search trees}

\date{14 January, 2016}

\usepackage{amssymb}
\usepackage{url}
\usepackage[square,numbers]{natbib}
\bibpunct[, ]{[}{]}{;}{n}{,}{,}

\author{Cecilia Holmgren}
\thanks{CH partly supported by the Swedish Research Council}

\address{Department of Mathematics, Uppsala University, PO Box 480,
SE-751~06 Uppsala, Sweden}
\email{cecilia.holmgren@math.uu.se}

\author{Svante Janson}
\thanks{SJ partly supported by the Knut and Alice Wallenberg Foundation}
\address{Department of Mathematics, Uppsala University, PO Box 480,
SE-751~06 Uppsala, Sweden}
\email{svante.janson@math.uu.se}
\urladdr{http://www.math.uu.se/svante-janson}



\numberwithin{equation}{section}

\renewcommand\le{\leqslant}
\renewcommand\ge{\geqslant}

\allowdisplaybreaks

\theoremstyle{plain}
\newtheorem{theorem}{Theorem}[section]
\newtheorem*{theorem*}{Theorem}
\newtheorem{lemma}[theorem]{Lemma}

\newtheorem{corollary}[theorem]{Corollary}

\theoremstyle{definition}
\newtheorem{example}[theorem]{Example}
\newtheorem{definition}[theorem]{Definition}
\newtheorem{problem}[theorem]{Problem}
\newtheorem{remark}[theorem]{Remark}

\theoremstyle{remark}

\newenvironment{romenumerate}[1][-10pt]{
\addtolength{\leftmargini}{#1}\begin{enumerate}
 \renewcommand{\labelenumi}{\textup{(\roman{enumi})}}%
 \renewcommand{\theenumi}{\textup{(\roman{enumi})}}%
 }{\end{enumerate}}

\newcounter{oldenumi}
{\setcounter{oldenumi}{\value{enumi}}
\begin{romenumerate} \setcounter{enumi}{\value{oldenumi}}}
{\end{romenumerate}}

\newcounter{thmenumerate}
\newenvironment{thmenumerate}
{\setcounter{thmenumerate}{0}%
 \def\item{\par
 \refstepcounter{thmenumerate}\textup{(\roman{thmenumerate})\enspace}}
}
{}

\newenvironment{xenumerate}[1][-10pt]{
\addtolength{\leftmargini}{#1}\begin{enumerate}}{\end{enumerate}}

\newcommand\pfitemx[1]{\par#1:}
\newcommand\pfitemref[1]{\pfitemx{\ref{#1}}}

\newcommand{\refT}[1]{Theorem~\ref{#1}}
\newcommand{\refC}[1]{Corollary~\ref{#1}}
\newcommand{\refL}[1]{Lemma~\ref{#1}}
\newcommand{\refR}[1]{Remark~\ref{#1}}
\newcommand{\refS}[1]{Section~\ref{#1}}
\newcommand{\refSS}[1]{Section~\ref{#1}}
\newcommand{\refP}[1]{Problem~\ref{#1}}
\newcommand{\refD}[1]{Definition~\ref{#1}}
\newcommand{\refE}[1]{Example~\ref{#1}}

\newcommand{\refApp}[1]{Appendix~\ref{#1}}

\newcommand\REM[1]{{\raggedright\texttt{[#1]}\par\marginal{XXX}}}

\begingroup
  \divide\count255 by 60
  \multiply\count255 by -60
  \advance\count255 by \time
  \ifnum \count255 < 10 \xdef\klockan{\the\count1.0\the\count255}
  \else\xdef\klockan{\the\count1.\the\count255}\fi
\endgroup

\newcommand\nopf{\qed}   

\newcommand{\sumk}{\sum_{k=0}^\infty}

\newcommand{\sumn}{\sum_{n=0}^\infty}
\newcommand{\sumni}{\sum_{n=1}^\infty}
\newcommand{\sumji}{\sum_{j=1}^\infty}
\newcommand{\sumii}{\sum_{i=1}^\infty}
\newcommand{\sumin}{\sum_{i=1}^n}

\newcommand{\sumjb}{\sum_{j=1}^b}

\newcommand{\prodji}{\prod_{j=1}^\infty}

\newcommand{\prodjn}{\prod_{j=1}^n}
\newcommand{\prodko}{\prod_{k=0}^\infty}

\newcommand\set[1]{\ensuremath{\{#1\}}}
\newcommand\bigset[1]{\ensuremath{\bigl\{#1\bigr\}}}
\newcommand\Bigset[1]{\ensuremath{\Bigl\{#1\Bigr\}}}

\newcommand\xpar[1]{(#1)}
\newcommand\bigpar[1]{\bigl(#1\bigr)}
\newcommand\Bigpar[1]{\Bigl(#1\Bigr)}

\newcommand\lrpar[1]{\left(#1\right)}

\newcommand\xcpar[1]{\{#1\}}
\newcommand\bigcpar[1]{\bigl\{#1\bigr\}}
\newcommand\Bigcpar[1]{\Bigl\{#1\Bigr\}}

\newcommand\bigabs[1]{\bigl|#1\bigr|}
\newcommand\Bigabs[1]{\Bigl|#1\Bigr|}

\newcommand\lrabs[1]{\left|#1\right|}
\def\rompar(#1){\textup(#1\textup)}    
\newcommand\xfrac[2]{#1/#2}

\newcommand\parfrac[2]{\lrpar{\frac{#1}{#2}}}

\newcommand\Bigparfrac[2]{\Bigpar{\frac{#1}{#2}}}

\def\xexp(#1){e^{#1}}
\newcommand\ceil[1]{\lceil#1\rceil}
\newcommand\floor[1]{\lfloor#1\rfloor}

\newcommand\ntoo{\ensuremath{{n\to\infty}}}

\newcommand\ktoo{\ensuremath{{k\to\infty}}}
\newcommand\mtoo{\ensuremath{{m\to\infty}}}
\newcommand\Mtoo{\ensuremath{{M\to\infty}}}

\newcommand\ttoo{\ensuremath{{t\to\infty}}}
\newcommand\xtoo{\ensuremath{{x\to\infty}}}

\newcommand\downto{\searrow}
\newcommand\upto{\nearrow}
\newcommand\punkt{.\spacefactor=1000}    
\newcommand\iid{i.i.d\punkt}    
\newcommand\ie{i.e\punkt}
\newcommand\eg{e.g\punkt}
\newcommand\viz{viz\punkt}
\newcommand\cf{cf\punkt}
\newcommand{\as}{a.s\punkt}

\newcommand\whp{w.h.p\punkt}

\newcommand{\tend}{\longrightarrow}
\newcommand\dto{\overset{\mathrm{d}}{\tend}}
\newcommand\pto{\overset{\mathrm{p}}{\tend}}
\newcommand\asto{\overset{\mathrm{a.s.}}{\tend}}
\newcommand\eqd{\overset{\mathrm{d}}{=}}

\newcommand\bbR{\mathbb R}
\newcommand\bbC{\mathbb C}
\newcommand\bbN{\mathbb N}

\newcommand\bbZ{\mathbb Z}
\newcommand\bbZleo{\mathbb Z_{\le0}}
\newcommand\bbZgeo{\mathbb Z_{\ge0}}

\newcounter{CC}
\newcommand{\CC}{\stepcounter{CC}\CCx} 
\newcommand{\CCx}{C_{\arabic{CC}}}     
\newcommand{\CCdef}[1]{\xdef#1{\CCx}}     
\newcommand{\CCreset}{\setcounter{CC}0} 
\newcounter{cc}

\renewcommand\Re{\operatorname{Re}}

\newcommand\E{\operatorname{\mathbb E{}}}
\renewcommand\P{\operatorname{\mathbb P{}}}

\newcommand\Var{\operatorname{Var}}

\newcommand\Exp{\operatorname{Exp}}
\newcommand\Po{\operatorname{Po}}
\newcommand\Bi{\operatorname{Bi}}
\newcommand\Bin{\operatorname{Bin}}

\newcommand\Geo{\operatorname{Ge_0}}
\newcommand\Gei{\operatorname{Ge_1}}
\newcommand\NBin{\operatorname{NBin}}

\newcommand\ga{\alpha}
\newcommand\gb{\beta}
\newcommand\gd{\delta}

\newcommand\gf{\phi}
\newcommand\gam{\gamma}
\newcommand\gG{\Gamma}
\newcommand\gk{\kappa}
\newcommand\gl{\lambda}
\newcommand\gL{\Lambda}
\newcommand\go{\omega}
\newcommand\gO{\Omega}
\newcommand\gs{\sigma}

\newcommand\gth{\theta}
\newcommand\eps{\varepsilon}

\newcommand\cE{\mathcal E}
\newcommand\cF{\mathcal F}

\newcommand\cI{\mathcal I}

\newcommand\cL{{\mathcal L}}

\newcommand\cP{\mathcal P}
\newcommand\cQ{\mathcal Q}

\newcommand\cS{{\mathcal S}}
\newcommand\cT{{\mathcal T}}

\newcommand\cX{{\mathcal X}}
\newcommand\cY{{\mathcal Y}}

\newcommand\ett[1]{\boldsymbol1\xcpar{#1}}

\newcommand\qw{^{-1}}
\newcommand\qww{^{-2}}

\newcommand\intoi{\int_0^1}
\newcommand\intot{\int_0^t}
\newcommand\intoo{\int_0^\infty}

\newcommand\oi{[0,1]}
\newcommand\ooo{[0,\infty)}

\newcommand\setoi{\set{0,1}}

\newcommand\dd{\,\mathrm{d}}
\newcommand\ddx{\mathrm{d}}
\newcommand\ddd[1]{\frac{\ddx}{\ddx#1}}

\newcommand{\pgf}{probability generating function}

\newcommand\lhs{left-hand side}
\newcommand\rhs{right-hand side}

\newcommand\heir{\txi^*}
\newcommand\Xia{\widehat\Xi(\ga)}
\newcommand\Xib{\widehat\Xi(\gth)}
\newcommand\zgf{Z^\gf}
\newcommand\zpsi{Z^\psi}
\newcommand\friT{\overline{\cT}}
\newcommand\sinT{\widetilde{\cT}}
\newcommand\sinTo{\sinT_{\text{\rm leaf}}}
\newcommand\sinTox[1]{\sinT_{\text{\rm leaf},#1}}
\newcommand\sinTQ{\sinT_{\cQ}}
\newcommand\taux{{\overline\tau}}
\newcommand\tauo{\tau_o}
\newcommand\tauxx{{\tau'}}
\newcommand\tauy[1]{\xi_{(#1)}}
\newcommand\mary{$m$-ary}
\newcommand\mst{\mary{} search tree}
\newcommand\emst{extended \mary{} search tree}
\newcommand\bst{binary search tree}
\newcommand\rrt{random recursive tree}
\newcommand\CMJ{Crump--Mode--Jagers}
\newcommand\CMJbp{\CMJ{} branching process}
\newcommand\CMJp{\CMJ{} process}
\newcommand\mphi{m_\phi}
\newcommand\mpsi{m_\psi}

\newcommand\xToo{\widetilde{\xT}}
\newcommand\mm[1]{^{(#1)}}
\newcommand\y[1]{^{-#1}}
\newcommand\Tyy[2]{T^{#1,-#2}}
\newcommand\Tvk{\Tyy{v}{k}}

\newcommand\xty[1]{\xT^{-#1}}
\newcommand\xtk{\xty{k}}
\newcommand\x{^*}
\newcommand\Tx{T\x}
\newcommand\Txx[1]{\Tyy*{#1}}
\newcommand\Txk{\Txx{k}}

\newcommand\tX{\widetilde X}
\newcommand\tY{\widetilde Y}
\newcommand\tS{\widetilde S}
\newcommand\hXi{\widehat \Xi}

\newcommand\hmu{\widehat\mu}
\newcommand\loghmu{\log\hmu}
\newcommand\tXi{\widetilde \Xi}
\newcommand\txi{\widetilde \xi}
\newcommand\hphi{\hat \phi}
\newcommand\hpsi{\hat \psi}

\newcommand\mx[1]{^{[m]}}
\newcommand\xT{\cF}
\newcommand\nun{n}
\newcommand\numc{n_{\mathsf{mc}}}
\newcommand\pmc{P_{\mathsf{mc}}}
\newcommand\nuxi{n_{\mathsf{x}}}

\newcommand\cPxi{\cP_{\mathsf{x}}}
\newcommand\ts{S}
\newcommand\sL{\mathsf{L}}
\newcommand\sR{\mathsf{R}}
\newcommand\iso{\approx}
\newcommand\voo{V_\infty}
\newcommand\dout{\degree^+}
\newcommand\degree{d} 
\newcommand\dep{h}
\newcommand\Xix[1]{\Xi([0,#1])}
\newcommand\Xit{\Xix{t}}

\newcommand\ott{(t)}
\newcommand\wws[1]{\frac{w_{#1}}{w_{#1}+s}}
\newcommand\prodwk{\prod_{k=0}^Mw_k}
\newcommand\xxi{\bar\xi}
\newcommand\ctt{\cT_t}

\newcommand\cttaux{\cT_\taux}
\newcommand\cpk{\cP_k}

\newcommand\refBP{\ref{BPfirst}--\ref{BPmub}}
\newcommand\refBPphi{\ref{BPfirst}--\ref{BPlast}}
\newcommand\bnu{\boldsymbol{\nu}}
\newcommand\FF{F}
\newcommand\FFx{{}_2F_1}
\newcommand\HG{\mathsf{HG}}
\newcommand\fall[2]{(#1)_{#2}} 
\newcommand\rise[2]{\langle#1\rangle_{#2}} 
\newcommand\gll{\gl_s+\gl_g}
\renewcommand\SS{S}
\newcommand\tSS{\tS}
\newcommand\gax{a}
\newcommand\gbx{b}
\newcommand\gaxx{\ga^*}
\newcommand\hgaxx{\widehat\ga^*}
\newcommand\tgaxx{\widetilde\ga^*}
\newcommand\gaxxx{\ga^{**}}
\newcommand\chgaxx{\check\ga^*}
\newcommand\chga{\check\ga}
\newcommand\chgam{\check\gamma}
\newcommand\chm{\check m}

\newcommand\htx{\widehat T^1}
\newcommand{\euler}[2]{\genfrac{ < }{ > }{0pt}{}{#1}{#2}}
\newcommand{\vv}{\ensuremath{\mathbf{V}}}
\newcommand{\vvb}{\ensuremath{(V_1,\dots,V_b)}}
\newcommand\fT{\widehat T}
\newcommand\fTxx{\fT_{x_0/x_1}}
\newcommand\Eio{\cE_{1}}
\renewcommand\ln{\log}
\newcommand\medianllbst{median-of-$(2\ell+1)$ binary search tree}
\newcommand\Next{N^{\mathrm e}}
\newcommand\Nint{N^{\mathrm i}}
\newcommand\nle[1]{n_{\le#1}}
\newcommand\nge[1]{n_{\ge#1}}
\newcommand\nz[1]{n_{#1}}
\newcommand\nzfloor[1]{n_{\floor{#1}}}

\newcommand\Ngechi[1]{N^{\chi}_{\ge#1}}
\newcommand\Ngei[1]{N^1_{\ge#1}}
\newcommand\chNgechi[1]{\check N^{\chi}_{\ge#1}}
\newcommand\chNgei[1]{\check N^1_{\ge#1}}
\newcommand\Nlechi[1]{N^{\chi}_{\le#1}}
\newcommand\Nlei[1]{N^1_{\le#1}}
\newcommand\bara{\bar x_+}
\newcommand\baram{\bar x_-}
\newcommand\barao{\bar x_0}
\newcommand\barai{\gamma_*}
\newcommand\chbaram{-\baram}
\newcommand\oichar{$0$--$1$ characteristic}
\newcommand\hx{h_*}
\newcommand\gabw{(\ga\gb)\qw} 
\newcommand\AAA{A_-}
\newcommand\AAAo{A_0}
\newcommand\AAAoo{[\AAAo,\infty)}
\newcommand\Iga{I_{\ga}} 
\newcommand\MM{^{(M)}}
\newcommand\MMo{^{(M_0)}}
\newcommand\hgaxxMM{\hgaxx{}^{(M)}}
\newcommand\oot{[0,t]}
\newcommand\xoot{(\oot)} 
\newcommand\TPL{L}
\newcommand\gamx{\gamma_-}
\newcommand\bgam{\bar\gamma} 
\newcommand\bgamx{\bgam_-}
\newcommand\xoo{[x,\infty)}
\newcommand\xxoo[1]{[#1,\infty)}
\newcommand\taul{\tau_L} 
\newcommand\iii{^{(i)}}
\newcommand\nb{\overline n}
\newcommand\no{n^0}
\newcommand\vo{V^0}
\newcommand\aaa{\tgaxx_1}
\newcommand\Sh{S_{1/2}}
\newcommand\bxi{\bar\xi}

\newcommand{\Holder}{H\"older}
\newcommand{\Polya}{P\'olya}

\newcommand{\maple}{\texttt{Maple}}

\begin{document}

\begin{abstract} 
This survey studies asymptotics of
random fringe trees and extended fringe trees in
random trees that can be constructed as family trees of a
Crump--Mode--Jagers branching process, stopped at a suitable time.
This includes random recursive trees, preferential attachment trees, 
fragmentation trees,
binary search trees and (more generally) $m$-ary search trees, as well 
as some
other classes of random trees.

We begin with general results, mainly due to Aldous (1991) and 
Jagers and Nerman (1984).
The general results are applied to fringe trees and extended fringe trees for 
several particular types of random trees, where the theory is developed in
detail. 
In particular, we consider fringe trees of $m$-ary search trees in detail;
this seems to be new.

Various applications are given,
including degree distribution, protected nodes and maximal clades
for various types of random trees. Again, we emphasise results for $m$-ary
search trees, and give for example new results on protected nodes in $m$-ary
search trees.

A separate section surveys results on height, saturation level, typical
depth and total path length,
due to Devroye (1986), Biggins (1995, 1997) and others.

This survey contains well-known basic results 
together with some additional general results 
as well as many new examples and applications for various classes of random
trees. 
\end{abstract}

\maketitle

\tableofcontents

\section{Introduction}\label{S:intro}

\citet{Aldous} introduced the concept of a random fringe subtree of a random
tree. (See \refS{Sfringe} below for definitions.)
This is a useful concept since many properties of a
tree can be formulated in terms of fringe trees, and thus results on 
the asymptotic distribution of fringe trees can
imply various other asymptotic results; 
a simple example is the degree distribution 
(considered already in \cite{Aldous})
and some other examples are given
in  \refS{Srank} (protected nodes and rank).
(See also
\citet{SJ283} 
and \citet{HolmgrenJanson1}
for some recent related applications of fringe trees.)
Moreover, \citet{Aldous} also introduced the extended fringe tree
that allows for consideration of \eg{} parents and siblings of a chosen
node; see \refS{Sclades} for some applications (\eg{}  maximal clades).

It is thus of interest to describe the asymptotic distribution of random
fringe trees and extended fringe trees
for various classes of (random) trees. 
\citet{Aldous} 
gave several examples of asymptotic fringe trees, 
including the case of random binary search trees;
he also, more briefly, gave examples of asymptotic extended fringe trees.
One of the purposes of the present paper is to extend these examples. In
particular, we describe asymptotic fringe trees and extended fringe trees
for \mst{s}
(see \refS{Smst-def} for a definition and \refS{Smst} for results).
We give some applications of these results for \mst{s} in Sections \ref{Srank}
and \ref{Sclades}. 

Our characterization uses some of the ideas in \citet{Aldous}, in particular the
reduction to results for continuous-time Crump--Mode--Jagers 
branching processes by \citet{JagersNerman1984}, \cite{NermanJagers1984}.
(The \mst{s} have earlier been studied
by similar methods by \citet{Pittel1994}; 
however there the focus was on the height of the trees and not on the fringe
trees.) 
In a sense, the results are implicit in \cite{Aldous}, 
and partly in \cite{JagersNerman1984,NermanJagers1984},
but 
the details are not completely trivial so we give a detailed explicit
treatment.

We therefore begin with a survey of fringe trees and extended fringe trees
for family trees of \CMJbp{es}, including many other examples besides the
\mst{s}. 
The general theory is described in
Sections \ref{Sfringe} and \ref{S:BP}. In \refS{Sex}, several examples are
studied in detail, in particular various versions of preferential attachment
trees, 
which earlier have been studied by these methods by \citet{OSpencer},
\citet{RudasTV} and \citet{RudasT}. 
We then specialise on  \mst{s}; explicit results for them are given in
\refS{Smst}. 
In  \refS{Smedian} we consider the random 
median-of-$(2\ell+1)$ binary search tree as yet another example.

Furthermore, 
as another novel example,
we consider in  \refS{Sfrag} 
the class of fragmentation trees; these too can be constructed using family
trees of \CMJbp{es}, but in a slightly different way from the preceding
examples. We extend  the results for the asymptotic distribution of random
(extended) fringe trees to this case too.

In Sections \ref{Srank} and  \ref{Sclades}, as mentioned above, 
we give  some applications of the results on
asymptotic fringe trees and extended fringe trees 
to protected nodes and maximal clades (and related properties). 
This serves partly to illustrate the general theory and its uses and
some results are old, but we also give a number of new results 
for \mst{s}.
In particular, we give a recursion that yields 
the asymptotic probability that a random node in an \mst{} is $k$-protected,
for general $m$ and $k$,
and a closed formula for the case $k=2$, together with
asymptotics as \mtoo{} of this probability for $k=2$.

In the main part of the paper, we consider the fringe tree or other
properties of a uniformly random node in the tree.
In Section \ref{Ssample} we consider variations for a random node
with a non-uniform distribution.
We study first restricted sampling, 
where we sample only nodes with some given property, for example a random leaf. 
For \mst{s}, 
we study also
the node containing a random key.

In Sections \ref{Sfringe}--\ref{Ssample}, we study (more or less) local
properties of the tree, that are related to (extended) fringe trees. 
Branching process methods have also
for a long time, 
beginning with \citet{Devroye1986}, 
been used to study
global properties of random trees, such as the height and other properties
related to the distance to the root from the nodes.
As a complement to the previous sections, we give 
in \refS{Sheight}
a survey of such results for the height, saturation level, profile,
typical depth and total path length.
This uses the same set-up as the preceding sections 
with random trees constructed as family trees of \CMJbp{es}, 
but the methods are different and  based on
results on branching random walks
by  Biggins \cite{Biggins76,Biggins77,Biggins95,Biggins97}.
This section is thus essentially independent of the previous sections,
except for definitions and some basic results.
The main results are well-known, but we believe that some results are new.

In this paper, we concentrate on results 
obtained by general branching  process methods,
in particular results 
on the asymptotic distribution of (extended) fringe trees 
and applications of such results.
Typical results can be expressed as convergence in  probability or almost surely
(\as) of the fraction of fringe trees that are isomorphic to some
given tree, 
see for example \eqref{tbpq}; see also \eqref{nupp}--\eqref{nupas} and
\refR{Rk}. 
Such results can be seen as a law of large numbers for fringe trees,
and
typical applications 
yield first-order results for the 
proportion or number of nodes that have a certain property
(see Sections \ref{Sex}--\ref{Ssample}).
In some special cases, 
for example for some properties of the binary search tree, much more precise
results have been derived by other methods. We give some references to such
results, but we do not attempt completeness. 

A natural next step would be to show a general central limit theorem, \ie,
asymptotic 
normality of the number of fringe trees of a given type, under suitable
conditions. This will not be attempted in the present paper, but we give some
comments and references in \refS{Sfurther}; 
in particular we note that such results have been proved, by other methods,
for some 
special cases (the binary search tree and \rrt), 
but that they do not hold in other cases (\mst{} with $m\ge27$).

The appendices contain some results
that are used in the main part of the paper.

\begin{remark}\label{RcGWt}
  In the present paper we consider random trees that are generated by
  stopping a supercritical branching process at a suitable (random) time,
for example when its size (the number of individuals) is a given number.

Note that the results are quite different from the results for fringe trees
of conditioned Galton--Watson trees, 
where we also start with a branching process but
instead of stopping it, we let it run until extinction and
condition on its total size being a given finite number, 
see \cite{Aldous,BenniesK,SJ264,SJ285}.
\end{remark}
 
\section{Some notation}\label{Snotation}

The trees considered here are \emph{rooted} and \emph{finite}, unless otherwise
indicated. (The infinite sin-trees, that arise as limits in \refS{Sfringe},
are important exceptions.) 
Furthermore, the trees are \emph{ordered}, again unless otherwise indicated;
unordered trees may be considered by giving them an arbitrary (\eg{} random)
ordering of the children of each node.

Moreover, there may be further information on the children of each node.
In a \emph{binary tree}, 
each child is labelled as left or right (with at most one child of
each type at any node); the tree is ordered, with a left child before a
right child, but also a single child is labelled left or right.
More generally, in an \emph{\mary} tree, see \refS{Smst-def},
a node has $m$ slots for children and
the children are labelled with distinct numbers in \set{1,\dots,m}; these
numbers determine the order of the children, but not conversely, since a
node may have less than $m$ children and thus only use a subset of these labels.
(In an extended \mary{} tree, each node has either $m$ children or 0, so
these labels are determined by the order and are therefore redundant.)

We write
$T_1\iso T_2$ when $T_1$ and $T_2$ are isomorphic rooted trees.
We often identify trees that are isomorphic. 
We may regard all finite rooted trees as subtrees of the infinite
Ulam--Harris tree with node set $\voo:=\bigcup_{n=0}^\infty \bbN^n$
consisting of 
all finite strings of natural numbers, where $\emptyset$ is the root and the
mother of $i_1\dotsm i_k$ is $i_1\dotsm i_{k-1}$, see \eg{} 
\cite[\S\,VI.2]{Harris} and \cite{Neveu}.

Let $|T|$ be the number of nodes in a tree $T$.

We regard the edges in a tree as directed from the root. Thus the
\emph{outdegree} $\dout(v)=\dout_T(v)$ of a node $v$ in a tree $T$ is its number
of children. 

The \emph{depth} $\dep(v)$ of a node $v$ is its distance from the root.

Given a tree $T$ and a node $v\in T$, let $T^v$ denote
the subtree rooted at $v$, \ie, the subtree consisting of $v$ and its
descendants.

If $T$ and $S$ are trees, let $\nun_S(T)$ be the number of nodes $v$ in $T$
such that  $T^v\iso S$.
Similarly,
given a property $\cP$ of nodes in a tree, let 
$\nun_{\cP}(T)$ be the number of nodes $v$
in $T$ that have the property $\cP$.
(Thus $\nun_S(T)=\nun_{\cP_S}(T)$ if $\cP_S$ is the property of $v$ that
$T^v\iso S$.)

For a random rooted tree $\cT$ and a fixed tree $S$, 
let $p_S(\cT)=\P(\cT\iso S)$. Furthermore, if $\cP$ is a property of nodes,
let $p_{\cP}(\cT)$ be the probability that the root of
$\cT$ has the  property $\cP$. 
(Note that $p_S(\cT)=p_{\cP_S}(\cT)$ with $\cP_S$ as in the preceding
paragraph,
 so the notation is consistent.)

Note that when we talk  about a property $\cP$ of nodes, it is implicit
that the property depends also on the tree containing the node, so it is
really a property of pairs $(v,T)$ with $v\in T$.
We will frequently consider properties of a node $v$ that depend only 
on $v$ and its descendants, \ie, 
on the subtree $T^v$. In this case (but not in general), we may also regard 
the property $\cP$ as a property of rooted trees: we say that a tree $T$ has
$\cP$ if the root of $T$ has $\cP$.
In this case we also use $\cP$ for the set of rooted trees that have the
property $\cP$;
thus a node $v$ in a tree $T$ has $\cP$ $\iff T^v\in\cP$. 

If $\cP$ is a property of nodes, we sometimes write $v\in\cP$ for the event
that $v$ has $\cP$.

By $\Exp(\gl)$, we mean the exponential distribution with \emph{rate}
$\gl>0$, and thus mean $1/\gl$: if $X\sim\Exp(\gl)$ then $\P(X>t)=e^{-\gl t}$.
$\Geo(p)$ denotes the geometric distribution with probability function
$p(1-p)^k$, $k\ge0$, and 
$\Gei(p)$ denotes the shifted geometric distribution with probability function
$p(1-p)^{k-1}$, $k\ge1$.

We let $\rise{x}k$ and $\fall xk$ denote the rising and falling factorials:
\begin{align}
\label{rise}
\rise{x}k&:=x(x+1)\dotsm(x+k-1)=\gG(x+k)/\gG(x),
\\
\fall xk&:=x(x-1)\dotsm(x-(k-1))=\gG(x+1)/\gG(x-k+1).   
\label{fall}
\end{align}

We let $\bbZgeo:=\set{0,1,2,\dots}$ and $\bbZleo:=\set{0,-1,-2,\dots}$.

We say that a function $f(x)$ is \emph{decreasing}
if $x<y$ implies $f(x)\ge f(y)$; note that we allow equality. 
(This is sometimes called \emph{weakly  decreasing}.) 
If $x<y$ implies $f(x)< f(y)$, we may say \emph{strictly decreasing}.
\emph{Increasing} and \emph{strictly increasing} are defined similarly.

We consider asymptotics of various random trees
when some parameter $n$ (for example the
number of nodes, or number of keys in an \mst)
tends to infinity. 
Similarly, for the continuous-time branching processes, 
we consider limits as the time $t$ tends to infinity.
As usual, \emph{\whp} (\emph{with high probability}) means with probability
tending to 1.

\section{$m$-ary search trees}\label{Smst-def}
An  \emph{$ m $-ary search tree}, where $ m\ge 2 $ is a fixed number, is
an \mary{} tree
constructed recursively from a sequence of distinct \emph{keys} (real numbers)
as follows, see \eg{} \cite{Mahmoud:Evolution} or \cite{Drmota}. 
(In the case $m=2$, we say \emph{binary search tree}.) 
The $ m $-ary search trees were first introduced in  \cite{Muntz71}.

Each node may store up to $ m-1 $ keys.  
We start with a tree containing just an empty root. 
The first $ m-1 $ keys are stored in the root.
When the $(m-1)$:th key is placed in the root, so the root becomes full,
we add $m$ new nodes, initially empty, as children of the root.
Furthermore, the $m-1$ keys in the root divide the set of real numbers into $m$
intervals $J_1,\dots,J_m$. Each further key is passed to one of the children
of the root
depending on which interval it belongs to; a key in $J_i$ is passed to the 
$i$:th child.

This construction yields the \emph{extended \mst}. Nodes containing
at least one key are called \emph{internal} and empty nodes are called
\emph{external}. Usually one eliminates all external nodes and consider the
tree consisting of the internal nodes only; this is the \emph{\mst}.

For both versions, we often wish to keep track of the number of keys in each
node, so we regard the trees as labelled trees where each node has a
label in \set{0,\dots,m-1} indicating the number of keys. (Thus external
nodes have label 0 while internal nodes have labels in \set{1,\dots,m-1}.)

We assume that the keys are \iid{} random variables with a continuous
distribution, for example $U\oi$. With a given number $n$ of keys, this 
gives a \emph{random \mst} $T_n$. 
(As is customary, we usually omit the word ``random'' for convenience.
Also, we regard $m$ as fixed, and omit it from the
notation.) 
Note that only the order of the keys matter; hence we obtain the
same random \mst{} $T_n$ also
if we instead let the keys be a uniformly random permutation
of \set{1,\dots,n}.

Note that in $T_n$ we have fixed the number of keys; not the number of nodes.
A node may contain $1,\dots,m-1$ keys, and the total number of
nodes will be random when $m\ge3$.
(The binary case $m=2$ is an exception; each internal node contains exactly
one key, 
so the number of (internal) nodes equals the number $n$ of keys, and the
number of external nodes is $n+1$.)

In an extended \mst, 
say that a node with $i\le m-2$ keys has $i+1$ \emph{gaps}, while a full
node has no gaps. It is easily seen that an extended $m$-ary search tree
with $n$ keys 
has $n+1$ gaps; the gaps correspond to the intervals of real numbers between
the keys (and $\pm\infty$), 
and a new key  
has the same probability $1/(n+1)$ of belonging to any of the gaps.
Thus the evolution of the extended
$m$-ary search tree may be described by choosing a
gap uniformly at random at each step. Equivalently, the probability that the
next key is added to a node is proportional to the number of gaps at that
node.
For the \mst{} (with only internal nodes) the same holds with minor
modifications; a full node now has one gap for each external node in the
extended version,  \ie, $m-d$ gaps if there are $d$ children, and a key
added to one of its gaps now starts a new node.

\section{Fringe trees and extended fringe trees}\label{Sfringe}

Given a (finite, rooted) tree $T$,
the \emph{random fringe tree} of $T$
is the  
random tree obtained by taking the subtree $T^v$ with
$v$ chosen uniformly at random from
the nodes of $T$; we denote the random fringe tree of $T$ by  $\Tx$.

Consider a sequence $T_n$ of (possibly random) trees
such that
the random fringe tree $T_n^*$ converges in distribution to some random tree
$\xT$:
\begin{equation}\label{fringe}
  T_n^*\dto\xT,
\end{equation}
which simply means (since the set of finite trees is countable)
\begin{equation}\label{fringet}
  \P(T_n^*\iso\ts) \to \P(\xT\iso\ts)
\end{equation}
for every finite rooted tree $\ts$.
We then say, following \citet{Aldous}, that $\xT$ (or rather its
distribution) is the asymptotic fringe distribution of $T_n$.

If the trees $T_n$ are deterministic, then \eqref{fringet}
can be written
\begin{equation}\label{nut}
 \frac{\nun_{S}(T_n)}{|T_n|}\to p_{S} (\xT), 
\end{equation}
for every tree $S$; this is  equivalent 
to the seemingly more general
\begin{equation}\label{nup}
 \frac{\nun_{\cP}(T_n)}{|T_n|}\to p_{\cP} (\xT), 
\end{equation}
for every property  $\cP$ of a node $v$ that depends only on the subtree
$T^v$, \ie, on $v$ and its descendants.

In the more general case when $T_n$ are random (which is the case we are
interested in),
\eqref{fringet} instead
can be written
\begin{equation}\label{nupt}
\E \frac{\nun_{S}(T_n)}{|T_n|}\to p_{S} (\xT)
\end{equation}
or, more generally but equivalently,
\begin{equation}\label{nupe}
\E \frac{\nun_{\cP}(T_n)}{|T_n|}\to p_{\cP} (\xT)
\end{equation}
for properties $\cP$ as above.
In interesting cases, we may typically strengthen \eqref{nupt}--\eqref{nupe}
to convergence in probability: 
\begin{equation}\label{nupp}
\frac{\nun_{\cP}(T_n)}{|T_n|}\pto p_{\cP} (\xT);
\end{equation}
\citet[Proposition 7]{Aldous} gives
a general criterion for this (the distribution of $\xT$
is extremal in the set of fringe distributions), but we will instead prove
\eqref{nupp} directly in the cases considered here; moreover,
we will in our cases prove convergence almost surely:
\begin{equation}\label{nupas}
\frac{\nun_{\cP}(T_n)}{|T_n|}\asto p_{\cP} (\xT).
\end{equation}

\begin{remark}\label{Rk}
Note that
\begin{equation}\label{rks}
\frac{\nun_{\ts}(T_n)}{|T_n|}=  \P\bigpar{T_n^*\iso \ts\mid T_n}
\end{equation}
and, more generally, for a property $\cP$ as above,
\begin{equation}\label{rkp}
\frac{\nun_{\cP}(T_n)}{|T_n|}=  \P\bigpar{T_n^*\in \cP\mid T_n}.
\end{equation}
It follows from \eqref{rks} that \eqref{nupp} and \eqref{nupas} 
(for all properties $\cP$ considered there)
are equivalent to  conditional versions of \eqref{fringe}:
\begin{align}
  \label{fringeCp}
\cL\bigpar{T_n^*\mid T_n}
&\pto\cL\bigpar{\xT}
\intertext{and}
\label{fringeCas}
\cL\bigpar{T_n^*\mid T_n}
&\asto\cL\bigpar{\xT},
\end{align}
respectively,
with convergence in probability or \as{} of the conditional distribution, in
the space of probability distributions on trees.
(Note that any such
property $\cP$ corresponds to a set of finite rooted trees $T$, and conversely.)

Results such as \eqref{fringeCp} and \eqref{fringeCas}, where we fix a
realization $T_n$ of a 
random tree and then study the distribution of its fringe tree (or something
else), as a random variable depending on $T_n$, are usually called
\emph{quenched},
while results such as \eqref{fringe}, where we consider
the random fringe tree of a random tree as a combined random event,
are called
\emph{annealed}. See further \eg{} \cite{SJ283} and \cite{SJ285}.
\end{remark}

\subsection{Extended fringe trees}\label{Sextended}

The fringe tree $\Tx$ considers only the descendants of a random node. 
\citet{Aldous} introduced also the \emph{extended fringe trees} that include
the nearest ancestors and other close relatives. 
If $k\ge0$ and $v\in T$ with $\dep(v)\ge k$, let $v\mm{k}$ be the
ancestor of $v$ that is $k$ generations earlier (\ie, with
$\dep(v\mm{k})=\dep(v)-k$),
and let $\Tvk$ be the subtree rooted at $v\mm{k}$, with the node $v$ marked.
(Or, equivalently, with the path from the root $v\mm{k}$ to $v$ marked.)
Thus $\Tvk$ is a rooted tree with a distinguished node of depth $k$.
(Note that $\Tyy{v}0=T^v$.)

We define the random extended fringe tree $\Txk$ as $\Tvk$ for a uniformly
random node $v\in T$; this is really not defined when $\dep(v)<k$, but we may
define $\Tvk$ in this case too by some supplementary definition, 
for example as a path of length
$k-\dep(v)$ with a copy of $T$ attached, with $v$ marked. We are only
interested in asymptotics of the random extended fringe trees for sequences
of trees $T$ such that 
\begin{equation}\label{doo}
\dep(v)\pto\infty   
\end{equation}
for a random node $v$, \ie, 
$\P(\dep(v)<k)\to0$ for every fixed $k$, and thus each $\Txk$ is
well-defined \whp, 
and then the supplementary definition does not matter.

\citet{Aldous} showed that if $T_n$ is a sequence of (possibly random) trees
such that \eqref{doo} holds and
an asymptotic fringe distribution
exists, \ie, \eqref{fringe} holds, 
then, more generally, 
each $\Txk_n$ converges in distribution to some random tree $\xT\y{k}$ with
a distinguished node $o$ of depth $k$.
Note that the trees $\Tvk_n$ are consistent in an obvious way, with
$\Tyy{v}{(k-1)}_n$ a subtree of $\Tvk_n$, and thus the same holds for 
the limits $\xtk$ (after a suitable coupling).
Hence it is possible to regard the trees $\xtk$ as subtrees
of a (random) infinite tree $\xToo$ with 
a distinguished node $o$ and
an infinite line $o, o\mm1,o\mm2,\dots$ of ancestors of  $o$,
such that
$\xtk=\xToo^{o,-k}=\xToo^{o\mm{k}}$. 
Furthermore, every node in $\xToo$ has a finite number of descendants; thus
there are no other infinite paths from $o$. (\citet{Aldous} calls such a
tree a \emph{sin-tree}, for single infinite path.)
We may then say that the extended fringe trees converge to 
the random sin-tree $\xToo$, in the sense that
$\Txk_n\dto\xToo^{o\mm{k}}$ for each $k$, or, equivalently, using the product
topology on the set of sequences of (finite) trees,
\begin{equation}\label{extended}
\bigpar{\Txk_n}_{k=0}^\infty
\dto\bigpar{\xToo^{o\mm{k}}}_{k=0}^\infty.
\end{equation}

For a random sin-tree $\xToo$ and a property $\cP$ of nodes,
let $p_{\cP}(\xToo)$ be the probability that the distinguished node $o$
has the  property $\cP$.
Then, \cf{} \eqref{nupe} (which is the  case $k=0$), 
\eqref{extended} implies, and is equivalent to,
\begin{equation}\label{xnup}
\E \frac{\nun_{\cP}(T_n)}{|T_n|}\to p_{\cP} (\xToo), 
\end{equation}
for every property  $\cP$ that depends only on $\Tvk$ for some $k$, 
\ie, on $v$ and its descendants and the
descendants of its ancestors at most a fixed number of generations back.
Again, we may typically strengthen \eqref{xnup} to convergence in
probability, and in our cases we shall prove convergence \as:
\begin{equation}\label{xnupas}
\frac{\nun_{\cP}(T_n)}{|T_n|}\asto p_{\cP} (\xToo).
\end{equation}

By standard truncation arguments, it may be possible to extend \eqref{xnup}
or \eqref{xnupas} also
to some more general properties $\cP$, depending on an unlimited number of
ancestors, see Sections \ref{SStrunc} and  \ref{Sclades} for some examples.

\begin{remark}
Similarly to \refR{Rk}, \eqref{xnupas}
is equivalent to a conditional version of \eqref{extended}:
\begin{equation}\label{extendedCas}
\cL\bigpar{\bigpar{\Txk_n}_{k=0}^\infty\mid T_n}
\asto\cL\bigpar{\bigpar{\xToo^{o\mm{k}}}_{k=0}^\infty}.
\end{equation}
\end{remark}

\section{Family trees of general branching processes}\label{S:BP}

A \emph{Crump--Mode--Jagers} process is a general branching process defined
as follows, see \eg{} \cite{Jagers} for further details and for basic facts
used below.

The branching process starts 
with a single individual born at time 0. 
This individual has a random number $N$ of children, born at random times
$(\xi_i)_{i=1}^N$; here $0\le N\le\infty$, and we assume $0\le\xi_1\le
\xi_2\le\dotsm$. 
It is convenient to describe the birth times $\set{\xi_i}_1^N$ as a 
point process $\Xi$ on $[0,\infty)$.
Every child that is
born evolves in the same way, \ie, every individual $x$ has its own copy 
$\Xi_x$ of $\Xi$ (where now $\xi_i$ means the age of the mother when child
$i$ is born); 
these copies are assumed to be independent and identically distributed.
Denote the time an individual $x$ is born by $\gs_x$.

Recall that 
formally a point process $\Xi$ is best defined as an integer-valued random
measure, where $\Xi(A)$ is the number of points in a set $A$, 
see \eg{}  \cite{Kallenberg}. In our setting,
$\Xi=\sum_{i=1}^N\gd_{\xi_i}$, where $\gd_t$ is a point mass 
(Dirac measure)
at $t$.
In particular, 
we have $N=\Xi\xpar{[0,\infty)}$. 
Furthermore, for $j\le N$, $\Xi([0,t])\ge j$ if and only if $\xi_j\le t$.
We let $\mu:=\E\Xi$ denote the intensity measure
 of $\Xi$, and write $\mu(t):=\mu\xpar{[0,t]}=\E\Xi\xpar{[0,t]}$.
In particular, $\mu(\infty)=\E N$.

Usually one also assumes that each individual has a random lifetime
$\gl\in[0,\infty]$; for our purposes this plays no role, so we ignore
it. (Formally, we may assume that $\gl=\infty$.) There may also be other
random variables associated to the individuals. Formally, we give each
possible individual $x$ its own copy $(\gO_x,\cF_x,\mu_x)$
of some probability space $(\gO,\cF,\mu)$ on which there are defined
some given functions defining $N$, $\xi_i$ (and thus $\Xi$),
and possibly other random variables describing the life history
such as the marks $\nu_i$ 
or label $\ell(t)$ in Remarks \ref{Rorder} and \ref{Rlabel}
below;
the branching process then is defined on the
product $\prod_x (\gO_x,\cF_x,\mu_x)$ of these probability spaces.
(The individuals may be labelled in a natural way by strings 
in $\voo:=\bigcup_{n=0}^\infty \bbN^n$; hence the set of individuals that
are realized in the branching process is a random subset of $\voo$, and we
may extend the product over $x\in\voo$.)

Let $Z_t$ be the number of individuals  at time $t\ge0$; since we
assume no deaths, this equals the number of individuals born in $[0,t]$.
(We follow standard custom and let all processes be right-continuous; thus
an individual born at $t$ exists at $t$ and is included.)
We say that the process is \emph{finite} 
(or \emph{dies out}) 
if $Z_\infty<\infty$, \ie, only a finite number of individuals are ever born.

Let $\cT_\infty$ be the \emph{family tree} of the branching process. This
is a (generally infinite) tree
obtained from the branching process by ignoring the time structure;
in other words, it has the individuals as nodes, with the
initial individual as the root, and the children of a node
in the tree are the same as the children in the branching process.
Let $\cT_t$ be the subtree consisting of all individuals born up to time $t$.
Note that the number of nodes $|\cT_t|=Z_t$. (We are mainly interested in
cases where $Z_t<\infty$ for every finite $t$, but $Z_\infty=\infty$.)

\begin{remark}\label{Rorder}
  This defines the family tree $\cT_t$ as an unordered tree. Sometimes
we want an ordered
  tree, so we have to add an ordering of the children of each individual.
This can be done by taking the children in order of birth 
(which is the standard custom), 
but in our examples we rather want a random order. In general, we 
can obtain ordered family trees by assuming that each individual has a
marked point 
process $\Xi^*$ (augmenting the plain $\Xi$ above),
where each point $\xi_i$ has a mark $\nu_i\in\set{1,\dots,i}$ telling at which
position the new child is inserted among the existing ones.
(This includes both the birth order case, with $\nu_i=i$, and the random
order case, with $\nu_i$ uniform and independent of everything else.)

For the $m$-ary search trees in \refS{Smst}, we want further
information; this is obtained by instead giving each of the $m$ children a
distinct mark $\nu_i\in\set{1,\dots,m}$ telling the position of the child
among all (existing and future) children. (Equivalently, we may equip each
individual with a random permutation of $\set{1,\dots,m}$ 
giving the order of birth of the children.)
\end{remark}

\begin{remark}\label{Rlabel}
We may also have labels on the nodes of $\cT_t$; this is important for our
application to \mst{s}, since they have nodes labelled with the number of
keys, see \refS{Smst-def}. In general, we may assume that each individual 
has a label given by some random function $\ell(t)$ of its age. We assume
that the set of possible labels is countable (with the discrete topology); we
may assume that the labels are integers. We also assume that the function 
$\ell(t)\in D\ooo$; thus $\ell(t)$ is constant on some intervals
$[t_i,t_{i+1})$. 
(As everything else in the branching process, the label may depend on $\Xi$
and other properties of the same individual, but not on other individuals,
and they have the same distribution for all individuals;
this is also a consequence of the formalism with probability spaces
$(\gO_x,\cF_x,\mu_x)$ above.)
\end{remark}

A \emph{characteristic} of an individual, 
see \eg{} \cite{Jagers,JagersNerman1984,Nerman,NermanJagers1984},
is a random
function $\gf(t)$ of the age $t\ge0$;
we assume that $\gf(t)\ge0$ and that 
$\gf$ belongs to the space $D\ooo$ of right-continuous functions with left
limits. 
(Note that we consider only $t\ge0$.
We may extend $\gf$ to $(-\infty,\infty)$ by setting $\gf(t)=0$ for $t<0$.)
We assume that each individual has its own copy $\gf_x$,
and we at first for simplicity assume 
that the pairs $(\Xi_x,\gf_x)$ for all individuals are independent
and identically distributed; this assumption can (and will) be relaxed, see
\refR{Rnerman7} below.

Given a characteristic $\gf$,
let 
\begin{equation}\label{zgf}
\zgf_t:=\sum_{x:\gs_x\le t} \gf_x(t-\gs_x)  
\end{equation}
be the total characteristic
at time 
$t$ of all individuals 
that have been born so far.
(Recall that $x$ is born at time $\gs_x$, and thus has age $t-\gs_x$ at time
$t$.) 

The random tree $\cT_t$ has a random size. We are usually interested in random
trees 
with a given number of nodes, or trees where something else is given,
for example the number of keys in an \mst{}. 
We can obtain such random trees by stopping the
branching process as follows. 
Fix a characteristic $\psi(t)$, which we shall call
\emph{weight}, 
and let 
$\tau(n):=\inf\set{t:\zpsi_t\ge n}$, \ie, the first time the total weight is
at least $n$. (As usual, we define $\inf\emptyset=\infty$.)
We  exclude the trivial case when
$\psi(t)=0$ for all $t\ge0$ \as{} (which would give $\tau(n)=\infty$ \as).
 Define $T_n:=\cT_{\tau(n)}$, the family tree at the time
the total weight reaches $n$ (provided this ever happens).

Random trees $T_n$ defined in this way, for some \CMJ{} branching process and
some weight $\psi(t)$, are the focus of the present paper.
We shall always denote the weight by $\psi$  and the random tree, stopped as
above, 
by $T_n$ 
omitting $\psi$ from the notation for simplicity. 
(In all our examples, $\psi$ is integer-valued, so it is natural to let $n$
be an integer. This is not necessary, however, and all our results are valid
for arbitrary real $n\to\infty$.)

\begin{example}\label{E1}
If $\psi(t)=1$, $t\ge0$, then $\zpsi_t=Z_t$, and $T_n$ is the
family tree of the branching process stopped when there are $n$ nodes or
more; if the birth times have continuous distributions
and there are no twins,
then \as{} no two
nodes are born simultaneously, and thus we stop when there are exactly
$n$ nodes, so $|T_n|=n$. 
(This weight is used in all examples in \refS{Sex}, but not for the \mst{s}
in \refS{Smst}.)
\end{example}

We define the Laplace transform of a function $f$ on $\ooo$ by
\begin{equation}\label{Lf}
  \widehat f(\gth) = \gth\intoo e^{-\gth t}f(t) \dd t,
\qquad \gth>0,
\end{equation}
and the Laplace transform of a measure $m$ on $\ooo$ by
\begin{equation}\label{Lm}
  \widehat m(\gth) = \intoo e^{-\gth t}m(\ddx t), \qquad -\infty<\gth<\infty.
\end{equation}
(Note that there is a factor $\gth$ in \eqref{Lf} but not in \eqref{Lm}.
A justification of this difference is that a measure $m$ has the same
Laplace transform $\widehat m$ as the function $m(t):=m([0,t])$, 
as is easily verified by an integration by parts, or by Fubini's theorem for
the integral $\iint_{s\le t}\gth e^{-\gth t} m(\ddx s)$.)

Some standing assumptions in this paper are:
\begin{xenumerate}
\renewcommand{\theenumi}{{\upshape{(A\arabic{enumi})}}}
\renewcommand{\labelenumi}{\theenumi}
\item \label{BPfirst}
$\mu\set0=\E\Xi\set0<1$. 
(This rules out a rather trivial case with explosions
already at the start.
In all our examples, $\mu\set0=0$.)
\item \label{BPnonlattice}
$\mu$ is not concentrated on any lattice $h\bbZ$, $h>0$.
(The results extend to the lattice case with suitable modifications, but we
  ignore it.) 
\item \label{BPsuper}
$\E N>1$. 
(This is known as  the \emph{supercritical} case.)
For simplicity, we further assume that $N\ge1$ a.s., but see \refR{Rsuper}.
(In this case, every individual has at least one child, so the process never
dies out and $Z_\infty=\infty$.)
\item \label{BPmalthus}
There exists a real number $\ga$ (the Malthusian parameter) such that
$\hmu(\ga)=1$, \ie, 
\begin{equation}\label{malthus}
\intoo e^{-\ga t}\mu(\ddx t) =1.  
\end{equation}
(By \ref{BPsuper}, $\ga>0$.)

\item \label{BPmub}
$\hmu(\gth)<\infty$ for some $\gth<\ga$.

\item \label{BP6.2}\label{BPlast}
\xdef\bplast{A\arabic{enumi}}
The random variable $\sup_t \bigpar{e^{-\gth t}\gf(t)}$ has finite
expectation for 
some $\gth<\ga$.
\end{xenumerate}

\newcommand\refBPlastpsi{{\upshape(\bplast$\psi$)}}

\citet[Theorem 6.3]{Nerman} 
(see also \citet[Section 6.10]{Jagers} for related results)
shows that under the conditions \refBPphi, 
as \ttoo,
\begin{equation}\label{olle}
\frac{\zgf_t}{Z_t} \asto m_\gf:=\E \hat\gf(\ga) = \widehat{\E\gf}(\ga)
=\ga\intoo e^{-\ga t}\E \gf(t)\dd t.
\end{equation}
The \rhs{} of \eqref{olle} is finite by \ref{BP6.2}. 
Thus, if we  exclude the trivial case when
$\gf(t)=0$ for all $t\ge0$ \as, $0<m_\gf<\infty$.

Note that \ref{BPfirst}--\ref{BPmub} are conditions on the branching
process, while \ref{BPlast} is a condition on the characteristic $\phi$ (and
$\ga$), and thus is relevant only we consider some $\phi$. When discussing
trees $T_n$ defined by stopping using a weight 
$\psi$ as above, we sometimes want \ref{BPlast} to hold for $\psi$; 
we denote
this version of the condition by \refBPlastpsi.
(However, for most of our results, \refBPlastpsi{} is not required.
In any case, in \refE{E1} and in all our examples in Sections \ref{Sex} and
\ref{Smst}, 
$\psi(t)$ is bounded, so \refBPlastpsi{} holds trivially.)

\begin{remark}
As a consequence of \ref{BPmalthus}, 
  $\mu(t)<\infty$ for every $t<\infty$. 
(However, $\mu(\infty)=\E N$ may be infinite.)
It is a standard result that this implies that $Z_t$ and $\E Z_t$ are finite
for every 
$t<\infty$.
\end{remark}

\begin{remark}\label{Rsuper}
We do not really need the assumption $N\ge1$ in \ref{BPsuper}; it suffices that 
$\E N>1$. 
In this case, the extinction probability $q:=\P(Z_\infty<\infty)<1$, so
there is a positive probability that the process is infinite,
and \eqref{olle}
and the results below hold conditioned on the event $Z_\infty=\infty$.
(This is the standard setting in
\cite{Nerman,JagersNerman1984,NermanJagers1984}.)
\end{remark}

\begin{remark}
\label{RBPmalthusage}
By \eqref{malthus}, $e^{-\ga t}\mu(\ddx t)$ is a probability measure on $\ooo$.
See \refR{Rage} for an interpretation of this distribution.
\end{remark}

\begin{remark}
 \label{RBPxib}
By the definitions,
\begin{equation}\label{xib}
  \Xib:=\intoo e^{-\gth t}\Xi(\ddx t)=\sum_{i=1}^N e^{-\gth\xi_i}.
\end{equation}
Since $\mu=\E\Xi$, we have $\hmu(\gth)=\E\hXi(\gth)$ and
\eqref{xib} yields
\begin{equation}\label{hmub}
  \hmu(\gth)=\E\sum_{i=1}^N e^{-\gth\xi_i}.
\end{equation}
Thus,
\eqref{malthus} can be written $\E\hXi(\ga)=1$, or
\begin{equation}\label{malthustau}
  \E\sum_{i=1}^N e^{-\ga\xi_i}=1.  
\end{equation}
Similarly,
\ref{BPmub} says that
the random variable 
$\Xib$
has finite expectation
$\E\Xib<\infty$ for some $\gth<\ga$.
\end{remark}

\begin{remark}
The conditions \ref{BPmub} and \ref{BP6.2} may be weakened somewhat
if we further assume $\E \Xia\log^+\Xia<\infty$, see
\cite[Conditions 5.1 and 5.2]{Nerman},
but the versions above are
sufficient for our applications.   
\end{remark}

\begin{remark}\label{Rmulti}
  The results can be extended 
to multi-type branching processes, see \citet{JagersNerman1996}.
\end{remark}

\begin{remark}
  \label{Rnerman7}
We have for simplicity assumed above that the characteristic
$\gf_x(t)$ 
associated to an individual $x$ is independent of the life histories of all
other individuals. As shown by \cite[Section 7]{Nerman}, the results 
extend to characteristics  $\gf_x(t)$ that may depend also on the
descendants of $x$; we may let $\gf_0(t)$ be any non-negative random
function that depends on the entire branching process
(and belongs to $D\ooo$ and satisfies \ref{BP6.2}), and
then define $\gf_x(t)$ as $\gf_0(t)$ evaluated for the branching process
consisting of $x$ and its descendants (shifting the origin of time to the
birth of $x$). This will be important below.
\end{remark}

\begin{remark}\label{RW}
\citet{Nerman} showed also that,
under the assumptions \refBP{} above, 
there exists a random
variable $W$ such that, as \ttoo,
\begin{equation}\label{olle1}
  e^{-\ga t}Z_t \asto W
\end{equation}
and, more generally, for every $\gf$ satisfying \ref{BP6.2},
\begin{equation}\label{olle2}
  e^{-\ga t}\zgf_t \asto \mphi W.
\end{equation}
If furthermore
\begin{equation}\label{xlogx}
\E \Xia\log^+\Xia<\infty, 
\end{equation}
which is the case in our applications, then
$W>0$ \as{} (on the event \set{Z_\infty=\infty}), see also \cite{Doney}, 
so $Z_t$ and $\zgf_t$ grow exactly at
rate $e^{\ga t}$; moreover, \eqref{olle} then follows from
\eqref{olle1}--\eqref{olle2}.
However, if \eqref{xlogx} fails, then $W=0$ \as, so $\zgf_t=o\xpar{e^{\ga t}}$
\as{}; nevertheless, also in this case
\begin{equation}\label{cam}
e^{-\gth t}\zgf_t\asto\infty  
\end{equation}
for every $\gth<\ga$,
as follows easily by truncating the offspring distribution $\Xi$ to at most
$M$ children, for some large $M$, and applying \eqref{olle2} to the
truncated process.
It follows easily from \eqref{olle2} and \eqref{cam} that if \refBPphi{}
hold, then
\begin{equation}
  \label{chu}
\log \zgf_t/t \asto\ga.
\end{equation}
(See also \cite[Theorem 2.1]{Biggins97} for this result under a slightly
weaker condition.)

Furthermore, 
the expectation $\E\zgf_t$ always grows as $e^{\ga t}$, even when \eqref{xlogx}
fails so $W=0$ a.s.; more precisely
\cite[Proposition 2.1]{Nerman},
with $\gb>0$ is given by \eqref{el} below,
\begin{equation}\label{pyret}
\E \zgf_t \sim(\ga\gb)\qw m_\gf e^{\ga t}.
\end{equation}
If \eqref{xlogx} holds, then 
\begin{equation}
  \label{EW}
\E W=(\ga\gb)\qw, 
\end{equation}
so \eqref{pyret}
says that the expectation converges in \eqref{olle1}. (However, as just said,
\eqref{pyret} holds also when \eqref{olle1} holds with $W=0$.)
\end{remark}

Our main results are now simple consequences of the
general results by \citet{Nerman} above.
We consider the random trees $T_n$ defined by stopping the branching process
according to some
fixed weight $\psi$ as above.
We begin by noting that $\tau(n)<\infty$ so that
$T_n$ really is well-defined.
(See \eg{} \cite{Devroye1998}.)

\begin{theorem}\label{Tex}
  \begin{thmenumerate}
  \item \label{tex1}
Under the assumptions \refBP, and for any weight $\psi$,
$\zpsi_t\asto\infty$  as \ttoo; thus \as{} $\tau(n)<\infty$ for
every $n\ge0$ and $T_n:=\cT_{\tau(n)}$ is a well-defined finite random tree.
Furthermore, $\tau(n)\asto\infty$ as \ntoo.
\item \label{tex2}
If moreover \refBPlastpsi{}  holds, then
\begin{equation}\label{tex2e}
\frac{|T_n|}{n}\asto \frac{1}{m_\psi}\in(0,\infty)
\end{equation}
and
\begin{equation}
  \label{tex2tau}
\frac{\tau(n)}{\log n} \asto \frac{1}{\ga}.
\end{equation}
  \end{thmenumerate}
\end{theorem}

\begin{proof}
\pfitemref{tex1}
If \refBPlastpsi{} holds, then \eqref{cam} (with $\gth=0$) shows that
$\zpsi_t\asto\infty$. 
In general, apply \eqref{cam} to the truncated weight $\psi_1(t):=\psi(t)\land1$
and deduce
$\zpsi_t\ge Z_t^{\psi_1}\asto\infty$.  
Hence, $\tau(n)<\infty$ for every $n$.
Finally, 
for every finite $t$,
$\zpsi_t<\infty$ \as, and thus $\tau(n)>t$ for large $n$; 
hence $\tau(n)\to\infty$.

\pfitemref{tex2}
By the definition of $\tau(n)$ (and right-continuity of the process and
$\psi$), $\zpsi_{\tau(n)}\ge n$, while
$\zpsi_{\tau(n)-\gd}< n$ for any $\gd>0$. 
Fix $\gd>0$ and define the characteristic
$\psi_\gd(t):=\psi(t-\gd)$, with $\psi_\gd(t):=0$ for $t<\gd$.
Then $Z^{\psi_\gd}_t=\zpsi_{t-\gd}$. Consequently,
\begin{equation}\label{ini}
  Z^{\psi_\gd}_{\tau(n)}=\zpsi_{\tau(n)-\gd}
<n
\le 
Z^{\psi}_{\tau(n)}.
\end{equation}
Recall also that $|T_n|=|\cT_{\tau(n)}|=Z_{\tau(n)}$.

We have assumed \ref{BP6.2} for $\psi$, and it follows that it holds for
$\psi_\gd$ too. 
Consequently, \eqref{olle} applies to both $\psi$ and $\psi_\gd$, which by
\eqref{ini} implies
\begin{align}\label{fu1}
\limsup_\ntoo \frac{n}{|T_n|}
&\le
\limsup_\ntoo \frac{\zpsi_{\tau(n)}}{Z_{\tau(n)}}
=\mpsi,
\\
\liminf_\ntoo \frac{n}{|T_n|}
&\ge
\liminf_\ntoo \frac{Z^{\psi_\gd}_{\tau(n)}}{Z_{\tau(n)}}
=m_{\psi_\gd}.\label{fu2}
\end{align}
Furthermore, 
\begin{equation}\label{fu3}
  \begin{split}
m_{\psi_\gd}
&=
\ga\intoo e^{-\ga t}\E \psi_\gd(t)\dd t    
=
\ga\int_\gd^\infty e^{-\ga t}\E \psi(t-\gd)\dd t    
\\&
=
\ga\int_0^\infty e^{-\ga (u+\gd)}\E \psi(u)\dd u
=e^{-\ga \gd}\mpsi.    
  \end{split}
\raisetag{\baselineskip}
\end{equation}
Since $\gd>0$ is arbitrary, the result \eqref{tex2e} follows from
\eqref{fu1}--\eqref{fu2} and \eqref{fu3}
by letting $\gd\to0$.

Similarly, \eqref{chu} applies to both $\psi$ and $\psi_\gd$, which by
\eqref{ini} yields $\log n/\tau(n)\asto\ga$, which is \eqref{tex2tau}.
\end{proof}

\begin{remark}
Note that \eqref{tex2e} does \emph{not} hold in the lattice case (in this
paper excluded
by \ref{BPnonlattice}), since then the population and $\zpsi_t$ grow
in discrete steps with asymptotically a fixed factor $>1$ each time.
\end{remark}

We next study the fringe tree $T_n^*$.
Note that the following theorem (and its proof)
applies both if we consider $\cT_t$ 
as an unordered tree and if we consider
it as an ordered (or \mary) tree as in \refR{Rorder}; in the latter case 
$T_n$  
and 
the fringe tree $T_n^*$ are random ordered (or \mary) trees, and $T$ below
should be an ordered (or \mary) tree. We may also have labels on the nodes,
defined by some random function $\ell(t)$ as in \refR{Rlabel}; then $T$
should be a tree with (arbitrary) labels on the nodes.

Recall from \refS{Snotation} that a property 
of a node $v$ that depends only on $v$ and its
descendants may also be regarded as a property of rooted trees (and conversely).

\begin{theorem}[Jagers, Nerman, Aldous] \label{TBP}
  Under the assumptions \refBP, 
the following hold:
\begin{romenumerate}
\item\label{TBPa}
 (Annealed version.)
The random fringe tree
$T_n^*$ converges in distribution as \ntoo{} to the random tree 
$\friT:=\cT_\taux$, where $\taux\sim\Exp(\ga)$ is a random time, independent
  of the   branching process.
\item \label{TBPq}
(Quenched version.)
For every finite tree $T$, 
as \ntoo,
\begin{align}\label{tbpq}
  \P\bigpar{T_n^*\iso T\mid T_n} &= \frac{\nun_T(T_n)}{|T_n|}
\asto \P(\friT\iso T).
\intertext{More generally,
for every property  $\cP$ of a node $v$ that depends only 
on $v$ and its descendants,}
\label{tbpp}
 \P\bigpar{T_n^*\in \cP\mid T_n} 
&=
 \frac{\nun_{\cP}(T_n)}{|T_n|}\asto p_{\cP} (\friT).
\end{align}
Furthermore, for a property of this type,
\begin{equation}
  \label{ak}
p_{\cP}(\friT)=\intoo \ga e^{-\ga t} p_{\cP}(\ctt)\dd t
=\mphi
=\E\hphi(\ga),
\end{equation}
where $\phi(t)$ is the characteristic $\ett{\ctt\in\cP}$.
\end{romenumerate}
\end{theorem}

More precisely, the characteristic $\phi$ in \eqref{ak}
is defined as in
\refR{Rnerman7} with $\phi_0(t):=\ett{\ctt\in\cP}$.

\begin{proof}
This is a special case of the main results in
\citet{JagersNerman1984} and \cite{NermanJagers1984}, and is one of the main
examples in \citet{Aldous}, but we give the simple proof for completeness
and in our setting.

Note first that \eqref{tbpq} is a special case of \eqref{tbpp}, with
$\cP=\cP_T$, see \refS{Snotation}, and thus $\phi(t)=\ett{\ctt\iso T}$.
(Some readers might prefer to consider this case first.)

The function $\phi(t)$ is clearly a \setoi-valued random function in
$D\ooo$,
so $\phi$ is a characteristic.
The assumption \ref{BP6.2} holds trivially since $\phi$ is
bounded.
Furthermore, \eqref{ak} is a consequence of the definition
$\friT:=\cT_{\taux}$,
the fact that $\tau\sim\Exp(\ga)$ has the density function $\ga e^{-\ga t}$, 
$\E\phi(t)=\P(\ctt\in\cP)=p_{\cP}(\ctt)$,
and \eqref{olle}.

Moreover, the characteristic $\phi_x(t-\gs_x)$ of $x$
at time $t$ is the indicator $\ett{\ctt^x\in\cP}$ 
that the subtree $\cT_{t}^x$ of $\cT_t$ rooted
at $x$ satisfies $\cP$.
Thus, the total characteristic $\zgf_t$ is
the number of nodes $v\in\ctt$ such that $\ctt^v\in\cP$,
which by definition holds if and only if $v$ has the property $\cP$;
hence, $\zgf_t=n_{\cP}(\cT_t)$.
Consequently, 
\eqref{olle} yields
\begin{equation}\label{london}
  \frac{n_\cP(\cT_t)}{|\cT_t|}
=
\frac{\zgf_t}{Z_t}
\asto \mphi.
\end{equation}
By \refT{Tex}, we also have (\as) $\tau(n)<\infty$ for every $n$
and $\tau(n)\to\infty$ as \ntoo; thus \eqref{london} implies,
as \ntoo,
\begin{equation}\label{paddington}
  \frac{n_\cP(T_n)}{|T_n|}
=
  \frac{n_\cP(\cT_{\tau(n)})}{|\cT_{\tau(n)}|}
\asto \mphi.
\end{equation}
The result \eqref{tbpq} follows from \eqref{paddington} and \eqref{ak}.
As said above, \eqref{tbpq} is a special case, and
the annealed version \ref{TBPa} follows by taking the expectation in
\eqref{tbpq}, yielding (by dominated convergence) 
$\cP(\Tx_n\iso T)\to\P(\friT\iso T)$ for every fixed tree $T$.
(Recall that there is only a countable set of finite trees $T$, so this
shows convergence in distribution. Alternatively, one can take the
expectation of \eqref{tbpp}.)
\end{proof}

\begin{remark}\label{RC}
As said above, \eqref{tbpq} is a special case of \eqref{tbpp}. Conversely,
again because  there is only a countable set of finite trees $T$,
\eqref{tbpq} is equivalent to the \as{} convergence 
of the distributions in \eqref{fringeCas},
and thus to \eqref{tbpp}, \cf{} \refR{Rk}.
(In general, for distributions on a countable sample space, convergence of
the individual point probabilities is equivalent to convergence in total
variation \cite[Theorem 5.6.4]{Gut}.)
Hence, \eqref{tbpq} and \eqref{tbpp} are equivalent. (We state both versions
for convenience in later applications.)
\end{remark}

\begin{remark}\label{RX}
We have stated the result \eqref{tbpq}
for the stopped trees $T_n$, but proved it 
by proving the corresponding result for
  the full branching process, see \eqref{london} and \eqref{paddington}.
In fact, the two types of results are equivalent; by choosing the weight
$\psi=1$ as in \refE{E1}, the trees $\cT_t$ run through the same (countable)
set of trees as $\ttoo$ as $T_n$ does as $\ntoo$; hence \eqref{london} and
\eqref{paddington} are equivalent. The same holds for 
\eqref{tbpp} and for \eqref{tbpxq} and \eqref{tbpxq+} in 
Theorems \ref{TBPX}
and \ref{TQ} below, where again we state the results for $T_n$, in view of
our applications in later sections, but the results also hold for $\cT_t$.
\end{remark}

\begin{remark}\label{Rsame}
  Note that the asymptotics in \refT{TBP} do not depend on the choice of
  weight $\psi$; any weight gives the same asymptotic fringe tree
  distribution.
Of course, this is an immediate consequence of the proof using
\eqref{paddington} and \eqref{london}, see also \refR{RX}.
Note that for this proof, it is essential that we consider convergence
almost surely (and not, \eg, in probability).
\end{remark}

\begin{remark}\label{Rnumber}
In cases when $|T_n|$ is  random, 
it is often of interest to study the number
$\nun_{\cP}(T_n)$ rather than the fraction $\nun_{\cP}(T_n)/|T_n|$ in
\eqref{tbpp}. Assuming \refBPlastpsi{}, we can combine \eqref{tbpp}
and \eqref{tex2e} and obtain
\begin{equation}
  \label{tbpp2}
\frac{\nun_{\cP}(T_n)}{n}\asto \frac{p_{\cP} (\friT)}{m_\psi}.
\end{equation}
\end{remark}

\begin{remark}
  \label{Rmean}
Since the fractions in \eqref{tbpq} and \eqref{tbpp} are bounded by 1, the
\as{} convergence above  immediately yields also convergence of the
expectation by the dominated convergence theorem; thus
\begin{equation}
  \label{etbpp}
\E \frac{\nun_{\cP}(T_n)}{|T_n|}\to p_{\cP} (\friT).
\end{equation}
This is particularly nice in the common case when the weight $\psi(t)=1$, so
$|T_n|=n$ deterministically;
then \eqref{etbpp} can be written
\begin{equation}
  \label{etbpp2}
\E \nun_{\cP}(T_n)= p_{\cP} (\friT) n+o(n).
\end{equation}
For other weights $\psi$, we can (assuming \refBPlastpsi) 
use \eqref{tbpp2}.
If we furthermore have a deterministic bound $|T_n|\le Cn$ for some constant
$C$ (which, for example, is the case for the \mst{s} in Sections \ref{SSemst}
and \ref{SSmst}), 
then dominated convergence applies again and yields
\begin{equation}
  \label{etbpp3}
\E\nun_{\cP}(T_n)= \frac{p_{\cP} (\friT)}{m_\psi}n+o(n).
\end{equation}
\end{remark}

We give a simple but important corollary to \refT{TBP}, showing that the degree
distribution in $T_n$ converges to the distribution of $D:=\Xix{\taux}$ (with
$\Xi$ and $\taux$ independent).
\begin{corollary}\label{Cdeg}
Let $n_k(T_n)$ be the number of nodes in $T_n$ with outdegree $k$.
Under the assumptions  \refBP{} above, 
  \begin{equation}\label{cdeg}
	\frac{n_k(T_n)}{|T_n|}\asto \P(D=k),
  \end{equation}
where $D=\Xix{\taux}$ is the degree of the root of $\friT$.
In other words, if $D_n$ denotes the outdegree of a uniformly random node in
$T_n$, then 
\begin{equation}
  \cL(D_n\mid T_n)\asto \cL(D).
\end{equation}
\end{corollary}

\begin{proof}
  Let $\cP$ be the property of a node that it has outdegree $k$. Then
  $n_k(T)=n_\cP(T)$. Hence, \eqref{tbpp} shows that $n_k(T_n)/|T_n|$ \as{}
  converges to the probability that the root of $\friT$ has (out)degree
  $k$. However, the root of $\cT_t$ has degree $\Xix{t}$, so the degree $D$
of the root of $\friT=\cT_\taux$ equals $\Xix{\taux}$ and \eqref{cdeg}
follows.  
\end{proof}

\begin{remark}\label{RD}
Since $D=\Xix{\taux}$ and
$\taux\sim\Exp(\ga)$,
\begin{equation}\label{D}
 \P(D\ge i)= \P(\taux \ge \xi_i)
=\E e^{-\ga \xi_i}. 
\end{equation}
See further \refR{Rheir} below.
\end{remark}

In order to extend \refT{TBP} to the extended fringe, we first define the 
limiting random sin-tree $\sinT$; this is the family tree of the 
\emph{doubly infinite pedigree process} in \cite{NermanJagers1984}
(\emph{doubly infinite stable population process} in \cite{JagersNerman1984}).
In the latter, we start with an individual $o$ (``ego'') born at time 0, and
grow a branching process starting with it as usual. We also give $o$ an
infinite line of 
ancestors $o\mm1,o\mm2,\dots$ having a modified distribution of their life
histories defined below, and let each child $x$ of each ancestor $o\mm k$,
except $x=o\mm{k-1}$, start a new branching process where all individuals
have the original distribution. We denote the (infinite) family tree of this
branching process by $\sinT_t$, $-\infty<t<\infty$. 
Finally, we stop the entire process at a random time $\taux\sim\Exp(\ga)$ as
before, and let $\sinT:=\sinT_{\taux}$ be the resulting sin-tree, with
distinguished node $o$. 
(Note that the subtree of $\sinT$ rooted at $o$ equals $\friT$ defined in
\refT{TBP}.) 

It remains to define 
the distribution of the life history of an ancestor. This is
really a distribution of a life history with a distinguished child, which we
call the \emph{heir}. The heir may be any child, but the probability
distribution is weighted by $e^{-\ga\tau}$, where $\tau$ is the time the heir
  is born. Thus, recalling that the children are born at times
  $(\xi_i)_{i=1}^N$, 
for any event $\cE$ in the life history,
  \begin{equation}\label{heir}
\P(\cE, \text{ and the heir is the $i$:th child})
=\int_{\cE} e^{-\ga\xi_i} dP,
  \end{equation}
where for $i>N$ we define $\xi_i=\infty$, so $e^{-\ga\xi_i}=0$.
In particular, 
  \begin{equation}\label{heir1}
q_i:=
\P(\text{the heir is the $i$:th child})
=\E e^{-\ga\xi_i}.
  \end{equation}
Note that \eqref{heir} defines a probability distribution, since the total
probability equals
\begin{equation}
  \sum_{i=1}^\infty \int_{\gO} e^{-\ga\xi_i} dP
=\E \sum_{i=1}^N e^{-\ga \xi_i} 
=1
\end{equation}
by 
\eqref{malthustau}.

We may give the children of the ancestor another order as in \refR{Rorder},
still using \eqref{heir}. Note that then \eqref{heir}--\eqref{heir1} 
hold also if we
consider the $i$:th child in the final order and redefine $\xi_i$ as the
birth time of that child; this is seen by summing over all children and
combinations of marks $\nu_j$ that put a certain child in place $i$ at a
given time.

The ancestors $o\mm k$ are given independent copies of this modified life
history distribution, and are put together so that the heir of $o\mm k$ is
$o\mm{k-1}$ (with $o\mm0=o$); 
this also defines recursively the birth times of all $o\mm k$.

\begin{remark}\label{Rage}
Let $\heir$ denote the age of an ancestor when its heir is born.
Then $\heir$ has by \eqref{heir} the distribution 
$e^{-\ga t}\mu(\ddx t)$, \ie, the distribution in \refR{RBPmalthusage}. 
Its Laplace transform is given by 
\begin{equation}\label{rage}
\E e^{-s\heir}=\intoo e^{-st-\ga t}\mu(\ddx t)
=\hmu(\ga+s)=\E\hXi(\ga+s)  ,
\end{equation}
\cf{} \eqref{heir} and \eqref{xib}--\eqref{hmub}.
Assumption \ref{BPmub} thus says that $\E e^{\eps \heir}<\infty$ for some
$\eps>0$. 
In particular, $\heir$ has a finite expectation
\begin{equation}\label{el}
\gb:=\E\heir=\intoo te^{-\ga t}\mu(\ddx t) <\infty.  
\end{equation}
By \eqref{rage}, we also have the formula
\begin{equation}\label{el2}
\gb=  \E\heir = -\ddd{s}\E e^{-s\heir}\Bigm|_{s=0}
=-\hmu'(\ga),
\end{equation}
and directly from \eqref{heir}, or by  \eqref{el},
\begin{equation}\label{el3}
  \gb = \E \sum_{i=1}^N \xi_i e^{-\ga\xi_i}.
\end{equation}
\end{remark}

\begin{remark}\label{Rheir}
Let, as in \eqref{heir1},
$q_i$ be the probability that the heir in the ancestor distribution
is child $i$ (in birth order), and 
let $D=\Xix{\taux}$ be the degree of the root in $\friT$, which by
  \refC{Cdeg} is the 
  limit in distribution of the outdegree of a random node in $T_n$.
  By \eqref{heir1} and \eqref{D},
\begin{equation}\label{heirD}
  q_i=\E e^{-\ga \xi_i} 
=  \P(D\ge i),
\end{equation}
so the two distributions are closely related. 
Note also that \eqref{heirD} implies
\begin{equation}\label{ED}
  \E D = \sumii \P(D\ge i)=\sumii q_i=1,
\end{equation}
so the average asymptotic outdegree is always 1.
This should not be surprising; it just is an asymptotic version of the fact
that in a tree with 
$n$ nodes, there are together $n-1$ children, and thus the average outdegree
is $1-1/n$; see also see \cite[Lemma 1]{Aldous}.
\end{remark}

\begin{remark}\label{RD2}
Recall that we may regard the node set of $\friT$ as a subset of $\voo$, the
node set of the infinite Ulam--Harris tree.
Let $v\in\voo$. 
  By the recursive definition of the branching process $\ctt$
and the
  memoryless property of the exponential random variable $\taux$,
it follows that conditioned on $v\in\friT=\cttaux$,
the subtree of $\friT$ rooted at $v$ has the same distribution as
$\friT$. In particular, conditioned on $v\in\friT$, the outdegree of $v$ has
the same distribution as $D$. 

It follows from this and \eqref{ED}, by induction, that for every $k\ge0$,
the expected number of nodes in the $k$:th generation of $\ctt$ is 1.
In particular, the expected size $\E|\ctt|=\infty$.

Note also that the outdegrees of two
different nodes are not independent, since they both depend on the common
stopping time $\taux$; it is easy too see that for any $v,w\in\voo$,
conditioned on $v,w\in\friT$, the outdegrees $\deg(v)$ and $\deg(w)$ are
(strictly) positively correlated.

In fact, the properties in this remark except the last one
hold for any fringe distribution in
the sense of \citet{Aldous}, see \cite[Section 2.1]{Aldous}.
However, the positive correlation of node degrees is not general; in
particular,
it makes the asymptotic fringe trees $\friT$ studied in this paper
different from the ones obtained from conditioned Galton--Watson trees,
since the latter are just unconditioned Galton--Watson trees, where all
outdegrees are 
independent, see \cite{Aldous}.

\end{remark}

\begin{theorem}[Jagers, Nerman, Aldous]\label{TBPX}
  Under the assumptions \refBP,
as \ntoo,
$\dep(v)\pto\infty$ for a random node $v\in T_n$ and thus each $T_n^{*,-k}$ is
well-defined \whp; moreover,
the following hold:
\begin{romenumerate}
\item\label{TBPXa}
 (Annealed version.)
The extended random fringe tree of
$T_n$ converges in distribution as \ntoo{} to the random sin-tree 
$\sinT$ defined above, in the sense \eqref{extended}, see \refS{Sextended}.
\item \label{TBPXq}
(Quenched version.)
The convergence in \ref{TBPXa} holds also
conditioned on $T_n$, a.s.
Equivalently, 
for every property $\cP$ of nodes $v$ that depends only on $v$, its
descendants and the descendants of 
its ancestors at most a fixed number of generations back,
we have if $v$ is a uniformly random node in $T_n$,
\begin{equation}\label{tbpxq}
\P(v\text{ has }\cP\mid T_n)=
\frac{\nun_{\cP}(T_n)}{|T_n|}\asto p_{\cP} (\sinT).
\end{equation}
\end{romenumerate}
\end{theorem}

\begin{proof}
Again, this is a special case of the main results in
\citet{JagersNerman1984} and \cite{NermanJagers1984}, and is at least
implicit in  \citet{Aldous}, but we give the proof for completeness.

First
consider the case of ordered trees (possibly with labels) with the children
taken in order of birth. 
Fix a finite tree $T$ with a distinguished node of depth $k\ge0$, and let
$v_0\dotsm v_k$ be the path in $T$ from the root to the distinguished node;
also, let $v_i$ be the $j_i$:th child of $v_{i-1}$.
Let $\cP=\cP(T)$ be the property of a node $v$ that it has depth at least
$k$ and 
that, if $w$ is its $k$:th ancestor, the subtree $T^w$, with $v$ as
distinguished node, is isomorphic to $T$.
Then $n_\cP(T_n)$, the number of $v\in T_n$ that have this property, equals
the number of $w\in T_n$ such that $T_n^w\iso T$, \ie, $n_T(T_n)$.
Thus, by \refT{TBP},
\begin{equation}\label{sn}
\frac{\nun_{\cP(T)}(T_n)}{|T_n|}
=
\frac{\nun_{T}(T_n)}{|T_n|}
\asto \P(\friT\iso T)=\P(\cT_\taux\iso T).
\end{equation}

Construct $\friT=\cT_{\taux}$ as above, and let $\tauy{i}=\xi_{v_i,j_i}$ be
the age of $v_i$ when its distinguished child 
$v_{i+1}$ is born. The distinguished node $v_k$
is thus born at time $V=\sum_{i=0}^{k-1}\tauy{i}$.  
If $\cT_{\taux}\iso T$, then necessarily $\taux\ge V$. Moreover, conditioned on
$\taux\ge V$, $\taux$ has the same distribution as $V+\tauxx$ with
$\tauxx\sim\Exp(\ga)$ and independent of everything else. Thus,
by conditioning on $V$,
\begin{equation}\label{sno}
  \begin{split}
\P(\cT_\taux\iso T)
=\E\bigpar{e^{-\ga V}\ett{\cT_{V+\tauxx}\iso T}}
=\E\Bigpar{\ett{\cT_{V+\tauxx}\iso T}\prod_{i=0}^{k-1} e^{-\ga \tauy{i}}},
  \end{split}
\end{equation}
By shifting the time parameter in $\cT$ by $V$, so that the distinguished
node $v_k$ becomes born at time 0, and recalling that the subtree
$\sinT^{o\mm k}$ has the modified distribution \eqref{heir} for the
ancestors of the distinguished node, we see that \eqref{sno} equals 
\begin{equation}\label{snq}
 \P\bigpar{\sinT^{o\mm k}\iso T}
=p_{\cP(T)}(\sinT).
\end{equation}

Consequently, \eqref{sn}--\eqref{snq} show that
\begin{equation}\label{snx}
\frac{\nun_{\cP(T)}(T_n)}{|T_n|}\asto \P\bigpar{\sinT^{o\mm k}\iso T}
=p_{\cP(T)}(\sinT)
\end{equation}
for every finite tree with a distinguished node of depth $k$.

More generally, for any fixed $k\ge0$ and any set $A$ of finite trees, each
having a distinguished node of depth $k$, let $\cP=\cP(A):=\bigcup_{T\in
  A}\cP(T)$ be the property of a node $v$ that it has depth at least $k$
and that $T^w\in A$, where $w$ is its $k$:th ancestor.
Then as in \eqref{sn}, by \refT{TBP} applied to the property $T^w\in A$, 
and using again \eqref{sno}--\eqref{snq},
\begin{equation}\label{sna}
  \begin{split}
\frac{\nun_{\cP(A)}(T_n)}{|T_n|}\asto \P(\cT_\taux\in A)
&= \sum_{T\in A} \P(\cT_\taux \iso T)
= \sum_{T\in A} \P\bigpar{\sinT^{o\mm k}\iso T}
\\&
=  \P\bigpar{\sinT^{o\mm k}\in A}
=p_{\cP(A)}(\sinT).	
  \end{split}
\end{equation}
In particular, taking $A$ to be the set of all finite trees, $\cP(A)$ is the
property that $\dep(v)\ge k$ and $p_{\cP(A)}(\sinT)=1$, so \eqref{sna} shows
that for any $k$, $\P(\dep(v)\ge k)\to1$ for a random node $v$ in $T_n$. 
Since $k$ is
arbitrary, thus $\dep(v)\pto\infty$.
Moreover, 
every property $\cP$ in \eqref{tbpxq} is of the form $\cP(A)$ for some $k$ and
$A$, and thus the result \eqref{tbpxq} follows.

As in \refT{TBP}, the annealed case follows from the quenched case by taking
expectations. 

The case of unordered trees follows by ignoring the order.

Finally, if $\cT_t$ is an ordered tree with the order of children
defined by marks
$\nu_i$ as in \refR{Rorder}, we first fix an integer $M$ and 
consider $\cT_t$ and $\sinT$ ordered by
birth order and with each node labelled with 
the sequence of marks $\bnu_M:=(\nu_i)_{i=1}^{M\land N}$
(in addition to existing labels, if any). 
(We use a cut-off $M$ in order to keep the space of labels countable.)
We have just shown that \eqref{sna} holds
for any set $A$ of ordered trees with such a label $\bnu_M$ on each
node. Since the 
birth order and the marks define the true order in the trees, it follows
immediately that \eqref{snx} holds also with the true order in $T_n$ and
$\sinT$, for any tree $T$ with such marks and with maximum degree at most
$M$. Since $M$ is arbitrary, it holds with the true order
for any $T$, and we may then
forget the marks. (For an \mary{} tree, we keep the marks.) 
Then \eqref{sna} and \eqref{tbpxq} follow as above.
\end{proof}

\subsection{An extension to some more general properties}\label{SStrunc}
In \refT{TBPX}, we consider only properties of a node $v$
that depend only on $v$, its ancestors at most a fixed number of generations
back, and their descendants. (\refT{TBP} is even more restrictive.)
In this subsection, we show how this result can be extended to some
properties that depend on all ancestors of $v$. A typical example is the
property that $v$ has no ancestor with outdegree 1; we consider this and
some related examples in \refS{Sclades}.
(This section can be omitted at the first reading.)

\begin{theorem}\label{TQ}
  Let $\cP_0$ and $\cQ$ be two properties of a node $v$ in a tree, such
  that both $\cP_0$ and $\cQ$ depend only on $v$ and its descendants.
Let $\cP$ be the property of a node $v$ that $v$ satisfies $\cP_0$  but no
ancestor of $v$ satisfies $\cQ$.
Suppose, in addition to \refBP, that
\begin{equation}\label{lq0}
  \E \hXi(\ga)^2<\infty
\end{equation}
and that if\/
$\gL:=\sup\set{t:\ctt\in\cQ}$, then
\begin{equation}\label{lqe}
  \E e^{\gd \gL}<\infty
\end{equation}
for some $\gd>0$.

Then, as \ntoo,
if $v$ is a uniformly random node in $T_n$,
\begin{equation}\label{tbpxq+}
\P(v\text{ has }\cP\mid T_n)=
\frac{\nun_{\cP}(T_n)}{|T_n|}\asto p_{\cP} (\sinT).
\end{equation}
\end{theorem}
In other words, \eqref{tbpxq} holds also for properties $\cP$ of this type,
although they are not covered by \refT{TBPX}.

\begin{remark}
The assumption \eqref{lqe} entails $\gL<\infty$ a.s.; moreover,
\eqref{lqe} is equivalent to assuming $\P(\gL\ge t)= O\bigpar{e^{-\gd t}}$
for some (possibly different) $\gd>0$, and thus  to
\begin{equation}\label{lq1}
  \P\bigpar{\cT_u\in\cQ \text{ for some $u\ge t$}}
= O\bigpar{e^{-\gd t}}.
\end{equation}

  In the examples in \refS{Sclades}, the property $\cQ$ is (or can be
  taken as) decreasing in the sense that if it holds for some rooted tree
  $T$, then it holds also for every subtree with the same root; hence if
  $\cQ$ holds for $\cT_u$ with $u\ge t$, then it holds for $\cT_t$, so
  \eqref{lq1} can be simplified to
\begin{equation}\label{lq1x}
  \P\bigpar{\cT_t \in \cQ}
= O\bigpar{e^{-\gd t}}.
\end{equation}
\end{remark}

Before the proof, we give a lemma.

\begin{lemma}\label{LQ}
Suppose that \refBP{} and \eqref{lq0} hold.
Let $\cQ$ be a property of rooted trees 
such that \eqref{lq1} holds for some $\gd>0$.
Then there exists $\eta>0$ and $a<\infty$ such that
\begin{equation}\label{lq}
  \frac{1}{|\cT_t|}\sum_{v\in\cT_t}|\cT_t^v|^{1+\eta}
\ett{\cT_t^v\in\cQ} \asto a.
\end{equation}
  \end{lemma}

\begin{proof}
  The \lhs{} of \eqref{lq} equals $\zgf_t/|\cT_t|$, where $\gf(t)$ is the
  characteristic given 
as in \refR{Rnerman7} 
with
  \begin{equation}\label{lqq}
	\gf_0(t):=|\cT_t|^{1+\eta} \ett{\cT_t\in\cQ} .
  \end{equation}
The result thus follows from \eqref{olle}, provided we can choose $\eta>0$
such that \ref{BP6.2} holds for this $\gf$.

To verify \ref{BP6.2}, we first note that \eqref{lq0} (together with the
other conditions)
implies 
\begin{equation}\label{lq3}
\E|\cT_t|^2\le \CC e^{2\ga t},  \CCdef\CCsquare
\end{equation}
for some $\CCx<\infty$, see
\cite[Theorem 6.4.3 with Note or Theorem 6.8.1]{Jagers}
or \cite[Theorem 3.5]{JagersNerman1984}.

If $0<\eta<1$, we thus have by 
\eqref{lqq}, \Holder's inequality, \eqref{lq3} and
\eqref{lq1},
\begin{equation}\label{lq5}
  \begin{split}
\E &\sup_{t\ge0}\bigpar{e^{-\gth t}\gf_0(t)	}
\le \sumn \E \sup_{n\le t\le n+1}\bigpar{e^{-\gth t}\gf_0(t)	}
\\&
\le \sumn e^{-\gth n}\E \Bigpar{|\cT_{n+1}|^{1+\eta}
  \ett{\cT_u\in\cQ\text{ for some $u\ge n$}}} 
\\&
\le \sumn e^{-\gth n}\bigpar{\E|\cT_{n+1}|^2}^{(1+\eta)/2}
\bigpar{ \P\bigpar{\cT_u\in\cQ\text{ for some $u\ge n$}}}^{(1-\eta)/2} 
\\&
\le \sumn e^{-\gth n}\bigpar{\CCsquare e^{2\ga (n+1)}}^{(1+\eta)/2}
\bigpar{\CC e^{-\gd n}}^{(1-\eta)/2} 
\\&
= \CC\sumn e^{n(-\gth+\ga(1+\eta)-\gd(1-\eta)/2)}, 
  \end{split}
\end{equation}
which is finite provided 
\begin{equation}\label{lq4}
(1+\eta)\ga<\gth+\gd(1-\eta)/2.  
\end{equation}
If $0<\eta<\min(\gd/4\ga,\xfrac12)$, then 
$\eta\ga< \gd/4<\gd(1-\eta)/2$, and thus there exists $\gth<\ga$ such that
\eqref{lq4} holds; hence the sum in \eqref{lq5} is finite and \ref{BP6.2}
holds, which completes the proof.
\end{proof}

\begin{proof}[Proof of \refT{TQ}]
  For each integer $M$, let $\cP_M$ be the truncated property ``$v$
  satisfies $\cP_0$ but no ancestor at most $M$ generations before $v$
  satisfies $\cQ$.'' Then $\cP_M$ is covered by \refT{TBPX}, so
\begin{equation}\label{tbpxqM}
\frac{\nun_{\cP_M}(T_n)}{|T_n|}\asto p_{\cP_M} (\sinT)
\end{equation}
as \ntoo, for each $M$.
Since $\cP$ is the intersection of the decreasing sequence of properties
$\cP_M$, it is clear that $p_{\cP_M}(\sinT)\to p_\cP(\sinT)$ as \Mtoo.
Furthermore,
$n_{\cP}(\ctt)\le n_{\cP_M}(\ctt)$ and for any $\eta>0$,
writing $w\prec v$ when $w$ is an ancestor of $v$,
\begin{equation}\label{tq5}
  \begin{split}
n_{\cP_M}(\ctt)-n_{\cP}(\ctt)	
&
\le \sum_{v\in\ctt}\sum_{w\in\ctt}
\ett{w\prec v, \dep(w)< \dep(v)-M \text{ and }\ctt^w\in\cQ}
\\&
\le \sum_{w\in\ctt}\sum_{v\succ w}
\ett{|\ctt^w|\ge M \text{ and }\ctt^w\in\cQ}
\\&
= \sum_{w\in\ctt}(|\ctt^w|-1)
\ett{|\ctt^w|\ge M \text{ and }\ctt^w\in\cQ}
\\&
\le M^{-\eta}\sum_{w\in\ctt}|\ctt^w|^{1+\eta}
\ett{\ctt^w\in\cQ}.
  \end{split}
\raisetag{1.5\baselineskip}
\end{equation}
By \refL{LQ} we can choose $\eta>0$ and $a$ such that \eqref{lq} holds.
Then for any $M$, by \eqref{tq5} and \eqref{lq}, a.s.,
\begin{equation}
\limsup_{\ntoo}\frac{\lrabs{n_{\cP_M}(T_n)-n_{\cP}(T_n)}}{|T_n|}
\le
\limsup_{\ttoo}\frac{\lrabs{n_{\cP_M}(\ctt)-n_{\cP}(\ctt)}}{|\ctt|}
\le M^{-\eta}a.
\end{equation}
Consequently, using also \eqref{tbpxqM}, \as,
\begin{multline}
\limsup_{\ntoo}\lrabs{\frac{n_{\cP}(T_n)}{|T_n|}-p_{\cP}(\sinT)}
\le	
\limsup_{\ntoo}\frac{\bigabs{n_{\cP_M}(T_n)-n_{\cP}(T_n)}}{|T_n|}
\\\shoveright{
+
\limsup_{\ntoo}\lrabs{\frac{n_{\cP_M}(T_n)}{|T_n|}-p_{\cP_M}(\sinT)}
+\bigabs{p_{\cP_M}(\sinT)-p_{\cP}(\sinT)}}
\\
\le M^{-\eta}a+0+\bigabs{p_{\cP_M}(\sinT)-p_{\cP}(\sinT)}.  
\end{multline}
The \rhs{} tends to 0 as \Mtoo, and the theorem follows.
\end{proof}

\section{Examples with uniform or preferential attachment}\label{Sex}

We begin with a few standard examples, where we repeat earlier results
by other authors, together with some new results on the limiting sin-trees.
In all examples in this section, $|T_n|=n$, so we stop the branching process
using the weight $\psi(t)=1$ as in \refE{E1}. Since this weight is bounded,
\refBPlastpsi{} holds trivially.

\begin{example}[random recursive tree]\label{ERRT}
 An important example, considered already by \citet{Aldous}, is the
 \emph{random recursive tree}.
This tree, usually considered as an unordered rooted tree, is constructed
recursively by adding nodes one by one, with each new node attached as a
child of a (uniformly) randomly chosen existing node, 
see \cite[Section 1.3.1]{Drmota}. 
It is easy to see, by the
memoryless property of the exponential distribution, that the random
recursive tree with $n$ nodes is the tree $T_n$ defined in \refS{S:BP}
for the branching process where each individual gives birth with constant
intensity 1, \ie{} with independent $\Exp(1)$ waiting times between births,
and weight function $\psi(t)=1$ as in \refE{E1}. In other words, the point
process $\Xi$ describing the births of the children of an individual
is a Poisson process with intensity 1.
This branching process 
(or just the sizes $(|\ctt|)_{t\ge0}$)
is often called the \emph{Yule process}, so the process
$(\cT_t)_t$ of trees is called the \emph{Yule tree process} \cite{Aldous}.
Note that Yule process formed by the size $|\ctt|$ is a pure birth process
where the birth rate $\gl_n=n$,
see \refE{EYule1}.

We will need some notation.
Let $X_i:=\xi_i-\xi_{i-1}$ (with $\xi_0:=0$) be the waiting
times between the births of the children of a given individual. Thus $X_i$
are \iid{} $\Exp(1)$, and $\xi_i=\sum_{j=1}^i X_j \sim \gG(i,1)$ has a Gamma
distribution. 

The intensity measure $\mu$ is Lebesgue measure on $\ooo$, so
\begin{equation}\label{hmuRRT}
  \hmu(\gth)=\intoo e^{-\gth t}\dd t =\frac{1}{\gth},
\qquad \gth>0,
\end{equation}
and
\eqref{malthus} holds with $\ga=1$. 
The conditions \refBP{} are trivially verified. 

As shown by \citet{Aldous}, the limiting fringe tree $\friT=\cT_\taux$ can
also be described as a random recursive tree with a random number $M$ nodes,
where
\begin{equation}\label{rrtM}
  \P(M=n)=\frac{1}{n(n+1)}, \qquad n\ge1.  
\end{equation}
In fact, 
by symmetry, if $M=|\friT|$ and we condition $\friT$ on $M=n$, we get a random
recursive tree on $n$ nodes. Moreover,
if we at some time have $n\ge1$ individuals in the branching process,
then a new child is born with intensity $n$, while the process stops (at
$\taux$) with intensity 1, so the probability that the process 
continues with at least one more individual is $n/(n+1)$. In other words,
$\P(M\ge n+1)=\frac{n}{n+1}\P(M\ge n)$, and thus by induction $\P(M\ge
n)=1/n$ and \eqref{rrtM} follows. 
(For an alternative argument, see \refE{Elinear+-} below.)

As noted by \cite{Aldous}, 
various results for the random recursive tree $T_n$ now follows from
\refT{TBP}.
For example, the asymptotic distribution of the size of a random fringe tree
is given by \eqref{rrtM}. 
Furthermore, 
the asymptotic distribution of the outdegree of the nodes in $T_n$
equals by \refC{Cdeg}  the distribution of the root degree $D$ in $\friT$,
which  is geometric $\Geo(1/2)$ as an immediate consequence of
\eqref{D}. (See \eqref{rrtJ} below and \eqref{heirD}.)
See \refSS{SSrankRRT} for yet another example.

In order to construct the random sin-tree $\sinT$, 
which enables applications of \refT{TBPX} on the extended fringe, we 
have to find the distribution of the life history of the ancestors, given by
\eqref{heir}. 
Consider an ancestor and
denote its successive birth times by $\txi_i$, $i\ge1$, and
let
$\tX_i:=\txi_i-\txi_{i-1}$ (with $\txi_0:=0$) be the succesive waiting
times. Furthermore, let $\tXi$ be the point process of all births of
children of this ancestor
(thus $\tXi=\sum_i\gd_{\txi_i}$)
and let
$J$ be the number of the heir (in birth order). Then, by \eqref{heir1},
\begin{equation}\label{rrtJ}
  \P(J=j)=\E e^{-\xi_j} = \E e^{-\sum_1^j X_i}
=\prod_{i=1}^j \E e^{-X_i}=2^{-j},
\qquad j\ge1.
\end{equation}
Thus $J$ has the (shifted) geometric distribution $\Gei(1/2)$.
Moreover, conditioned on $J=j$, the joint density of $(\tX_1,\dots,\tX_m)$,
for any $m\ge j$, is by \eqref{heir}
\begin{equation}\label{rrttX}
  \frac{1}{P(J=j)}  e^{-\sum_1^j x_i}\prod_{i=1}^m e^{-x_i}
= \prod_{i=1}^j 2e^{-2x_i} \prod_{i=j+1}^m e^{-x_i}.
\end{equation}
Consequently, conditioned on $J=j$, the waiting times $\tX_i$ between the
births for an ancestor are independent, with $\tX_i\sim\Exp(2)$ for $i\le J$
and $\tX_i\sim\Exp(1)$ for $i>J$.

We claim that we can describe $\tXi$ in a simpler way as a Poisson process
$\Xi$ with intensity 1, plus an extra point $Z\sim \Exp(1)$, independent of
$\Xi$, with $Z$ the heir.
To see this, note that with this description, the first point of $\tXi$ is
either the first point of $\Xi$ or $Z$; these two first points are both 
$\Exp(1)$
and independent, 
so the first point $\tX_1$,
which is the smallest of these two points, 
will be $\Exp(2)$. Furthermore, with probability
$1/2$, this point is the heir $Z$, so $J=1$, and then the rest of the
process is $\Xi$, with 
independent $\Exp(1)$ waiting times. And with probability $1/2$, $\tX_1$
comes from $\Xi$, and then the whole process repeats from $\tX_1$, so the
next waiting time $\tX_2\sim\Exp(2)$, and so on. A simple induction shows
that this yields both the distribution of $J$ in \eqref{rrtJ} and the right
conditional distribution of $(\tX_i)_1^\infty$ given $J=j$ for each $j$,
which proves the claim.

In particular, an ancestor's age $\heir$ when its heir is born has distribution
$\Exp(1)$, which also follows directly from \refR{Rage}.
As a consequence, 
$\gb=\E\heir=1$, \cf{} \eqref{el} and \eqref{el2}.

We can thus describe the random sin-tree $\sinT$ as follows: First construct 
an infinite chain of ancestors $o\mm 1,o\mm2,\dots$ of $o$
(backwards in time), with the times between their births \iid{} $\Exp(1)$;
in other word, $(o\mm{k})_{k\ge1}$ are born according to a Poisson process
with intensity 1 on $(-\infty,0)$. Then grow independent Yule tree processes
from all $o\mm k$, $k\ge0$. Finally, stop everything at $\taux\sim\Exp(1)$.
(Cf.\ \citet[Section 4]{Aldous}, where the description is less explicit.)
For an application, see \refT{Tpx1rrt}.
\end{example}

\begin{example}[binary search tree]\label{EBST}
Another important example studied by \citet{Aldous} is the (random)
\emph{\bst}.
This is the case $m=2$ of the \mary{} search tree in \refS{Smst-def}, but it is
simpler than the general case, so we treat it separately, using a slightly
different but
equivalent formulation.
(Since each (internal) node has exactly one key, the number of keys equals
the number of nodes, and we can ignore the keys completely.)

The \bst{} can be grown recursively as follows. (See \eg{} \cite{Drmota} for
other, equivalent, constructions.)
Start with a single node. Since we grow a binary tree, each node may have a
left child and a right child. When the tree has $n$ nodes, there are $n+1$
empty places for children (these places are the external nodes in the
description in \refS{Smst-def}). The tree grows by adding a node to one of
these $n+1$ places, chosen uniformly at random.
Similarly as in \refE{ERRT}, it is easy to see that the binary search tree
is the tree $T_n$ produced by the branching process where each individual
has two children, labelled left and right and born at age $\xi_\sL$ and
$\xi_\sR$, say, with $\xi_\sL$ and $\xi_\sR$ both
$\Exp(1)$ and independent; furthermore we use
again the weight function $\psi(t)=1$ as in \refE{E1}.
(This continuous-time branching process seems to have been first used to
study the \bst{} by \citet{Pittel1984}, who considered the height and
saturation level, see \refS{Sheight}.)

We thus have $N=2$. Since each child is born with the density function
$e^{-x}$, the intensity measure $\mu$ of $\Xi$ has density $2e^{-x}$.
Thus 
\begin{equation}\label{hmuBST}
  \hmu(\gth)=\intoo 2e^{-\gth x-x}\dd x =\frac{2}{1+\gth},
\qquad \gth>-1,
\end{equation}
and
\eqref{malthus} holds with the Malthusian parameter $\ga=1$. 
The conditions \refBP{} are trivially verified. 

Note that if we order the children in order of birth as usual, then 
$\xi_1=\min(\xi_\sL,\xi_\sR)$, and thus $\xi_1\sim\Exp(2)$, while
the waiting time $\xi_2-\xi_1$ for the second child is $\Exp(1)$ and
independent of $\xi_1$.

We see also that the size $|\ctt|$ grows as a pure birth process with birth
rate $\gl_n=n+1$, see \refApp{AA}. Equivalently, $|\ctt|+1$, which can be
interpreted as the number of external nodes, is a pure birth process with rate
$\gl_n=n$, \ie, the Yule process in \refE{EYule1}, and in \refE{ERRT}, 
but started at 2 instead of 1.

As shown by \citet{Aldous}, the limiting fringe tree $\friT=\cT_\taux$ can
be described as a \bst{} with a random number $M$ nodes,
where
\begin{equation}\label{bstM}
  \P(M=n)=\frac{2}{(n+1)(n+2)}, \qquad n\ge1;
\end{equation}
\cf{} the similar result \eqref{rrtM} for the \rrt.
To see this we argue as in \refE{ERRT}; the difference is that when
there are $n$ individuals, there are now $n+1$ places to add a new node, and
thus $n+1$ 
independent $\Exp(1)$ for these, competing with the random time $\taux$ that
stops the process; hence the probability of adding another node is
$(n+1)/(n+2)$ and thus by induction 
$\P(M\ge n)=2/(n+1)$ and \eqref{bstM} follows. 
(For an alternative argument, see \refE{Elinear+-} below.)

By \refT{TBP}, the asymptotic distribution of the size of a random fringe tree
is given by \eqref{bstM}. Another simple calculation 
in \cite{Aldous} shows
that the asymptotic distribution of the outdegree of the nodes in $T_n$,
which by \refC{Cdeg} equals the distribution of the root degree $D$ in $\friT$,
is uniform on \set{0,1,2}, see \eqref{D}. This can also be seen without
calculation: 
$\xi_\sL,\xi_\sR$ and $\taux$ are three
\iid{} $\Exp(1)$ random variables, so the three
events that $\taux$ is the smallest, the middle, or the largest of these three
have by symmetry all the same probability $1/3$. These events equal the
events that the root in $\friT$ has degree $0,1,2$.

To find the random sin-tree $\sinT$, note that by the comments after
\eqref{heir}--\eqref{heir1}, \eqref{heir} holds also when taking the
children in order left--right. 
For an 
individual in $\friT$,
the pair $(\xi_\sL,\xi_\sR)$ has the density function $e^{-x_\sL-x_\sR}$. 
For an ancestor, the probability that the heir is the
left child is $1/2$ (by symmetry or by \eqref{heir1}), and it follows that
conditioned on the heir being the left child, 
the pair $(\xi_\sL,\xi_\sR)$ has the density function 
$2e^{-x_\sL} e^{-x_\sL-x_\sR}= 2e^{-2x_\sL}e^{-x_\sR}$. 
In other words, for an ancestor, given that the heir is 
the left child, 
the age $\heir$ when the heir is born is $\Exp(2)$ and the age
when the other child is born is $\Exp(1)$, and these two ages are independent.
The same holds given that the heir is the right child.
In particular, $\heir\sim\Exp(2)$ and thus $\gb=\E\heir=1/2$, 
\cf{} \refR{Rage}.  

Consequently, the random sin-tree $\sinT$ can be described as follows, \cf{}
the case of the random recursive tree in \refE{ERRT}: 
First construct 
an infinite chain of ancestors $o\mm 1,o\mm2,\dots$ of $o$
(backwards in time), with the times between their births \iid{} $\Exp(2)$;
in other word, $(o\mm{k})_{k\ge1}$ are born according to a Poisson process
with intensity 2 on $(-\infty,0)$. Moreover, make a random choice (uniform
and independent of everything else) for each ancestor to decide whether its
heir is the left or right child.
Then grow independent binary tree processes at all empty places (external
nodes), with independent $\Exp(1)$ waiting times for all new nodes.
Finally, stop everything at $\taux\sim\Exp(1)$.
(Applications are given in \refS{Sclades}.)

\end{example}

\begin{example}[general preferential attachment trees]\label{Epref}

We can generalise the preceding examples as follows,
see \citet{RudasTV} and \citet{RudasT} where this example is studied
using the branching process method described here; 
see also \citet{Bhamidi}. 
(The branching process below was also earlier used by \citet{BigginsGrey1996} to
study the height of these trees.)
We thus give only a
summary and some complements, in particular on sin-trees.
Some special cases are treated in Examples \ref{Elinear+-}--\ref{Epreflast}
below, 
see in particular 
\refE{Elinear+};  
these cases have been studied by many authors, using various methods. 
(Further references are given below, but we do not attempt a complete
history.) 

Suppose that we are given a sequence of non-negative weights
$(w_k)_{k=0}^\infty$, 
with $w_0>0$. Grow a random tree $T_n$ (with $n$ nodes) recursively, starting
with a single node and adding nodes one by one. Each new node is added as a
child of some randomly chosen existing node;
when a new node is added to $T_{n-1}$, the
probability of choosing a node $v\in T_{n-1}$ as the parent is proportional to
$w_{\dout(v)}$, where $\dout(v)$ is the outdegree of $v$ in $T_{n-1}$.
(More formally, this is the conditional probability, given $T_{n-1}$ and the
previous history.
The sequence $(T_n)_{n=1}^\infty$ thus constitutes a Markov process.)
If we want the trees $T_n$ to be ordered trees, we also insert the new child
of $v$ among the existing $\dout(v)$ children in a random position, uniformly
chosen among the $\dout(v)+1$ possibilities.

The random recursive tree in \refE{ERRT} is the special case $w_k=1$, $k\ge0$,
and the \bst{} in \refE{EBST} is the special case with $w_0=2$, $w_1=1$
and $w_k=0$, $k\ge2$ (and, furthermore, each first child randomly assigned
to be left or right).

Note that we require $w_0>0$ (and $w_1>0$ will be implicitly assumed, as a
consequence of \eqref{cond} below), but we allow $w_m=0$ for some larger
$m$, as in the example of the \bst. In this case, no individual will ever
get more than $m$ children; in fact (provided $m$ is chosen minimal), $N=m$
a.s. In this case, the weights $w_{m+1},w_{m+2},\dots$ are irrelevant, so it
suffices to prescribe $w_k$ for $k\le m$. (In this case, we interpret
$1/w_m=\infty$ below, and the infinite sums become finite. We leave such
obvious modifications to the reader.)

In some important examples, for example \refE{Elinear+} below, $w_k$ is a
strictly increasing function of $k$, which means that nodes with a 
high degree are more likely to attract a new node than nodes with a low
degree; hence the name \emph{preferential attachment}, which comes from
\citet{BarabasiA} where this type of model
was introduced (in a more general version, in general yielding graphs and
not trees), see \refE{Elinear+}. The tree version of their model had been
studied earlier under a different name by \citet{Szymanski} and others, see
\refE{EPORT}.  
The model with general $w_k$ was considered by \citet{Mori}.


As in the examples above, the tree $T_n$ can be constructed by a branching
process as in \refS{S:BP}, again with weight $\psi(t)=1$ and
taking the birth times $\xi_i:=\sum_{j=1}^iX_j$, now
with the waiting times between births  
$X_j=\xi_j-\xi_{j-1}\sim\Exp(w_{j-1})$ and independent.
In other words, the stochastic process $\Xit$, $t\ge0$, (\ie, the number
of children of a given individual born up to age $t$) is a pure birth
process, starting at 0 and with birth rate $w_k$ when the  state is $k$.
(See \refApp{AA}.)

Let $\xi_\infty:=\lim_\ntoo\xi_n=\sumji X_j \in(0,\infty]$.
Then
\begin{equation}\label{etau}
  \E \xi_\infty=\sumji\E X_j = \sumk \frac{1}{w_k}.
\end{equation}
We  distinguish between two cases, depending on whether this sum is
finite or not. 

In the \emph{explosive case},
\begin{equation}\label{explosion}
  \E \xi_\infty= 
\sumk \frac{1}{w_k}<\infty.
\end{equation}
Thus $\xi_\infty<\infty$ a.s., \ie, an individual will have an infinite
number of children in a finite time (the point process $\Xi$ explodes).
In this case, the branching process will explode in finite time, and several
of the assumptions in \refS{S:BP} fail.
Nevertheless, this case can be treated separately. It turns out that the
random fringe tree $\Tx_n$ is asymptotically degenerate and
\whp{} consists of a single node only, \ie, $|\Tx_n|\pto1$, 
see \refT{Texp} below. (Equivalently, the proportion of leaves in the tree
$T_n$ tends to 1.)
The case $w_k=(k+1)^p$ 
for some $p>1$ is studied 
by \citet{KrapivskyRL}, \citet{KrapivskyR}
and (rigorously and in detail) by \citet{OSpencer}, who show that if
$p>2$ (but not if  $1<p\le2$), 
the random tree process $T_n$, $n\ge1$, is even more strongly
degenerate: \as{}
there exists a (random) $n_0$ and a node $v\in T_{n_0}$ such that all nodes
added after 
time $n_0$ become children of $v$ (and thus remain leaves forever). 
See also \citet{Athreya}.

In the sequel we consider the \emph{non-explosive case}
\begin{equation}\label{nonexplosion}
  \sumk \frac{1}{w_k}=\infty.
\end{equation}
In this case, by \eqref{etau}, $\E\xi_\infty=\infty$; moreover, 
it is easy to see that  $\xi_\infty=\lim_\ntoo\xi_n=\infty$
a.s., for example by calculating 
\begin{equation}
  \E e^{-\gl\xi_\infty}=\prodji \E e^{-\gl X_j}
=\prodko\frac{w_k}{\gl+w_k}
=\prodko\frac{1}{1+\gl/w_k}
=0
\end{equation}
for any $\gl>0$, see \cite{Athreya}.
Hence, an individual has \as{} only a finite number of children in each
finite interval, \ie, $\Xix{t}<\infty$ for every $t<\infty$.
Furthermore, using \eqref{hmub},
\begin{equation}\label{hmupref}
\hmu(\gth)
=\E\sumni e^{-\gth\xi_n}  
=\sumni \prodjn \E e^{-\gth X_j}
=\sumni \prod_{k=0}^{n-1} \frac{1}{1+\gth/w_k}.
\end{equation}
We assume that there exists $\gth>0$ such that
$1<  \hmu(\gth)<\infty$, \ie,
\begin{equation}\label{cond}
1<\sumn \prod_{k=0}^n \frac{1}{1+\gth/w_k}<\infty.
\end{equation}
(This is also easily seen to imply
\eqref{nonexplosion}.) This implies, by dominated convergence, that
$\gl\mapsto\hmu(\gl)$ is continuous on $[\gth,\infty]$, with
$\hmu(\infty)=0$,
and thus there exists $\ga>\gth$ such that $\hmu(\ga)=1$. Hence
the assumptions \ref{BPmalthus} and \ref{BPmub} hold. The remaining
assumptions are trivially satisfied, and thus 
Theorems \ref{TBP} and \ref{TBPX} apply.

The asymptotic degree distribution is by \refC{Cdeg} and \eqref{D} given by
\begin{equation}\label{prefD}
  \P(D\ge i) =\E e^{-\ga\xi_i}=\prod_{j=1}^{i} \E e^{-\ga X_j}
=\prod_{k=0}^{i-1} \frac{w_k}{w_k+\ga}
\end{equation}
and thus
\begin{equation}\label{prefD=}
  \begin{split}
\P(D=i)&=  \P(D\ge i)-\P(D\ge i+1) 
=\Bigpar{1-\frac{w_i}{w_i+\ga}}\P(D\ge i)
\\&
=\frac{\ga}{w_i+\ga}\prod_{k=0}^{i-1} \frac{w_k}{w_k+\ga}.	
  \end{split}
\end{equation}
This can also be seen as an example of \refT{TA1}.

To describe the life of an ancestor, 
let $\cE_i$ be the event that the heir of the ancestor is child $i$.
We note first that 
if we fix $M<\infty$, then
in the point process $\xi$, the waiting times $X_1,\dots,X_M$ have the joint
density function $\prod_{j=1}^Mw_{j-1}e^{-w_{j-1}x_j}$.
It follows from \eqref{heir} that for any $i$ and $M$ with $1\le i\le M$,
conditioned on $\cE_i$,
the waiting times $(\tX_j)_{j=1}^M$ between the $M$ first
children of the ancestor have a joint density function that is proportional
to
\begin{equation}
e^{-\ga \sum_{j=1}^i x_j}\prod_{j=1}^Me^{-w_{j-1}x_j}
=
\prod_{j=1}^ie^{-(w_{j-1}+\ga)x_j}
\prod_{j=i+1}^Me^{-w_{j-1}x_j}.
\end{equation}
Furthermore, by \eqref{heirD} and \eqref{prefD}
(or by tracking constants in the argument just given),
\begin{equation}\label{prefheir}
q_i:=\P(\cE_i)=\prod_{k=0}^{i-1} \frac{w_k}{w_k+\ga},
\qquad i\ge1.
\end{equation}
Consequently, the point process $\tXi$ describing the births of
the children of an ancestor can be constructed as follows: 
Select the number $I$ of the heir at random, with the
distribution \eqref{prefheir}. Then, conditioned on $I=i$, let the waiting
times $\tX_j$ be independent exponential variables, with
$\tX_j\sim\Exp(w_{j-1}+\ga)$ for $j\le i$ and
$\tX_j\sim\Exp(w_{j-1})$ for $j> i$.

The limiting random sin-tree $\sinT$ then is constructed as in \refS{S:BP}.

In Examples \ref{ERRT} and \ref{EBST}, we have seen alternative, simpler,
constructions of $\tXi$. This will be extended to the linear case in  
\refE{Elinear+-} and \refT{Tancestor}, 
but it does not seem possible to extend it further.
In particular, we show in \refT{Tonlylinear} that the age $\heir$ 
when the heir is
born has an exponential distribution only in the linear case. See also
\refE{E11} for a simple non-linear example.
\end{example}
 
\begin{example}[linear preferential attachment]\label{Elinear+-}
  The simplest, and most studied, case of preferential attachment as in
  \refE{Epref} is the
  linear case 
  \begin{equation}\label{wlinear+-}
w_k=\chi k+\rho, 	
  \end{equation}
for some real parameters $\chi$ and $\rho$,
  with $\rho=w_0>0$.
Note that we obtain the same random trees $T_n$ if we multiply all $w_k$ by a
positive constant. (In the branching processes, only the time scale
changes.) Hence, only the quotient $\chi/\rho$ matters, and 
it suffices to consider $\chi\in\set{1,0,-1}$. 

The case $\chi=0$ is the
(non-preferential) random recursive tree in \refE{ERRT}.
(In this case $\rho$ is irrelevant and we take $\rho=1$.)

The case $\chi=1$ (the increasing case) is studied in \refE{Elinear+}.

In the case $\chi=-1$, so $w_k=\rho-k$, $w_k$ is eventually negative. This
is impossible, and violates our basic assumption in \refE{Epref}. However,
this is harmless if (and only if) $\rho=m$ is an integer; then $w_m=0$ and,
as said above in \refE{Epref}, 
the values $w_k$ for $k>m$ do not matter.
This is the \mary{} case studied in \refE{Elinear-}; the binary search tree
in \refE{EBST} is the special case $\chi=-1$, $\rho=2$.

We continue with some results valid for any linear weight \eqref{wlinear+-},
and refer to Examples \ref{ERRT},  \ref{Elinear+} and \ref{Elinear-} for
further results for the different cases $\chi=0,1,-1$.

Since $\Xit$ is a pure birth process 
with a rate that is a linear function $\chi k+\rho$ of the current state $k$,
and with initial value $0$,
it is easy to see, 
see 
\refT{TAext},
that the expectation $\E\Xit=\mu\ott$ 
is given by 
\begin{equation}
  \mu\ott=
  \begin{cases}
  \frac{\rho}{\chi} \bigpar{e^{\chi t}-1}, & \chi\neq0,
\\
\rho t, &\chi=0.
  \end{cases}
\end{equation}
Hence, $\mu$ has density $\rho e^{\chi t}$ 
(also when $\chi=0$), \cf{} \eqref{ddext}, and thus
\begin{equation}\label{hmulin+-}
  \hmu(\gth)=\intoo \rho e^{\chi t-\gth t}\dd t =\frac{\rho}{\gth-\chi},
\qquad \gth>\chi.
\end{equation}
It follows that \eqref{cond} holds, and that \eqref{malthus} holds with
\begin{equation}\label{alphalin}
  \ga=\chi+\rho = w_1.
\end{equation}
(Alternatively, \eqref{hmulin+-} can be verified algebraically, see
\eqref{hmulin} and \eqref{hmulin-} below.) 

By \refR{Rage} and \eqref{alphalin}, the age $\heir$ when an heir is born to an
ancestor has 
density 
$e^{-\ga t}\mu(\ddx t)=e^{-\ga t}\rho e^{\chi t}\dd t=\rho e^{-\rho t}\dd t$;
thus $\heir$
has an exponential distribution $\Exp(\rho)$. (This also follows from 
\eqref{hmulin+-} and the
formula for the Laplace transform in \refR{Rage}.)
As a consequence, 
generalizing the values of $\gb$ found in Examples \ref{ERRT} and \ref{EBST},
\begin{equation}\label{betalin}
  \gb = \E\heir = \rho\qw=w_0\qw.
\end{equation}

We claim that the life history $\tXi$
of an ancestor can be described as follows
(as a simpler alternative to the general construction in \refE{Epref}),
\cf{} the special cases in Examples \ref{ERRT}--\ref{EBST}; we
postpone the proof to \refT{Tancestor} below:
For an ancestor, the ordinary children are born according to a point process
$\Xi'$ which is a pure birth process, with birth rate $w_{k+1}=w_k+\chi$
when the state (number of ordinary children so far) is $k$, and the heir is
born at an age $\heir\sim\Exp(\rho)$, independent of $\Xi'$.
Consequently, 
the limiting random sin-tree $\sinT$ can be constructed as
follows, generalising the constructions in Examples \ref{ERRT}--\ref{EBST}:
First construct 
an infinite chain of ancestors $o\mm 1,o\mm2,\dots$ of $o$
(backwards in time), with the times between their births \iid{} $\Exp(\rho)$;
in other word, $(o\mm{k})_{k\ge1}$ are born according to a Poisson process
with intensity $\rho$ on $(-\infty,0)$. Give each ancestor additional
children according to independent copies of $\Xi'$ (where the intensities
are shifted from $\Xi$, as said above). 
Then,  every other individual gets
children according to independent copies of $\Xi$.
Finally, stop everything at $\taux\sim\Exp(\ga)=\Exp(\chi+\rho)$.

The linear case \eqref{wlinear+-} treated in this example is simpler than
the general 
case in \refE{Epref} in several ways. 
For example, we have shown that the age $\heir$
when the heir of an ancestor is born has an exponential distribution,
and (as said earlier) it will be shown in 
\refT{Tonlylinear} that this holds only in the linear case.
An important reason 
(perhaps the main reason)
that the linear case is simpler is that the total weight
in a tree depends only on the size of the tree: if $|T|=n$, then the total
weight of the nodes in $T$, which we may label by $1,\dots,n$, is
\begin{equation}\label{totalw}
  \sumin w_{\dout(i)}
=
  \sumin (\chi\dout(i)+\rho)
=
 \chi \sumin \dout(i)+n\rho
=\chi(n-1)+n\rho
=n\ga-\chi.
\end{equation}
This property has several important consequences. 
First, it follows 
(as remarked for the random recursive tree and the \bst{} above)
that if $M=|\friT|$ and we condition $\friT$ on $M=n$, we get 
the random tree $T_n$. 
(The property called \emph{coherence} by \citet[Section 2.6]{Aldous}.)
The distribution of $M$ can be found by the same argument as for the \rrt{}
in \refE{ERRT}, which now, using \eqref{totalw} and \eqref{alphalin}, yields
\begin{equation}\label{Mgelin}
  \P(M\ge n) 
= \frac{\rho}{n\ga-\chi}
= \frac{\rho}{(n-1)\ga+\rho}
\end{equation}
and hence
\begin{equation}\label{Mlin}
  \P(M=n)
= \frac{\rho\ga}{((n-1)\ga+\rho)(n\ga+\rho)}
= \frac{\gk}{(n+\gk-1)(n+\gk)}
\end{equation}
with
\begin{equation}
  \gk:=\frac{\rho}{\ga}=\frac{\rho}{\chi+\rho}=\frac{w_0}{w_1}.
\end{equation}
Consequently, $\friT$ can be described as the random tree $T_M$ with a
random size $M$ given by \eqref{Mlin}.

An alternative way to see \eqref{Mlin} is to note that \eqref{totalw}
implies that the size $Z_t=|\ctt|$ of the branching process is a pure birth
process with 
birth rates $\gl_n=n\ga-\chi=n(\chi+\rho)-\chi$, 
and thus $|\ctt|-1$ is a pure birth process with
birth rates $\gl_n=(n+1)\ga-\chi=n\ga+\rho$.
(The special case $\chi=0$, $\rho=1$, when $|\ctt|$ is a Yule process, was
noted in \refR{ERRT}.) 
\refT{TAlin} shows that 
\begin{equation}
|\friT|-1=|\cttaux|-1\sim\HG(\rho/\ga,1;(\rho+\ga)/\ga+1)
=\HG(\gk,1;\gk+2),
\end{equation}
which by \eqref{hg} and \eqref{hgC}, or simpler
by \eqref{waring}, yields \eqref{Mlin}.
Note also that \refT{TApxt} shows that the size $|\ctt|$ at a fixed time,
minus 1, has a negative binomial distribution. 

Furthermore, \eqref{totalw} implies that if we
label the nodes of $T_n$ by $1,\dots,n$ in the order they are added to
the tree, so that $T_n$ becomes an \emph{increasing tree} 
(or \emph{recursive tree} \cite[Section 1.3]{Drmota}),
then the probability that $T_n$ equals a given ordered increasing tree $T$ (with
$|T|=n$) 
is, by the definition and a simple rearrangement,
\begin{equation}
\frac{\prod_{i=1}^n
\bigpar{\frac{1}{\dout(i)!} \prod_{k=0}^{\dout(i)-1} w_k}}
{\prod_{j=1}^{n-1}(j\ga-\chi)}
\end{equation}
which is proportional to $\prod_{i=1}^n \phi_{\dout(i)}$ with
$\phi_d:=\frac{1}{d!}\prod_{k=0}^{d-1}w_k$. 
Hence $T_n$ has the distribution of a 
\emph{simply generated random increasing tree} \cite[Section 1.3.3]{Drmota},
with  weight sequence 
$\phi_d=\frac{1}{d!} \prod_{k=0}^{d-1} w_k$.
Conversely, a simply generated random increasing tree can be generated by a
random evolution where nodes are added one by one
only when its weight sequence is of this form, for some
$w_k$ of the form \eqref{wlinear+-} 
\cite{KubaPanholzer}, \cite{PanholzerProdinger}. 
(Such trees are called \emph{very simple increasing trees} in
\cite{KubaPanholzer}, \cite{PanholzerProdinger}.) 
In other words,
the random increasing tree generated by a general sequence of
weights  $w_k$ (as in \refE{Epref})
is a simply generated increasing tree if and only if
the weights are of the linear type \eqref{wlinear+-}. (I.e., we are in the
case of the present example.)

Finally,
\eqref{totalw} is very useful when using martingale methods (which we do
not do in the present paper). 
\end{example}

\begin{example}[plane oriented recursive tree]\label{EPORT}
  A random \emph{plane oriented recursive tree}, 
 introduced by \citet{Szymanski},
is constructed similarly
  to the random recursive tree in \refE{ERRT}, but we now consider the trees
  as ordered; an existing node with $k$ children thus has $k+1$ position
  in which a new node can be added, and we give all possible positions of the
  new node the same probability. 
The probability of choosing a node $v$ as
  the parent is thus proportional to $\dout(v)+1$, so the plane oriented
  recursive tree is the case $w_k=k+1$ of \refE{Epref}.
This is the special case $\chi=\rho=1$ of \refE{Elinear+-}, and thus 
the special case $\rho=1$ of the following example
(\refE{Elinear+}), where some results and further
references are given.
\end{example}

\begin{example}[positive linear preferential attachment]\label{Elinear+}

Consider the case $\chi=1$ of \eqref{wlinear+-}, \ie,
\begin{equation}\label{wlinear+}
  w_k=k+\rho, \qquad k\ge0,
\end{equation}
where $\rho>0$ is a parameter. 

Thus, $w_k$ is a strictly increasing function of $k$, so this is a model with
preferential attachment as mentioned in \refE{Epref}.
This is a popular model, that has been studied by many authors (often by
methods different from the branching processes used here).
The original preferential attachment model by \citet{BarabasiA} was the case
$\rho=1$, so $w_k=k+1$; thus the probability of attaching a new node to an
existing node $v$ is proportional to $\dout(v)+1=\degree(v)$, 
the total degree of
the node (except for the root). 
As said above, trees of this type had earlier been studied by
\citet{Szymanski}.
(\citet{BarabasiA}  considered a
more general model where a new node may be attached to more than one
existing node, thus creating graphs that are not trees. We only consider the
tree case here.) \citet{Bollobasetal} made a precise formulation of the
definition, and found (and proved rigorously) the asymptotic degree
distribution (in the general, graph case). 
See also \citet[Chapter 8]{Hofstad}, with many details and references.
The tree model with a general $\rho$ was studied by \citet{Mori}.
See also \citet{AthreyaGS} for an extension with multiple edges, treated by an
extension of the methods used here.
\citet{RudasTV} and \citet{RudasT} also used 
the branching process method described here.

In the case \eqref{wlinear+}, \eqref{hmupref} becomes a hypergeometric series
\begin{equation}\label{hmulin1}
  \hmu(\gth)=\sumni\prod_{k=0}^{n-1}\frac{k+\rho}{k+\rho+\gth}
=\FF(\rho,1;\rho+\gth;1)-1,
\end{equation}
where $\FF$ is a hypergeometric function, 
see \eqref{hyper} in \refApp{ASShyper};
the series converges for $\gth>1$, and then 
\eqref{hmulin1} and \eqref{gauss} yield
\begin{equation}\label{hmulin}
  \hmu(\gth)
=\frac{\gG(\rho+\gth)\gG(\gth-1)}{\gG(\rho+\gth-1)\gG(\gth)}-1
=\frac{\rho+\gth-1}{\gth-1}-1=\frac{\rho}{\gth-1},
\end{equation}
as we have seen by another method in \eqref{hmulin+-}.
Consequently, 
or by \eqref{alphalin},
the Malthusian parameter is
\begin{equation}
  \label{ga+}
\ga=\rho+1. 
\end{equation}

The asymptotic degree distribution is by \eqref{prefD}--\eqref{prefD=} 
and \eqref{ga+}
given by
\begin{equation}\label{pdgelin}
  \P(D\ge i) 
= \prod_{k=0}^{i-1} \frac{k+\rho}{k+2\rho+1}
=\frac{\gG(2\rho+1)\gG(i+\rho)}{\gG(\rho)\gG(i+2\rho+1)},
\qquad i\ge0,
\end{equation}
and thus
\begin{equation}\label{pdlin}
  \begin{split}
\P(D=i)&
=  \P(D\ge i) -\P(D\ge i+1)
=\frac{(\rho+1)\gG(2\rho+1)\gG(i+\rho)}{\gG(\rho)\gG(i+2\rho+2)}	,
\quad i\ge0.
  \end{split}
\end{equation}
This is the hypergeometric distribution $\HG(\rho,1;2\rho+2)$, see \refD{DHG}.
(This also follows from \refT{TAlin}, using \refE{EA1} and \eqref{ga+}.)
This degree distribution has a power-law tail: \eqref{pdlin}
implies,
see \refT{THGtail}, 
\begin{equation}
\P(D=i)\sim c(\rho) i^{-\rho-2}, \qquad \text{as }i\to\infty,   
\end{equation}
for the constant
$c(\rho)=(\rho+1)\gG(2\rho+1)/\gG(\rho)$, as shown by \citet{Mori}.
(Note that this power-law is quite sensitive to the choice of $w_k$, with
the exponent depending on the constant term $\rho$ in \eqref{wlinear+}.)

In the special case $\rho=1$, 
$D\sim\HG(1,1;4)$ and
\eqref{pdlin} becomes
\begin{equation}
  \P(D=i)=\frac{4}{(i+1)(i+2)(i+3)},
\qquad i\ge0,
\end{equation}
found by \citet{Szymanski} (showing the annealed version);
see also \citet{MahmoudSS}, \citet{LuFeng},  
\citet{Bollobasetal}  (quenched version),
and \citet{SJ155}.

By \eqref{heirD} and \eqref{pdgelin}, 
\begin{equation}\label{qilin}
q_i=  \P(D\ge i) 
= \prod_{k=0}^{i-1} \frac{k+\rho}{k+2\rho+1}
= \frac{\rho}{i+2\rho}\prod_{k=0}^{i-2} \frac{k+\rho+1}{k+2\rho+1}
\qquad i\ge1,
\end{equation}
which in the special case $\rho=1$ simplifies to 
\begin{equation}
  q_i=\frac{2}{(i+1)(i+2)},
\qquad i\ge1.
\end{equation}
Again using \refD{DHG}, this says that if $I$ is the index of the heir of an
ancestor, then $I-1\sim\HG(\rho+1,1;2\rho+2)$.

The limiting random sin-tree is given by the construction in \refE{Elinear+-}.
\end{example}

\begin{example}[\mary{} increasing tree, 
negative linear preferentialattachment] 
\label{Elinear-}

We may generalise the binary case \refE{EBST} and grow a random \mary{}
tree as follows, for any $m\ge2$. 
This is sometimes called an \emph{\mary{} increasing tree}.
Note that for $m>2$, this will \emph{not} give the \mst{}
defined in \refS{Smst-def}.
(One difference is that we here fix the number of nodes to
be $n$, while the \mst{} has a random number of nodes,
but this is a minor technicality, see \refR{Rfixednodes}. 
A more essential difference is seen in the asymptotic degree distribution
$D$, see \refT{TMSTdegree})

Start with a single node. Let each node have $m$ positions for children,
labelled $1,\dots,m$. Add each new node to an empty child position in the tree,
chosen uniformly at random. (We may, as in \refS{Smst-def}, regard the empty
child positions as external nodes.)

Since a node with outdegree $d$ has $m-d$ empty positions for children, this
is an instance of the general preferential attachment in \refE{Elinear+-},
with
\begin{equation}\label{wlinear-}
  w_k=m-k, \qquad k=0,\dots, m.
\end{equation}
This is thus the case $\chi=-1$ of the linear case in \refE{Elinear+-}
(with $\rho=m$),
so all results there hold. In particular, 
by \eqref{hmulin+-}--\eqref{alphalin},
$\mu$ has density $m e^{-t}$,
\begin{equation}\label{hmu-}
  \hmu(\gth)=\frac{m}{\gth+1}
\end{equation}
and
\begin{equation}\label{ga-}
  \ga=m-1.
\end{equation}

Also in the case \eqref{wlinear-}, 
\eqref{hmupref} becomes a hypergeometric series; in this case we obtain, 
\cf{} \eqref{hmulin1} and \eqref{hyper},
\begin{equation}\label{hmulin-}
  \hmu(\gth)=\sumni\prod_{k=0}^{n-1}\frac{m-k}{m-k+\gth}
={}_2F_1(-m,1;-m-\gth;1)-1.
\end{equation}
(This is a case where the hypergeometric series is finite.)
Gauss' formula \eqref{gauss} yields another proof of \eqref{hmu-}.

The asymptotic degree distribution is by \eqref{prefD}--\eqref{prefD=}  
and \eqref{ga-}
given by
\begin{equation}\label{pdge-}
  \P(D\ge i) 
  = \prod_{k=0}^{i-1} \frac{m-k}{2m-1-k}
=\frac{m!\,(2m-1-i)!}{(2m-1)!\,(m-i)!}
=\frac{\binom{2m-1-i}{m-1}}{\binom{2m-1}{m-1}}
\end{equation}
and thus
\begin{equation}\label{pd-}
  \begin{split}
\P(D=i)&
=\frac{(m-1)m!\,(2m-2-i)!}{(2m-1)!\,(m-i)!}
=\frac{\binom{2m-2-i}{m-2}}{\binom{2m-1}{m-1}}	,
\qquad i=0,\dots,m.
  \end{split}
\end{equation}
Thus, using \refD{DHG} or \refR{RHG}, $D$ has the hypergeometric
distribution $\HG(-m,1,2-2m)$. 
(Again, this also follows by \refT{TAlin} and \refE{EA1}.)
Note that the distribution \eqref{pd-} 
is the same as the distribution of balls in a given box when $m$
indistinguishable balls are distributed in $m$ distinguishable boxes, with
all distributions having the same probability 
(Bose--Einstein statistics, see \refE{EBose}); 
equivalently,
it is the distribution of the
first term in a random composition of $m$ into $m$ (possibly empty) parts.

It follows easily 
from \eqref{pd-} that as \mtoo, the distribution of $D$ converges
to the Geometric distribution $\Geo(1/2)$, which is the distribution of $D$ for
the \rrt. (Cf.\ \cite[Example 12.2 and Theorem 11.7]{SJ264}.)

The point process $\Xi$ contains $m$ points, with successive
exponential waiting times with rates $m$, $m-1$, \dots,1. 
As is well-known, this process can
also be constructed by taking $m$ \iid{}
$\xxi_1,\dots,\xxi_m\sim\Exp(1)$ and ordering them as
$\xi_1<\dots<\xi_m$.
Since the construction of the \mary{} tree also involves randomly labelling
the children, it follows that if $\xxi_i$ denotes the age when child at
position $i$ is born, then 
$\xxi_1,\dots,\xxi_m$ are \iid{} $\Exp(1)$.
The growing tree  $\cT_t$ is thus the subtree of the (rooted) infinite \mary{}
tree, where each child of each node is born after an $\Exp(1)$ waiting time
(with all these waiting times independent).

Similarly, for an ancestor, the process $\Xi'$ of its ordinary children
described in \refE{Elinear+-} and \refT{Tancestor} simply
consists of $m-1$ \iid{} $\Exp(1)$ points. Furthermore, the age $\heir$ when
the heir is born is $\Exp(m)$ and independent of $\Xi'$. Consequently,
the description of the limiting random sin-tree $\sinT$ in \refE{Elinear+-} 
can be simplified as follows, \cf{} the binary case in \refE{EBST}:
Construct an infinite chain of ancestors $o\mm 1,o\mm2,\dots$ of $o$
(backwards in time), with the times between their births \iid{} $\Exp(m)$;
in other word, $(o\mm{k})_{k\ge1}$ are born according to a Poisson process
with intensity $m$ on $(-\infty,0)$. Moreover, make a random choice (uniform
and independent of everything else) for each ancestor to decide which of 
its $m$ children that is its heir.
Then grow independent \mary{} tree processes at all empty places (external
nodes), with independent $\Exp(1)$ waiting times for all new nodes.
Finally, stop everything at $\taux\sim\Exp(m-1)$.
\end{example}

The  examples above are all cases of \refE{Epref}, and all except the
general \refE{Epref} itself are special cases of
\refE{Elinear+-}. We have seen that in the latter cases, the age $\heir$
when the heir is born to an ancestor has an
exponential distribution, and is independent of the births of the other
children. 
We give a simple example showing that this is not always the case.

\begin{example}[Binary pyramids]\label{E11}\label{Epreflast}
Let  $w_0=w_1=1$ and $w_k=0$ for $k>1$. Thus no node ever gets more than 2
children, and we can regard the result as a binary tree by randomly
labelling children as left or right as in \refE{EBST}, but the difference
is that we here have $w_0=w_1$; hence, when adding a new node, the parent of
the new node is chosen uniformly among all existing nodes with less than 2
children. (I.e., as in \refE{ERRT} but with a cut-off at 2 children.)
This random tree was called a \emph{binary pyramid} by
\citet{Mahmoud:pyramids} who studied their height.
(The name comes from pyramid schemes for chain letters, see
\citet{GastBhatt84}.
As said in \cite{Mahmoud:pyramids}, the definition can be generalized to an
arbitrary cut-off $m\ge2$; we leave this case to the reader.)

We have $\xi_1=X_1\sim\Exp(1)$ with density $e^{-x}$ and
$\xi_2=X_1+X_2\sim\gG(2,1)$ with density $xe^{-x}$. Hence, the intensity
$\mu$ has density $(1+x)e^{-x}$ and Laplace transform, by \eqref{hmub} or
\eqref{hmupref}, 
\begin{equation}\label{hmupyr}
  \hmu(\gth)=\frac{1}{1+\gth}+\frac{1}{(1+\gth)^2},
\qquad \gth>-1.
\end{equation}
Hence \eqref{malthus} is satisfied with $\ga=(\sqrt5-1)/2$
(the inverse golden ratio).
By \refR{Rage}, the age $\heir$ when an heir is born has the density 
$(1+t)e^{-(1+\ga)t}=(1+t)e^{-\frac12(\sqrt5+1)t}$,
and thus, or by \eqref{el2}, 
$\gb=\E\heir=(3\sqrt5-5)/2$.

Furthermore, by \eqref{heir1},
\begin{align}
 q_1&=1/(1+\ga)=\ga=(\sqrt5-1)/2,
\\
q_2&=1/(1+\ga)^2=\ga^2=1-\ga=(3-\sqrt5)/2.
 \end{align}
Thus, by \eqref{heir}, 
$\tXi$, describing the life history of an ancestor, can be described as
a mixture: with probability $q_1=(\sqrt5-1)/2$, an heir is born at age
$\txi_1\sim\Exp(1+\ga)$, and then another child is born after an
independent waiting time $\txi_2-\txi_1\sim \Exp(1)$;
with probability $1-q_1$, first another child is born at  age 
$\txi_1\sim\Exp(1+\ga)$, and then an heir is born after an
independent waiting time $\txi_2-\txi_1\sim \Exp(1+\ga)$.
We obtain also, by this or
directly from \eqref{heir}, the joint density $f(x,y)$ of
the ages when the ordinary  child and the heir is born as
\begin{equation}
  f(x,y)=
  \begin{cases}
	e^{-x-\ga y}, &0<y<x,\\
e^{-(1+\ga)y},&0<x<y.
  \end{cases}  
\end{equation}
Consequently, the two births are not independent (unlike the linear case
in \refE{Elinear+-}).

Since $\heir$ is not exponential, the times of births of the ancestors
$o\mm1, o\mm2, \dots $ do not form a Poisson process on $(-\infty,0)$. 

The asymptotic degree distribution is by \eqref{heirD} given by 
$\P(D=0)=\P(D=2)=1-\ga=(3-\sqrt5)/2$, $\P(D=1)=2\ga-1=\sqrt5-2$.
\end{example}

We end this section by proving some claims made above. 
First the ancestor in the linear case.

\begin{theorem}\label{Tancestor}
  For the linear preferential attachment in \refE{Elinear+-}, 
with weights $w_k=\chi k+\rho$,
the life history $\tXi$ of an ancestor 
consists of an heir born at age $\heir\sim\Exp(\rho)$ together with
ordinary children born according to a pure birth process $\Xi'$, with rate
$w_{k+1}$ when there are $k$ ordinary children, 
with $\heir$ and $\Xi'$ independent.
\end{theorem}
\begin{proof}
Consider an ancestor,
let $\heir$ be its age when the heir is born, and denote its age at the births
of the other children by $\xi_1'<\xi_2'<\dots$. (Also, let $\xi'_0=0$.)  
Let $\cE_i$ be the event that the heir is child $i$. Thus $\cE_i$ \as{}
equals the event \set{\xi'_{i-1}<\heir<\xi'_i}; furthermore, 
if this event holds, then
$\heir=\txi_i$ and
\begin{equation}\label{xi'}
  \xi_j'=
  \begin{cases}
	\txi_j,&j<i,\\
\txi_{j+1},& j\ge i.
  \end{cases}
\end{equation}

Fix $i$ and $M>i$. 
For an ordinary individual,
the joint distribution of $(\xi_1,\dots,\xi_{M+1})$ 
has density, on \set{0<x_1<\dots<x_{M+1}} and with $x_0=0$,
\begin{equation}
  \begin{split}
  \prod_{j=0}^{M}w_{j}e^{-w_{j}(x_{j+1}-x_{j})}
&=\prodwk
  \cdot e^{\sum_{j=1}^M(w_j-w_{j-1})x_j-w_Mx_{M+1}}	
\\
&=
  \prodwk\cdot e^{\sum_{j=1}^M \chi x_j-w_Mx_{M+1}}	
.
  \end{split}
\end{equation}
Hence, 
for an ancestor,
 \eqref{heir} shows that
restricted to the event $\cE_i$, the joint density of 
$(\txi_1,\dots,\txi_{M+1})$ is,
on \set{0<x_1<\dots<x_{M+1}},
\begin{equation}
  \begin{split}
e^{-\ga x_i}
  \prodwk\cdot e^{\sum_{j=1}^M \chi x_j-w_Mx_{M+1}}	
=
  \prodwk\cdot e^{\chi\sum_{j=1}^M  x_j-w_Mx_{M+1}-\ga x_i}	
  \end{split}
\end{equation}
and, using \eqref{xi'}, the joint distribution of 
$(\xi'_1,\dots,\xi'_M,\heir)$ on the set
$\set{0<x_1<\dots<x_{i-1}<y<x_i<\dots<x_M}$ is, recalling \eqref{alphalin}
and $w_0=\rho$,
\begin{equation}
  \begin{split}
  \prodwk\cdot 
e^{\chi\lrpar{\sum_{j=1}^{M-1} x_j+ y}-w_Mx_{M}-\ga y}	
&=
w_0  \prod_{k=1}^{M}w_{k}\cdot 
e^{\sum_{j=1}^{M-1}\chi x_j-w_Mx_{M}-\rho y}	
\\
&=
\rho e^{-\rho y}  \prod_{k=1}^{M}w_{k}\cdot 
e^{\sum_{j=1}^{M-1}(w_{j+1}-w_j) x_j-w_Mx_{M}}	
\\
&=
\rho e^{-\rho y}  
\prod_{j=0}^{M-1}w_{j+1}e^{-w_{j+1}(x_{j+1}-x_{j})}
.
  \end{split}
\raisetag{1.5\baselineskip}
\end{equation}
This equals the joint density  of the first $M$ points of the birth process
$\Xi'$ defined in the statement, together with an independent
$\heir\sim\Exp(\rho)$.  
The result follows, since $M$ is arbitrary.
\end{proof}

We have shown in \refE{Elinear+-} that the age $\heir$ when the heir is born
to an ancestor has an exponential distribution in the linear case.
We now show the converse: this happens only in the linear case.
(Recall that if $w_m=0$ for some $m$, the weights $w_{k}$ for $k>m$ are
irrelevant.) 
\begin{theorem}\label{Tonlylinear}
Consider a general preferential attachment tree defined as in \refE{Epref}
by a sequence $(w_k)_0^\infty$ of weights. If the age $\heir$
when an ancestor gets
an heir has an exponential 
distribution, then  $w_k=\chi k+\rho$ for some $\chi\in\bbR$ and $\rho>0$
(at least until $w_k=0$, if that ever happens).  
\end{theorem}

\begin{proof}
  The Laplace transform of $\heir$ is by \eqref{rage} $\hmu(s+\ga)$. Hence,
  if $\heir\sim\Exp(\gl)$ for some $\gl>0$, then for all $s>0$,
  \begin{equation}
	\hmu(s+\ga)=\E e^{-s\heir}=\frac{\gl}{\gl+s}.
  \end{equation}
Consequently, by \eqref{hmupref},
\begin{equation}\label{rud0}
\wws0+\wws0\cdot\wws1+\dots = \hmu(s)=\frac{\gl}{\gl-\ga+s},
\end{equation}
at least for $s>\ga$.

Consider, more generally, the equation
\begin{equation}\label{rud}
\wws0+\wws0\cdot\wws1+\dots=\frac{a}{b+s}.
\end{equation}
for some real $a$ and $b$, and all large $s$.
Multiply \eqref{rud} by $(w_0+s)/w_0$. This yields
\begin{equation}\label{olof}
1+\wws1+\wws1\cdot\wws2+\dots = \frac{a(w_0+s)}{w_0(b+s)}.
\end{equation}
Now let $s\to\infty$. On the \lhs, each term except the first decreases to
0, and by dominated convergence, the sum converges to $1+0+\dots$; thus
\eqref{olof} implies
\begin{equation}
1=
1+0+\dots 
=\lim_{s\to\infty} \frac{a(w_0+s)}{w_0(b+s)} =\frac{a}{w_0}.
\end{equation}
Consequently, \eqref{rud} implies $w_0=a$.
Use this in \eqref{olof} and subtract 1 to obtain
\begin{equation}\label{rud2}
\wws1+\wws1\cdot\wws2+\dots = \frac{w_0+s}{b+s}-1
=\frac{w_0-b}{b+s}.
\end{equation}
If $w_1=0$, then the \lhs{} vanishes, so $w_0-b=0$. Otherwise, \eqref{rud2}
is of the same type as \eqref{rud}, with the weights $(w_k)$ shifted to
$(w_{k+1})$, and $a$ replaced by $w_0-b$.
Hence, the argument above yields $w_1=w_0-b$.

Thus, in both cases $w_1=w_0-b$. Moreover, if $w_1\neq0$, we can iterate the
argument, and find $w_2=w_1-b$, $w_3=w_2-b$, and so on, as long as the
weights are non-zero. Thus $w_k=w_0-kb=\chi k+\rho$, with $\chi=-b$ and
$\rho=w_0$.  
\end{proof}

Finally, we prove the result claimed above in the explosive
case \eqref{explosion}.

\begin{theorem}\label{Texp}
Let $T_n$ be a
general preferential attachment tree, defined by a sequence $w_k$,
and assume that the explosion condition \eqref{explosion} holds. Then
$\P(|\Tx_n|=1)\asto1$.
\end{theorem}

\begin{proof}
Let $T_\infty:=\cT_{\tau(\infty)}$, the (infinite) tree obtained by stopping
when 
the process explodes. Thus $T_n\subset T_\infty$ for every $n$.
Let, for $1\le i\le n\le \infty$, $I_{i,n}:=\ett{\dout_{T_n}(v_i)>0}$,  
the indicator of the event that the $i$:th node 
(in order of appearance)
$v_i$ has at least one
child in $T_n$.

Fix $\gd>0$, and let $\cE_{i,\gd}$ be the event that the $i$:th individual
(in order of birth) in the branching process gets at least one child before
age $\gd$, \ie, that it has $\xi_1<\gd$. Further, let
$J_{i,\gd}:=\ett{\cE_{i,\gd}}$. The events $\cE_{i,\gd}$ are independent
and have the same probability $\P(\xi_1<\gd)=\P(X_1<\gd)$. 
Thus, by the law of large numbers,
\begin{equation}
\frac{\sumin J_{i,\gd}}{n}\asto\P(X_1<\gd)=1-e^{-w_0\gd}<w_0\gd.
\end{equation}
Furthermore, \as{} $\tau(\infty)<\infty$, and then $\gs_i>\tau(\infty)-\gd$
for all but a finite number of $i$, \ie, all but a finite number of
individuals have age less than $\gd$ when the process explodes. Hence,
$I_{i,\infty}\le J_{i,\gd}$ for all  but a finite number of $i$ and, \as,
\begin{equation}
  \limsup_\ntoo \frac{\sumin I_{i,\infty}}n
\le
\lim_\ntoo \frac{\sumin J_{i,\gd}}{n}<w_0\gd.
\end{equation}
Since $\gd>0$ is arbitrary, this shows 
$  \limsup_\ntoo \xfrac{\sumin I_{i,\infty}}n=0$ a.s.
Furthermore, the finite tree $T_n$ is a subtree of $T_{\infty}$; hence,
a.s.,
\begin{equation}
  \P\bigpar{|\Tx_n|>1}=\frac{\sumin I_{i,n}}{n}
\le \frac{\sumin I_{i,\infty}}{n} \to0
\end{equation}
\end{proof}

\section{$m$-ary search trees and branching processes}\label{Smst}

In this section, as always when we discuss \mst{s}, $m$ is a fixed integer
with $m\ge2$.
We apply the general theory in \refS{S:BP} to the \mst{} in
\refS{Smst-def}.
Recall from \refS{Smst-def} that besides the \mst, we may also consider the
\emst{} (including external nodes).
It turns out that both versions can be described by stopped branching
processes.
It is easy to go between the two versions, but we find it instructive to
treat them separately, and describe the two related but different branching
processes connected to them. The reader is recommended to compare the two
versions, even when we do not explicitly do so.

\begin{remark}\label{Rfixednodes}
  The random \mst{} is defined as in \refS{Smst-def} to have a given number
  of keys, which makes the number of nodes random (in general).
We can also define a random \mst{} with a given number of nodes, by adding
keys until the desired number of nodes is obtained. This is obtained by the
branching processes below, stopping when the number of nodes is a given
number $n$; we thus use the weight $\psi(t)=1$ in \refE{E1} (as in
\refS{Sex}).
The asymptotics are the same for this version, see \refR{Rsame}.
We therefore ignore this version in the sequel, and consider only the
standard version with a given number of keys.
\end{remark}

\subsection{Extended \mst} \label{SSemst}

Recall from \refS{Smst-def}
that we can grow an extended $m$-ary search tree by starting with an
empty 
tree (a single external node) and then adding keys, each new key added with
equal probability to each existing gap. Hence, we can also grow the \emst{}
in continuous time by adding a key to each gap after an exponential
$\Exp(1)$ waiting time (independent of everything else). 
By the construction of the \emst{} in \refS{Smst-def},
this is a \CMJ{} branching process, where the life of each individual is as
follows (\citet{Pittel1994}):

An individual is born as an external node with no keys. It acquires $m-1$  keys
after successive independent waiting times $Y_1,\dots,Y_{m-1}$, where
$Y_i\sim\Exp(i)$ 
(since the node has $i$ gaps when there are $i-1$ keys).
When the $(m-1)$:th key arrives, the individual immediately gets $m$
children.

We let $\psi(t)$ be the number of keys stored at the individual at age $t$. Thus
$\zpsi_t$ is the total number of keys at time $t$ and $\tau(n)$ is the time
the $n$:th key is added. Hence $T_n$ is a random \mst{} with $n$ keys, as we
want. 

Let $S_k:=\sum_{i=1}^k Y_i$, $k=0,\dots,m-1$;
for $1\le k\le m-1$, this is the time the $k$:th key arrives.
Let further $S_m:=\infty$.
Then $\psi(t)=k$ for $S_k\le t<S_{k+1}$.
 For $\gth\ge0$ (in fact, for $\gth>-1$)
and $k\le m-1$, 
\begin{equation}\label{esk}
\E e^{-\gth S_{k}}
=\prod_{i=1}^{k} \E e^{-\gth Y_i}	
=\prod_{i=1}^{k} \frac{i}{i+\gth}
=\frac{k!}{\prod_{i=1}^{k}(i+\gth)}.
\end{equation}
(See also \refT{TC}, which further gives more the distribution of
$S_k$;  in the notation used in \refApp{AC}, $S_k\eqd V_{k,k}$.)

Furthermore,
all children are born at the same time with
$\xi_1=\dots=\xi_m=S_{m-1}$, and thus
the random variable $\Xib$ in
\refR{RBPxib}  equals $me^{-\gth S_{m-1}}$. 
Hence, see \eqref{hmub} and \eqref{esk}, 
\begin{equation}\label{hmumst}
  \begin{split}
\hmu(\gth) = m \E e^{-\gth S_{m-1}}
=\frac{m!}{\prod_{i=1}^{m-1}(i+\gth)},
\qquad \gth>-1.
  \end{split}
\end{equation}
In particular, we see that $\hmu(1)=1$, so 
the Malthusian condition \eqref{malthus} is satisfied with $\ga=1$.
It is easy to see that all other conditions \refBP{}
 are satisfied. (Note that in this case, $N=m$ is non-random.
Furthermore, $\psi$ is bounded, so \refBPlastpsi{} holds too.)
Consequently, \refT{TBP} applies, and shows (in particular) that the
random fringe tree 
$\Tx_n$ converges in distribution to $\friT$, which is obtained by
running the branching process above and stopping it after a
random time $\taux\sim\Exp(1)$.

Similarly, \refT{TBPX} applies. In order to find the sin-tree $\sinT$, note
that since all children of an individual are born at the same time, 
so $\xi_1=\dots=\xi_m=S_{m-1}$,
it does not matter which one is the heir. It thus follows from \eqref{heir},
that if we let $\tY_1,\dots,\tY_{m-1}$ be the successive waiting times
between the arrival of keys for an ancestor, 
so all $m$ children are born at time
$\txi=\sum_{i=1}^{m-1}\tY_i$, then 
$\tY_1,\dots,\tY_{m-1}$  have joint density
\begin{equation}
  me^{-(y_1+\dots+y_{m-1})}\prod_{i=1}^{m-1} i e^{-iy_i}
=
m!\prod_{i=1}^{m-1}  e^{-(i+1)y_i}
= \prod_{i=1}^{m-1} (i+1) e^{-(i+1)y_i}.
\end{equation}
Thus $\tY_1,\dots,\tY_{m-1}$ are independent with $\tY_i\sim\Exp(i+1)$.
(Cf.\ \eqref{rrttX} and the proof of \refT{Tancestor}, with similar
calculations in  different but related situations.) 
The $m$ children are numbered $1,\dots,m$, with the heir chosen uniformly at
random among them.

\begin{remark}
The age $\heir$ when the heir is born is $\heir=\sum_{i=1}^{m-1}\tY_i$, with
mean 
\begin{equation}\label{gbemst}
\gb=
\E \heir=\sum_{i=1}^{m-1}\E \tY_i=\sum_{i=1}^{m-1}\frac{1}{i+1}=H_m-1,  
\end{equation}
where $H_m:=\sum_1^m\frac{1}i$ denotes the $m$:th harmonic number.
(See also \eqref{el2} and \eqref{hmumst}.)
The distribution of $\heir$ is given by \refT{TC}; using the notation there
\begin{equation}\label{heiremst}
\heir=\sum_{i=1}^{m-1}\tY_i\eqd\sum_{i=2}^m Y_i \eqd V_{m,m-1}.  
\end{equation}
In particular, $\heir$ is \emph{not} exponentially distributed unless $m=2$.
\end{remark}

In the construction of $\ctt$ above, the number of gaps is always $1$ $+$ the
number of keys, and we add keys (and thus gaps) with an intensity equal to
the number of gaps. Hence, the number of gaps at time $t$ forms a pure birth
process with birth rates $\gl_k=k$, starting at 1 (this is again the Yule
process in \refE{ERRT}, see \refE{EYule1}), 
and thus the number of keys at time $t$ forms a
pure birth process with birth rates $\gl_k=k+1$, starting at 0.
(Note that this is independent of the choice of $m$.)

Since $\taux$ has the same distribution $\Exp(1)$ here as in \refE{ERRT}, it
follows that the number of gaps in $\friT=\cttaux$ has the same distribution
as the number of nodes $M$ in $\cttaux$ in \refE{ERRT}, given by \eqref{rrtM}. 
Moreover, by symmetry, conditioned on the number of keys $K=k$ in
$\cttaux=\friT$, $\friT$ has the same distribution as the 
random \emst{} $T_k$ with $k$ keys. Hence, we get the following result: 

\begin{theorem}\label{fringedistributionextended}
 The number $K$ of keys in 
the asymptotic fringe tree 
$\friT$ has the distribution
\begin{equation}\label{mstkeys}
  \P(K=k)=\P(M=k+1)=\frac{1}{(k+1)(k+2)},
\qquad k\ge0.
\end{equation}
Furthermore,
$\friT$ can also be described as an \emst{} with
a random number $K$ keys, where $K$ has the distribution \eqref{mstkeys}.
\nopf
\end{theorem}
 
\begin{remark}\label{REMSTgaps} 
Using the notation in \refD{DHG}, $K\sim\HG(1,1;3)$.
(This also follows from \refT{TAlin}, with $\chi=\rho=\ga=1$.)

The property in the second part of the theorem, describing the asymptotic
fringe tree $\friT$ as an \emst{} with 
a random number $K$ keys is called \emph{coherence} by 
\citet[Section 2.6]{Aldous},
and was seen also in \refE{Elinear+-}.
(In the present case with respect to the number of keys; we might call this
\emph{key-coherent}.) 
\end{remark}

We proceed to derive some properties of the random \emst{} $T_n$.
Note that, unlike the examples in \refS{Sex}, $T_n$ does not have $n$ nodes;
$n$ is the number of keys, while the number of nodes is random for $m\ge3$.
(For $m=2$, the number of nodes is $2n+1$, of which $n$ are internal, see
\refS{Smst-def}.)
To find the asymptotic number of nodes, we use \refT{Tex} and obtain the
following result.

\begin{theorem}\label{TEMSTnodes}
  For the \emst{} $T_n$ with $n$ keys,
\begin{equation}\label{kull}
  \frac{|T_n|}n\asto \frac{1}{\mpsi}=\frac{1}{H_m-1}.
\end{equation}
\end{theorem}

 The asymptotic value of the expectation $\E |T_n|/n$ was found by
\citet{Baeza-Yates}. We do not know any reference where \eqref{kull} is
stated explicitly, but closely related results for the number of internal
nodes have been shown in several papers, see \refR{RMSTnodes}; the result
follows also immediately from the main result by \citet{KalMahmoud:Degree}.

\begin{proof}
This follows from \refT{Tex}\ref{tex2}, except for the value of $\mpsi$,
which we calculate as follows.
Since $\psi(t)=\sum_{i=1}^{m-1}\ett{S_i\le t}$, 
\begin{equation}\label{aku}
\hpsi(1)=\intoo e^{-t}\psi(t) \dd t 
=\sum_{i=1}^{m-1}\intoo e^{-t}\ett{S_i\le t} \dd t 
=\sum_{i=1}^{m-1} e^{-S_i}.
\end{equation}
By \eqref{esk},
\begin{equation}\label{esk1}
\E e^{-S_i}
=\frac{i!}{(i+1)!}
=\frac{1}{i+1},
\qquad i=0,\dots, m-1,
\end{equation}
and thus, 
\begin{equation}\label{akull}
\mpsi=
\E\hpsi(1)
=\sum_{i=1}^{m-1}\E e^{-S_i}
=\sum_{i=1}^{m-1}\frac{1}{i+1}
=H_m-1.
\end{equation}
\end{proof}

\begin{theorem}\label{TEMSTkeys}
Let $N_k(T_n)$ be the number of nodes in $T_n$ with
$k$ keys, for $k=0,\dots,m-1$. 
Then,
\begin{equation}\label{nk}
  \begin{split}
\frac{N_k(T_n)}{|T_n|}
\asto
\begin{cases}
\frac{1}{(k+1)(k+2)}, & 0\le k<m-1,	
\\
\frac{1}{m}, & k=m-1.
\end{cases}
  \end{split}
\end{equation}
\end{theorem}

Again,
we do not know any reference where this is stated explicitly;
the asymptotic values of the expectations $\E N_k(T_n)/n$ were found by
\citet{Baeza-Yates}; see also the references in \refR{RMSTnodes}.
The result can also easily be shown using \Polya{} urns, see 
\cite[Example 7.8]{SJ154},
\cite{KalMahmoud:Degree} and
\cite{HolmgrenJanson2}.

\begin{proof}
 $N_k(T_n)=n_\cP(T_n)$, where $\cP$ is the property of a node $v$ that
it contains $k$ keys. Hence, $p_\cP(\friT)$ is the probability that the root
of $\friT=\cttaux$ contains $k$ keys, \ie, that $\psi(\taux)=k$ or,
equivalently, 
$S_k\le\taux<S_{k+1}$. 
We apply \refT{TBP}, and note that the characteristic $\phi$ there is
$\gf(t)=\ett{\psi(t)=k}=\ett{S_k\le t<S_{k+1}}$.
By \eqref{ak}, arguing similarly to \eqref{aku}--\eqref{akull} and in
particular using \eqref{esk1}, and recalling
that 
$S_m:=\infty$, 
\begin{equation}
  \begin{split}
p_\cP(\friT)=
\E\hphi(1)
&
=\E\intoo e^{-t}\ett{S_k\le t<S_{k+1}} \dd t 
= \E\bigpar{e^{-S_k}-e^{-S_{k+1}}}
\\&
=
\begin{cases}
\frac{1}{k+1}-\frac{1}{k+2}, & 0\le k<m-1,	
\\
\frac{1}{k+1}, & k=m-1.
\end{cases}
  \end{split}
\raisetag{\baselineskip}
\end{equation}
The result follows by  \refT{TBP}.
(Alternatively, one can use \refT{TA1}.)
\end{proof}

\begin{remark}
In particular, the fraction of external nodes 
\begin{equation}
  \label{n0}
\frac{N_0(T_n)}{|T_n|}\asto\frac12.
\end{equation}  
and thus the same holds for the number of
internal nodes;
the numbers of external and internal nodes are thus asymptotically the same.
(Perhaps surprisingly, the asymptotic fractions of external and  internal
nodes are thus independent of $m$.)
\end{remark}

\begin{remark}
The asymptotic degree distribution $D$ is not very interesting for the
\emst, 
since every internal node has outdegree $m$ and every external node has
outdegree 0;  
thus, as a corollary of \eqref{n0},
$\P(D=0)=\P(D=m)=1/2$.  
\end{remark}

\subsection{\mst, internal nodes only}\label{SSmst}

Usually, we consider an \mary{} search tree as consisting only of the
internal nodes. This can be obtained from the tree with external nodes in
\refSS{SSemst} by deleting all external nodes, but it may also be constructed
directly as follows, using a different \CMJ{} process.

We now start with a node containing a single key.
Thus each individual is born as a  node with 1 key. It acquires more keys
after successive waiting times $Y_2,\dots,Y_{m-1}$, where $Y_i\sim\Exp(i)$.
At the  arrival of the $(m-1)$:th key, at time
$\SS:=\sum_{i=2}^{m-1}Y_i$, 
the individual becomes fertile;
it then gets $m$ children,
marked by $1,\dots,m$,
with child $i$ born after a further waiting time $X_i$,
\ie, at time $\xi_i=\SS+X_i$,
where $X_1,\dots,X_m$ are independent and $\Exp(1)$. 
(The children are here marked by their final positions in the tree, 
see \refR{Rorder}.) 

Alternatively, taking the children in order of birth,
we may say that after the $(m-1)$:th key, there
are $m$ children born after successive waiting times $X_1',\dots,X_m'$, with
$X_i'\sim\Exp(m+1-i)$, all waiting times independent.

We let again the weight $\psi(t)$ be the number of keys at time $t$
in an individual. 
It is easy to see that then
$T_n$ is a random \mst{} with $n$ keys, as defined in \refS{Smst-def}.

The random variable $\Xib$ in
\refR{RBPxib} is now given by
\begin{equation}
\Xib=
\sum_{j=1}^m e^{-\gth(\SS+X_j)}
=
\sum_{j=1}^m e^{-\gth(\sum_{i=2}^{m-1}Y_i+X_j)}.   
\end{equation}
Its distribution
is not the same as in \refSS{SSemst}, but the mean $\E\Xib=\hmu(\gth)$
is easily seen to be the
same as in \eqref{hmumst}, and thus we still have $\ga=1$;
similarly, by \eqref{el2}, $\gb$ is the same as in \eqref{gbemst},
\ie,
\begin{equation}\label{gbmst}
\gb=
H_m-1.
\end{equation}
(That $\ga$ has to be the same for the \mst{} with and without external nodes is
rather obvious, since the number of internal nodes in $\cT_t$
is the same for both versions, and grows like $e^{\ga t}$ by \eqref{olle2}
and \eqref{n0}.)  
The conditions \refBP{} and \refBPlastpsi{} are satisfied, and
\refT{TBP} shows that the random fringe tree
$\Tx_n$ converges in distribution to $\friT$, which is obtained by
running this branching process and stopping it after 
a random time $\taux\sim\Exp(1)$.

Moreover, the random sin-tree $\sinT$ is constructed
by the general procedure in \refS{S:BP}.
To find the distribution of an ancestor, we note that by symmetry, each
child has the same probability $1/m$ of being the heir.
Furthermore, using $\tY_i$ and $\tX_j$ to denote the waiting times
(corresponding to $Y_i$ and $X_j$ above) 
for an ancestor,  
it follows from \eqref{heir} that conditioned
on the heir being the child marked $k$, the joint distribution of
$\tY_2,\dots,\tY_{m-1},\tX_1,\dots,\tX_m$ has density
\begin{multline}
m e^{-\sum_{i=2}^{m-1} y_i-x_k} \prod_{i=2}^{m-1} ie^{-iy_i}
\prod_{j=1}^{m} je^{-jx_j}	
\\= 
\prod_{i=2}^{m-1} (i+1)e^{-(i+1)y_i}
\prod_{j\neq k} je^{-jx_j}\cdot 2e^{-2x_k}.  
\end{multline}
Consequently, $\tY_i\sim\Exp(i+1)$ and, given that the heir is child $k$,
$\tX_k\sim\Exp(2)$ while $\tX_j\eqd X_j\sim \Exp(1)$ for $j\neq k$, all
waiting times independent (conditioned on $k$). 

\begin{remark}\label{Rmstbirths}
The distributions of the birth times can be obtained from \refT{TC};
it follows that using the notation there, 
$\SS\eqd V_{m-1,m-2}$, and $\xi_i\eqd V_{m-1,m-1}$ (with marks as above),
while the $i$:th child in birth order is born at time $\xi_i'=\SS+ V_{m,i}$,
with $\SS$ and $V_{m,i}$ independent.

Similarly,
an ancestor
becomes fertile at time $\tSS=\sum_{i=2}^{m-1} \tY_i\sim V_{m,m-2}$
and the heir is born at time
\begin{equation}\label{heirmst}
  \heir = \tSS+\tX_k\eqd \sum_{j=2}^{m} Y_j\sim V_{m,m-1}.
\end{equation}
Thus $\heir$ has the same distribution as for the \emst, see
\eqref{heiremst}. (This is not surprising since we really construct the same
trees in two somewhat different ways.)  
\end{remark}

As for the extended \mst{} in \refSS{SSemst},
the number of gaps in the process
(= 1 + the number of keys, \ie, $1+\zpsi_t$) 
forms a Yule process (see \refE{EYule1}),
but in the present case it starts at 2, while it starts with 1 for the
\emst{} in \refSS{SSemst}.
(In other words, the number of gaps is the sum of two independent standard
Yule processes.)
The number of keys in $\ctt$ thus evolves in exactly
the same way for every $m\ge2$, and hence is the same as for the 
binary case $m=2$ treated in \refE{EBST}.
In particular, since also $\taux\sim\Exp(1)$ is the same for all $m$,
the
number $K$ of keys in $\cttaux$ has the distribution  \eqref{bstM}.
Moreover, as for the \emst,
if we condition on $K=k$, then $\friT$ has the same distribution as the
random \mst{} $T_k$ with $k$ keys. Hence, we get the following result:

\begin{theorem}\label{fringedistributionmary}
The number $K$ of keys in the asymptotic fringe tree 
$\friT$ has the distribution  
 \begin{align}\label{msts2}
  \P(K=k)=\frac{2}{(k+1)(k+2)}
  , \qquad k\ge1.
 \end{align} 
 Furthermore,  $\friT$ can also be described as an \mst{} 
$T_K$ with
a random number $K$ keys, where $K$ has the distribution \eqref{msts2}.
\nopf
\end{theorem}
Cf.\ \eqref{mstkeys}, the similar result for the extended $m$-ary search
tree, and  note that the distribution \eqref{msts2} equals the distribution
\eqref{mstkeys} conditioned on $K\ge1$.
Furthermore, the number of keys thus grows as $e^t$, so the number of nodes
has to grow at the same rate, which again shows that $\ga=1$.  
Note that the second part of the theorem is another instance of key-coherence.

As for the extended \mst{} in \refSS{SSemst},
the number of nodes is random for $m\ge3$.
We can again find the asymptotics from \refT{Tex}, yielding the following
theorem. 
(Alternatively, we can obtain the result from
 \eqref{kull} and \eqref{nk} for the extended \mst.)

\begin{theorem}\label{TMSTnodes}
  For the \mst{} $T_n$ with $n$ keys,
\begin{equation}\label{kull-e}
  \frac{|T_n|}n\asto \frac{1}{\mpsi}=\frac{1}{2(H_m-1)}.
\end{equation}
\end{theorem}
In other words, the average number of keys per node converges \as{} to
$2(H_{m-1}-1)$. 

\begin{remark}\label{RMSTnodes}
This result was first shown, in a weaker form with convergence in
probability, by \citet{MahmoudPittel1989}; convergence of the expectation 
had been shown earlier by
\citet[answer to exercise 6.2.4-8]{Knuth1998}
and \citet{Baeza-Yates}.

For the variance and asymptotic distribution (which we do not consider in
the present paper), there is an interesting phase transition:
the variance is linear in $n$ and the distribution asymptotically normal if
$m\le26$ but not if $m\ge27$, see
\cite{MahmoudPittel1989}, 
\cite{LewMahmoud1994},
\cite{ChernHwang2001},
\cite{ChauvinPouyanne}.
\end{remark}

\begin{proof}
For the present branching
process, $\psi(t)=1+\sum_{i=2}^{m-1}\ett{S_i'\le t}$, where
$S_i':=\sum_{j=2}^{i}Y_j$ is the time the $i$:th key comes to the node.
Arguing as in \eqref{aku}--\eqref{akull} we find (omitting some details)
\begin{equation}\label{aku2-e}
\hpsi(1)
=\intoo e^{-t}\psi(t) \dd t 
=1+\sum_{i=2}^{m-1} e^{-S'_i}
\end{equation}
and thus,
\begin{equation}\label{akull2-e}
\mpsi=
\E\hpsi(1)
=1+\sum_{i=2}^{m-1}\prod_{j=2}^i\E e^{-Y_j}
=1+\sum_{i=2}^{m-1}\frac{2}{i+1}
=2(H_m-1).
\end{equation}
Hence, \refT{Tex}\ref{tex2} yields
\begin{equation}\label{kull2}
  \frac{|T_n|}n\asto \frac{1}{\mpsi}=\frac{1}{2(H_m-1)}.
\end{equation}  
\end{proof}

The asymptotic number of nodes with a given number of keys can be found
similarly. Note that the tree is constructed so that each node contains at
least one key. 
(This theorem is also an immediate corollary of results 
by \citet{KalMahmoud:Degree},
shown using a \Polya{} urn, see also \cite[Example 7.8]{SJ154}.)

\begin{theorem}\label{TMSTkeys}
Let $N_k(T_n)$ be the number of nodes in $T_n$ with
$k$ keys, for $k=1,\dots,m-1$. 
Then,
\begin{equation}\label{nk2}
  \begin{split}
\frac{N_k(T_n)}{|T_n|}
\asto
\begin{cases}
\frac{2}{(k+1)(k+2)}, & 1\le k\le m-2,	
\\
\frac{2}{m}, & k=m-1.
\end{cases}
  \end{split}
\end{equation}
\end{theorem}

\begin{proof}
This follows either from \eqref{nk} for the extended \mst{} or by a similar
argument as in the proof of \refT{TEMSTkeys} (which we omit).
\end{proof}

Finally, we give the asymptotic degree distribution $D$.
(This was found, using a \Polya{} urn, by \citet{KalMahmoud:Degree},
generalizing the special case of leaves ($k=0$) given in
\cite{HolmgrenJanson2}.)

\begin{theorem}
  \label{TMSTdegree}
Let $n_k(T_n)$ be the number of nodes of outdegree $k$ in  $T_n$.
Then
\begin{equation}\label{mstD}
\frac{n_k(T_n)}{|T_n|}\asto
\P(D=k)=
\begin{cases}
\frac{m-1}{m+1}, & k=0,
\\
\frac{2}{m(m+1)}, & 1\le k\le m.
\end{cases}
\end{equation}
\end{theorem}

The asymptotic degree distribution is thus uniform on \set{1,\dots,m}, but
with a large proportion of the nodes being leaves (outdegree 0). (For $m=2$,
the distribution is uniform on \set{0,1,2}.) 
Note that
$\E D=1$,
as always,  
see \eqref{ED}.

\begin{proof}
  This follows by straightforward calculations from \eqref{D} and
  \refR{Rmstbirths}, for example using \refT{TC}.
However, we find it illuminating to instead give
a less computational proof, using the properties of the exponential
distributions. Recalling that $D$ is the degree of the root of $\friT$,
we consider the life of an individual (the root), 
stopped at $\taux$;
we regard $\taux$ as an exponential clock 
(the \emph{doomsday clock})
that strikes at a random time, and
then stops the process.

After the creation (at $t=0$, and with a single key), 
the next thing that happens is either the
arrival of the second key, or that the doomsday clock strikes. 
Since the second key
arrives with intensity 2 and the clock strikes with intensity 1, 
the probability is $2/3$ that the second key will arrive before the clock
strikes. 
Conditioned on this event, the same argument shows that the
probability that also the third key arrives before the clock strikes is
$3/4$, and so on. It follows that the probability that the node acquires all
$m-1$ keys  before the clock strikes is 
\begin{equation}\label{hag}
\frac23\frac{3}4\dotsm\frac{m-1}m=\frac{2}m.  
\end{equation}
(Note that this argument also yields another proof of \refT{TMSTkeys}.)

After the arrival of all $m-1$ keys, assuming that the doomsday
clock still has not
struck, we wait for the $m$ children. Each child arrives with intensity 1,
and the clock strikes with the same intensity, so by symmetry (and
independence),
the order of the $m$ births
and the strike of the clock is uniform among all $(m+1)!$ possibilities.
In particular, the position of the clock strike is uniform among these $m+1$
events, \ie, the number of children born before the clock strikes is uniform
on \set{0,\dots,m}.
Combining this and \eqref{hag} we obtain, for $1\le k\le m$,
\begin{equation}
  \P(D=k)=\frac{2}m \frac{1}{m+1}
\end{equation}
and, including the cases where less than $m-1$ keys arrive before the clock
strikes, 
\begin{equation}
  \P(D=0)=1-\frac{2}m+\frac{2}m \frac{1}{m+1}.
\end{equation}
The result follows.
\end{proof}

\begin{remark}
   Note that the degree distribution in \eqref{mstD} differs
from the degree distribution \eqref{pd-} for the random \mary{} tree defined
in \refE{Elinear-}; as said there, the two different types of
random \mary{} trees
are thus not even asymptotically equivalent.
\end{remark}

\section{Median-of-$(2\ell+1)$ binary search tree}\label{Smedian}
Let $\ell\ge1$ be a fixed integer.
The random \emph{median-of-$(2\ell+1)$ binary search tree},
see \eg{} \cite{Devroye93},
is a modification of
the \bst{} in \refE{EBST}, where each internal node still contains exactly
one key, but each external node can contain up to $2\ell$ keys.
(We can also include the case $\ell=0$; this is just the extended \bst, \ie,
the special case $m=2$ of \refS{SSemst}.)

The tree is grown recursively, starting with a single external node without
any keys. The first $2\ell$ keys are placed in this node. When the
$(2\ell+1)$:th 
key arrives to the node (or to another external node later in the process),
the node becomes an internal node with two new external nodes as children,
say $v_L$ and $v_R$; moreover, 
the median of the $2\ell+1$ keys now at the node is found and
put in the internal node, while the $\ell$ keys that are smaller than the
median are put in the left child $v_L$
and the $\ell$ keys that are larger than the median are put in the right child
$v_R$. 

In order to model this by a branching process, we start the tree with $\ell$
keys in the root. (This is no restriction, since the first $\ell$ keys always
go there.) Then each external node will contain between $\ell$ and $2\ell$ keys,
throughout the process, and the 
median-of-$(2\ell+1)$ binary search tree is produced by a branching
process with the following life histories:
An individual is born as an external node with $\ell$ keys. It acquires $\ell+1$
additional keys 
after successive independent waiting times $Y_1,\dots,Y_{\ell+1}$, where
$Y_i\sim\Exp(\ell+i)$ 
(since the node has $\ell+i$ gaps when there are $\ell+i-1$ keys).
When the $(\ell+1)$:th key arrives, the individual immediately gets $2$
children.

We let the weight $\psi(t)$ be the number of keys stored at the individual
at age $t$. Thus 
$\zpsi_t$ is the total number of keys at time $t$ and $\tau(n)$ is the time
the $n$:th key is added. Hence, assuming $n\ge \ell$, 
$T_n$ is a random 
median-of-$(2\ell+1)$ binary search tree with $n$ keys.

Note that this construction is very similar to the one for the \emst{} in
\refSS{SSemst}, and we analyse it in the same way.
Let $S_k:=\sum_{i=1}^k Y_i$, $k=0,\dots,\ell+1$;
this is the time the node gets its $(\ell+k)$:th key.
Then 
\begin{equation}
\psi(t)=
  \begin{cases}
\ell+k, & S_k\le t<S_{k+1},\; 0\le k\le \ell,
\\
1, & S_{k+1}\le t.
  \end{cases}
\end{equation}

For $\gth>-\ell-1$
and $k\le \ell+1$, 
\begin{equation}\label{wp}
\E e^{-\gth S_{k}}
=\prod_{i=1}^{k} \E e^{-\gth Y_i}	
=\prod_{i=1}^{k} \frac{\ell+i}{\ell+i+\gth}
=\frac{\rise{\ell+1}k}{\rise{\ell+1+\gth}k}
.
\end{equation}
(See also \refT{TC};
in the notation used in \refApp{AC}, $S_k\eqd V_{\ell+k,k}$.)

Furthermore,
$\xi_1=\xi_2=S_{\ell+1}$, and thus
the random variable $\Xib$ in
\refR{RBPxib}  equals $2e^{-\gth S_{\ell+1}}$. 
Hence, see \eqref{hmub} and \eqref{esk},
\begin{equation}\label{hmumedian}
  \begin{split}
\hmu(\gth) = 2 \E e^{-\gth S_{\ell+1}}
=\frac{2\rise{\ell+1}{\ell+1}}{\rise{\ell+1+\gth}{\ell+1}}
=\frac{\rise{\ell+2}{\ell+1}}{\rise{\ell+1+\gth}{\ell+1}}.
  \end{split}
\end{equation}
In particular, we see that $\hmu(1)=1$, so 
the Malthusian condition \eqref{malthus} is satisfied with $\ga=1$.
(Again, $\ga=1$ has to hold since the number of keys is a Yule process,
although now started with $\ell$ keys.)
It is easy to see that all other conditions \refBP{}
 are satisfied.
Consequently, \refT{TBP} applies;
the asymptotic random fringe tree $\friT$ is obtained by
running the branching process above and stopping it after a
random time $\taux\sim\Exp(1)$.

Theorems \ref{TEMSTnodes} and \ref{TEMSTkeys}  can be
adapted with minor modifications as follows; 
we omit the proofs which are similar to the ones in \refSS{SSemst}, now
using \eqref{wp}.

\begin{theorem}\label{Tmediannodes}
  For the \medianllbst{} $T_n$ with $n$ keys,
\begin{equation}\label{mediankull}
  \frac{|T_n|}n
\asto \frac{1}{\mpsi}
=\frac{1}{(\ell+1)(H_{2\ell+2}-H_{\ell+1})}.
\end{equation}
\nopf
\end{theorem}

\begin{theorem}\label{Tmediankeys}
Let $\Next_k(T_n)$ be the number of external nodes in $T_n$ with
$k$ keys, for $k=\ell,\dots,2\ell$, and let 
$\Nint_k(T_n)$ be the number of internal nodes (all having one key).
Then, for the \medianllbst,
\begin{align}
  \frac{\Next_k(T_n)}{|T_n|}
&\asto \frac{\ell+1}{(k+1)(k+2)}, 
\qquad \ell\le k \le 2\ell,	
\\
\frac{\Nint_k(T_n)}{|T_n|}
&\asto\frac12.
\end{align}
\nopf
\end{theorem}

\begin{remark}\label{Rrandompivot}
\citet{ChernHwangTsai2002} consider (using different methods)
a more general class of trees, where 
an external node has up to $r-1$ keys; when the $r$:th key arrives to the
node, 
a pivot is selected among them at random, such that its rank $R$ (\ie, its
number if the $r$ keys are ordered) has 
some fixed distribution on \set{1,\dots,r}. (The case above is thus
$r=2\ell+1$ and $R=\ell+1$; in this case $R$ is deterministic.) 
The pivot is put in the internal node, and its
children get $R-1$ and $r-R$ keys. 
Translated to the branching process, this means (in general)
that the individuals start with different number of keys, which would
require a multi-type version of the results above (see \refR{Rmulti}).
However, it is possible to modify the branching process by including the
external nodes in the life of their parent. Thus the individuals now
are the internal nodes. (Properties of external nodes can be found using
suitable characteristics.) The life of an individual
starts with $r$ keys; these are immediately split up 
with a random $R$ as above, and we regard the individual as carrying two
unborn children (fetuses) with initially $R-1$ and $r-R$ keys. The fetuses 
get new keys, independently of each other and each with rate 1 + the number of
existing keys, and each is born when it has got
$r$ keys. We omit the details.
\end{remark}

\begin{example}\label{Ehennequin}
  An \mary{} generalisation,
introduced by \citet{Hennequin91} and further studied by
\eg{} \cite{ChernHwang2001}, \cite{ChernHwangTsai2002},
\cite{ChauvinDrmota2006} and \cite{DrmotaJansonNeininger2008},
  has internal nodes with $m-1$ keys and external nodes with up to
  $m\ell+m-2$ keys. 
(Here $m\ge2$ and $\ell\ge0$.)
When an external node gets $m\ell+m-1$ keys, 
it is converted to an internal node with $m$ external children;
the $m\ell+m-1$
keys are ordered and keys number $\ell+1, 2(\ell+1),\dots,(m-1)(\ell+1)$ are
put in the internal node, while the external children get $\ell$ keys each.
The \mst{} in \refSS{SSemst} is the case $\ell=0$.

This version can be treated as above; again each individual starts with
$\ell$ keys, but now it acquires $(m-1)(\ell+1)$ more keys, after waiting
times $Y_i\sim\Exp(\ell+i)$, $i=1,\dots,(m-1)(\ell+1)$. At time
$S_{(m-1)(\ell+1)}$, $m$ children are born. 
Note that \eqref{esk}--\eqref{hmumst} generalize to, 
\cf{} the special case \eqref{hmumedian},
\begin{equation}\label{hmumstl}
  \begin{split}
\hmu(\gth)&=
m \E e^{-\gth S_{(m-1)(\ell+1)}}
=m\prod_{i=1}^{(m-1)(\ell+1)} \E e^{-\gth Y_i}	
\\&
=m\prod_{i=1}^{(m-1)(\ell+1)}  \frac{i+\ell}{i+\ell+\gth}
=\frac{m\rise{\ell+1}{(m-1)(\ell+1)}}{\rise{\ell+1+\gth}{(m-1)(\ell+1)}}
\\&
=\frac{\rise{\ell+2}{(m-1)(\ell+1)}}{\rise{\ell+1+\gth}{(m-1)(\ell+1)}}
  \end{split}
\end{equation}
and hence again $\ga=1$. 
It then follows from \eqref{hmumstl} and \eqref{el2} that
\begin{equation}\label{gbmstl}
  \gb = \sum_{i=1}^{(m-1)(\ell+1)}\frac{1}{i+\ell+1}
= H_{m\ell+m}-H_{\ell+1},
\end{equation}
\cf{} \eqref{gbemst} for the case $\ell=0$.
Results for this model can be derived as
  above, but we leave this to the readers.

More generally, one can similarly make an \mary{} version of the model with
random pivot in \refR{Rrandompivot}, see  \cite{ChernHwangTsai2002}; a
corresponding \CMJbp{} (with the internal nodes as individuals) 
can be constructed as there.
\end{example}

\section{Fragmentation trees}\label{Sfrag}
Another type of example is provided by the following fragmentation process,
introduced by \citet{Kolmogorov}, see also
\citet[Chapter 1]{Bertoin} 
and \citet{Janson2008}, and the further references given there.
Fix $b\ge2$  and the
law for a random vector $\vv=\vvb$; this is
commonly called the \emph{dislocation law}.
We assume that
$0\le V_j\le1$, $j=1,\dots,b$, and
\begin{equation}
  \label{a1}
\sumjb V_j=1,
\end{equation}
\ie, that \vvb{} belongs to the standard simplex.
For simplicity we also assume that each $V_j<1$
a.s. (We allow $V_j=0$.)

Starting with an object of mass  $x_0\ge1$, break
it into $b$ pieces with masses $V_1x_0,\dots,V_bx_0$.
For a given threshold $x_1\in(0,x_0]$,
continue recursively with each piece of mass
$\ge x_1$, using new (independent) copies of the
random vector \vvb{} each time. 
The process
terminates a.s.\ after a finite number of steps,
leaving a finite set of fragments of masses $<x_1$.
We regard the fragments of mass $\ge x_1$ that occur during this process as
the (internal) nodes of a (random) tree, the \emph{fragmentation tree}; the
resulting fragments of mass $<x_1$ can be added as external nodes.

Obviously, the fragmentation tree depends only on the ratio $x_0/x_1$, so we
denote it by $\fT_{x_0/x_1}$.
(We may assume either $x_0=1$ or $x_1=1$ without loss of generality, but we
prefer to be more flexible.)

We can translate the fragmentation process to a \CMJ{} branching process
by regarding a fragment of mass $x$ as born at time $\log(x_0/x)$; an
individual will have $b$ children, born at ages $\xi_1,\dots,\xi_b$ with
$\xi_i:=-\log V_i$. (If some $V_i=0$, we get $\xi_i=\infty$, meaning that
this child is not born at all, so there are fewer than $b$ children. Note
also that in this section, we do not require that $\xi_1,\xi_2,\dots$ are
ordered in increasing order.)
It is easy to see that the fragmentation tree $\fTxx$ defined
above for a threshold $x_1$ is the same as the family tree 
$\cT_{\log(x_0/x_1)}$
of this branching
process at time $\log(x_0/x_1)$.

The relation \eqref{a1} can be written as
\begin{equation}\label{a2}
  \sumjb e^{-\xi_i}=1.
\end{equation}
Taking the expectation we find, see \eqref{hmub}, $\hmu(1)=1$, so the
Malthusian parameter $\ga=1$. It is easy to see that the assumptions 
\refBP{} hold, except possibly \ref{BPnonlattice}; 
we say that the fragmentation process is \emph{non-lattice} if
\ref{BPnonlattice} holds, \ie, if not every $V_i$ is concentrated on 
$\set{r,r^2,r^3\dots}$ for some $r\in(0,1)$.
(A sufficient condition
for \ref{BPnonlattice} is thus that $V_1$ has a continuous distribution.)

Furthermore, \eqref{a2} and \eqref{xib} say that $\Xia=1$ is non-random.
This has the consequence that the random variable $W$ in \refR{RW} also is
deterministic; more precisely, see \cite[Theorem (6.8.1)]{Jagers},
\begin{equation}\label{a3}
  W=1/\gb,
\end{equation}
where by \eqref{el3},
\begin{equation}
  \gb=\E\sumjb \xi_je^{-\xi_j}
=\E\sumjb V_j \log(1/V_j).
\end{equation}
\begin{remark}
The Laplace transform $\Psi(s):=\E e^{-s W}$ of the limit $W$ in \refR{RW}
satisfies the  functional equation 
\begin{equation}\label{Psi}
\Psi(s)=\E\prod_{i=1}^N \Psi\bigpar{se^{-\ga\xi_i}},
\end{equation}
and this equation together 
with \eqref{EW} (provided \eqref{xlogx} holds) determines $W$ uniquely, see
\cite{Doney}. 
Assuming $\ga=1$ (which can be regarded as a normalisation of the time scale),
it is easy to see that a constant $W$ satisfies \eqref{Psi}
if \eqref{a2} holds,
which gives an alternative proof of \eqref{a3}.
Yet another proof of \eqref{a3} is obtained by noting that
when \eqref{a2} holds, the
martingales $R_n$ and $Y_t$ in \cite{Nerman} are constant 1, and then \eqref{a3}
follows from \cite[Corollary 3.2]{Nerman}. (The related intrinsic martingale
in \cite{Bertoin} is constant 1 too, see \cite[Section 1.2.2]{Bertoin}.) 
Note also that the converse holds: $W$ is constant (and then $1/\gb$)
if and only if \eqref{a2} holds; this too follows from the functional equation
\eqref{Psi}.
\end{remark}

\begin{remark}\label{Rfragtau}
Note that unlike the trees studied in the previous sections, we consider the 
family tree $\cT_t$ at a fixed time $t=\log(x_0/x_1)$ instead of stopping
when some weight $\zpsi_t$ reaches a given value. 
However, since $W$ is constant, this makes a very
small difference. In fact, by \eqref{olle1} and \eqref{a3}, $Z_t\sim \gb\qw e^t$
\as, and thus, if we use the characteristic $\psi(t)=1$ again, 
the stopping time $\tau(n)$ when the tree has $n$ nodes satisfies \as{}
\begin{equation}\label{atau}
\tau(n)=\log(\gb n)+o(1)=\log n+\log\gb+o(1).  
\end{equation}
We may define a fragmentation tree $T_n$ of fixed size $n$ by stopping at
$\tau(n)$; in the original formulation this means that we choose the
threshold $x_1$ to be the size of the $n$:th largest fragment in the process,
so that there will be exactly $n$ fragment of size $\ge x_1$ (unless there
is a tie). We see from \eqref{atau} that asymptotically, this is almost the
same as taking a constant time $t=\log n+\log\gb$. For more precise results on
$|Z_t|$, and thus on $\tau(n)$, see \cite{Janson2008}.
\end{remark}

\begin{theorem}\label{Tfrag}
Let $\fTxx$ be a random fragmentation tree defined as above,
for a non-lattice fragmentation process.
Then \refT{TBP} holds also (as $x_0/x_1\to\infty$, and with other obvious
notational modifications) 
for the random fringe tree $\fTxx\x$. 

The limiting random fringe tree $\friT$ can be constructed by the fragmentation
process above, starting at $x_0=1$ and with a random threshold $x_1=U\sim
U(0,1)$, with $U$ independent of the fragmentation. 
\end{theorem}

\begin{proof}
By the equivalence above of the fragmentation process and the \CMJ{}
branching process, $\fTxx=\cT_{\log(x_0/x_1)}$, and the first part follows
from  
\refT{TBP} (and its proof).

The limiting fringe tree $\friT$ is obtained by stopping the branching
process $\ctt$ at a random time $\taux\sim\Exp(1)$; by the equivalence
above, this 
is equivalent to starting the fragmentation process at $x_0=1$ and
stopping at a threshold
$x_1=\exp(-\taux)$. This completes the proof, since
$\exp(-\taux)\sim U(0,1)$.  
\end{proof}

Similarly, \refT{TBPX} holds, and the random sin-tree $\sinT$ can be defined
by a suitable extension of this
random fragmentation process; we leave the details of the general case to the
reader, and discuss only one case in the example below.

\begin{example}[Binary splitting]\label{Ebinaryfrag}
  Let $b=2$ and $\vv=(V_1,V_2)=(V_1,1-V_1)$ with $V_1\sim U(0,1)$.
Thus, at each fragmentation event, the object is split into two parts, with
uniformly random sizes. 

In the corresponding \CMJbp, each individual gets two
children, born at ages $\xi_1$ and $\xi_2$, where $\xi_1,\xi_2\sim\Exp(1)$
and one of them determines the other by 
\begin{equation}\label{xibin}
e^{-\xi_1}+e^{-\xi_2}=1.  
\end{equation}
Note the similarities with the \CMJbp{} for the \bst{} in \refE{EBST}; the
difference is that there $\xi_1$ and $\xi_2$ are independent, while here
they are dependent. For properties that depend only on the individual (marginal)
distributions of $\xi_1,\xi_2$ and not on their joint distribution, we thus
have the same results for both processes; some examples are the intensity
$\mu$, the distribution of $\heir\sim\Exp(2)$ and its mean $\gb=1/2$, and
the expected size of the population $\E Z_t=2(e^t-1)$.
However, many properties really depend on the joint distribution of the
times of birth of the children, and are thus in general different for the
two processes. 
For example, although $\E Z_t$ is the same for both processes, the
distributions of $Z_t$ are not:  for the present process, 
there is by \eqref{xibin} always one child of the root born before time
$\log 2$, so $Z_{\log 2}\ge2$, while for the  process in \refE{EBST},
$\P(Z_t=1)=e^{-2t}>0$ for every $t\ge0$.
Also the fringe tree distributions will be different, as is seen below. 

Let us first consider the asymptotic outdegree distribution in the
fragmentation tree, which equals the distribution of the root $D$ in
$\friT$. We have, using the construction of $\friT$ in \refT{Tfrag},
\begin{equation}\label{pyrD}
  \begin{cases}
	D=0\iff U\ge \max(V_1,1-V_1)
\\
	D=1\iff \min(V_1,1-V_1)<U\le\max(V_1,1-V_1)
\\
	D=2\iff U<\min(V_1,1-V_1)
  \end{cases}
\end{equation}
and simple calculations yield
\begin{equation}\label{pyr0}
\P(D=0)=\P(1-U\le V_1\le U)=\E(2U-1)_+=1/4  
\end{equation}
and similarly
$\P(D=1)=1/2$ and $\P(D=2)=1/4$. Consequently, $D\sim\Bi(2,1/2)$ has a
binomial distribution. (This differs from \refE{EBST}, where $D$ has a uniform
distribution.) 

Furthermore, let $X_1:=V_1/U$ and $X_2:=(1-V_1)/U$ be the masses of the two
children of the root, relative to the threshold $U$. Then $X_1,X_2>0$ and
$X_1+X_2\ge1$, and a calculation of the Jacobian of the mapping
$(U,V_1)\mapsto(X_1,X_2)$ shows that in this region, $(X_1,X_2)$ has the
density $f(x_1,x_2)=(x_1+x_2)^{-3}$.  This enables us to again compute the
distribution of $D$; for example, $D=0\iff X_1,X_2<1$. Moreover, we can now
easily find the distribution of nodes in the second generation too; we give
a few examples.

Denote the children of the root by $v_1$ and $v_2$.
Then $\friT$ contains $v_1$ but not $v_2$
if and only if $X_1\ge1>X_2$; denote this event by $\Eio$.
Conditioned on $\Eio$, the density of
$X_1$ is, by a small calculation, $2\bigpar{x\qww-(x+1)\qww}$, $x\ge1$.
Furthermore, the outdegree of $v_1$ is given by \eqref{pyrD} with $U$
replaced by $1/X_1$ (and $V_1$ by an independent copy $V_{11}$); hence, in
analogy to \eqref{pyr0},
\begin{align}
\P\bigpar{\deg(v_1)=0\mid\Eio}
&
=\E \Bigpar{\Bigpar{\frac{2}{X_1}-1}_+\Bigm|\Eio}
\notag\\&
=\int_1^2\Bigpar{\frac{2}x-1}\Bigpar{\frac{2}{x^2}-\frac{2}{(x+1)^2}}\dd x
\notag\\&
=\frac{3}2-8\ln 2+4\ln 3 \approx0.34927,\label{pyr10}
\intertext{and, similarly,}
\P\bigpar{\deg(v_1)=1\mid\Eio}
&
=2\E \Bigpar{\min\Bigpar{\frac{1}{X_1},1-\frac{1}{X_1}}\Bigm|\Eio}
\notag
\\&
=1+12\ln 2-8\ln3 \approx0.52887,\label{pyr11}
\\
\P\bigpar{\deg(v_1)=2\mid\Eio} 
&
=\E \Bigpar{\Bigpar{1-\frac{2}{X_1}}_+\Bigm|\Eio}\notag
\\&
=-\frac{3}2-4\ln 2+4\ln 3 \approx0.12186.\label{pyr12}
\end{align}
Recall that if we just condition on $v_1\in\friT$, its outdegree distribution
equals the unconditional distribution of $D$, \ie, $\Bin(2,1/2)$;
hence, \eqref{pyr10}--\eqref{pyr12} illustrate the dependencies between the
outdegrees of different nodes (in this case, $o$ and $v_1$), see \refR{RD2}.

We obtain also
$\P(|\friT|=1)=\P(D=0)=1/4$ and, from \eqref{pyr10},
\begin{equation}
  \begin{split}
  \P(|\friT|=2)&=2\P\bigpar{\Eio\text{ and }\deg(v_1)=0}
=2\P\bigpar{\deg(v_1)=0\mid\Eio}\P(\Eio)
\\&
=\frac{3}4-4\ln 2+2\ln 3 \approx0.17464.\label{pyr2}
  \end{split}
\raisetag{\baselineskip}
\end{equation}
Again, this differs from the \bst{} in \refE{EBST}.
We do not know any general formula for the probability distribution of the
size $|\friT|$.

The irrational probabilities in \eqref{pyr10}--\eqref{pyr12} and \eqref{pyr2}
seem to 
exclude any simple combinatorial construction or interpretation of the
asymptotic
fringe tree $\friT$.

To construct the limiting sin-tree $\sinT$ in \refT{TBPX}, we note that the
heir $\heir$ of an ancestor has distribution $\Exp(2)$ by \eqref{heir} (as in
\refE{EBST}). Going back to the mass scale, we note that $Y:=\exp(\heir)$
has the Pareto(2) distribution
\begin{equation}
  \label{Pareto2}
\P(Y>x)=\P(\heir>\log x)=x\qww,
\qquad x>1.
\end{equation}
The random sin-tree $\sinT$ can thus be constructed as follows:
Start with a root $o$ of mass 1 (as in \refT{Tfrag})
and give it an infinite sequence of ancestors of mass $Y_1$, $Y_1Y_2$,
$Y_1Y_2Y_3$, \dots, where $Y_i$ are \iid{} with the distribution
\eqref{Pareto2}; the other child of the ancestors thus has mass $Y_1-1$,
$Y_1(Y_2-1)$, $Y_1Y_2(Y_3-1)$, \dots. Grow independent  fragmentation
trees from these other children and from $o$, using uniformly random binary
splittings, and stop at a common threshold $x_1=U\sim U(0,1)$.
\end{example}

\begin{remark}
  We have, for simplicity, assumed that the branching factor
$b$ is a constant finite integer. (Although we may allow fewer than $b$
  fragments by letting some $V_i=0$.)
We can also allow $b=\infty$, or a random $b$ (which can be reduced to
$b=\infty$ by adding variables $V_i$ that are 0). The results above extend,
provided \ref{BPmub} holds.
\end{remark}

\begin{remark}
  As noted in \refS{S:BP}, the results extend also to the lattice case, with
  minor modifications, but for simplicity we ignore that case.
Only very special fragmentation processes are lattice; one trivial example
is the deterministic symmetric binary splitting
$V_1=V_2=1/2$. More generally, the deterministic binary splitting
$V_1=p$, $V_2=q=1-p$ is lattice if and only if $\log p/\log q$ is rational.
For a random example,  let $r=(\sqrt5-1)/2$,  take $b=3$ and let
$(V_1,V_2,V_3)$ be either $(r,r^2,0)$ or $(r^2,r^2,r^3)$ with probability
$1/2$ each.
\end{remark}

\begin{remark}\label{Rsplit1}
  The \emph{split trees} defined by \citet{Devroye98split} are related to
  fragmentation trees. A split tree is a $b$-ary tree defined using a
  number of balls that enter the root and are distributed (randomly and
  recursively) to  the subtrees of the root and further down in the tree
according to certain rules that are based on a splitting law 
$\vv=\vvb$ satisfying \eqref{a1},
see \cite{Devroye98split} for details. 
(A splitting law is thus the same as a dislocation law.) 
Far away from the fringe, where there are many balls and the law of large
numbers applies, the numbers of balls in different subtrees are distributed
asymptotically as the masses in the corresponding fragmentation tree, so
there are many similarities between the two types of random trees. However,
at the fringe, the details differ, and the asymptotic fringe distributions
are in general not the same. For example, the \bst{} in \refE{EBST} can be
defined as a split tree, where the splitting law $\vv=(V_1,V_2)=(V_1,1-V_2)$
with $V_1\sim U(0,1)$ uniform. The corresponding fragmentation tree is thus
the tree studied in \refE{Ebinaryfrag}, and as noted there the asymptotic
fringe tree distribution is not the same as for the \bst; for example, the
degree distributions differ.
Fringe distributions of split trees will be studied in another paper.
\end{remark}

\section{Rank}\label{Srank}

Define,
following \citet{BonaPittel}, 
the \emph{rank} of a node in a rooted tree to be the smallest
  distance to a descendant that is a leaf. Thus a leaf has rank 0, while a
  non-leaf has rank  $\ge1$. A node with rank $\ge k$ is also said to be
  \emph{$k$-protected}. (For example, $1$-protected = non-leaf; $2$-protected =
  non-leaf and no child is a leaf.)
The simplest ``non-trivial'' case is $2$-protected, which sometimes is
called just \emph{protected}. There has in recent years 
been a number of papers on the number of $2$-protected nodes in various
random trees, or (equivalently) the probability that a random node is
$2$-protected, 
and a few papers on $k$-protected nodes for higher $k$;
see \eg{} \citet{SJ283} and the references therein.
Such results can equivalently be described as results on the distribution of
the rank of a random node.

For a tree $T$ (deterministic or random), let $R(T)$ be the rank of a
uniformly random
node in $T$, and let $R_0(T)$ be the rank of the root of $T$.
(Thus $R(T)$ is a random variable, while $R_0(T)$ is deterministic if $T$ is.)
Since the rank of $v$ depends only on the subtree $T^v$,
$R(T)=R_0(\Tx)$,  the rank of the root of the random fringe tree $\Tx$.
This reduces the study of rank and $k$-protected nodes to the study of
random fringe trees. (This was the method by \citet{SJ283}, there applied to
several classes of random trees, including \rrt{s} and \bst{s} but also
conditioned Galton--Watson trees 
which are not of the type considered in the present paper.)

For the random trees considered here,
\refT{TBP} applies, 
for any fixed $k\ge0$,
to the property that a node has
rank $\ge k$ (\ie, is $k$-protected); we denote this property by $\cpk$ 
(in this section) and deduce the following.
(Note that, depending on one's point of view, \eqref{trank} can be seen 
both as a limit result for the distribution of the rank, and as a limit
result for the proportion of $k$-protected nodes.)

\begin{theorem}\label{Trank}
Suppose that \refBP{} hold.
Then, for any $k\ge0$, as \ntoo,
\begin{equation}\label{trank}
  \P\bigpar{R(T_n)\ge k\mid T_n}=\frac{n_{\cpk}(T_n)}{|T_n|}
\asto
P_k:=
\P\bigpar{R_0(\friT)\ge k}.
\end{equation}
In other words, 
the conditioned random variables $\bigpar{R(T_n)\mid T_n}$ converge in
distribution \as{} as \ntoo,
\begin{equation}
\cL\bigpar{R(T_n)\mid T_n}\asto \cL\bigpar{R_0(\friT) } .
\end{equation}
In particular, the same holds for the unconditioned random variables,
\ie,
\begin{equation}
  R(T_n)\dto R_0(\friT).
\end{equation}
\end{theorem}

\begin{proof}
  An immediate application of \refT{TBP}.
\end{proof}

\subsection{\mst}\label{SSprotectedmst}
We consider here the rank and $k$-pro\-tected nodes in the \mst{} in
\refSS{SSmst}.  
(The binary case $m=2$ has been studied by 
\citet{MahmoudWard-bst}, \citet{Bona}, \citet{SJ283}, \citet{BonaPittel},
\citet{HolmgrenJanson1}; 
the case $m=3$ by \citet{HolmgrenJanson2}
and some higher $m$ by \citet{Heimburger}.)

We let as in \eqref{trank} $P_k=P_k(m):=\P\bigpar{R_0(\friT)\ge k}$. Thus, by
\eqref{trank}, the fraction of $k$-protected nodes in an \mst{} $T_n$
converges \as{} to $P_k(m)$. Recall that (for $m\ge3$) the number of nodes
$|T_n|$ is random. Hence it is interesting to study not only the fraction of
$k$-protected nodes, but also the (random) number 
$n_{\cpk}(T_n)$ of them in $T_n$.
(The results are formulated in this way in some of the references above.)
We note that, as an immediate consequence of \eqref{trank} and
\eqref{kull-e},
see \refR{Rnumber},
\begin{equation}\label{mrank}
\frac{n_{\cpk}(T_n)}{n}
\asto
\frac{P_k(m)}{2(H_m-1)}.
\end{equation}
Note also that
this implies asymptotics for the expectation $\E n_{\cpk}(T_n)$, see 
\refR{Rmean}. 

We proceed to the calculation of the numbers $P_k(m)$.
It will be convenient to use the extended \mst{} in \refSS{SSemst}, but note
that we really are interested in the subtree of internal nodes; to emphasize
this we say \emph{internally $k$-protected} for the $k$-protected nodes in
the tree of internal nodes.
As usual, $m$ is fixed and will often be omitted from the notation.

With this in mind, define, for $k\ge0$,
\begin{equation}\label{hk}
  h_k(t):=
\P\xpar{\text{the root of $\ctt$ is internal and internally $k$-protected}}.
\end{equation}
The root becomes an internal node at time $S_{1}\sim\Exp(1)$, when 
it receives its first key. 
Every node is 0-protected, so $h_0(t)$ is the probability
that the root is internal; thus
\begin{equation}\label{h0}
  h_0(t)=\P(S_1\le t)=1-e^{-t}.
\end{equation}

Recall that all $m$ children of the root are born at time $\xi=S_{m-1}$. 
Let $f(t)$ be the density function of $\xi=S_{m-1}$.
For $k\ge1$, the root of $\ctt$ is internal and internally
$k$-protected if and only if $t>\xi$ and the $m$ children of the root either
are external or internally $(k-1)$-protected, but not all external.
Conditioned on $\xi$, with $\xi<t$, the $m$ subtrees of the root of $\ctt$
are independent, and distributed as $\cT_{t-\xi}$. Hence,
the conditional probability (given $\xi$) that a given child is internally
$(k-1)$-protected is $h_{k-1}(t-\xi)$, while the probability that it is
external is $1-h_0(t-\xi)=e^{-(t-\xi)}$. Consequently,
\begin{equation}
  h_k(t)=\intot f(s) \bigpar{(h_{k-1}(t-s)+e^{-(t-s)})^m-e^{-m(t-s)}}\dd s.
\end{equation}
This can be written as a convolution
\begin{equation}\label{axel}
  h_k(t)= f(t)* \bigpar{(h_{k-1}(t)+e^{-t})^m-e^{-mt}}.
\end{equation}
Furthermore, $f(t)$ is the density function of the sum
$S_{m-1}=\sum_{i=1}^{m-1}Y_i$ of independent  random variables, and thus $f(t)$
is itself a convolution of their densities $ie^{-it}$.
Hence, \eqref{axel} can be written
\begin{equation}\label{jb}
  h_k(t)=  \bigpar{(h_{k-1}(t)+e^{-t})^m-e^{-mt}}* 
  \mathop{*}_{i=1}^{m-1} ie^{-it}.
\end{equation}
This, with the initial \eqref{h0}, makes it possible to calculate any
$h_k(t)$ by recursion (preferably using computer algebra).

By \refT{TBP} and \eqref{hk}, the fraction of nodes in $T_n$ that are
internal and internally $k$-protected converges \as{} to
\begin{equation}
  \begin{split}
\P\xpar{\text{the root of $\friT$ is internal and internally $k$-protected}}
=\intoo h_k(t)e^{-t} \dd t.
  \end{split}
\end{equation}

Recall that we are really interested in the internal nodes only. By
\eqref{n0}, asymptotically half of the nodes in the extended \mst{} are
internal, and thus the fraction of $k$-protected nodes in the (internal)
\mst{} $T_n$ converges \as{} to
\begin{equation}\label{pk1}
P_k= P_k(m) := 2\intoo h_k(t)e^{-t} \dd t.
\end{equation}
This can be regarded as a value of a Laplace transform.
Since the Laplace transform transforms convolutions to multiplications,
and 
$\intoo f(t) e^{-t}\dd t=\E e^{-S_{m-1}}=1/m$ by \eqref{esk1},
we have by \eqref{axel} also the alternative formula (that might be better for
numerical calculations)
\begin{equation}\label{pk2}
  P_k = \frac{2}{m}\intoo e^{-t}\bigpar{(h_{k-1}(t)+e^{-t})^m-e^{-mt}} \dd t.
\end{equation}

In the binary case $m=2$, these formulas are equivalent to the formulas
derived by a very similar argument in \citet{SJ283}. (The function $r_k(t)$
in \cite{SJ283} equals $1-e^{-t}-h_k(t)$.)
For a different method to find $P_k(2)$, see \citet{Bona} and
\citet{BonaPittel}.

Consequently, 
\refT{Trank} says that for 
an \mst{} 
the asymptotic distribution of the rank 
is
\begin{equation}
  \P(R(T_n)=k\mid T_n)\asto \P(R_0(\friT)=k)=P_k-P_{k+1},
\end{equation}
where $P_k$ is given by \eqref{pk1}--\eqref{pk2}.

Note that (by induction), each $h_k(t)$ is a polynomial in $e^{-t}$ and $t$
with rational coefficients. Hence, each $P_k(m)$ is a rational number.

Trivially, $P_0=1$ for every $m$ by \eqref{trank}.

For $k=1$, $h_0(t)+e^{-t}=1$ by \eqref{h0}, and thus by \eqref{pk2},
\begin{equation}\label{Pm1}
  P_1=\frac{2}m\intoo e^{-t}(1-e^{-mt})\dd t
= \frac{2}m\Bigpar{1-\frac{1}{m+1}}=\frac{2}{m+1},
\end{equation}
in accordance with \eqref{mstD} (recall that the 1-protected nodes are
precisely the non-leaves).  

Also for $P_2$, we may deduce a rather simple formula.

\begin{lemma}
  For the \mst, with $m\ge2$,
  \begin{equation}\label{h1}
h_1(t)=\xpar{1-e^{-t}}^m.	
  \end{equation}
\end{lemma}
\begin{proof}
  We extend the notation of \refSS{SSemst} and let $S_j:=\sum_{i=1}^j Y_i$ for
  any integer $j\ge0$, where $Y_i\sim\Exp(i)$ are independent.
(We thus change the earlier special definition of $S_m$.)
By \refT{TC}, 
\begin{equation}
  \label{maxe}
S_j\eqd V_{j,j}=
\max_{1\le i\le j} E_i,
\end{equation}
where
$E_1,E_2,\dots,E_j \sim\Exp(1)$ are i.i.d.

Since $h_1(t)+e^{-t}=1$ by \eqref{h0}, \eqref{axel} yields, recalling that 
$f$ is the density function of $S_{m-1}$ and using \eqref{maxe},
\begin{equation*}
  \begin{split}
h_1(t)&
=f(t)*\xpar{1-e^{-mt}}	
=f(t)*\P(Y_m\ge t)
=\intot f(s)\P(Y_m\le t-s)\dd s 
\\&
= \P(S_{m-1}+Y_m\le t)
= \P(S_m\le t)
=\P\bigpar{\max_{1\le i\le m} E_i \le t}
\\&
=\prod_{i=1}^m\P(E_i\le t)
=(1-e^{-t})^m.
  \end{split}
\qedhere
\end{equation*}
\end{proof}

\begin{theorem}
  \label{Tmprot}
The asymptotic probability that a random node in an \mst{} $T_n$ is
$2$-protected is
\begin{equation}\label{tmprot}
  P_2=P_2(m)=\frac{2}m\sum_{\ell=0}^{m-1}
  \frac{m!}{(m-\ell)!}\cdot
  \frac{(m(m-\ell))!}{(m(m-\ell)+\ell+1)!}.
\end{equation}
\end{theorem}

\begin{proof}
  By \eqref{pk2},
a binomial expansion, 
 \eqref{h1}, 
the change of variables $x=e^{-t}$ and a
standard evaluation of a beta integral,
\begin{equation}
  \begin{split}
P_2&
=\frac{2}m\intoo e^{-t}\sum_{\ell=0}^{m-1} \binom{m}{\ell}
  h_1(t)^{m-\ell}e^{-\ell t}\dd t	
\\&
=\frac{2}m \sum_{\ell=0}^{m-1} \binom{m}{\ell}
\intoo (1-e^{-t})^{m(m-\ell)}e^{-(\ell+1) t}\dd t	
\\&
=\frac{2}m \sum_{\ell=0}^{m-1} \binom{m}{\ell}
\intoi (1-x)^{m(m-\ell)}x^{\ell}\dd t	
\\&
=\frac{2}m \sum_{\ell=0}^{m-1} \binom{m}{\ell} B\bigpar{m(m-\ell)+1,\ell+1}
\\&
=\frac{2}m \sum_{\ell=0}^{m-1} \binom{m}{\ell} 
\frac{(m(m-\ell))!\,\ell!}{(m(m-\ell)+\ell+1)!}.
  \end{split}
\end{equation}
\end{proof}

\begin{remark}
We can also prove this result using a more combinatorial proof with balls
and boxes. 
We recall from \refT{fringedistributionmary}
that the asymptotic fringe tree $\friT$
can be constructed as an \mst{} with a random number $K$ of keys, 
where by
\eqref{msts2}
$\P(K=k)=\frac{2}{(k+1)(k+2)}=1/\binom{k+2}2$, $k\ge1$.
We condition on $K=k$ and find the probability that the root of $T_k$ is
2-protected.

Recall that a node is 2-protected if it is
not a leaf and has no child that is a leaf. 
Thus, the root of $T_k$ is 2-protected if and only if it is filled with
$m-1$ keys and
each of the $ m $ subtrees of the  root has the property that it is either
empty or contains at least $m$ keys, and at least one of the subtrees is
nonempty; in particular, we must have $k\ge 2m-1$.

For $k\ge 2m-1$ we order the keys in increasing order and represent the
$k-m+1$ keys that are distributed to the $m$ subtrees of the root by 0's
and the $m-1$ keys that stay in the root by 1's. We also add two additional
1's first and last. 
This gives a string of length $k+2$, beginning and ending with 1, and with
$m-1$ additional 1's. There are 
$\binom{k}{m-1}$  such strings, and all occur with the same probability.

Furthermore, the corresponding tree $T_k$ is 2-protected
if and only if between every pair of 1's in the string,
there is either no 0's, or at least $m$ 0's. In other words, the 1's appear
in clusters, separated by at least $m$ 0's. Let the number of clusters be
$r+1$, and note that $1\le r\le m$.
To count the number of strings of length $k+2$ such that these properties
are satisfied for a given $r$, 
we first  distribute the $m+1$ 1's into $r+1$
boxes, such that no box is empty. This gives $\binom{m}{r}$
different choices. 
We then distribute the $k-m+1$ 0's into the $r$ gaps between the clusters;
it is required that there should be at least $m$ 0's in each gap, 
but the remaining $k-m+1-mr$ 0's can be distributed arbitrarily into the $r$
gaps. This  can be done in $\binom{k+r-m(r+1)}{r-1}$ ways.  
Hence, summing over $r$ and $k$ and using (\ref{msts2}),  we obtain
\begin{align}\label{2prot}
P_2=P_2(m)
=\sum_{r=1}^{m}\sum_{k=m(r+1)-1}^{\infty}
 \frac{\binom{m}{r}\binom{k+r-m(r+1)}{r-1}}{\binom{k+2}{2}\binom{k}{m-1}}.
\end{align}
The sum over $k$ can be written as a hypergeometric sum \eqref{hyper}, 
which using Gauss' formula \eqref{gauss} simplifies and yields
\begin{equation}
  \begin{split}
P_2(m)&
=\sum_{r=1}^{m} 2\binom{m}{r}\frac{(m-1)!\,(mr)!}{(mr+m+1)!}F(mr+1,r;mr+m+2;1)
\\&
=2\sum_{r=1}^{m} \frac{(m-1)!\,(mr)!}{r!\,(mr+m-r+1)!},
  \end{split}
\end{equation}
which is equivalent to \eqref{tmprot} by letting $r=m-\ell$.
\end{remark}

We can find the asymptotics of $P_2(m)$ as \mtoo{} from \eqref{tmprot}.

\begin{theorem}
  As \mtoo, the probability $P_2=P_2(m)$ in \refT{Tmprot} is
  \begin{equation}
	P_2(m)=\frac{2}{m^3}+O\bigpar{m^{-4}}.
  \end{equation}
\end{theorem}

\begin{proof}
  Write \eqref{tmprot} as
  \begin{equation}
	\label{fris}
P_2(m)=\frac{2}m\sum_{\ell=0}^{m-1} a_\ell(m)
  \end{equation}
with
\begin{equation}\label{uk}
a_\ell(m)=
    \frac{m!}{(m-\ell)!}\cdot
  \frac{(m(m-\ell))!}{(m(m-\ell)+\ell+1)!}.
\end{equation}
In particular,
\begin{equation}\label{a0}
  a_0(m)=\frac{(m^2)!}{(m^2+1)!}=\frac{1}{m^2+1}.
\end{equation}

For each fixed $\ell$, as \mtoo,
\begin{equation}
  a_\ell(m)\sim m^\ell(m(m-\ell))^{-\ell-1}
\sim m^\ell m^{-2(\ell+1)}
=m^{-\ell-2}.
\end{equation}
Furthermore, rather crudely, if $\ell\le m/2$, then
\begin{equation}\label{sop}
  a_\ell(m)\le m^\ell\frac{1}{(m(m-\ell))^{\ell+1}}
\le \frac{m^\ell}{(m^2/2)^{\ell+1}}
= \frac{2^{\ell+1}}{m^{\ell+2}}
\end{equation}
and thus
\begin{equation}\label{sop-}
  \sum_{\ell=1}^{m/2}a_\ell(m) 
\le \frac{2^{2}}{m^{3}}\Bigpar{1-\frac{2}m}\qw
= O\bigpar{m^{-3}}.
\end{equation}
If $m/2<\ell\le m-1$, we instead note that $\ell-m/4+O(1)\ge m/4+O(1)$ of
the factors $(m-i)$, $i=0,\dots,\ell-1$, are less than $\frac{3}4 m$, and
thus, similarly to \eqref{sop},
\begin{equation}\label{sop2}
  a_\ell(m)\le \Bigparfrac34^{m/4+O(1)}m^\ell\frac{1}{(m(m-\ell))^{\ell+1}}
\le \Bigparfrac34^{m/4+O(1)}
\end{equation}
 and thus
 \begin{equation}\label{sop+}
   \sum_{\ell=\floor{m/2}+1}^{m-1} a_\ell(m)
=
O\lrpar{m\Bigparfrac34^{m/4}}
=O\bigpar{m^{-3}}.
 \end{equation}
Consequently, by \eqref{a0}, \eqref{sop-}
 and \eqref{sop+},
\begin{equation}
  \sum_{\ell=0}^{m-1}a_\ell(m)=a_0(m)+O\bigpar{m^{-3}}
= \frac{1}{m^2}+O\bigpar{m^{-3}}
\end{equation}
and the result follows from \eqref{fris}.
\end{proof}

\begin{remark}
  The proof shows that $\sum_{\ell=k+1}^{m-1}=O\bigpar{m^{-k-3}}$
for any fixed $k$, and thus
\begin{equation}
  P_2(m)=\frac{2}m\sum_{\ell=0}^k a_\ell(m)+O\bigpar{m^{-k-4}}
\end{equation}
which together with \eqref{uk} gives an asymptotic expansion of $P_2(m)$ in
powers of $m\qw$ to arbitrary precision. For example, taking $k=2$,
\begin{equation}
  \begin{split}
  P_2(m)&=\frac{2}m\Bigpar{\frac{1}{m^2+1}+\frac{m}{\rise{m^2-m+1}2}
          +\frac{m(m-1)}{\rise{m^2-2m+1}3}} + O\bigpar{m^{-6}}
\\&
= \frac{2}{m^3}+\frac{2}{m^4}+\frac{4}{m^5}+ O\bigpar{m^{-6}}.	
  \end{split}
\raisetag{\baselineskip}
\end{equation}
\end{remark}

We give some numerical examples for small $m$ and $k$, calculated by \maple{}
using \eqref{pk1}, \eqref{pk2} or (for $k=2$) \eqref{tmprot}.
Recall that
$P_0(m)=1$ and $P_1(m)=2/(m+1)$ by \eqref{Pm1}.
The value $P_2(2)$ was first found by \citet{MahmoudWard-bst}.
\citet{Bona} found also $P_3(2)$ and $P_4(2)$ in an equivalent form; in
our notation he computed $\P(R_0(\friT)=k)$ for $k\le 4$; this was extended
to $k\le6$ by \citet{BonaPittel}, see also \citet{SJ283}.
For $m>2$, $P_2(m)$  was found using \Polya{} urns for $ m=3$ by
\citet{HolmgrenJanson2} and for $m=4,5,6,7 $ by \citet{Heimburger}.
The values of $P_3(3)$ and $P_3(4)$ are new.
(We have also calculated \eg{} $P_4(3)$ and $P_3(5)$, 
but they have too many digits to fit on a line.)
\begin{align*}
P_2(2)&={\tfrac {11}{30}}
\\
P_3(2)&=
{\tfrac {1249}{8100}}
\\
P_4(2)&=
{\tfrac {103365591157608217}{2294809143026400000}}
\\
P_5(2)&=
{\tfrac {
28988432119470126428745503472450231049113704894255010839147852677}{
3353377025022449199852900725670960067418280803797231788288000000000}}
\\
P_2(3)&={\tfrac {19}{140}}
\\
P_3(3)&
={\tfrac {1550707922167467531619}{109171218839281719120000}}
\\
P_2(4)&=
{\tfrac {54731}{1021020}}
\\
P_3(4)&=
{\tfrac {
18669293609273671848329391002724727078204420654827428710651571011801}{
10647619779410811926633805505705710276401405786992612983909862200000000
}}
\\
P_2(5)&=
{\tfrac {3491}{145860}}  
\\
P_2(6)&=
{\tfrac {760687429}{61618399920}}
\\
P_2(7)&=
{\tfrac {30706935422593}{4254556530624400}}
\\
P_2(8)&=
{\tfrac {351290243384177}{75964420635503400}}
\\
P_2(9)&=
{\tfrac {4396518809902327}{1391843538882476680}}
\\
P_2(10)&=
{\tfrac {2342020375280850167304355177}{1036489488884911417519833710040}}
\\
P_2(13)&=
{\tfrac {26484106080648861012732023951265705943439}{
26631299618409351014679549721918492033886700}}
\\
P_2(15)&=
{\tfrac {3345241368583279619305030461877889463409229379479549149}{
5237499866293002299315802678302280652458939709867850337600}}
\\
P_2(18)&=
{\tfrac {
291057072374356381926366502629609543999495527334257770304419581448683067
}{
798438501453518964732193240994834409677783210033312650160515213913065316400
}}
\end{align*}

The numerators and denominators of these
rational numbers evidently grow very rapidly with $k$; \citet{BonaPittel}
note that 
(in our notation)
the denominator of $P_6(2)$ has 274 digits, but the largest prime
factor is only 61, and they show that in general, the largest prime factor
of the denominator of $P_k(2)$ is at most $2^k+1$.
This can be generalized to arbitrary $m$, using the recursion above;
this also gives a new and simpler proof for $m=2$.
(Nothing similar seems to hold for the numerators; they typically have only
a few and often large prime factors in these examples. The numerator of
$P_2(13)$ happens to be a prime with 41 digits.)

\begin{theorem}
The largest prime factor of the denominator of $P_k(m)$ is at most $m^k+1$,
for any $k\ge1$ and $m\ge2$.  
\end{theorem}

\begin{proof}
A simple calculation shows that
  for integers $j,a,b\ge0$ with $a\neq b$, 
the convolution $t^j e^{-at}*e^{-bt}$ is of the form
$\sum_{i=0}^j c_i t^i e^{-at}+c'e^{-bt}$ with coefficients
$c_i,c'\in(b-a)^{-j-1}\bbZ$; if $a=b$, we obtain instead
$\frac{1}{j+1}t^{j+1}e^{-at}$. 
It follows by \eqref{jb}, \eqref{h1} and induction that for $k\ge1$, 
$h_k(t)$ is a polynomial in $e^{-t}$ and $t$ of degree $m^k$ in $e^{-t}$ and
of degree (at most) $1+m+\dots+m^{k-2}=(m^{k-1}-1)/(m-1)$ in $t$, 
with rational coefficients whose denominators have all their prime factors
$<m^k$. The result then follows from \eqref{pk1}.
\end{proof}

For the binary case $m=2$, the probabilities $P_k(2)$ where shown to have an
exponential decay by \citet{BonaPittel}.
We conjecture that this holds for $m\ge3$ too, but leave that as an open
problem.

\subsection{Random recursive tree}  \label{SSrankRRT}
Consider the \rrt{} in \refE{ERRT}. This has been studied by
\citet{MahmoudWard-rrt} and \citet{SJ283};
we follow here \cite{SJ283}.
Let $\ctt$ be the Yule tree process in
\refE{ERRT} and define
\begin{equation}
p_k(t):=\P(R_0(\ctt)\ge k)
,  
\end{equation}
the probability that the root of $\ctt$ is $k$-protected.
By the construction of the fringe tree $\friT=\cttaux$, with
$\taux\sim\Exp(1)$,
the limit $P_k$ in \eqref{trank} is given by
\begin{equation}\label{pkrrt}
  P_k=\intoo p_k(t)e^{-t}\dd t.
\end{equation}

The functions $p_k(t)$ can, in principle, be found by recursion.
The children of the root in $\ctt$ arrive according to a Poisson process
with intensity 1, and
a child that is born at time $s\le t$ is not $(k-1)$-protected at time $t$
with probability $1-p_{k-1}(t-s)$. 
Hence, for any $k\ge1$, 
the number of children of the root in $\ctt$ that are 
not $(k-1)$-protected at time $t$ is Poisson distributed
with mean 
$\intot \bigpar{1-p_{k-1}(t-s)}\dd s=\intot \bigpar{1-p_{k-1}(s)}\dd s$.
Since the root is $k$-protected if and only if there is no such child,
but there is at least one child, and the probability that there is no child
at all is $e^{-t}$, we obtain the recursion
\begin{equation}\label{rrtrec}
  \begin{split}
p_k(t)&=\exp\lrpar{-\intot \bigpar{1-p_{k-1}(s)}\dd s}-e^{-t}	
\\&
=
e^{-t}\lrpar{ \exp\lrpar{\intot p_{k-1}(s)\dd s}-1},
\qquad k\ge1,\, t\ge0,
  \end{split}
\end{equation}
with $p_0(t)=1$.

Taking $k=1$ in \eqref{rrtrec} we obtain the obvious $p_1(t)=1-e^{-t}$.
Taking $k=2$ in \eqref{rrtrec}, we then find
  \begin{equation}
	p_2(t)=\exp\bigpar{e^{-t}-1}-e^{-t}
  \end{equation}
and thus by \eqref{pkrrt}
  \begin{equation}
	\begin{split}
P_2&
=\intoo \exp\bigpar{e^{-t}-1}e^{-t}\dd t- \intoo e^{-2t}\dd t
\\&
=\intoi \exp(x-1)\dd x-\frac12
=
\frac12-e^{-1},	  
	\end{split}
  \end{equation}
in accordance with \citet{MahmoudWard-rrt}.

In principle, the recursion \eqref{rrtrec} yields $p_k(t)$ and $P_k$ 
for larger $k$ too, but we do not know any closed form for $k\ge3$.

\section{Maximal clades in $m$-ary search trees}\label{Sclades}

We define a \emph{clade} in an $m$-ary tree to be a node with less than $m$
children. 
(In the formulation using extended \mst{s} with external nodes, a clade is
thus a node with at 
least one external child.)  
A \emph{maximal clade} is a clade such that no ancestor is a clade.

\begin{remark}
The reason for this somewhat strange terminology comes from applications 
of the binary case $m=2$ to
mathematical biology,
where the clade is regarded as a set of external nodes, see
\eg{} \citet{BlumF}, \citet{DBF07}, \citet{ChangF10}, \citet{DF10}, 
\citet{DFL} and
(for the elementary equivalence with the definition here)
\citet{SJ299}. We consider here the natural extension to $m$-ary trees.
(As a mathematically interesting example; we do not claim any biological
applications.) 
\end{remark}

The number of clades is thus the number of nodes with outdegree less than
$m$, and the fraction of such nodes is by \refT{TBP} asymptotically given 
by the probability that the root of the asymptotic fringe tree $\friT$ has
outdegree less than $m$. (This is found to be $1-\frac{2}{m(m+1)}$ in
\refT{TMSTdegree}.) 

The property that a clade is maximal, however, depends also on  its ancestors,
and therefore we need the extended fringe and the random sin-tree $\sinT$;
moreover, we have to consider all 
ancestors, so \refT{TBPX} does not apply and we use \refT{TQ}.

\begin{theorem}\label{TMC}
  Let $\numc(T)$ be the number of maximal clades in $T$. 
If $T_n$ is a random \mst{} with $n$ keys, then
  \begin{equation}
	\frac{\numc(T_n)}{|T_n|}\asto 
\pmc=\pmc(m)=\P\bigpar{o\text{ is a maximal clade in }\sinT}.
  \end{equation}
\end{theorem}

\begin{proof}
We apply \refT{TQ} with $\cP_0=\cQ=$ ``the outdegree is $<m$'',
\ie, the property that a node is a clade. Then $\cP$ in \refT{TQ} is the
property that a node is a maximal clade.
The assumption \eqref{lq0} holds trivially, since $\Xia\le m$.
The random variable  $\gL$ is the time the root of $\ctt$ gets it final
child; by \refR{Rmstbirths}, this can be written as a sum of a number of
exponential variables (with different rates), and thus
\eqref{lqe} holds for some small $\gd>0$. (In fact, for all $\gd<1$, by
\refR{Rmstbirths} and \refT{TC}.)

Hence, \refT{TQ} applies and the result follows.  
\end{proof}

The constant $\pmc(2)$, \ie, the asymptotic proportion of maximal clades in
a binary search tree, was found to be $(1-e^{-2})/4$ by \citet{DF10}, see also
\cite{DFL} and \cite{SJ299}. We give a different proof of this, using the
properties of the sin-tree $\sinT$ in \refE{EBST}.

\begin{theorem}\label{TMC2}
  \begin{equation}
	\pmc(2)=\frac{1-e\qww}4.
  \end{equation}
\end{theorem}

\begin{proof}
  Recall the general construction of the sin-tree $\sinT$ in \refS{S:BP} and
  the specific version for the binary search tree in \refE{EBST}.
In the construction, we stop the tree at $\taux\sim\Exp(1)$, but we first
consider the tree $\sinT_t$ at a fixed time $t\ge0$. (Equivalently, we
condition on $\taux=t$.)

We thus want to compute the probability 
$\P\bigpar{o\text{ is a maximal clade in }\sinT_t}$.
We first note that $o$ is a clade unless it already has got its two
children; each child has appeared with probability $1-e^{-t}$ and thus
\begin{equation}\label{cp0}
\P\bigpar{o\text{ is a  clade in }\sinT_t}
=1-\xpar{1-e^{-t}}^2
=2e^{-t}-e^{-2t}.
\end{equation}

We also require that no ancestor is a clade, \ie, that each ancestor has two
children. Note that each ancestor has an heir, so it is not a clade if and
only if the other child is not yet born.
Suppose that the ancestors are born at times $-\eta_1, -\eta_2,
\dots$, 
and condition on these times. Ancestor $o\mm i$ thus has age $\eta_i+t$ at
time $t$, so the probability that it is not a clade is $e^{-(\eta_i+t)}$.
Consequently, using the independence of different parts of the sin-tree,
\begin{equation}\label{cp}
  \P\bigpar{o\text{ is a maximal clade in }\sinT_t\mid \eta_1,\dots}
=\bigpar{2e^{-t}-e^{-2t}}
\prod_{i=1}^\infty \bigpar{1-e^{-(\eta_i+t)}}.
\end{equation}

The next step is to find the expectation of \eqref{cp} over all
$\set{\eta_i}$.
In the present case, this is not difficult since, by \refE{EBST}, 
\set{-\eta_i} is a Poisson process with intensity 2 on $(-\infty,0)$,
and thus 
\set{\eta_i} is a Poisson process with intensity 2 on $(0,\infty)$.
For any Poisson process $\Xi=\set{\xi_i}$ on some space $\cS$, with
intensity measure $\gl$, and any function  $f$ on $\cS$ with 
$0\le f(x)\le 1$, there is a standard formula
\begin{equation}\label{pop}
  \E \prod_{\xi\in\Xi}f(\xi) = e^{-\int_{\cS}\xpar{1-f(x)}\gl(\ddx x)}
.
\end{equation}
(See \eg{} \cite[Lemma~12.2]{Kallenberg}.
Or note that \eqref{pop} follows easily if $f$ takes only a finite number of
values, and the general case follows by monotone convergence.)
Consequently, taking $f(x)=1-e^{-x-t}$,
\eqref{cp} yields
\begin{equation}
  \begin{split}
  \P\bigpar{o\text{ is a maximal clade in }\sinT_t}
&
=\bigpar{2e^{-t}-e^{-2t}}
\E\prod_{i=1}^\infty f(\eta_i)
\\&
=\bigpar{2e^{-t}-e^{-2t}}
e^{-\intoo e^{-x-t}2\dd x}	
\\&
=\bigpar{2e^{-t}-e^{-2t}}
e^{-2 e^{-t}}.	
  \end{split}
\end{equation}
Finally, recalling that $\sinT=\sinT_{\taux}$ with $\taux\sim\Exp(1)$,
\begin{equation}
  \begin{split}
  \P\bigpar{o\text{ is a maximal clade in }\sinT}
&
=\intoo\bigpar{2e^{-t}-e^{-2t}}e^{-2 e^{-t}} e^{-t}\dd t
\\&
=\intoi\bigpar{2x-x^{2}}e^{-2 x}\dd x =\frac{1}4-\frac{1}4e^{-2}.
  \end{split}
\end{equation}
\end{proof}

For further, somewhat surprising, results on the number of maximal clades in
the binary case (moments and asymptotic distribution), see
\citet{DFL} and
\citet{SJ299}.

\begin{problem}
Unfortunately, we do not know how to compute $\pmc(m)$ for $m>2$, and we
leave this as an open problem.
Using the description in \refSS{SSmst} of $\sinT$,
it is straightforward to modify \eqref{cp0}--\eqref{cp} (although the result
is more complicated since the birth times do not have exponential
distributions, see \refR{Rmstbirths}), 
but the birth times of the ancestors do not
form a Poisson process so \eqref{pop} does not apply and we do not know how 
compute the expectation.  
\end{problem}

We can use the same method for other, related, problems.
We give two examples.
Let us first consider again the \bst, but we simplify the property
of being a maximal clade studied above
by considering only the condition for the ancestors but ignoring the
number of children.
Thus, let $\cPxi$ be the property of a node that none of its ancestors has
only one child, and let
$\nuxi(T)$ be the number of nodes in $T$ with this property.

\begin{theorem}\label{Tpx1}
If $T_n$ is a random \bst{} with $n$ keys, then
  \begin{equation}
	\frac{\nuxi(T_n)}{|T_n|}\asto 
\frac{1-e\qww}2.
  \end{equation}
\end{theorem}
\begin{proof}
  We apply \refT{TQ} as in the proof of \refT{TMC}, but with $\cP_0$ 
the  trivial property ``true''. This yields convergence almost surely,
to the limit $\P(o \text{ has $\cPxi$ in }\sinT)$.
This probability is computed as in the proof of \refT{TMC2}, replacing the
factor \eqref{cp0} by 1, which yields the result
\begin{equation}\label{tupx}
\intoo e^{-2 e^{-t}} e^{-t}\dd t
\\
=\intoi e^{-2 x}\dd x =\frac{1}2-\frac{1}2e^{-2}.
\end{equation}
\end{proof}

The property $\cPxi$, as formulated above, can be studied also in other trees.
We consider the random recursive tree as a different simple example.
\begin{theorem}\label{Tpx1rrt}
If $T_n$ is a \rrt{} with $n$ keys, then
  \begin{equation}
	\frac{\nuxi(T_n)}{|T_n|}\asto 
{1-e\qw}.
  \end{equation}
\end{theorem}
\begin{proof}
We argue as in the proof of \refT{Tpx1}, now using the description of the
sin-tree $\sinT$ in \refE{ERRT}.
In this sin-tree, 
the ancestors form a Poisson process with intensity 1 on $(-\infty,0)$, 
and, as in the \bst{} case, for an ancestor, 
the time until birth of the first non-heir
is $\Exp(1)$. Hence the limit 
$\P(o \text{ has $\cPxi$ in }\sinT)$ 
can be calculated by the method above, now yielding, \cf{} \eqref{tupx},
\begin{equation}
\intoo e^{- e^{-t}} e^{-t}\dd t
\\
=\intoi e^{- x}\dd x =1-e^{-1}.
\end{equation}
\end{proof}

\section{Restricted sampling and sampling by a random key}\label{Ssample}

We have so far considered the properties of a random node in the tree $T_n$.
As pointed out by   \citet{JagersNerman1984}, 
one can similarly obtain results for a random node sampled with some
restriction.
(For example, a random leaf, a random non-leaf, a random node with no
sibling, \dots .)

In general, let $\cQ$ be a property of the type in \refT{TBP} or \ref{TBPX}
and sample $v$ uniformly among all nodes in $T_n$ that satisfy $\cQ$.
If $\cP$ is another such property, then, by \refT{TBPX},
\begin{equation}\label{sample}
  \P(v \text{ has } \cP\mid T_n)
= \frac{n_{\cP\land\cQ}(T_n)}{n_{\cQ}(T_n)}
\asto
\frac{p_{\cP\land\cQ}(\sinT)}{p_{\cQ}(\sinT)}.
\end{equation}
If we let $\sinTQ$ denote $\sinT$ conditioned  on $o\in\cQ$, then we can 
write \eqref{sample} as 
\begin{equation}\label{sampleQ}
  \P(v \text{ has } \cP\mid T_n)
= \frac{n_{\cP\land\cQ}(T_n)}{n_{\cQ}(T_n)}
\asto
p_{\cP}(\sinTQ).
\end{equation}

If  \refT{TBP} applies, we can replace $\sinT$ by $\friT$ in
\eqref{sample}--\eqref{sampleQ}. 

\begin{example}
  We have already seen an example of this in the first suggested proof of
  \refT{TMSTkeys}, where we note that sampling a node uniformly in an \mst{}
  is the same as sampling an internal node uniformly in the corresponding \emst.
Thus $T_n$ is the \emst{} and $\cQ$ is ``internal''. 
Furthermore, $\cP=\cP_k$ is the property of having exactly $k$ keys.
(In this example, $\cQ$ is the complement of $\cP_0$, so $\cP_0\land\cQ$ is
the empty property while $\cP_k\land\cQ=\cP_k$ for $k\ge1$.)
\end{example}

Let us consider the example of sampling a random leaf $v$ in more detail.
Of course, the fringe tree $T^v$ rooted at $v$ is trivial, so the interest
is in the extended fringe and in properties of the type in \refT{TBPX}.
For example, \citet{Drmota:shape} study 
the number of internal nodes (and the number of leaves) in the
subtree rooted at the father of a randomly chosen leaf, for a variety of
different types of random trees.

We have the following general result.

\begin{theorem}\label{Tleaf}
Suppose that \refBP{} hold, and that $\cP$ is a property as in  \refT{TBPX}.
If $v$ is a uniformly random leaf in $T_n$, then
\begin{equation}\label{tleaf}
\P(v\text{ has }\cP\mid T_n)
\asto p_{\cP} (\sinTo),
\end{equation}
where $\sinTo$ is $\sinT$ conditioned on $o$ being a leaf.

The random sin-tree $\sinTo$ may be constructed directly from the 
tree process $(\sinT_t)$ in \refS{S:BP} by removing all descendants of $o$
and stopping at a random time $\tauo$ with the density function
\begin{equation}\label{tleaf2}
  \frac{e^{-\ga t}\P(\xi_1>t)}{\intoo e^{-\ga t}\P(\xi_1>t)\dd t},
\qquad t>0,
\end{equation}
where $\xi_1$ is the time of birth of the first child of an individual in
the branching process.
In particular, if $\xi_1\sim\Exp(a)$ for some $a>0$, then
$\tauo\sim\Exp(a+\ga)$.
\end{theorem}

\begin{proof}
Let $\cQ$ be the property of a node that it is a leaf. Then \eqref{tleaf} is
the same as \eqref{sampleQ}, with $\sinTo=\sinTQ$, \ie, $\sinT$ conditioned
on $o$ being a leaf.
 
To see that $\sinTo$ can be constructed as stated, note that in the
construction of the tree process
 $(\sinT_t)$ in \refS{S:BP}, the descendants of $o$
and the 
rest of the tree are  independent. 
Since $\sinT$ is obtained by stopping $\sinT_t$ at $\taux$,
it follows that if we ignore descendants of $o$,
$\sinTo$ is obtained by stopping 
$\sinT_t$ at an independent 
random time $\tauo$ having the distribution of $\taux$
conditioned on $o$ being a leaf in $\sinT_\taux$.
Moreover, if the first child of $o$ is
born at $\xi_1$, then $o$ is a leaf in $\sinT=\sinT_\taux$ if and only if
$\xi_1>\taux$. Since $\taux$ has the density function $\ga e^{-\ga t}$, 
it follows that conditioned on the event $\xi_1>\taux$, 
$\taux$ has the density function \eqref{tleaf2}.
\end{proof}

\begin{example}\label{EBSTleaf}
  Let $T_n$ be an extended binary search tree with $n$ internal nodes, let
  $v$ be a (uniformly) randomly chosen external node in $T_n$ and let 
$X_n$ be the number of internal nodes in $T_n^{v,-1}$, \ie, the number of
  internal nodes that are descendants of the parent of $v$. Equivalently,
if $v'$ denotes the sister of $v$, 
  $X_n$ is 1 plus the number of internal nodes in the subtree $T_n^{v}$.
It follows from \refT{Tleaf} that if $X$ similarly is 
the number of  internal nodes of $\sinTo$ that are descendants of the parent
$o\mm1$ of $o$, then
$\P(X_n=k\mid T_n)\asto \P(X=k)$ for every $k\ge1$,
\ie, $X_n\dto X$, also conditioned on $T_n$ in the sense 
$\cL(X_n\mid T_n)\asto\cL(X)$, \cf{} \refR{Rk}.

For the extended binary search tree, we use the branching process in
\refSS{SSemst} (with $m=2$), where each individual gets 2 children at the
same time,  at age $\xi_1\sim\Exp(1)$. Hence, the last statement of
\refT{Tleaf} applies with $a=1$. Furthermore, $\ga=1$, and thus
$\tauo\sim\Exp(2)$. 

Since $X$ equals 1 plus the number of internal nodes in the subtree
$\sinTo^{o'}$ rooted at the sister $o'$ of $o$, we do not have to consider the
ancestor $o\mm1$; we just note that $o'$ and $o$ are twins, and thus $o'$
too is born at time 0. The number of internal nodes in $\sinT_t^{o'}$ is a
pure birth process with birth rates $\gl_k=k+1$, started at 0.
(In other words, it  is $\cY_t-1$, where $\cY_t$ is a Yule process, see
\refE{EYule1}.) Stopping this at $\tauo\sim\Exp(2)$ we find by \refT{TAlin}
(with $\chi=\rho=1$, $\ga=2$) $X-1\sim\HG(1,1;4)$.
In other words, using \eqref{hg}--\eqref{hgC},
\begin{equation}
  \P(X=k)=\frac{4}{(k+1)(k+2)(k+3)},
\qquad k\ge1.
\end{equation}
This limit distribution was found by \citet{Drmota:shape} (in a slightly
different setting, keeping track of the position of the nodes).
\end{example}

\begin{example}\label{ERRTleaf}
  Let $T_n$ be a random recursive tree with $n$ nodes, as in \refE{ERRT},
let $v$ be a uniformly random leaf in $T_n$ and let 
$\htx_n:=T_n^{v,-1}\setminus\set{v}$, the random tree consisting
of the parent $v\mm1$ of $v$ and all its decendants except $v$.
Further, let $X_n:=|\htx_n|$, the number of 
nodes that are descendants of the parent  of $v$, including
the parent but excluding $v$. 

As in \refE{EBSTleaf}, it follows from \refT{Tleaf} that if $X$ similarly is 
the number of nodes in $\htx:=\sinTo^{o,-1}\setminus\set{o}$, then
$\P(X_n=k\mid T_n)\asto \P(X=k)$ for every $k\ge1$,
and thus $\cL(X_n\mid T_n)\asto\cL(X)$, \cf{} \refR{Rk}.

By the description at the end of \refE{ERRT}, the tree process
$\sinTox{t}^{o,-1}$, minus $o$ and its descendants, is a Yule tree process,
starting at time $-\heir$ when the parent $o\mm1$ is born. Furthermore,
$\heir\sim\Exp(1)$. 
We stop this tree process at time $\tauo$, where
by the last statement in \refT{Tleaf} (with $a=\ga=1$),
$\tauo\sim\Exp(2)$.
Consequently, $X\eqd \cY_{\heir+\tauo}$, a Yule process (started at time 0 as
usual) stopped at $\heir+\tauo$, with $\heir\sim\Exp(1)$ and
$\tauo\sim\Exp(2)$ independent.

A simple calculation shows that $\heir+\tauo$ has the density function
$2e^{-t}\bigpar{1-e^{-t}}$, $t>0$, while $\cY_t\sim\Gei(e^{-t})$,
see \refE{EYuleA}.
Hence, for any $k\ge1$, with  $x=e^{-t}$,
\begin{equation}\label{blakulla}
  \begin{split}
	  \P(X=k)&
=2\intoo\P(\cY_t=k)e^{-t}\bigpar{1-e^{-t}}\dd t
\\&
=2\intoo e^{-t}\xpar{1-e^{-t}}^{k-1}e^{-t}\bigpar{1-e^{-t}}\dd t
=2\intoi x(1-x)^k\dd x
\\&
=\frac{2}{(k+1)(k+2)}.
  \end{split}
\raisetag{\baselineskip}
\end{equation}

We have here considered only the number of nodes in
$\htx$. However, it is furthermore clear from
the symmetry of the Yule tree process that given $X=|\htx|$, the random tree
$\htx$ is distributed as a random recursive tree of order $X$, \ie, $\htx$
is a random recursive tree with random order $X$ given by \eqref{blakulla}. 
This describes, at least in principle, any properties of $\htx$.
For example, we may as in \cite{Drmota:shape} count leaves and non-leaves
separately in $\htx$. 
It is easy to see by induction that if $k\ge2$, then  
\begin{equation}
  \P(T_k \text{ has $i$ leaves})
=   \P(T_k \text{ has $i$ non-leaves})
= \frac{\euler{k-1}{i-1}}{(k-1)!},
\end{equation}
where $\euler{k-1}{i-1}$ denotes the Eulerian number 
(see \eg{} \cite{CM} or \cite{NIST});
\ie, the number of leaves in $T_k$ is distributed as 1 + the number of ascents
in a random permutation of length $k-1$. (In fact, both random vectors 
$(\text{\#leaves},\text{\#non-leaves})$ and 
$(1+\#\text{ascents},1+\#\text{descents})$
evolve when $k$ is increased as generalized \Polya{} urns with balls of two
colours where we draw a ball and return it together with a ball of the
opposite colour.)
Consequently we find, for $k\ge2$ and $1\le i\le k$, 
\begin{equation}\label{skartorsdag}
  \begin{split}
\hskip4em&\hskip-4em  \P\bigpar{\htx_n\text{ has $i$ non-leaves and $k-i$ leaves}}
\\&
\asto  
\P\bigpar{\htx\text{ has $i$ non-leaves and $k-i$ leaves}}
\\&
=\frac{2}{(k+1)(k+2)}\frac{\euler{k-1}{i-1}}{(k-1)!}	
=\frac{2k\euler{k-1}{i-1}}{(k+2)!}.	
  \end{split}
\end{equation}
Let $p_i:=\P\bigpar{\htx\text{ has $i$ non-leaves}}$.
Summing \eqref{skartorsdag} over $k$ we find for example, after short
calculations (partly assisted by \maple), 
$p_1=6-2e\approx0.563$,
$p_2=11-4e\approx0.127$,
$p_3=\frac{857}{54}-5e-\frac12e^2+\frac{2}{27}e^3\approx0.072$.
\end{example}
Using \cite[(26.14.6)]{NIST}, it is easy to see that $p_i$ is a polynomial
in $e$ with rational coefficients, of degree at most $i$, but we do not know
any simple general formula for $p_i$.

\subsection{Sampling a random key}

Similarly, in an \mst, one might sample a \emph{key} uniformly at random
and consider the properties of the node containing that key.

\begin{theorem}\label{TK}
  Let $T_n$ be a random \mst, and let $\cP$ be a property as in \refT{TBP}.
Sample a random key uniformly, and let $v$ be the node containing that
key. Then, as \ntoo, letting $R(T)$ denote the number of keys in the root of
$T$,
\begin{equation}\label{mp}
  \P\bigpar{v\text{ has }\cP}\asto
\frac{\E\bigpar{R(\friT)\ett{\friT\in\cP}}}{\E R(\friT)}
=\frac{\E\bigpar{R(\friT)\ett{\friT\in\cP}}}{2(H_m-1)}.
\end{equation}
\end{theorem}

\begin{proof}
  Let $Q_k$ be the property of a node $v$ that it contains $k$ keys.
Then, by \refT{TBP},
\begin{equation}
  \begin{split}
  \P\bigpar{v\text{ has }\cP}
= \frac{\sum_k k n_{\cP\land\cQ_k}(T_n)}{\sum_k k n_{\cQ_k}(T_n)}
&\asto
 \frac{\sum_k k \P(\friT\in\cP\land\cQ_k)}{\sum_k k \P(\friT\in\cQ_k)}
\\&
=
 \frac{\E\bigpar{\sum_k k \ett{\friT\in\cP\land\cQ_k}}}
{\E\bigpar{\sum_k k \ett{\friT\in\cQ_k}}},
  \end{split}
\end{equation}
which equals the second term in \eqref{mp} because $\friT\in\cQ_k\iff
R(\friT)=k$. 
Furthermore, for the same reason,
\begin{equation}\label{gw}
  \begin{split}
 \frac{\sum_k k n_{\cQ_k}(T_n)}{|T_n|}
\asto
\sum_k k \P(\friT\in\cQ_k)
=
\E R(\friT).
  \end{split}
\end{equation}
Since $\sum_k k n_{\cQ_k}(T_n)=n$, the total number of keys,
\eqref{gw} and \refT{TMSTnodes} imply
\begin{equation}\label{ER}
\E R(\friT)=2(H_m-1)
\end{equation}
which completes the proof.
(Alternatively, \eqref{ER} follows from \refT{TMSTkeys}, noting that the
limits in \eqref{nk2} are the probabilities $\P(R=k)$.)
\end{proof}

\begin{remark}
  \refT{TK} extends to properties as in \refT{TBPX} 
(or \refT{TQ})
with only notational
  changes: replace  
$R(\friT)$ by the number of keys in the distinguished node $o$
and
$\ett{\friT\in\cP}$ by $\ett{o\text{ has }\cP}$.
\end{remark}

\begin{example}
Let $K'$ be the number of keys in the node containing a random key in an
  \mst{} $T_n$.
Theorems \ref{TK} and \ref{TMSTkeys} imply that
\begin{equation}\label{nk2k}
  \begin{split}
\P(K'=k)
\asto
\begin{cases}
\frac{1}{H_m-1}\frac{k}{(k+1)(k+2)}, & 1\le k\le m-2,	
\\
\frac{m-1}{m(H_m-1)}, & k=m-1.
\end{cases}
  \end{split}
\end{equation}
For $m=3,4,5$, this yields the limit distributions
$\bigpar{\frac{1}5,\frac{4}5}$,
$\bigpar{\frac{2}{13},\frac{2}{13},\frac{9}{13}}$,
and 
$\bigpar{\frac{10}{77},\frac{10}{77},\frac{9}{77},\frac{48}{77}}$.
\end{example}

\begin{example}
Let $D'$ be the number of children (the outdegree) of the node containing a
random key in an   \mst{} $T_n$.
It follows from Theorems \ref{TK} and \ref{TMSTdegree},
noting that a node with outdegree different from 0 always contains $m-1$ keys,
\begin{equation}\label{mstD'}
  \begin{split}
\P(D'=k)
\asto
\begin{cases}
1-\frac{(m-1)}{(m+1)(H_m-1)}, & k=0,
\\
\frac{(m-1)}{m(m+1)(H_m-1)}, & 1\le k\le m.
\end{cases}
  \end{split}
\end{equation}
For $m=3,4,5$, this yields the limit distributions
$\bigpar{\frac{2}5,\frac{1}5,\frac{1}5,\frac{1}5}$,
$\bigpar{\frac{29}{65},\frac{9}{65},\frac{9}{65},\frac{9}{65},\frac{9}{65}}$
and 
$\bigpar{\frac{37}{77},\frac{8}{77},\frac{8}{77},\frac{8}{77},\frac{8}{77}, 
\frac{8}{77}}$.
\end{example}

\section{Height, profile and typical depth}\label{Sheight}
We consider in this paper fringe properties of random trees. However, 
the connection with \CMJbp{es} has also been used very fruitfully to study
properties related to the distance to the root, in particular the height of
the tree.
This was pioneered by Devroye \cite{Devroye1986}
using results by Kingman \cite{Kingman} 
and Biggins \cite{Biggins76,Biggins77}
for branching random walks with discrete time (based on Galton--Watson
processes), see also
\citet{Devroye1987},
\citet{Mahmoud:Evolution},
the survey 
\citet{Devroye1998}, and 
\citet{BroutinDevroye2006}.
(Partial results for the binary search tree had been proved earlier by
\citet{Pittel1984}, using the same continuous-time branching process
as \cite{Devroye1986} in a somewhat different way.)
The method was further developed by 
Biggins \cite{Biggins95,Biggins97} using the continuous-time \CMJbp{es} used
in the present paper.
We give in this section a description
of the method and some applications and examples;
see the papers just mentioned for further
details and results. (In particular, note the second order results in
\cite{Devroye1987, Devroye1998}.)

Recall that the \emph{depth} $\dep(v)$ of a node $v$ is its distance from
the root. The \emph{height} $H(T)$ of a tree $T$ is defined as 
$\max_{v\in  T} \dep(v)$, the maximum depth of a node. If we consider
$m$-ary trees, we define the \emph{saturation level} $S(T)$ 
(also called \emph{fill-up level}) to be the last
generation that is full, \ie, the largest $k$ such that there are $m^k$
nodes of depth $k$; this equals the minimum depth of a node with outdegree
$<m$. 

The key idea that makes it possible to apply results
on branching random walks 
is to plot the individuals in a branching process in the plane,
using two coordinates that we call \emph{time} and \emph{position}; 
time is the usual time
of birth in the branching process and position is an additional variable.
We assume that for each individual is defined, besides the sequence
$(\xi_i)_{i=1}^N$ of birth times of the children (relative to the birth of
the parent),
also
a sequence $(\eta_i)_{i=1}^N$ (of the same length $N$) of random displacements,
with $-\infty<\eta_i<\infty$;
if the parent is born at time and position $(\gs,y)$, then child $i$ is born
at time and position $(\gs+\xi_i,y+\eta_i)$.
(The general results in \cite{Biggins95}, \cite{Biggins97} allow also a
further random component, describing a random motion of each individual
during its life. For our purposes, we put that motion equal to 0 and
let each individual be static.) 

Results for branching random walks have been applied to the height 
(and other properties) of random trees in two different ways.
In the original application of Devroye \cite{Devroye1986},
see also \cite{Devroye1987,Devroye1998},
the ``position'' is what we have called time in the
\CMJbp, while 
``time'' is the number of the generation, \ie, the depth in the family tree
$\ctt$. 
This means that ``time'' is discrete
and that we consider a
Galton--Watson process where each individual has a position that is its time
of birth in the \CMJ{} process studied elsewhere in the present paper.
(Furthermore, in this application, the Galton--Watson process is
deterministic; in the original application to binary search trees, we
consider an infinite binary tree.)
Note that $H(\ctt)\ge n$ if and only if the minimum position of an
individual in generation $n$ is $\le t$, which gives the required connection
with the theorems on branching random walks.

The alternative approach, described by \citet{Biggins97}, reverses the two
coordinates and lets ``time'' be time in the \CMJbp{} while  ``position'' is
the generation number, \ie, the depth in the family tree. 
The offsets $\eta_i$ are thus non-random with $\eta_i=1$. 
(We sometimes reverse signs and take $\eta=-1$.)
We use this approach in the present section, 
referring to \cite{Biggins97} for further
details on branching random walks and to \cite{Biggins95} for proofs of the
theorems used here.

\subsection{Height}\label{SSheight}

Let $B_t$  be the maximum of the positions $y_x$ of all individuals $x$
that are born before or on time $t$, \ie, with birth time $\gs_x\le t$.
In our case with $\eta_i\equiv1$, $y_x$ is the generation number
of $x$, and thus
$B_t=H(\ctt)$, the height of $\ctt$.

Define the two-dimensional Laplace transform, for
$\zeta,\gth\in(-\infty,\infty)$, 
\begin{equation}\label{m}
  m(\zeta,\gth):=\E\sum_{i=1}^N e^{-\zeta\eta_i-\gth\xi_i}\in (0,\infty].
\end{equation}
Note that by \eqref{hmub}, $m(0,\gth)=\hmu(\gth)$.
In our case with $\eta_i\equiv1$, we thus simply have
\begin{equation}
  \label{mus}
m(\zeta,\gth)=e^{-\zeta}\hmu(\gth).
\end{equation}
Furthermore, let 
\begin{equation}\label{gamma0}
  \begin{split}
  \gamma&=\inf\Bigset{a:\inf_{\zeta<0}\log m(\zeta,-a\zeta)<0}.
  \end{split}
\end{equation}

One of the main results of
\citet{Biggins95,Biggins97} is the following
(valid for general $\eta_i$ under some conditions that are satisfied in our
case, \cf{} \refR{RBIGG} below): 
\begin{theorem}[Biggins \cite{Biggins95,Biggins97}]\label{TBiggins1}
As \ttoo,
\begin{equation}
  B_t/t\asto\gamma. 
\end{equation}
\qed
\end{theorem}

In our case $H(\ctt)=B_t$, so this yields the asymptotic height of $\ctt$;
this translates to the height of $T_n=\cT_{\tau(n)}$ as follows.

\begin{theorem}\label{Theight}
Under the assumptions \refBP{} and \refBPlastpsi,  as \ntoo,
  \begin{equation}\label{theight}
\frac{H(T_n)}{\log n} \asto \bgam:=\frac{\gamma}{\ga}.
  \end{equation}
\end{theorem}

\begin{proof}
  By \refT{TBiggins1}, $H(T_n)/\tau(n)\asto\gamma$, and the result follows
  by \eqref{tex2tau}. 
\end{proof}

\begin{remark}\label{RHfrag} 
The fragmentation trees in \refS{Sfrag} are of  a slightly different type
than the trees $T_n$ that are our main object of study, since they appear as
the family tree $\ctt$ stopped at a fixed time $t=\log(x_0/x_1)$ instead of
a random time $\tau(n)$, see \refR{Rfragtau}. This means that asymptotics
for the height of fragmentation trees follow directly from \refT{TBiggins1}
rather than from \refT{Theight}. 
In this section we usually consider only trees of the type $T_n$,
and leave corresponding results for fragmentation trees to the reader.
\end{remark}

\begin{remark}\label{Rsplit2}
Also the split trees defined by \citet{Devroye98split}, see \refR{Rsplit1},
are in general not exactly of the type of trees studied here, but for the
purpose of studying the height, they can be approximated by fragmentation trees
and similar results can be obtained, see 
\citet{BroutinDevroye2006} and \citet{BroutinDevroyeMcLeish2008}.
\end{remark}

By \eqref{theight}, $\gamma\ge0$, and thus
the fundamental constant $\gamma$ in \eqref{gamma0} and \eqref{theight}
can also be evaluated as
\begin{equation}
  \begin{split}
  \gamma&=\inf\bigset{a>0:\inf_{\zeta<0}\log m(\zeta,-a\zeta)<0}
\\&
=\inf\bigset{a>0:\inf_{\gth>0}\log m(-\gth/a,\gth)<0}.	
  \end{split}
\end{equation}
In our case, when \eqref{mus} holds, this simplifies to
\begin{equation}
  \begin{split}
\gam&
=\inf\bigset{a>0:\inf_{\gth>0}\bigcpar{\gth/a+\log \hmu(\gth)}<0}
\\&
=\inf\Bigset{a>0:\inf_{\gth>0}\Bigcpar{\frac{1}{a}+\frac{\log \hmu(\gth)}{\gth}}<0}
\\&
=\inf\Bigset{a>0:\frac{1}{a}<-\inf_{\gth>0}\frac{\log \hmu(\gth)}{\gth}}
  \end{split}
\end{equation}
and thus 
\begin{equation}\label{gamma2}
  \begin{split}
\gam\qw&
=-\inf_{\gth>0}\frac{\log \hmu(\gth)}{\gth}
=-\inf_{\gth\ge\ga}\frac{\log \hmu(\gth)}{\gth}.
  \end{split}
\end{equation}

Geometrically, \eqref{gamma2} says that $-\gam\qw$ is the slope of the
tangent from the origin to the curve $\log\hmu(\gth)$, $\gth>0$, provided
such a tangent exists. (Otherwise, $-\gam\qw$ is the slope of the asymptote,
as follows from \refL{Lgamma}\ref{Lgamma2} and \refR{Rgamma}  below.)
Analytically, $\gamma$ can be found as follows.
\begin{lemma}\label{Lgamma}
$0<\gb\qw\le \gamma<\infty$.
\begin{romenumerate}[-10pt]
\item \label{Lgamma1}
If $\gth>0$ is a solution of 
the equation
  \begin{equation}\label{mul}
	\gth\frac{\hmu'(\gth)}{\hmu(\gth)}=\log\hmu(\gth),
  \end{equation}
then
\begin{equation}\label{gamma3}
  \gam\qw
=-\frac{\log\hmu(\gth)}{\gth} 
=-\frac{\hmu'(\gth)}{\hmu(\gth)}
=-\bigpar{\log\hmu}'(\gth).
\end{equation}
Furthermore, \eqref{mul}
has at most one positive solution.

\item \label{Lgamma2}
If \eqref{mul} has no positive solution, then
\begin{equation}\label{gamma3b}
  \gam\qw
=-\lim_{\gth\to\infty}\frac{\log\hmu(\gth)}{\gth} 
=\sup\set{x:\mu[0,x)=0}.
\end{equation}
\end{romenumerate}
\end{lemma}

\begin{proof}
  Since $\log\hmu(\gth)<0$ for $\gth>\ga$, \eqref{gamma2} yields $\gam\qw>0$
  and thus $\gam<\infty$. 

It is easy to see, see \refL{LX}\ref{lxhmu} below,
that $\log\hmu(\gth)$ is convex and that it is differentiable 
in the interval $(\AAA,\infty)$ for some $\AAA<\ga$, and thus in particular
at
$\gth=\ga$;
furthermore, using \eqref{el2},
\begin{equation}\label{hmu'gb}
\bigpar{\log\hmu}'(\ga)
=\frac{\hmu'(\ga)}{\hmu(\ga)}
=\hmu'(\ga)=-\gb.
\end{equation}
Hence, for any real $\gth$,
\begin{equation}\label{bertil}
\log\hmu(\gth)
\ge \log\hmu(\ga)+(\gth-\ga)(\log\hmu)'(\ga)
=-\gb(\gth-\ga)
>-\gb\gth.
\end{equation}
Hence, for $\gth>0$,
$\log\hmu(\gth)/\gth> -\gb$, 
and \eqref{gamma2} yields $\gam\qw\le\gb$.
Thus $\gam\ge\gb\qw$.

Next, for any differentiable convex function $f(\gth)$ defined on an interval
$\cI\subseteq(0,\infty)$, the function 
$g(\gth):=\gth f'(\gth)-f(\gth)$ is increasing since
$g'(\gth)=\gth f''(\gth)\ge0$.
Hence, $g(\gth)$ has at most one zero in $\cI$, 
and since
$\bigpar{f(\gth)/\gth}'=g(\gth)/\gth^2$,
a zero of $g(\gth)$ is a global minimum point of $f(\gth)/\gth$ in $\cI$.
Taking $f(\gth):=\log\hmu(\gth)$ on $\cI:=\set{\gth>0:\hmu(\gth)<\infty}$,
we see that $g(\gth)=0$ is equivalent to \eqref{mul}. If \eqref{mul} has a
positive solution, it is thus unique and a minimum point of
$\log\hmu(\gth)/\gth$ which yields \eqref{gamma3} by \eqref{gamma2} and
\eqref{mul}. 

On the other hand, 
if \eqref{mul} has no positive solution, then $g(\gth)$ has a fixed
sign in $\cI$. 
Since $g(\ga)=\ga(\log\hmu)'(\ga)<0$, 
$g(\gth)<0$ for all $\gth\in\cI$
and $f(\gth)/\gth$ is strictly decreasing. 
Thus, the infimum in \eqref{gamma2} is the
limit as $\gth\to\infty$, which yields the first equality in \eqref{gamma3b}.
The final equality is a straightforward property of Laplace transforms.
\end{proof}

\begin{remark}\label{Rgamma}
The case \ref{Lgamma2} in \refL{Lgamma}
is exceptional. We see from \eqref{gamma3b} that
$\mu$ has no mass in $[0,\gam\qw)$, so no child is ever born to a parent of
  age less than $\gam\qw$. 
Moreover, by \eqref{gamma2}, $\hmu(\gth)\ge e^{-\gam\qw\gth}$ for all
$\gth>0$, and it follows easily that $\mu\set{\gam\qw}\ge1$, so $\mu$ has a
point mass at $\gam\qw$.
This case is thus exceptional, and does not appear in any of our examples.  
\end{remark}

The formula \eqref{gamma0} for $\gamma$ is given in 
\cite{Biggins95,Biggins97}, where it is shown that it is equivalent to the
following definition, which is more indirect but perhaps 
more fundamental, involving some other quantities of interest (see \eg{}
\refSS{SSprofile} below).
Let
\begin{equation}\label{gal}
  \ga(\zeta):=\inf\set{\gth:m(\zeta,\gth)\le1}
=\inf\bigset{\gth:\hmu(\gth)\le e^{\zeta}}
=\inf\bigset{\gth:\log\hmu(\gth)\le {\zeta}},
\end{equation}
noting that $\ga(0)=\ga$,
and define its one-sided Legendre transform
\begin{equation}\label{gaxx}
  \gaxx(x):=\inf_{\zeta<0}\bigset{x\zeta+\ga(\zeta)}.
\end{equation}
Then
\begin{equation}\label{gamma1}
  \gamma=\inf\set{x:\gaxx(x)<0}.
\end{equation}

Using \eqref{gal}, we can rewrite \eqref{gaxx} as, for $x\ge0$, 
\begin{equation}\label{gaxx2}
  \begin{split}
\gaxx(x)&
=\inf_{\substack{\zeta<0\\ \log\hmu(\gth)\le\zeta}}\bigset{x\zeta+\gth}	
=\inf_{\hmu(\gth)<1}\bigset{x\log\hmu(\gth)+\gth}
\\&
=\inf_{\gth>\ga}\bigset{x\log\hmu(\gth)+\gth}.	
  \end{split}
\end{equation}

We collect some elementary properties of these functions in a lemma.
(Some of these extend to general $\eta_i$, but not all; \cf{} \cite{Biggins95}.)
Define
\begin{equation}\label{AAA}
  \AAA:=
\inf\set{\gth:\hmu(\gth)<\infty}\in[-\infty,-\infty).
\end{equation}
Thus $\hmu(\gth)<\infty$ for
$\gth>\AAA$, but note that in the case $\AAA>-\infty$,
both $\hmu(\AAA)=\infty$  and $\hmu(\AAA)<\infty$
are possible.
Note also that \ref{BPmub} says that 
\begin{equation}\label{AAAga}
-\infty\le \AAA<\ga.  
\end{equation}
Furthermore, define $\bara\in(0,\infty]$ and $\baram\in\ooo$
by 
\begin{align}
  \bara\qw &= \sup\set{y\ge0:\mu[0,y)=0},\label{supp+}
\\
\baram\qw& = \inf\set{y>0:\mu(y,\infty)=0}, \label{supp-}
\end{align}
\ie, the lower and upper limits of the support of $\mu$.
Note that in all our examples in Sections \ref{Sex}--\ref{Smedian},
$\bara=\infty$ and $\baram=0$.
By \eqref{supp+}--\eqref{supp-}, 
\as{} each $\xi_i\in[\bara\qw,\baram\qw]$. In particular, see
\refR{Rage}, $\heir\in [\bara\qw,\baram\qw]$ \as, and since $\heir$ is not
concentrated at a single value as a consequence of \ref{BPnonlattice},
$\gb=\E\heir\in(\bara\qw,\baram\qw)$, see also \eqref{el3} and
\eqref{malthustau}. In other words,
\begin{equation}
0\le \baram <\gb\qw<\bara\le\infty.
\end{equation}
\begin{lemma}
  \label{LX}
  \begin{thmenumerate}
  \item \label{lxhmu}
$\log\hmu(\gth)$ is a convex and decreasing function on
$(-\infty,\infty)$ with $-\infty<\log\hmu(\gth)\le\infty$.

On $(\AAA,\infty)$, $\log\hmu(\gth)$ is
analytic, strictly decreasing and strictly convex
with $(\log\hmu)'(\gth)<0$
and $(\log\hmu)''(\gth)>0$.
If $\AAA>-\infty$, then $\loghmu(\gth)$ is right-continuous at $\AAA$ in the
sense that
$\lim_{\gth\downto\AAA}\loghmu(\gth)=\loghmu(\AAA)\le\infty$; hence
$\loghmu(\gth)$ is continuous on $[\AAA,\infty)$.

Furthermore, 
\begin{align}\label{lxi-}
\log\hmu(-\infty)&:=\lim_{\gth\to-\infty}\log\hmu(\gth)=\infty.  
\\
\log\hmu(\infty)&:=\lim_{\gth\to\infty}\log\hmu(\gth)=\log\mu\set0
\in[-\infty,0), \label{lxi+}
\end{align}
and
\begin{align}\label{lbaram}
\lim_{\gth\to-\infty}\frac{\log\hmu(\gth)}{-\gth}&=\baram\qw\in(0,\infty],
\\
\lim_{\gth\to\infty}\frac{-\log\hmu(\gth)}{\gth}&=\bara\qw\in\ooo. 
\label{lbara}
\end{align}

\item \label{lxga}
$\ga(\zeta)$ is a convex and decreasing function on
  $(-\infty,\infty)$ with
$-\infty<\ga(\zeta)\le\infty$.

On the interval $\Iga:=(\log\mu\set0,\log\hmu(\AAA))$,
$\ga(\zeta)$ is the inverse function of $\log\hmu:(\AAA,\infty)\to\Iga$; 
hence,
on $\Iga$, 
$\ga(\zeta)$ is analytic and strictly decreasing with $\ga'(\zeta)<0$ 
and $\ga''(\zeta)>0$.
Furthermore, $0\in\Iga$ and $\ga(0)=\ga$; 
hence $\ga(\zeta)>\ga$ for $\zeta<0$ 
and $\ga(\zeta)<\ga$ for $\zeta>0$. 

Moreover, 
see \eqref{el2},
\begin{equation}\label{ga'0}
  \ga'(0)
=\frac{1}{\hmu'(\ga)}=-\frac{1}{\gb},
\end{equation}
and
\begin{align}\label{lxii-}
\ga(-\infty)&:=\lim_{\zeta\to-\infty}\ga(\zeta)=\infty.  
\\
\ga(\infty)&:=\lim_{\zeta\to\infty}\ga(\zeta)
=\AAA
\in[-\infty,\ga), \label{lxii+}
\end{align}
and
$\ga(\zeta)=\infty$ if\/ $-\infty<\zeta\le\log\mu\set0$
and
$\ga(\zeta)=\AAA$ if\/ $\log\hmu(\AAA)\le\zeta<\infty$.
In particular, $\ga(\zeta)<\infty\iff \zeta>\log\mu\set0$;
hence $\ga(\zeta)<\infty$ for all $\zeta$ if and only if $\mu\set0=0$.
Furthermore,  $\ga(\zeta)<\infty$ for some $\zeta<0$.

\item \label{lxgaxx}
$\gaxx(x)$ is a concave and decreasing function on
  $(-\infty,\infty)$ with
$-\infty\le \gaxx(x)<\infty$.
We have
\begin{equation}\label{gaxxgb}
  \begin{cases}
	\gaxx(x) = \ga, & x\le 1/\gb,
\\
	\gaxx(x) < \ga, & x> 1/\gb
  \end{cases}
\end{equation}
and 
\begin{equation}
  \label{gaxxoo}
\lim_{\xtoo}\gaxx(x)=-\infty.
\end{equation}
Furthermore, where the last two cases apply only when $\bara<\infty$,
\begin{equation}\label{bara123}
  \begin{cases}
\gaxx(x)>-\infty,&x<\bara,\\
\gaxx(x)=\bara\log\mu\set{\bara\qw}\ge-\infty,&x=\bara,\\
\gaxx(x)=-\infty,&x>\bara,\\
  \end{cases}
\end{equation}
In particular, if the birth times $\xi_i$ can be arbitrarily small (as in
all our examples), so $\bara=+\infty$, then
$\gaxx(x)>-\infty$ for every real $x$.

Consequently,
$\gaxx(x):\bbR\to[-\infty,\infty)$ is continuous everywhere except possibly
  at $\bara$, and left-continuous everywhere.
  \end{thmenumerate}
\end{lemma}

\begin{proof}
  \pfitemref{lxhmu} The log-convexity of $\hmu$
 is well-known and follows from \Holder's inequality. The remaining
 statements are also well-known properties of Laplace transforms,
and follow easily from the definition
 \eqref{Lm}, 
using monotone and dominated convergence together with \ref{BPfirst} and
\ref{BPmub} (or \ref{BPmalthus}) for \eqref{lxi-}--\eqref{lxi+}
and simple estimates for \eqref{lbaram}--\eqref{lbara};
note also that \ref{BPnonlattice} implies that $\mu$ is not concentrated at
one point.

\pfitemref{lxga}
That $\ga$ is convex follows easily from \eqref{gal} and the convexity of
$\log\hmu$. Furthermore, \eqref{gal} implies that $\ga$ is decreasing. 
The remaining properties also follow easily from \eqref{gal} and
\ref{lxhmu};
that $0\in\Iga$ follows from \ref{BPfirst} and \ref{BPmub}, and
the fact that $\ga(\zeta)$ is the inverse function of $\log\hmu(\gth)$ implies
\begin{equation}
  \ga'(\zeta)=\frac{1}{(\log\hmu)'(\ga(\zeta))}
=\frac{\hmu(\ga(\zeta))}{\hmu'(\ga(\zeta))}
=\frac{e^\zeta}{\hmu'(\ga(\zeta))},
\qquad \zeta\in\Iga,
\end{equation}
which in particular yields \eqref{ga'0}.

\pfitemref{lxgaxx}
That $\gaxx$ is concave and decreasing follows from \eqref{gaxx}.
Furthermore, by \ref{lxga}, $\ga(\zeta)<\infty$ for some $\zeta<0$, and thus
$\gaxx(x)\le x\zeta+\ga(\zeta)<\infty$ for every $x$.
This implies also \eqref{gaxxoo} by letting $\xtoo$ (with $\zeta<0$ fixed).

Since $\ga(\zeta)$ is convex, \eqref{ga'0} implies that
$\ga'(\zeta)\le-1/\gb$ for $\zeta<0$.
It follows that the infimum in \eqref{gaxx} is attained at $\zeta=0$ if
$x\le1/\gb$, but at some $\zeta<0$ (possibly $-\infty$) if $x>1/\gb$, and
\eqref{gaxxgb} follows, recalling $\ga(0)=\ga$.

Next, \eqref{gaxx2} implies that for any $x>0$,
\begin{equation}\label{2016}
  \begin{split}
\exp\bigpar{\gaxx(x)/x}
=\inf_{\gth>\ga} \bigset{\exp(\gth/x)\hmu(\gth)}
=\inf_{\gth>\ga} \Bigset{\intoo e^{\gth(x\qw-t)}\mu(\ddx t)}.
  \end{split}
\end{equation}
If $x=\bara<\infty$, then $t\ge x\qw$ on the support of $\mu$ by \eqref{supp+};
hence the integral in  \eqref{2016} is a decreasing function of $\gth$,  
and dominated convergence 
as $\gth\to\infty$
shows that the infimum
equals $\mu\set{\bara\qw}$.
Similarly, if $x>\bara$, then the infimum in \eqref{2016} is 0.
On the other hand, if $x<\bara$, then
$\intoo e^{\gth(x\qw-t)}\mu(\ddx t)\ge
\mu\set{t:t\le x\qw}>0$ for every $\gth>0$, and thus the infimum in
\eqref{2016} is positive.
This shows \eqref{bara123}. 

It follows that $\gaxx(x)$ is concave and finite
for $x\in(-\infty,\bara)$, and thus continuous there. By \eqref{bara123},
$\gaxx(x)$ is also trivially continuous for $x>\bara$.
Finally, 
$\gaxx(x)$ is left-continuous everywhere 
since the definition
\eqref{gaxx} implies
\begin{equation}\label{kastrup}
  \lim_{ y\upto x} \gaxx(y)
=
  \inf_{ y< x} \gaxx(y)
=
  \inf_{ y< x,\; \zeta<0}\set{y\zeta+\ga(\zeta)}
=
  \inf_{ \zeta<0}\set{x\zeta+\ga(\zeta)}
=\gaxx(x).
\end{equation}
\end{proof}

\begin{example}\label{EHRRT}
For the \rrt{} in \refE{ERRT}, 
originally treated by \citet{Devroye1987} in a related but somewhat
different way, see also \cite{Devroye1998}, 
we have $\hmu(\gth)=1/\gth$, $\gth>0$, see \eqref{hmuRRT}. 
Consequently, \eqref{mul} is
\begin{equation}
  -1 = -\log\gth
\end{equation}
with the solution $\gth=e$,
and then \eqref{gamma3} yields
\begin{equation}
  \gam\qw 
=-\bigpar{\log\hmu}'(\gth)
= \frac{1}{\gth}=e\qw
\end{equation}
\ie, $\gam=\gth=e$.
Since $\ga=1$, the limit in \eqref{theight} is $\bgam=\gamma$.

Furthermore, \eqref{gal} and \eqref{gaxx} yield after short calculations,
\begin{equation}
  \ga(\zeta):=\inf\Bigcpar{\gth>0:\frac{1}{\gth}\le e^{\zeta}} 
=e^{-\zeta},
\end{equation}
\begin{equation}\label{gaxxRRT}
  \gaxx(x)=\inf_{\zeta<0}\bigset{x\zeta+e^{-\zeta}}
=\inf_{t>0}\bigset{-xt+e^{t}}
=
\begin{cases}
  1, & x\le 1,\\
-x\log x +x, & x>1.
\end{cases}
\end{equation}
Since $\gb=1$, this agrees with \eqref{gaxxgb}.
\end{example}

\begin{example}\label{EHBST}
For the \bst{} in \refE{EBST}, originally treated by
\citet{Devroye1986}, 
see also \cite{Devroye1987}, \cite{Biggins97} and \cite{Devroye1998},
we have $\hmu(\gth)=2/(1+\gth)$, $\gth>-1$, see \eqref{hmuBST}. 
Consequently, \eqref{mul} is
\begin{equation}\label{pyrx}
  -\frac{\gth}{1+\gth} = \log\frac{2}{1+\gth}
\end{equation}
and \eqref{gamma3} is
\begin{equation}\label{pyrxx}
  \gam\qw = -\bigpar{\log\hmu}'(\gth)=\frac{1}{1+\gth},
\end{equation}
\ie, $\gam=1+\gth$.
Since $\ga=1$,  $\bgam=\gamma$.
We may substitute \eqref{pyrxx} in \eqref{pyrx}, noting that $\gth>0$
corresponds to $\gam>1$, and obtain
\begin{equation}
  \gam\qw-1=\log\bigpar{2\gam\qw}=-\log(\gam/2)
\end{equation}
or
\begin{equation}
  \label{gammabst}
\gam\log(\gam/2)+1-\gam=0,
\end{equation}
which has the root $\gam\doteq4.311070$.
(The theory above implies that \eqref{gammabst} has a unique root
$\gamma>1$.)

Furthermore, \eqref{gal} and \eqref{gaxx} yield after short calculations,
\begin{align}\label{gahbst}
  \ga(\zeta)&=\inf\Bigcpar{\gth>0:\frac{2}{1+\gth}\le e^{\zeta}} 
=2e^{-\zeta}-1,
\\
  \gaxx(x)&=\inf_{\zeta<0}\bigset{x\zeta+2e^{-\zeta}-1}
=\inf_{t>0}\bigset{-xt+2e^{t}-1}
\notag
\\&
=
\begin{cases}
  1, & x\le 2,\\
-x\log (x/2) +x-1, & x>2.
\end{cases}	
\label{gaxxbst}
\end{align}
Since $\gb=1/2$ for the \bst, this agrees with \eqref{gaxxgb}.
\end{example}

\begin{example}\label{EHlinear}
More generally, consider a linear preferential attachment tree as in 
\refE{Elinear+-}, with weights $w_k=\chi k+\rho$ as in \eqref{wlinear+-}.
(This was originally treated by \citet{Pittel1994}, at least in the case
$\chi\ge0$; see also \citet{BigginsGrey1996}.) 
Since only the quotient $\chi/\rho$ matters, and $\rho>0$,  we may and shall
for simplicity assume $\rho=1$. 
(It follows that \eqref{pyrrhus} below holds also in the general case
provided we replace $\chi$ by $\chi/\rho$.)

We then have $\hmu(\gth)=1/(\gth-\chi)$ by \eqref{hmulin+-}, 
and thus \eqref{mul} is, with $\gth>\chi$,
\begin{equation}\label{pyry}
-\frac{\gth}{\gth-\chi} = -\log(\gth-\chi)
\end{equation}
and \eqref{gamma3} is 
\begin{equation}\label{pyryy}
  \gam\qw 
=\frac{1}{\gth-\chi},
\end{equation}
\ie, $\gam=\gth-\chi$.
We may substitute this in \eqref{pyry} and obtain
\begin{equation}
\log \gam=\frac{\gth}{\gth-\chi}=\frac{\gam+\chi}{\gam}
\end{equation}
or
\begin{equation}\label{pyrg}
\gam\log\gam-\gam=\chi.
\end{equation}
The limit in \eqref{theight} is $\bgam:=\gam/\ga$; we have $\ga=\chi+1$ by
\eqref{alphalin}, and thus \eqref{pyrg} yields
\begin{equation}\label{pyrrhus}
\bgam\log\bigpar{(\chi+1)\bgam}-\bgam=\frac{\chi}{\chi+1}.
\end{equation}

\refE{EHRRT} is the case $\chi=0$. For the binary search tree in
\refE{EHBST} we have $\chi=-1$ and $\rho=2$, which is equivalent to
$\chi=-\frac12$ and $\rho=1$; thus we take $\chi=-\frac12$ above. Indeed,
\eqref{pyrrhus} with $\chi=-\frac12$ is the same as \eqref{gammabst}.
More generally, the \mary{} increasing tree in \refE{Elinear-} is obtained
by $\chi=-1/m$. 
For this \mary{} tree, \eqref{pyrrhus} yields,
\cf{} the binary case  \eqref{gammabst},
\begin{equation}\label{pyrrhus-mary}
\bgam\log\Bigpar{\frac{m-1}{m}\bgam}-\bgam+\frac{1}{m-1}=0.
\end{equation}
The height of the \mary{} increasing tree was found by 
\citet{BroutinDevroyeEtal2008}, 
see also \citet[Theorem 6.47]{Drmota}.
\end{example}

\begin{example}\label{EHmst}
  A somewhat more complicated example is the \mst{} in Section \ref{SSemst} or
  \ref{SSmst}. (For this example, it does not matter whether we include
  external nodes or not, since this only changes the height by
  1. Furthermore, $\mu$ is the same for both versions, so the calculations
  are the same.)
This was originally treated by \citet{Devroye1990}, see also
\citet{Mahmoud:Evolution}, 
\citet{Pittel1994}, 
\citet{Biggins97} and \citet{Devroye1998}. 

Recall that $\ga=1$, so $\bgam=\gamma$.
By \eqref{hmumst},
\begin{equation}
  \log\hmu(\gth) = \log m! - \sum_{i=1}^{m-1}\log(i+\gth).
\end{equation}
Hence \eqref{mul} is
\begin{equation}\label{mulmst}
  \gth\sum_{i=1}^{m-1}\frac{1}{\gth+i}
-\sum_{i=1}^{m-1}\log\xpar{\gth+i} +\log m! =0
\end{equation}
and \eqref{gamma3} is
\begin{equation}\label{gamma3mst}
  \gam\qw = \sum_{i=1}^{m-1}\frac{1}{\gth+i},
\end{equation}
which yields $\gam$ after (numerically) finding the unique positive root of
\eqref{mulmst}. 
\end{example}

\begin{example}\label{EHmstl}
  We can extend \refE{EHmst} for \mst{s}
to the generalization in \refE{Ehennequin}.
(The case $m=2$,
the median-of-$(2\ell+1)$ binary search tree, was treated by
\citet{Devroye93}; the general case was studied, by different methods,  
by \citet{ChauvinDrmota2006}, see also \citet[Section 6.5.2]{Drmota}.)
We still have $\ga=1$, so $\bgam=\gamma$.
It follows from \eqref{hmumstl} that
\eqref{mul} is
\begin{equation}\label{mulmstl}
 \sum_{i=1}^{(m-1)(\ell+1)}\frac{\gth}{\gth+i+\ell}
-\sum_{i=1}^{(m-1)(\ell+1)}\log\xpar{\gth+i+\ell} 
+\log\frac{(m\ell+m)!}{(\ell+1)!} =0 
\end{equation}
and \eqref{gamma3} is
\begin{equation}\label{gamma3mstl}
  \gam\qw = \sum_{i=1}^{(m-1)(\ell+1)}\frac{1}{\gth+i+\ell}.
\end{equation}
\end{example}

\begin{example}\label{EHpyramids}
Another example is the binary pyramids in \refE{E11}, where the height was
found by \citet{Mahmoud:pyramids} by similar methods,
see also \citet{BigginsGrey1996} and \citet{BroutinDevroyeMcLeish2008}.
The equations, now using \eqref{hmupyr}, become a bit involved;
we refer to \cite{Mahmoud:pyramids},  \cite{BigginsGrey1996}
and \cite{BroutinDevroyeMcLeish2008}
for a numerical solution.  
\end{example}

\begin{example}\label{EHfrag}
For the fragmentation tree in \refE{Ebinaryfrag}, we have a branching
process that differs from the one for the \bst{} in \refE{EHBST}, 
but the intensity $\mu$ is
the same, so all calculations in \refE{EHBST} are valid for this tree too.
Thus, see \refT{TBiggins1} and \refR{RHfrag},
$H(\ctt)/t\asto\gam$ with $\gam$ given by \eqref{gammabst}.
Furthermore, if we stop at $n$ nodes as in \refR{Rfragtau},
 $H(T_n)/\log n\asto\gam$, 
just as for the \bst.
More precise results for the height of this fragmentation tree, and \mary{}
generalizations of it, are given  by \citet{ChauvinDrmota2006}. 
\end{example}

\subsection{Moment convergence}

We can also obtain moment convergence in \refT{Theight}, in particular
convergence of the expectation $\E H(T_n)/\log n$ to $\bgam$,
at least if we assume the following additional condition on the birth times
for an individual in the \CMJp.
\begin{xenumerate}
\renewcommand{\theenumi}{{\upshape{(A*)}}}
\renewcommand{\labelenumi}{\theenumi}
\item \label{BP+}
There exists $\gd>0$  such that 
$\E e^{\gd \xi_2}<\infty$.
\end{xenumerate}
In other words, each individual gets at least two children ($N\ge2$),
and the age when the second child is born has an exponential
moment. (Equivalently, it has exponentially decreasing tails.)

The condition \ref{BP+} is satisfied
in all examples in Sections \ref{Sex}--\ref{Smedian},
since $\xi_2$ is the sum of one or several exponential waiting times. 

\begin{remark}
  We use \ref{BP+} in the proof of \refL{LM2} below. Some extra condition is
  clearly needed for \refL{LM2} (at least $\E e^{\gd\xi_1}<\infty$ for some
  $\gd>0$, since $\tau(n)\ge\xi_1$ if $n>1$ and, say, $\psi(t)=1$.)
However, we do not know whether \ref{BP+} really is needed for Lemmas
\ref{LM3-} and \ref{LM3} and for \refT{Ton}.
In fact, we conjecture that \refT{Ton} holds assuming only
\refBP{} and \refBPlastpsi.
\end{remark}

We begin with some lemmas. The first two are stated somewhat more generally
than actually needed here.
\begin{lemma}
  \label{LM1}
For every $r>0$, there exists $c_r$ such that, for large $t$,
\begin{equation}\label{lm1}
\P\bigpar{H(\ctt)\ge c_r t} \le e^{-rt}. 
\end{equation}
\end{lemma}

\begin{proof}
  Let $N_t\xoo$ be the number of nodes in $\ctt$ that have depth $\ge x$,
  and let $\eta_t\xoo:=\E \bigpar{N_t\xoo}$.
Then, by \cite[Theorem 4 and its proof]{Biggins95}, for any real $x$,
\begin{equation}\label{lme}
  \limsup_\ttoo\frac{\log\eta_t\xxoo{tx}}{t}\le\gaxx(x).
\end{equation}
In fact, \cite[Theorem 4]{Biggins95} shows that the limit exists and equals
$\gaxx(x)$ except possibly for one exceptional $x$.
The upper bound \eqref{lme} (for every $x$)
is the simple part of the proof and follows 
from \cite[Corollary 1]{Biggins95}.

By \eqref{gaxxoo}, we can choose $c_r$ such that $\gaxx(c_r)<-r$, and then
\eqref{lme} yields 
$\log\eta_t\xxoo{c_rt}\le -rt$ for large $t$.
This yields the result \eqref{lm1}, since
\begin{equation}
  \P\bigpar{H(\ctt)\ge c_rt} 
=\P\bigpar{N_t[c_rt,\infty)\ge1}
\le \eta_t[c_rt,\infty)
\le e^{-rt}.
\end{equation}
\end{proof}

\begin{lemma}\label{LM2}
  Assume \refBP{} and \ref{BP+}.
Then, for every $r\ge0$ there exists $c'_r$ such that
$\P\bigpar{\tau(n)\ge c'_r\log n}=O\bigpar{n^{-r}}$.
\end{lemma}

\begin{proof}
  First, if \refBPlastpsi{} holds, then $\tau(n)/\log n \asto \ga\qw$ by
  \refT{Tex}. In general, we may as in the proof of \refT{Tex} consider the
  truncation $\psi_1(t):=\psi(t)\land 1$ and the corresponding stopping time
  $\tau_1(n)$. Then $\tau_1(n)/\log n\asto\ga\qw$ and $\tau(n)\le\tau_1(n)$.
Hence, in any case, if $b:=2\ga\qw$, then 
$\P(\tau(n)<b\log n)\to1$; in particular, 
if $n$ is large enough (which we assume in the rest of the proof;
the result is trivial for small $n$),
\begin{equation}
  \label{landings}
\P(\tau(n)\ge b\log n)<e\qw.
\end{equation}

Let $L:=\ceil{r\log n}$. Let $v_0=o$, the root, and let $v_i$ be the second
child of $v_{i-1}$, $i\ge1$.
(Thus, denoting the individuals by strings in 
$\voo:=\bigcup_{n=0}^\infty \bbN^n$, see \refS{S:BP}, the nodes $v_i$
are $\emptyset,2,22,222,\dots$.)
Let $\taul:=\gs_{v_L}$, the time $v_L$ is born. 
Let $A:=\E e^{\gd\xi_2}$, which is finite by \ref{BP+},
and let $B:=r(\log A+1)/\gd$.
Then,
\begin{equation}\label{ofvandahl}
  \begin{split}
\P(\taul\ge B\log n)&
\le e^{-\gd B\log n}\E e^{\gd\taul}	
=  e^{-\gd B\log n}\bigpar{\E e^{\gd\xi_2}}^L
=  e^{-\gd B\log n}A^L
\\&
\le e^{-\gd B\log n + (1+r\log n)\log A}
= A n^{-\gd B+r\log A}
=A n^{-r}.
  \end{split}
\end{equation}

Let $w_i$ be the  first child of $v_{i-1}$, $i\ge1$.
Then $w_i$ is born before (or at the same time as) $v_i$, so $w_1,\dots,w_L$
are all born before (or at) $\taul$. Each $w_i$ 
starts an independent copy $\ctt\iii$ of the branching process
and its family tree (we do not distinguish between these two), with
local time $t=0$ when $w_i$ is born, \ie, at global time $\gs_{w_i}$. 
Let $Z_t^{\psi,(i)}$ be the total weight of $\ctt\iii$
and 
let $\tau\iii(n)$ be the corresponding stopping time
$\tau\iii(n):=\inf\set{t:Z_t^{\psi,(i)}\ge n}$.

For every $t\ge0$, the individuals (nodes)  in $\ctt\iii$ form a subset of
$\cT_{\gs_{w_i}+t}$, and they have the same ages in both trees; hence
the definition \eqref{zgf} implies that 
$\zpsi_{\gs_{w_i}+t}\ge Z_t^{\psi,(i)}$.
Consequently, $\zpsi_{\gs_{w_i}+\tau\iii(n)}\ge n$ and thus,
for every $i\le L$,
\begin{equation}
 \tau(n)\le \gs_{w_i}+\tau\iii(n)\le\taul+\tau\iii(n). 
\end{equation}
In other words,
\begin{equation}\label{lyman}
 \tau(n)\le \taul+\min_{i\le L}\tau\iii(n). 
\end{equation}

Moreover, the stopping times $\tau\iii(n)$ are independent, and have the
same distribution as $\tau(n)$.
Hence, \eqref{landings} implies that
\begin{equation}\label{serrander}
  \P\bigpar{\min_{i\le L}\tau\iii(n)\ge b\log n}
=
  \P\bigpar{\tau(n)\ge b\log n}^L
<e^{-L}
\le n^{-r}. 
\end{equation}
Combining \eqref{lyman}, \eqref{ofvandahl} and \eqref{serrander},
we obtain
\begin{equation}
  \P\bigpar{\tau(n)\ge (b+B)\log n} \le (A+1)n^{-r}.
\end{equation}
\end{proof}

The next lemma will immediately be improved in \refL{LM3}. 
Lemma \ref{LM3} is
trivially true for $n=1$ too; however, we assume $n\ge2$ since as said in
\refS{S:BP}, in principle we do not require $n$ to be an integer; any real
positive $n$ is possible. (We use this in the proof of \refL{LM3} below, for
convenience, when we do not round $m$ to an integer.)

\begin{lemma}\label{LM3-}
  Assume \refBP{} and \ref{BP+}.
Then there exists $c$ such that, for all $n$ 
\begin{equation}\label{lm4}
\P\bigpar{H(T_n)\ge c\log n}=O\bigpar{n\qw}.  
\end{equation} 
\end{lemma}

\begin{proof}
  Let $c_1$ and $c'_1$ be as in Lemmas \ref{LM1} and \ref{LM2} with $r=1$;
  we may assume $c'_1\ge1$.
Then by Lemmas \ref{LM2} and \ref{LM1}, recalling that $T_n:=\cT_{\tau(n)}$,
for large $n$,
  \begin{equation}
	\begin{split}
\P\bigpar{H(T_n) \ge c_1c'_1\log n}
&\le 
\P\bigpar{\tau(n)>c'_1\log n} + \P\bigpar{H(\cT_{c'_1\log n})\ge c_1c'_1\log n}
\\&
\le O\bigpar{n\qw}+e^{-c'_1\log n} 
=  O\bigpar{n\qw}.
	\end{split}
\raisetag{0.9\baselineskip}
  \end{equation}
This completes the proof, since \eqref{lm4} is trivial for small $n$.
\end{proof}

\begin{lemma}\label{LM3}
  Assume \refBP{} and \ref{BP+}.
Then there exists $C$ such that, for all $n\ge2$ and $t\ge1$,
\begin{equation}\label{lm3}
\P\bigpar{H(T_n)\ge Ct\log n}\le n^{-t}.  
\end{equation}
\end{lemma}
\begin{proof}
By \refL{LM3-}, there exists $c$ and $A$ such that,
for all $n$,
\begin{equation}\label{th}
\P\bigpar{H(T_n)\ge c\log n}\le A/n.   
\end{equation}
We may assume $A\ge1$.

Given $n\ge2$ and $t\ge1$, let $m:=An^t$. Then $m\ge n$ and thus 
$\tau(m)\ge \tau(n)$ and $H(T_m)\ge H(T_n)$. Consequently,
\eqref{th} applied to $m$ yields
\begin{equation}\label{bo}
\P\bigpar{H(T_n)\ge c\log m}
\le \P\bigpar{H(T_m)\ge c\log m}  \le A/m =n^{-t}.
\end{equation}
Moreover, $\log m=\log A+t\log n\le (1+\log A/\log2)t\log n$.
Hence, with $C:=c(1+\log A/\log 2)$, \eqref{bo} implies
\begin{equation}
\P\bigpar{H(T_n)\ge Ct\log n}
\le
\P\bigpar{H(T_n)\ge c\log m}
\le n^{-t}.
\end{equation}
\end{proof}

\begin{theorem}\label{Ton}
  Assume \refBP, \refBPlastpsi{} and \ref{BP+}.
Then the convergence in \eqref{theight} holds also in $L^r$ for every
$r>0$, \ie,
\begin{equation}\label{ton1}
\E |H(T_n)/\log n-\bgam|^r\to0.  
\end{equation}
In particular, 
\begin{equation}\label{ton2}
  \frac{\E \bigpar{H(T_n)^r}}{\log^r n}\to\bgam^r,
\qquad r>0.
\end{equation}
\end{theorem}

\begin{proof}
  Let $X_n:=H(T_n)/\log n$. By \refL{LM3}, for $n\ge2$ and $t\ge1$,
$\P(X_n\ge Ct)\le n^{-t}\le 2^{-t}\le 2^{1-t}$, which obviously holds also
  for $t<1$.
Hence, for $r>0$,
\begin{equation}
  \E X_n^r = r\intoo t^{r-1} \P(X_n \ge t)\dd t
\le r\intoo t^{r-1} 2^{1-t/C}\dd t <\infty.
\end{equation}
This shows that each moment $\E X_n^r$ is uniformly bounded for $n\ge2$.
As is well-known, this implies uniform integrability of $X_n^r$ for each
$r$, and thus also of $|X_n-\bgam|^r$; since $X_n\asto\bgam$ by
\refT{Theight}, this implies \eqref{ton1} and \eqref{ton2}.
(See \eg{} \cite[Theorems 5.4.2 and 5.5.2]{Gut} for uniform integrability.)
\end{proof}

In particular, \refT{Ton}
shows that $\E H(T_n)\sim \bgam \log n$, and similarly for
higher moments, in Examples \ref{EHRRT}--\ref{EHpyramids}. 
We obtain also corresponding results for the fragmentation tree
in \refE{EHfrag}, using \refL{LM1}.

\begin{remark}
It follows from \eqref{ton1}, with $r=2$, that the variance
$\E|H(T_n)-\E H(T_n)|^2=o(\log^2 n)$.
In the case of a \bst, \citet{Reed} showed the much sharper result that
$\E|H(T_n)-\E H(T_n)|^2=O(1)$; this was extended to higher central moments
and to \mst{s} by \citet{DrmotaII2003}.
\end{remark}

\subsection{Saturation level}\label{SSsat}
In this subsection we will often assume that the random tree is \mary; 
more precisely,  that $N=m$ for some (non-random) integer $m$, 
\ie, every individual in the branching process gets $m$ children.
(There is no risk of confusion between the integer parameter $m$ and the
function $m$ in \eqref{m}; they never appear together.)
We call this the \emph{\mary{} case} in the present section.
(We have previously defined an \mary{} tree to be a tree where the children
of each nodes have labels in the set \set{1,\dots,m}. In the present
section, such labels are irrelevant, as are the order of the children, so we
can use this simpler definition.)

In the \mary{} case,
the infinite complete family tree $\cT_\infty$ is thus a complete \mary{}
tree; however, we are interested in the trees $\ctt$ for finite $t$, and in
particular in $T_n$, and there the outdegrees may be smaller than $m$ (but
never larger); note that any given node will get $m$ children eventually
(\ie, for large $t$ or $n$).
As said above, the saturation level $S(T_n)$ is defined as the last level
(generation) $k$ where all possible $m^k$ nodes exist; equivalently, it is the
first generation where some node has less than $m$ children.

We study the saturation level in basically
the same way as the height in the preceding
subsection, but now using a feature of 
\citet{Biggins95,Biggins97} that was not needed above: 
Let $\chi$ be a \oichar, \ie{} a
characteristic that takes the values 0 and 1 only
(excluding the trivial case when \as{}
$\chi(t)=0$ for all $t\ge0$),
and now let $B_t$ be the maximum of the position
$y_x$ of all individuals $x$ born at time $t$
such that the characteristic $\chi_x(t-\gs_x)=1$.
(I.e., only individuals with $\chi=1$ count.)
Then \refT{TBiggins1} still holds, for general $\eta_i$,
provided the following two conditions are
satisfied \cite{Biggins95,Biggins97}:

\begin{xenumerate}
\renewcommand{\theenumi}{{\upshape{(B\arabic{enumi})}}}
\renewcommand{\labelenumi}{\theenumi}
\item \label{BIGG1}
$\ga(\zeta)<\infty$ for some $\zeta<0$.
\item \label{BIGG2}
For every $\zeta<0$ such that $\ga(\zeta)<\infty$,
\begin{equation}\label{bigg2}
  \E \sup_{t>0}\bigpar{e^{-\ga(\zeta) t}\chi(t)}
<\infty.
\end{equation}
(The property \ref{BIGG2} is called \emph{well-regulated} in \cite{Biggins97}.)
\end{xenumerate}

\begin{remark}\label{RBIGG}
The case considered in \refSS{SSheight} above is the special
case when $\chi(t)=1$ for all $t\ge0$, and further $\eta_i=1$. 
We noted \ref{BIGG1} in
\refL{LX}\ref{lxga}, 
and \ref{BIGG2} is trivial 
(for this choice of $\eta_i$)
since $\ga(\zeta)\ge\ga>0$ when
$\zeta<0$, as also noted in \refL{LX}\ref{lxga}.
\end{remark}

\begin{remark}\label{RBiggins-Nerman}
  The characteristics treated in \citet{Biggins95} are independent of the
  life histories of all other individuals. However, the proofs
hold also for the more general characteristics in \refR{Rnerman7} that may
depend on the entire process of descendants of the individual. 
(The proofs in \cite{Biggins95}
use \cite[Theorem 5.4]{Nerman}, which is shown in \cite[Section 7]{Nerman}
to extend to such characteristics.)
\end{remark}

\begin{example}\label{Echim}
In the \mary{} case, 
denote, as usual, the birth times of the children of a typical individual
by
$\xi_1\le\dots\le\xi_m$, and define the characteristic 
\begin{equation}
  \label{chim}
\chi(t):=\ett{t<\xi_m},
\end{equation}
the indicator that not all $m$ children are born. 
Furthermore, let $\eta_i=-1$ for every $i$.
Then $-B_t$ is the minimum generation number of an individual that does not have
$m$ children at time $t$, \ie, $-B_t = S(\ctt)$, the saturation level of
$\ctt$.    
\end{example}

To study the saturation level,
we thus intend to use \refE{Echim} and apply \refT{TBiggins1} with $\eta_i=-1$.
For later use, 
we continue to use the notations in \eqref{m}--\eqref{gamma0} and
\eqref{gal}--\eqref{gaxx} for the case $\eta_i=1$, and let 
$\check m(\zeta,\gth)$, $\chgam$, $\chga(\zeta)$ and $\chgaxx(x)$ 
be the corresponding quantities for $\eta_i=-1$ (for general branching
processes satisfying our conditions; not just in the \mary{} case);
furthermore, let $\gamx:=-\chgam$.
Thus, \cf{} \eqref{mus},
\begin{align}\label{chm}
  \check m(\zeta,\gth)
=e^{\zeta}\hmu(\gth)
=m(-\zeta,\gth)
\end{align}
and consequently, \cf{} \eqref{gal}--\eqref{gamma1} and \eqref{gamma0}, 
\begin{align}
\label{chga}
  \chga(\zeta)&=\inf\bigset{\gth:\hmu(\gth)\le e^{-\zeta}}=\ga(-\zeta),
\\
\label{chgaxx}
\chgaxx(x)&=\inf_{\zeta<0}\bigset{x\zeta+\ga(-\zeta)}
=\inf_{\zeta>0}\bigset{-x\zeta+\ga(\zeta)},
\\
\gamx \label{chgamma1}
&= 
-\chgam
=-\inf\Bigset{a:\inf_{\zeta<0}\log \chm(\zeta,-a\zeta)<0}
\\&\label{chgamma2}
=-\inf\set{x:\chgaxx(x)<0}
=\sup\set{x:\chgaxx(-x)<0}.
\end{align}

For $x<0$ we also have, using \eqref{gal},
in analogy with \eqref{gaxx2},
\begin{equation}\label{chgaxx2}
  \begin{split}
\chgaxx(x)&
=\inf_{\substack{\zeta>0\\ \log\hmu(\gth)\le\zeta}}\bigset{-x\zeta+\gth}	
=\inf_{\gth}\bigset{-x\log_+\hmu(\gth)+\gth}
\\&
=\inf_{\gth\le\ga}\bigset{-x\log\hmu(\gth)+\gth}.	
  \end{split}
\end{equation}

We give some properties of these functions in \refL{LXch} below,
but we first state the main result for the saturation level in the \mary{}
case.
\begin{theorem}
  \label{Tfillup}
In the \mary{} case $N=m$,
under the assumptions \refBP{} and \refBPlastpsi, 
  \begin{equation}\label{tfillup}
	\frac{S(T_n)}{\log n} \asto \bgamx:=\frac{\gamx}{\ga} .
  \end{equation}
\end{theorem}

\begin{proof}
By \eqref{chga} and \refL{LX}\ref{lxga},
$\check\ga(\zeta)=\ga(-\zeta)<\infty$ for every
$\zeta<0$, so \ref{BIGG1} holds for $\chga$.

Furthermore, \eqref{bigg2} (with $\chga(\zeta)$)
is trivial if $\chga(\zeta)\ge0$, so assume
$\chga(\zeta)<0$. Then
\eqref{chim} and \eqref{hmub} yield
\begin{equation}\label{mixmax}
\E  \sup_{t>0}\bigpar{e^{-\chga(\zeta) t}\chi(t)}
= \E e^{-\chga(\zeta)\xi_m}
\le \hmu(\chga(\zeta)).
\end{equation}

Recall $\AAA:=\inf\set{\gth:\hmu(\gth)<\infty}$ defined in \eqref{AAA}.
There are three cases:
\begin{romenumerate}[-5pt]
\item 
$\AAA=-\infty$. Then $\hmu(\gth)<\infty$ for every real $\gth$.
Furthermore, $\chga(\zeta)=\ga(-\zeta)>-\infty$, 
see \refL{LX}\ref{lxga}, and thus
  $\hmu(\chga(\zeta))<\infty$. 
\item 
$\AAA>-\infty$ and $\hmu(\AAA)<\infty$.
Then $\chga(\zeta)\ge \AAA$ for every $\zeta$, and thus 
$\hmu(\chga(\zeta))\le\hmu(\AAA)<\infty$. 
\item 
$\AAA>-\infty$ and $\hmu(\AAA)=\infty$.
Then, by monotone convergence, $\hmu(\zeta)\to\infty$ as $\zeta\downto \AAA$
and thus $\chga(\zeta)> \AAA$ for every $\zeta$, see \eqref{chga}; hence
$\hmu(\chga(\zeta))<\infty$. 
\end{romenumerate}

Consequently, in all three cases, 
$\hmu(\chga(\zeta))<\infty$ for all real $\zeta$ such that
$\chga(\zeta)<0$ (and, in fact, for all real $\zeta$), 
and thus \eqref{mixmax} shows that \ref{BIGG2} holds
for $\chga$.

Consequently, \refT{TBiggins1} applies to our $\eta_i=-1$ and $\chi$ in
\eqref{chim}, which 
yields
\begin{equation}\label{bt-}
  \frac{-S(\ctt)}{t}=\frac{B_t}{t}
\asto \chgam = -\gamx
\end{equation}
and the result follows by \eqref{tex2tau}
as in the proof of \refT{Theight}.
\end{proof}

As a corollary, we obtain moment convergence and convergence in $L^r$ as in
\refT{Ton}, assuming also \ref{BP+}. 

\begin{theorem}\label{TonS}
  Assume \refBP, \refBPlastpsi{} and \ref{BP+}
and $N=m$.
Then the convergence in \eqref{tfillup} holds also in $L^r$ for every
$r>0$. 
In particular, 
\begin{equation}\label{tonS2}
  \frac{\E \bigpar{S(T_n)^r}}{\log^r n}\to\bgamx^r,
\qquad r>0.
\end{equation}
\end{theorem}
\begin{proof}
  Since $S(T_n)\le H(T_n)+1$, \refL{LM3} holds also for $S(T_n)$; hence the
result follows from \refT{Tfillup} by the argument in the proof of \refT{Ton}.
(In fact, in all cases of \mary{} trees with $|T_n|=n$
and in many other cases, \eg{} for \mst{s}, 
$S(T_n)\le C\log n$ deterministically for some $C$; then the results follow
from \refT{Tfillup} 
by dominated convergence
without using \ref{BP+} and \refL{LM3}.)
\end{proof}

\begin{remark}\label{RBIGG2}
  We see that \ref{BIGG2} (or some similar condition) is needed for
\refT{TBiggins1}; if we let
  $\eta_i=-1$ as above but take $\chi(t)=1$ for all $t\ge0$ (as in
  \refS{SSheight}), then obviously $B_t=0$, and $B_t/t$ does not converge to
  $\chgam$ (in general), so \refT{TBiggins1} does not hold.
\end{remark}

We give some formulas for $\gamx$, similar to \refL{Lgamma}.

\begin{lemma}
  \label{Lgamx}
We have
\begin{equation}\label{gamx2a}
  \begin{split}
\gamx\qw&
=\inf_{\gth<0}\frac{\log \hmu(\gth)}{-\gth}
=\inf_{\AAA<\gth<0}\frac{\log \hmu(\gth)}{|\gth|}.
  \end{split}
\end{equation}
Furthermore, 
$0\le\gamx\le\gb\qw$, and 
with $\AAA$  given by \eqref{AAA}:
\begin{romenumerate}
\item \label{Lgamx0}
If $\AAA\ge0$, then $\gamx=0$.

\item \label{Lgamx>0}
If $\AAA<0$, then $0<\gamx\le\gb\qw$.
Moreover,
if\/ $\gth\in(\AAA,0)$ is a solution of 
the equation \eqref{mul},
then
\begin{equation}\label{gamx3}
  \gamx\qw
  =\frac{\log\hmu(\gth)}{-\gth} 
=-\frac{\hmu'(\gth)}{\hmu(\gth)}
=-\bigpar{\log\hmu}'(\gth).
\end{equation}
Furthermore, \eqref{mul}
has at most one negative solution.

\item \label{Lgamx3}
If $\AAA<0$ but \eqref{mul} has no negative solution, then
\begin{equation}\label{gamx3b}
  \gamx\qw
=\lim_{\gth\to \AAA}\frac{\log\hmu(\gth)}{|\gth|}. 
\end{equation}

\end{romenumerate}
\end{lemma}

\begin{proof}
By \eqref{chgamma1} and \eqref{chm},
 $ \chgam=\inf\set{a:a\in\cE}$,
where
\begin{equation}\label{cE}
  \cE:=\Bigset{a\in\bbR:\inf_{\zeta<0}\bigpar{\zeta+\log \hmu(-a\zeta)}<0}.
\end{equation}
If $a>0$, then as $\zeta\to-\infty$,
$\log\hmu(-a\zeta)\to\log\hmu(\infty)<0$, see \eqref{lxi+}, and thus
$\zeta+\log\hmu(-a\zeta)\to-\infty$; hence $a\in\cE$. Consequently
$\chgam\le0$,
so $\gamx=-\chgam\ge0$. 

\pfitemref{Lgamx0}
If $\AAA\ge0$, then $\hmu(\gth)=\infty$ for every $\gth<0$, and it follows from
the definition \eqref{cE} that
$a\notin\cE$ for every $a<0$; hence $\gamx=\chgam=0$.
In this case, \eqref{gamx2a} is trivial.

\pfitemref{Lgamx>0}
If $\AAA<0$, then $\hmu(\gth)<\infty$ for some $\gth<0$;
furthermore, $\hmu(\gth)\ge\hmu(0)=\E N>1$ by \ref{BPsuper}. Taking first
$\zeta<-\log\hmu(\gth)<0$ and then $a=-\gth/\zeta<0$, we see that
$a\in\cE$ and thus $\chgam\le a<0$; hence $\gamx>0$.
In this case we  have
\begin{equation}
    \chgam
=\inf\set{a<0:a\in\cE}
=\inf\bigset{a<0:\inf_{\zeta<0}\bigpar{\zeta+\log \hmu(-a\zeta)}<0}
\end{equation}
and thus,
\begin{equation}
  \begin{split}
\gamx
=-\chgam
&
=\sup\bigset{x>0:\inf_{\zeta<0}\bigpar{\zeta+\log \hmu(x\zeta)}<0}
\\&
=\sup\bigset{x>0:\inf_{\gth>0}\bigcpar{-\gth/x+\log \hmu(-\gth)}<0}
\\&
=\sup\Bigset{x>0:\inf_{\gth>0}
 \Bigcpar{-\frac{1}{x}+\frac{\log \hmu(-\gth)}{\gth}}<0}
\\&
=\sup\Bigset{x>0:\frac{1}{x}>\inf_{\gth>0}\frac{\log \hmu(-\gth)}{\gth}}
.
  \end{split}
\end{equation}
Hence,
\begin{equation}\label{gamx2b}
  \begin{split}
\gamx\qw&
=\inf_{\gth>0}\frac{\log \hmu(-\gth)}{\gth}
  \end{split}
\end{equation}
which yields \eqref{gamx2a}.

Extreme points of $\log\hmu(\gth)/\gth$ are given by \eqref{mul}, and the
argument in the proof of \refL{Lgamma}, applied to
$f(\gth)=\log\hmu(-\gth)$, shows that there is at most one extreme point in
$(\AAA,0)$ and that any extreme point is a minimum; 
this yields \eqref{gamx3}.
Similarly, if there is no negative solution of \eqref{mul}, then
$\log\hmu(\gth)/|\gth|$ is increasing on $(\AAA,0)$ and 
\eqref{gamx3b}  follows.

Finally, 
if $\gth< 0$, then, \eqref{bertil} implies
$\log\hmu(\gth)/|\gth|>\gb$
and thus \eqref{gamx2a} yields $\gamx\qw\ge\gb$, so $\gamx\le\gb\qw$.
\end{proof}

\begin{remark}
Geometrically, \eqref{gamx2a} says that $-\gamx\qw$ is the slope of the
tangent from the origin to the curve $\log\hmu(\gth)$, $\gth<0$, provided
such a tangent exists. 
Consequently, 
$-\gam\qw$ and $-\gamx\qw$ are the slopes of the two tangents from the
origin to this curve, provided there are two such tangents. (I.e., when
\eqref{mul} has one negative and one positive solution, which typically
is the case.)
Note that, analytically, $\gam$ and $\gam_-$ are given by the same formula
\eqref{gamma3} or \eqref{gamx3}, with $\gth$ the positive or negative root
of \eqref{mul}.  
\end{remark}

\begin{remark}
  \label{Rgamx}
\refL{Lgamx} shows that $\gamx>0\iff \AAA<0$, \ie, 
$\gamx>0$ if and only if
$\hmu(\gth)<\infty$ for some $\gth<0$.
In the \mary{} case $N=m$, this is by \eqref{hmub} equivalent to $\E
e^{s\xi_i}<\infty$ for every $i$ and some $s>0$, \ie, that the child birth
ages $\xi_i$, $i=1,\dots,m$, 
have some exponential moment. 
(With the usual ordering
$\xi_1\le\dots\le\xi_m$; this is equivalent to $\xi_m$ having an exponential moment.)
\end{remark}

We give some properties of $\chga(\zeta)$ and $\chgaxx(x)$, \cf{} \refL{LX}. 

\begin{lemma}\label{LXch}
  \begin{thmenumerate}
\item \label{lxchga}
$\chga(\zeta)=\ga(-\zeta)$ is a convex and increasing function on
  $(-\infty,\infty)$ with
$-\infty<\chga(\zeta)\le\infty$.
Furthermore, $\chga(0)=\ga$,
$\chga(\zeta)<\ga$ for $\zeta<0$ 
and $\chga(\zeta)>\ga$ for $\zeta>0$. 

\item \label{lxchgaxx}
$\chgaxx(x)$ is a concave and decreasing function on
  $(-\infty,\infty)$ with
$-\infty\le \chgaxx(x)<\infty$.
Moreover,
\begin{equation}\label{chgaxxgb}
  \begin{cases}
	\chgaxx(x) = \ga, & x\le -1/\gb,
\\
	\chgaxx(x) < \ga, & x> -1/\gb,
\\ 
\chgaxx(x)=\AAA, & x=0,
\\ 
\chgaxx(x)=-\infty, & x>0,
  \end{cases}
\end{equation}
and furthermore
\begin{equation}\label{chbaram123}
  \begin{cases}
\chgaxx(x)>-\infty,&x<\chbaram,\\
\chgaxx(x)=\baram\log\mu\set{\baram\qw}\ge-\infty,&x=\chbaram\neq0,\\
\chgaxx(x)=\AAA\ge-\infty,&x=\chbaram=0,\\
\chgaxx(x)=-\infty,&x>\chbaram,\\
  \end{cases}
\end{equation}
and thus
\begin{equation}\label{chbaram0}
\sup\set{x:\chgaxx(-x)=-\infty}=\baram.
\end{equation}
In particular, if the birth times $\xi_i$ are unbounded (as in all our
examples),  
so $\baram=0$, then
$\chgaxx(x)>-\infty$ for every $x<0$.

Consequently,
$\chgaxx(x):\bbR\to[-\infty,\infty)$ is continuous everywhere except possibly
  at $\chbaram$, and left-continuous everywhere.
  \end{thmenumerate}
\end{lemma}

\begin{proof}
\pfitemref{lxchga}
Follows directly from \eqref{chga} and \refL{LX}\ref{lxga}.

\pfitemref{lxchgaxx}
Since $\ga(\zeta)$ is convex, 
 \eqref{ga'0} implies
$\ga'(\zeta)\ge-1/\gb$ for $\zeta\ge0$, with equality for $\zeta=0$.
It follows that the second infimum in \eqref{chgaxx} is attained at $\zeta=0$ 
(extending the infimum to $\zeta\ge0$)
if
$x\le-1/\gb$, but not if $x>-1/\gb$, and
thus the two first cases in \eqref{chgaxxgb} follow, recalling $\ga(0)=\ga$.
The third case follows directly from \eqref{chgaxx} by \eqref{lxii+} and the
fact that $\ga(\zeta)$ is decreasing, 
and the fourth case follows by letting
$\zeta\to\infty$ in the second infimum in \eqref{chgaxx}.

Next, \eqref{chgaxx2} implies that for any $x<0$,
\begin{equation}\label{2016-}
  \begin{split}
\exp\bigpar{\chgaxx(x)/|x|}
=\inf_{\gth\le\ga} \bigset{\exp(\gth/|x|)\hmu(\gth)}
=\inf_{\gth\le\ga} \Bigset{\intoo e^{\gth(|x|\qw-t)}\mu(\ddx t)}.
  \end{split}
\end{equation}
We argue as in the proof of \eqref{bara123} in \refL{LX}.
If $x=-\baram<0$, then $t\le |x|\qw$ on the support of $\mu$ by \eqref{supp-};
hence the integral in  \eqref{2016-} is an increasing function of $\gth$,  
and dominated convergence 
as $\gth\to-\infty$
shows that the infimum
equals $\mu\set{\baram\qw}$.
Similarly, if $-\baram<x<0$, then the infimum in \eqref{2016-} is 0.
On the other hand, if $x<-\baram$, then
$\intoo e^{\gth(|x|\qw-t)}\mu(\ddx t)\ge
\mu\set{t:t\ge |x|\qw}>0$ for every $\gth\le0$, and thus the infimum in
\eqref{2016-} is positive.
This shows \eqref{chbaram123} for $x<0$; the case $x\ge0$ follows from
the last two parts of \eqref{chgaxxgb}.
Obviously, \eqref{chbaram0} follows from \eqref{chbaram123}.

Since $\chgaxx$ is concave, it follows from \eqref{chbaram123} that it is
continuous except possibly at $\chbaram$.
The left-continuity of $\chgaxx(x)$, or equivalently right-continuity of
$\chgaxx(-x)$, follows from \eqref{chgaxx} which yields, \cf{} \eqref{kastrup},
\begin{equation}
  \lim_{ y\downto x} \chgaxx(-y)
=
  \inf_{ y> x} \chgaxx(-y)
=
  \inf_{ y> x,\; \zeta>0}\set{y\zeta+\ga(\zeta)}
=\chgaxx(-x).
\end{equation}
\end{proof}

\begin{remark} 
By \eqref{chgamma2} and \eqref{chbaram123},
\begin{equation}\label{gamxbaram}
\gamx\ge\baram.
\end{equation}
\end{remark}

\begin{remark}
  \label{Rgaxxx}
Let, \cf{} \eqref{gaxx} and \eqref{chgaxx}, 
\begin{equation}\label{gaxxx}
  \gaxxx(x):=\min\bigpar{\gaxx(x),\chgaxx(-x)}
=\inf_{-\infty<\zeta<\infty}\bigset{x\zeta+\ga(\zeta)}.
\end{equation}
By \eqref{gaxxgb} and \eqref{chgaxxgb}, 
for every $x$,
$\max\bigpar{\gaxx(x),\chgaxx(-x)}=\ga$ and thus
one of 
$\gaxx(x)$ and $\chgaxx(-x)$ equals $\gaxxx(x)$ while the other equals
$\ga$; which one that equals $\gaxxx(x)$ depends on whether $x$ is smaller
or greater than $1/\gb$. (For $x=1/\gb$, $\gaxxx(x)=\gaxx(x)=\chgaxx(-x)=\ga$;
for all other $x$, $\gaxxx(x)<\ga$ and $\gaxx(x)\neq\chgaxx(-x)$.)

For $x>0$,  we also have by \eqref{gaxx2} and \eqref{chgaxx2},
\begin{equation}\label{gaxxx2}
  \begin{split}
\gaxxx(x)&
=\inf_{-\infty<\gth<\infty}\bigset{x\log\hmu(\gth)+\gth}.	
  \end{split}
\end{equation}
Furthermore, by \eqref{bara123} and \eqref{chbaram123}, 
$\gaxxx(x)$ is finite in $(\baram,\bara)$ but $-\infty$ outside
$[\baram,\bara]$. 

By the convexity of $\ga(\zeta)$ and $\log\hmu(\gth)$,
with strict convexity in $\Iga$ and $(\AAA,\infty)$
and continuity on $(-\infty,\infty)$ and $[\AAA,\infty)$,
respectively, see
Lemma \ref{LX},
it follows that
for $x\in(\baram,\bara)$, where $\gaxxx(x)$ thus is finite, 
the
infima in \eqref{gaxxx} and \eqref{gaxxx2} are attained at some 
(unique) finite $\zeta$
and $\gth$. 
Moreover, if $\AAA>-\infty$ 
(which implies $\baram=0$)
and both $\hmu(\AAA)$ and the right derivative
$\hmu'(\AAA+)$ are finite,
let 
\begin{equation}\label{barao}
\barao:=\frac{1}{-(\log\hmu)'(\AAA+)}
=\frac{\hmu(\AAA)}{-\hmu'(\AAA+)}>0=\baram;
\end{equation}
otherwise, let $\barao:=\baram$.
In the case $x\in (\baram,\barao]$, 
the infimum in \eqref{gaxxx} is attained at $\zeta=\log\hmu(\AAA)$,
and the infimum in \eqref{gaxxx2} is attained at $\gth=\AAA$, so 
$\gaxxx(x)=x\log\hmu(\AAA)+\AAA$.
For $x\in(\barao,\bara)$, 
the infimum in \eqref{gaxxx} is attained at the $\zeta=\zeta(x)$ for
which $\ga'(\zeta)=-x$, and it follows from \refL{LX} that $\zeta(x)$ is an
analytic function of $x$, and thus $\gaxxx(x)$ is analytic in
$(\barao,\bara)$.
(Similarly, in this case, the infimum in \eqref{gaxxx2} is attained at
$\gth=\gth(x)>-\AAA$ with $(\log\hmu)'(\gth)=-1/x$.)
Hence, the function $\gaxxx(x)$ is analytic in $(\baram,\bara)$ except 
that
in the case $\barao>\baram$, \ie{} when \eqref{barao} applies, 
it is \emph{not} analytic at $\barao$. 
Note also that $\barao\qw=-(\log\hmu)'(\AAA+)
> -(\log\hmu)'(\ga)=\gb$, so $\barao<\gb\qw$.
\end{remark}

\begin{example}
  \label{EFBST}
Consider the binary search tree in \refE{EBST}.
By \refT{Tfillup}, the saturation level satisfies $S(T_n)\sim \gamx\log n$
\as{} (since $\ga=1$).
For the constant $\gamx$, we obtain
from \refL{Lgamx} 
the same equations
\eqref{pyrx} and \eqref{pyrxx} as for $\gam$ (and the height) 
in \refE{EHBST}, but now
with $-1<\gth<0$ and thus $\gamx<1$. Consequently, 
$\gamx$ is the (unique) 
root of \eqref{gammabst} in the interval $(0,1)$.
Numerically, $\gamx\doteq0.37336$.
(This result for the saturation level was shown by \citet{Devroye1986}.) 

Furthermore, by simple calculations, \cf{} \eqref{gahbst}--\eqref{gaxxbst},
\begin{align}
  \chga(\zeta)&=\ga(-\zeta)=2e^\zeta-1,
\\
\label{chgaxxbst}
\chgaxx(x)&=
\begin{cases}
  1,&x\le-2, \\
x\log(|x|/2)-x-1,&-2\le x\le 0,\\
-\infty, & x>0.
\end{cases}
\end{align}
Note that $\chgaxx(-x)$ for $0\le x\le 2$ and $\gaxx(x)$ for $x\ge2$
are given by the same analytic
expressions, but for different ranges of $x$; this is in accordance with 
\refR{Rgaxxx}, which shows that both functions are restrictions of the
single analytic function $\gaxxx(x)$, which in this case is given by,
\cf{} \eqref{gaxxbst} and \eqref{chgaxxbst},
\begin{align}
  \label{gaxxxbst}
\gaxxx(x)&=
\begin{cases}
  -\infty,&x<0, \\
-x\log(x/2)+x-1,& x\ge 0.\\
\end{cases}
\end{align}
The maximum is $\gaxxx(2)=1=\ga$ attained at $x=2=\gb\qw$.
\end{example}

\begin{example}  \label{EFmst}
More generally, for the \mst{} in \refSS{SSmst},
\refT{Tfillup} and \refL{Lgamx} 
show that
the saturation level satisfies $S(T_n)\sim \gamx\log n$
\as{},
where the constant $\gamx$
is given by
the same equations
\eqref{mulmst} and \eqref{gamma3mst} as for $\gam$
in \refE{EHmst}, but now
with $-1<\gth<0$.

Similarly, for the generalization in \refE{Ehennequin}, $\gamx$ is given by
\eqref{mulmstl} and \eqref{gamma3mstl}, with $-\ell-1<\gth<0$.

These results were, as the corresponding results for the height,
proved by a different method by
\citet{ChauvinDrmota2006}, see also \citet[Section 6.5.2]{Drmota}. 
\end{example}

\subsection{Profile} \label{SSprofile}

We have in this section so far considered the height and the saturation
level, which are the maximum and minimum depths of nodes, in the latter case
considering only nodes that are not \emph{full}, \ie, with less than the
maximum number of children.
The results of \citet{Biggins95,Biggins97} are more general, and
yield also results for the distribution of intermediate depths; the
distribution of depths of the nodes is called the \emph{profile} of the
tree.

Let in this section, for a rooted tree $T$ and 
a real number $s$,
\begin{align}
  \nle{s}(T)&:=\bigabs{\set{v\in T:\dep(v)\le s}},
\\
\nge{s}(T)&:=\bigabs{\set{v\in T:\dep(v)\ge s}},
\end{align}
 the number of nodes of depth $\le s$ and $\ge s$, respectively.
We later consider also the number of nodes with exactly a given depth, see
 \eqref{nzk}.

More generally, for a branching random walk with general offsets $\eta_i$
and a \oichar{} $\chi$ as in \refS{SSsat}, let $\Ngechi{s}(t)$
be the number of individuals at time $t$ that have $\chi=1$ and
position $\ge s$, i.e., 
\begin{equation}
  \label{Ngechi}
\Ngechi{s}(t):=\bigabs{\set{x:\chi_x(t-\gs_x)=1\text{ and } y_x\ge s}}.
\end{equation}
This includes the two quantities just defined by 
suitable choices of the
offsets $\eta_i$
and taking the characteristic $\chi(t)=1$ (for all $t\ge0$):
\begin{align}
  \nge{s}(\ctt)&=\Ngechi{s}(t), &&\eta_i=1 \text{ and } \chi(t)=1,\label{nge}
\\
  \nle{s}(\ctt)&=\Ngechi{-s}(t), &&\eta_i=-1 \text{ and } \chi(t)=1\label{nle}
.
\end{align}

One of the main results of \citet{Biggins95,Biggins97} is the following,
which we for convenience state first in the original form
(valid for general $\eta_i$, with the corresponding $\gaxx$).
\begin{theorem}[\citet{Biggins95,Biggins97}]
  \label{TBiggins2}
Suppose that 
$\eta_i$ and $\chi$ are such that
\ref{BIGG1} and \ref{BIGG2} hold. 
If $x\in\bbR$ is such
that $\gaxx(x)>0$ and $\gaxx(x+\eps)>-\infty$ for some $\eps>0$,
then
\begin{equation}
  \frac{\log \Ngechi{xt}(t)}t\asto \gaxx(x).  
\end{equation}
\qed
\end{theorem}

We shall use the following version of \refT{TBiggins2}, for the special cases
$\eta_i=\pm1$ of interest to us.
(The uniform convergence \as{} in \eqref{tb2+} means that 
$\sup_{x\le x_1}\bigabs{\xfrac{\log \Ngechi{xt}(t)}t- \gaxx(x)}\asto0$,
and similarly in \eqref{tb2-}.)

\begin{theorem}
  \label{TBiggins2u}
  \begin{thmenumerate}
  \item \label{tbiggins2u+}
Consider the case $\eta_i=1$.
If $x<\gam$,
then,
\begin{equation}\label{tb2+}
  \frac{\log \Ngechi{xt}(t)}t\asto \gaxx(x).
\end{equation}
Moreover, this holds
 uniformly for $x\in(-\infty,x_1]$
for any $x_1<\gam$.
  \item \label{tbiggins2u-}
Consider the case $\eta_i=-1$, and assume that 
$\chi$ is such that \ref{BIGG2} holds.
If $x<\chgam$,
then,
\begin{equation}\label{tb2-}
  \frac{\log \Ngechi{xt}(t)}t\asto \chgaxx(x).
\end{equation}
Moreover, this holds uniformly for $x\in(-\infty,x_1]$
for every $x_1<\chgam$.
  \end{thmenumerate}
\end{theorem}

\begin{proof}
\pfitemref{tbiggins2u+}
Let $0<x_1<\gam$ and let $\eps>0$.
Since the function $\gaxx(x)$ is concave, it is continuous on any open set 
where it is finite, and in particular on $(-\infty,\gam)$, see
\refL{LX}\ref{lxgaxx} and \eqref{gamma1}.
Hence, $\gaxx(x)$ is continuous on the compact set $[0,x_1]$, and it follows
that  we can choose an integer $M$ such that
$|\gaxx(x')-\gaxx(x'')|<\eps$ whenever $0\le x'<x''\le x_1$ with 
$x''-x'\le x_1/M$.
Let $z_i:=ix_1/M$, $i=0,\dots,M$. Then thus
\begin{equation}\label{dgaxx}
  \bigabs{\gaxx(z_i)-\gaxx(z_{i-1})}<\eps,
\qquad i=1,\dots,M.
\end{equation}

If $x\le x_1$ and $\eps'<\gam-x_1$, then
$x+\eps'<\gam$ so $\gaxx(x+\eps')\ge0>-\infty$ by \eqref{gamma1}. 
Furthermore, this implies $\gaxx(x)>0$ by \eqref{gaxxgb} and the concavity
of $\gaxx$.
Moreover,
 \ref{BIGG1} and   \ref{BIGG2} hold, see \refR{RBIGG}.
Hence,
  \refT{TBiggins2} 
applies and yields
  \begin{equation}\label{kk}
\frac{\log\Ngechi{x t}(t)}{t}\asto\gaxx(x).
  \end{equation}
Consequently, there exists \as{} $t_0$ such that for $t\ge t_0$,
\begin{equation}\label{kka}
  \Bigabs{\frac{\log\Ngechi{z_i t}(t)}{t}-\gaxx(z_i)}<\eps,
\qquad
i=0,\dots,M.
\end{equation}
If $x\in[0,x_1]$, then $z_{i-1}\le x\le z_i$ for some $i$ and thus, for $t\ge
t_0$,
by \eqref{kka} and \eqref{dgaxx} and the fact that $\gaxx$ is decreasing, 
\begin{equation}
 \frac{\log\Ngechi{x t}(t)}{t}
\le
\frac{\log\Ngechi{z_{i-1} t}(t)}{t}<\gaxx(z_{i-1})+\eps
<\gaxx(x)+2\eps
\end{equation}
and, similarly,
\begin{equation}
 \frac{\log\Ngechi{x t}(t)}{t}
\ge
\frac{\log\Ngechi{z_{i} t}(t)}{t}>\gaxx(z_{i})-\eps
>\gaxx(x)-2\eps.
\end{equation}
Consequently, \eqref{kk} holds uniformly for $x\in[0,x_1]$, for any
$x_1<\gam$.
This extends trivially to $x<0$ too, since then 
$\Ngechi{x t}(t)=\Ngechi0(t)$ and $\gaxx(x)=\gaxx(0)$, see \eqref{gaxxgb}.

\pfitemref{tbiggins2u-}
This is similar to part \ref{tbiggins2u+}, using \eqref{chgamma2} and
\eqref{chgaxxgb}. 
The most important difference is that \ref{BIGG2} now is not automatic, and
has to be assumed. 
(It was noted in the proof of \refT{Tfillup} that \ref{BIGG1} always holds.)
Furthermore, fix 
$x_1<\chgam$ and 
some $x_2<-\gam$, and consider first $x\in[x_2,x_1]$, 
showing \eqref{tb2-} uniformly for such $x$  by the argument above.
The result extends to $x<x_2$ too, since each position is $\ge - H(\ctt)$
and $-H(\ctt)/t\asto -\gam>x_2$ by \refT{TBiggins1}; thus \as, for large $t$,
$-H(\ctt)\ge x_2 t$ and then all positions are $\ge x_2t$, so
$\Ngechi{xt}(t)=\Ngechi{x_2t}(t)$ for every $x<x_2$; furthermore, 
$\chgaxx(x)=\ga=\chgaxx(x_2)$ by \eqref{chgaxxgb}.
(Recall that $\gam<\infty$ by \refL{Lgamma}, which also implies
$x_2<-\gam\le-\gb\qw$.) 
\end{proof}
We immediately get results for the upper part of the profile.

\begin{theorem}\label{Tnge}
  Under the assumptions \refBP{} and \refBPlastpsi,  
for every $x<\gam/\ga$, as \ntoo, a.s.
  \begin{equation}\label{tnge}
\nge{x\log n}(T_n) = n^{\gaxx(\ga x)/\ga+o(1)}	.
  \end{equation}
Moreover, \eqref{tnge} holds uniformly for $x\in[0,x_1]$ 
for every $x_1<\gam/\ga$.
\end{theorem}

\begin{proof}
Let $x_1<\gam/\ga$ and fix $x'$ with $x_1\ga<x'<\gam$.

Take $\eta_i=1$ and $\chi(t)=1$, as in \eqref{nge}.
Then \eqref{nge} and \eqref{tb2+} yield
\begin{equation}\label{kk+}
  \frac{\log \nge{xt}(\ctt)}t
=   \frac{\log \Ngechi{xt}(t)}t
\asto \gaxx(x),
\end{equation}
uniformly on $[0,x']$.
Replace $t$ by $\tau(n)$ and $x$ by $x\log n/\tau(n)$ in \eqref{kk+}.
Since $\log n/\tau(n)\asto\ga$ by \eqref{tex2tau},
it follows that \as, uniformly for $0\le x\le x_1$,
  \begin{equation}
\frac{\log\nge{x\log n}(T_n)}{\tau(n)}=\gaxx(x\log n/\tau(n))+o(1)
=\gaxx(x\ga)+o(1),
  \end{equation}
again using the fact that $\gaxx$ is continuous on (at least) $(-\infty,\gam)$.
Using \eqref{tex2tau} again, this yields
  \begin{equation}
\frac{\log\nge{x\log n}(T_n)}{\log n}
=\gaxx(x\ga)/\ga+o(1),
  \end{equation}
uniformly for $x\in[0,x_1]$, \as, which is \eqref{tnge}.
The case $x<0$, included for completeness, follows trivially from the case
$x=0$ and \eqref{gaxxgb}.
\end{proof}

\begin{remark}
  \label{Rtnge}
\refT{Tnge} does not hold (in general) for $x>\gam/\ga$. Indeed, for such
$x$, by \refT{Theight}, \as{} $H(T_n)<x\log n$ for large $n$ and thus
$\nge{x\log n}(T_n)=0$
and $\log\bigpar{\nge{x\log n}(T_n)}/\log n=-\infty$, while typically
$\gaxx(\ga x)>-\infty$, see \eqref{supp+}.
Furthermore, 
\refT{Tnge} fails for $x=\gam/\ga$ too, as is seen for example by 
the binary search tree, where it follows from \citet{Biggins77martingale}
that $H(T_n)-(\gam/\ga)\log n\to -\infty$ \as,
and thus \as{} $\nge{(\gam/\ga)\log n}(T_n)=0$ for all large $n$;
see \cite{McDiarmid1995} and \cite{Reed} for more precise results.
(Consequently, \eqref{tnge} does not hold uniformly on $(-\infty,\gam/\ga)$;
only on closed subintervals $(-\infty,x_1]$.)
\end{remark}

\begin{remark}
  \label{Rge}
By \eqref{gaxxgb}, the exponent $\gaxx(\ga x)/\ga$ in \eqref{tnge}
is less than 1 for $x>(\ga\gb)\qw$; hence only a polynomially small fraction
of the nodes have depth at least $x\log n$ for such $x$.
On the other hand, for $x\le(\ga\gb)\qw$, \eqref{tnge} is pretty useless and
says only that a fraction $n^{o(1)}$ of all nodes have depth at least
$\gabw$. We shall see in \refSS{SSdepth} below the stronger fact that 
most nodes have depth $\bigpar{\gabw+o(1)}\log n$.
\end{remark}

\begin{remark}\label{Rkork}
The proof shows that \eqref{tnge} holds also if 
we count nodes according to some
\oichar{} $\chi$. 
For example, 
we may choose $\chi(t):=\ett{t<\xi_1}$
and consider leaves only. 
Similarly, in the $m$-ary case $N=m$ we may 
use \eqref{chim} and
count only nodes that are not full, as in \refSS{SSsat}.
For the \mary{} search tree in \refS{SSemst}, we
may similarly consider external nodes (\ie, nodes without a key) only; this
yields the \emph{external profile}.
\end{remark}

For the lower part of the profile, we cannot use \refT{TBiggins2u} directly,
since \ref{BIGG2} does not hold for $\chi=1$, 
but we can use other characteristics.
Before proceeding, 
we define yet another function $\tgaxx(x)$, for $x\ge0$, by
\begin{equation}\label{tgaxx}
  \tgaxx(x):=
  \begin{cases}
x\sup_{y\ge x} \frac{\chgaxx(-y)}{y},
 & x>0,
\\
\max\bigpar{\chgaxx(0),0}, &x=0.
  \end{cases}
\end{equation}
Recall that $\chgaxx(0)=\ga(\infty)=\AAA$, see \eqref{chgaxxgb} and
\eqref{lxii+}. 

Furthermore, if $\hmu(0)$ and 
the right derivative
$\hmu'(0+)$ are finite, let
\begin{equation}\label{barai}
\barai:=\frac{1}{-(\log\hmu)'(0+)}
=\frac{\hmu(0)}{-\hmu'(0+)}
>0;
\end{equation}
if $\hmu(0)$ or $\hmu'(0+)$ is infinite, let $\barai:=0$.
(If $\AAA<0$, then $\hmu(\gth)$ is analytic at $0$, and we may simply write
$\hmu'(0)$.) 
 Note that $\hmu(0)=\E N$ by \eqref{hmub}, which also implies, using the
monotone convergence theorem, that the right derivative
\begin{equation}
  \begin{split}
\hmu'(0+)&
=
-\lim_{\gth\downto0}\frac{\hmu(0)-\hmu(\gth)}{\gth}	
=
-\lim_{\gth\downto0}\E \sum_{i=1}^N \ett{i\le N}\frac{1-e^{-\gth\xi_i}}{\gth}	
\\&
=-\E\sum_{i=1}^N \xi_i.
  \end{split}
\end{equation}
Hence, when $\barai>0$, or more generally when $\E N<\infty$,
\begin{equation}\label{barai2}
  \barai=\frac{\E N}{\E\sum_{i=1}^N \xi_i}.
\end{equation}

\begin{lemma}
\label{Ltgaxx}
\begin{thmenumerate}
\item \label{Ltgaxx-x}
The constant $\barai$ satisfies
\begin{equation}\label{baraim}
  0\le\baram\le\gamx\le\barai<\gb\qw.
\end{equation}
and
furthermore either
\begin{align}\label{baraig}
 0\le\gamx<\barai
\quad\text{or}\quad
0=\gamx=\barai.
\end{align}

\item \label{Ltgaxx-ga}
The function $\tgaxx(x)$ is increasing and continuous on $\ooo$, 
and it is the least non-negative concave majorant of
$\chgaxx(-x)$ on $\ooo$.

\item \label{Ltgaxx-gagb}
We have
\begin{equation}\label{tgaxxgb}
	\tgaxx(x) = 
  \begin{cases}
x\log\hmu(0)
= x\log(\E N)
, & 0\le x\le\barai \text{ and }\barai>0,
\\
\chgaxx(-x), & x\ge \barai,
\\
 \ga, & x\ge \gb\qw,
  \end{cases}
\end{equation}
with
  \begin{align}
&&\tgaxx(x)&<x\log(\E N),&& x>\barai.&&	\label{tgaxxgb3}
\intertext{and}
&&\tgaxx(x)&<\ga, && 0\le x<\gb\qw. &&\label{tgaxxgb2}
  \end{align}
Consequently,
if $\barai>0$, then $\tgaxx(x)$ is analytic except at the
two points $\barai$ and $\gb\qw$.

Furthermore, 
\begin{equation}\label{AAAA}
\tgaxx(0)=0
\iff  
\chgaxx(0)\le 0
\iff
\AAA\le0.
\end{equation}

\item \label{Ltgaxx-inf}
If $x>0$, or if $x\ge0$ and $\AAA\le0$, then
\begin{equation}\label{tgaxxinf}
  \begin{split}
\tgaxx(x)&
=\inf_{0\le\gth\le\ga}\bigset{x\log\hmu(\gth)+\gth},
  \end{split}
\end{equation}
where the infimum is attained at $\gth=0$ for $x\le\barai$, at $\gth=\ga$ for
$x\ge\gb\qw$ and in $(0,\ga)$ for $x\in(\barai,\gb\qw)$.

\item \label{Ltgaxx->}
If $\AAA<0$, then $\tgaxx(x)>\chgaxx(-x)$ for $0\le x<\barai$
and $\tgaxx(x)=\chgaxx(-x)$ for $x\ge\barai$.
On the other hand, if 
$\AAA\ge0$, then $\tgaxx(x)=\chgaxx(-x)$ for all $x\ge0$.
\end{thmenumerate}
\end{lemma}

\begin{proof}
\pfitemref{Ltgaxx-x}
If $\gamx=0$, then \eqref{baraig} is trivial. 
Suppose now $\gamx>0$. Then, by \refL{Lgamx}, $\AAA<0$. Furthermore,
by \refL{LX}, $\log\hmu(\gth)$ is strictly convex on $(\AAA,\infty)$.
Thus \eqref{barai} applies, and, moreover, for $\gth\in(\AAA,0)$,
\begin{equation}
  \frac{\log\hmu(\gth)}{|\gth|}
=
  \frac{\log\hmu(0)}{|\gth|}
-
  \frac{\log\hmu(0)-\log\hmu(\gth)}{-\gth}
>
  \frac{\log\hmu(0)}{|\gth|}
-(\log\hmu)'(0).
\end{equation}
Hence, the infimum  of
$  \xfrac{\log\hmu(\gth)}{|\gth|}$ over any finite subinterval of $(\AAA,0)$
is strictly larger than $-(\log\hmu)'(0)=\barai\qw$.
Moreover, if $\AAA=-\infty$, then the strict convexity also implies
\begin{equation}
\lim_{\gth\to-\infty} \frac{\log\hmu(\gth)}{|\gth|}
=-\lim_{\gth\to-\infty}(\log\hmu)'(\gth)>-(\log\hmu)'(0).  
\end{equation}
Hence, in any case, \eqref{gamx2a} yields $\gamx\qw>\barai\qw$ and thus
$\gamx<\barai$
when
$\gamx>0$, which completes the proof of \eqref{baraig}.

By \eqref{baraig} and
\eqref{gamxbaram}, $0\le\baram\le\gamx\le\barai$, so to show \eqref{baraim},
it remains only to verify
$\barai<\gb\qw$. This is trivial if $\barai=0$, so assume $\barai>0$.
Then \eqref{barai} holds and $\AAA\le0$.
Hence $\log\hmu(\gth)$ is strictly convex on $[0,\ga]$ by \refL{LX}
and thus $(\log\hmu)'(0)< (\log\hmu)'(\ga)$, and thus by
\eqref{barai} and \eqref{hmu'gb},
\begin{equation}
\barai\qw= 
-(\log\hmu)'(0)>- (\log\hmu)'(\ga)
=\gb. 
\end{equation}

\pfitemref{Ltgaxx-gagb}
First, \eqref{AAAA} is immediate by \eqref{tgaxx} and \eqref{chgaxxgb}.

We now show \eqref{tgaxxgb}. 
For $x\ge\gb\qw$, \eqref{chgaxxgb} yields
$\chgaxx(-y)=\ga$ for every $y\ge x$, 
and thus \eqref{tgaxx} yields $\tgaxx(x)=\ga=\gaxx(-x)$.
Hence it suffices to consider $0\le x\le \gb\qw$.
In this range, we have $\chgaxx(-x)=\gaxxx(x)$, see \refR{Rgaxxx},
and it follows that for such $x$,
we may replace $\chgaxx(-y)$ by $\gaxxx(y)$ in
\eqref{tgaxx}, which will be a (minor) convenience.
Note first that by \eqref{gaxxx2}, for $x>0$,
\begin{equation}\label{gaxxx2x}
  \begin{split}
\frac{\gaxxx(x)}x
=\inf_{-\infty<\gth<\infty}\Bigset{\log\hmu(\gth)+\frac{\gth}x}.	
  \end{split}
\end{equation}
Assume that $\baram<x<\bara$ and 
let $\gth(x)$ be the (unique) $\gth$ where the infima in 
\eqref{gaxxx2} and \eqref{gaxxx2x}
are attained, see \refR{Rgaxxx}.
Let $f$ be the convex function $f(\gth):=x\log\hmu(\gth)+\gth$ in
\eqref{gaxxx2}, and note that if $\barai>0$, then \eqref{barai} implies
that the right derivative
\begin{equation}\label{f'0}
  f'(0+)=x\frac{\hmu'(0+)}{\hmu(0)}+1=-\frac{x}{\barai}+1.
\end{equation}

Suppose first that $\barai\le x\le\gb\qw$ and $x>0$.
If $\barai>0$, then \eqref{f'0} yields $f'(0+)=-x/\barai+1\le0$;
if $\barai=0$, then $f(0)=\infty$ or $f'(0+)=-\infty$. In all three cases,
the minimum point satisfies $\gth(x)\ge 0$. Hence, if 
$y\ge x$, then by \eqref{gaxxx2x},
\begin{equation}\label{qb}
\frac{\gaxxx(y)}{y}
\le
{\log\hmu(\gth(x))+\frac{\gth(x)}y}	
\le
{\log\hmu(\gth(x))+\frac{\gth(x)}x}	
= \frac{\gaxxx(x)}{x}.
\end{equation}
Thus, 
$\gaxxx(x)/x$  is decreasing on $[\barai,\infty)\cap(0,\infty)$
(and so is $\chgaxx(-x)/x$),
and it follows from \eqref{tgaxx} that $\tgaxx(x)=\gaxxx(x)=\chgaxx(-x)$
for $x\ge\barai$ with $0<x\le\gb\qw$.
Thus, the second line of \eqref{tgaxxgb} holds for $x\ge \barai$ with $x>0$.

On the other hand, if $\baram<x\le\barai$ (which entails $\barai>0$ and thus
$\AAA\le0$), 
we similarly have by \eqref{f'0}
$f'(0)\ge0$, and thus the minimum point $\gth(x)\le 0$. 
Hence, if also
$0< z\le x$, then by \eqref{gaxxx2x},
\begin{equation}\label{qb-}
\frac{\gaxxx(z)}{z}
\le
{\log\hmu(\gth(x))+\frac{\gth(x)}z}	
\le
{\log\hmu(\gth(x))+\frac{\gth(x)}x}	
= \frac{\gaxxx(x)}{x}.
\end{equation}
If $\baram=0$, we have shown that $\gaxxx(x)/x$ is increasing on $(0,\barai]$.
If $\baram>0$, we have shown that $\gaxxx(x)/x$ is increasing on
$[\baram,\barai]$, and since $\gaxxx(x)=-\infty$ for $x<\baram$ by
\eqref{gaxxx} and \eqref{chbaram123}, $\gaxxx(x)/x$ is increasing on
$(0,\barai]$ 
in this case too.

Hence, if $\barai>0$, then $\gaxxx(x)/x$ is increasing on $(0,\barai]$
and decreasing on $[\barai,\infty)$;  consequently
it has a maximum at $\barai$
and, for every $x\in(0,\barai]$,
\begin{equation}\label{maxgaxxx1}
  \max_{y\ge x}\frac{\gaxxx(y)}y  
=
  \frac{\gaxxx(\barai)}{\barai} . 
\end{equation}
Moreover, again by \eqref{f'0}, if $x=\barai>0$, then
$f'(0+)=0$, and thus the infimum in \eqref{gaxxx2x} is attained at $\gth=0$.
This yields the value
\begin{equation}\label{maxgaxxx2}
\frac{\gaxxx(\barai)}{\barai}
=\log\hmu(0)=\log\mu\ooo=\log(\E N),
\qquad \barai>0.
\end{equation}
It follows from \eqref{maxgaxxx1}--\eqref{maxgaxxx2} and \eqref{tgaxx} that
the first line of \eqref{tgaxxgb} holds for $0<x\le\barai$.  

Finally, consider $x=0$. By \refL{LXch}\ref{lxchgaxx}, $\chgaxx(0)=\AAA$.
If $\barai>0$, then $\hmu(0)<\infty$ and thus
$\AAA\le0$; hence \eqref{tgaxx} yields $\tgaxx(0)=0$. 
(Cf.\ \eqref{AAAA}.) 
Similarly,
if $\barai=0$, then $\hmu(0)$ or $\hmu'(0)$ is infinite, and thus
$\AAA\ge0$; hence \eqref{tgaxx} yields
$\tgaxx(0)=\AAA=\chgaxx(0)$. In both cases \eqref{tgaxxgb} holds. 
This completes the proof of \eqref{tgaxxgb}.

Moreover, 
if $\barai<x\le\gb\qw$,
then \eqref{f'0} yields $f'(0+)<0$,
and thus the infimum in \eqref{gaxxx2}
is not attained at $0$; 
hence 
$\gaxxx(x)<x\log\hmu(0)$. 
Furthermore, by \eqref{tgaxxgb} and \refR{Rgaxxx},
$\gaxx(x)=\chgaxx(-x)=\gaxxx(x)$; hence 
\eqref{tgaxxgb3} holds for $\barai<x\le\gb\qw$.
This extends to $x>\gb\qw$ since
$\tgaxx(x)$ is constant there by \eqref{tgaxxgb}.

Furthermore, it follows from \eqref{tgaxxgb} and
\eqref{chgaxxgb} that $\tgaxx(x)<\ga$ if $\barai\le x<\gb\qw$; 
if $\barai>0$, 
this extends to $0\le x<\gb\qw$, 
because
$\barai<\gb\qw$ and
$\tgaxx(x)$ is increasing on $[0,\barai]$ by \eqref{tgaxxgb}.
Hence \eqref{tgaxxgb2} holds.

\pfitemref{Ltgaxx-ga}
By \eqref{tgaxxgb}, $\tgaxx(x)$ is linear, continuous, and increasing on
$[0,\barai]$.
Furthermore,
by \refL{LXch}, $\chgaxx(-x)$ is concave, increasing and continuous on
$[\baram,\infty)$; 
hence, by \eqref{baraim} and \eqref{tgaxxgb},
$\tgaxx(x)=\chgaxx(-x)$ is concave, increasing and continuous on
$[\barai,\infty)$.
Consequently, $\tgaxx(x)$ is increasing and continuous on $\ooo$, and to
show that it is concave, it suffices to verify that, in the case $\barai>0$,
the left derivative at $\barai$ is at least as large as the right derivative.
To see this, note that 
$\tgaxx(x)=x\log(\E N)$ for $0\le x\le \barai$ 
by \eqref{tgaxxgb}, while
$\tgaxx(x)=\gaxxx(x)\le x\log(\E N)$ for $\barai\le x\le\gb\qw$ 
by \eqref{tgaxxgb} and
\eqref{maxgaxxx1}--\eqref{maxgaxxx2};
hence the left derivative equals $\log(\E N)$, and the right derivative
cannot be larger.

Moreover, $\tgaxx(x)$ is by the definition \eqref{tgaxx} 
a non-negative majorant of $\chgaxx(-x)$,
and we have just shown that it is concave.
It follows easily from \eqref{tgaxx} that it is the
least concave non-negative majorant.

\pfitemref{Ltgaxx-inf}
If $x=0$ and $\AAA\le0$, then \eqref{tgaxxinf}
follows by \eqref{AAAA}, so we may assume $x>0$.
If $0< x\le\barai$, then we have as above $f'(0+)\ge0$ 
by \eqref{f'0}
and thus the convex
function $f(\gth)$ is increasing on $\ooo$, 
so the infimum in \eqref{tgaxxinf} is attained at $\gth=0$
and equals $x\log\hmu(0)$, which
verifies \eqref{tgaxxinf} in this case by \eqref{tgaxxgb}.
Similarly, if $x\ge\gb\qw$, then, recalling \eqref{el2} and \ref{BPmalthus},
\begin{equation}\label{f'ga}
  f'(\ga)=x\frac{\hmu'(\ga)}{\hmu(\ga)}+1=-x\gb+1
\le0
\end{equation}
and thus $f(\gth)$ is decreasing for $\gth\le\ga$,
so the infimum in \eqref{tgaxxinf} is attained at $\gth=\ga$ 
and equals $x\log\hmu(\ga)+\ga=\ga$, which
verifies \eqref{tgaxxinf} in this case, again using \eqref{tgaxxgb}.
If $0<\barai<x<\gb\qw$, then $f'(0)<0<f'(\ga)$ so the minimum point
$\gth(x)$ in \eqref{gaxxx2} satisfies $0<\gth(x)<\ga$,
and a minor variation yields $0<\gth(x)<\ga$ also when 
$0=\barai<x<\gb\qw$.
Hence, if
$\barai<x<\gb\qw$, then  
the infima in \eqref{gaxxx2} and \eqref{tgaxxinf} are equal and
\eqref{tgaxxinf} follows in this case too. (Cf.\ \eqref{chgaxx2}.)

\pfitemref{Ltgaxx->}
By \eqref{tgaxx}, $\tgaxx(x)\ge\chgaxx(-x)$ for every $x$, and equality
holds for $x\ge\barai$ by \ref{Ltgaxx-gagb}.

For $x=0$,
\eqref{tgaxx} yields $\tgaxx(0)=\chgaxx(0)\iff \chgaxx(0)\ge0$,
and the result follows since $\chgaxx(0)=\AAA$ by \eqref{chgaxxgb}.

It remains to consider $0<x<\barai$. If $\AAA\ge0$, then $\hmu(\gth)=\infty$
for $\gth<0$. Hence \eqref{chgaxx2} and \eqref{tgaxxinf} yield
\begin{equation}
\chgaxx(-x)=\inf_{-\infty<\gth\le\ga}\bigset{x\log\hmu(\gth)+\gth}  
=\inf_{0\le\gth\le\ga}\bigset{x\log\hmu(\gth)+\gth}  
=\tgaxx(x).
\end{equation}
On the other hand, if $\AAA<0$, then $f(\gth)$ is analytic at $\gth=0$, and
$f'(0)>0$ by \eqref{f'0}; hence the infimum in \eqref{gaxxx2} is strictly
less than the value at $\gth=0$, \ie,
\begin{equation}
\gaxxx(x)
<x\log\hmu(0),
\end{equation}
which yields $\gaxxx(x)<\tgaxx(x)$ by \eqref{tgaxxgb}.
\end{proof}

\begin{theorem}\label{Tnle}
  Under the assumptions \refBP{} and \refBPlastpsi,  
as \ntoo, a.s.,
  \begin{equation}\label{tnle}
\nle{x\log n}(T_n) = n^{\tgaxx(\ga x)/\ga+o(1)},
  \end{equation}
uniformly in $x\ge x_0$ for every $x_0>0$;
if further $\AAA\le0$, then \eqref{tnle} holds  uniformly for all $x\ge0$.
\end{theorem}

\begin{proof}
  Let $\chi$ be the \oichar{} 
  \begin{equation}\label{chi01}
\chi(t):=\ett{0\le t<1}, 	
  \end{equation}
meaning that we
  count only individuals that are less than one unit old.

In the present proof we consider both this characteristic $\chi$ and the
characteristic 1, and both offsets $\eta_i=1$ and $\eta_i=-1$; we therefore
introduce more notation to distinguish between the different cases.
Let $\Ngechi{s}(t)$ and $\Ngei{s}(t)$ denote the number in \eqref{Ngechi}
counted with the characteristic $\chi$ in \eqref{chi01}
and the characteristic $1$, respectively,
in both cases for the offsets $\eta_i=1$. We similarly
let $\Nlechi{s}(t)$ and $\Nlei{s}(t)$ denote the corresponding numbers of
individuals with position  less that $s$. Finally, 
we let $\chNgechi{s}(t)$ and $\chNgei{s}(t)$ denote the corresponding
numbers for the offsets $\eta_i=-1$. Note that
\begin{equation}\label{nchn}
  \Nlechi s(t)=\chNgechi{-s}(t)
\quad\text{and}\quad
  \Nlei s(t)=\chNgei{-s}(t).
\end{equation}
Furthermore,
\begin{equation}\label{nchi1}
  \Ngechi{s}(t)=\Ngei{s}(t)-\Ngei{s}(t-1).
\end{equation}

Note further that \ref{BIGG2} holds trivially for $\chi$ and $\eta_i=-1$, since
$\chga(\zeta)=\ga(-\zeta)>-\infty$ for each $\zeta$ by \refL{LXch}\ref{lxchga}
and thus
\begin{equation}
  \sup_{t>0}\bigpar{e^{-\ga(\zeta)t}\chi(t)}
=\max\bigset{0,e^{-\ga(\zeta)}}<\infty.
\end{equation}
Consequently, \refT{TBiggins2u}\ref{tbiggins2u-} applies to $\chi$ and
$\eta_i=-1$ and yields, for
every $x<\chgam$,
\begin{equation}
  \label{chkod}
\frac{\log \chNgechi{xt}(t)}t 
\asto 
\chgaxx(x).
\end{equation}
Using \eqref{nchn}, this means that for
every $x>\gamx=-\chgam$,
\begin{equation}
  \label{kork}
\frac{\log \Nlechi{xt}(t)}t 
=\frac{\log \chNgechi{-xt}(t)}t 
\asto 
\chgaxx(-x)
\le \tgaxx(x).
\end{equation}

Furthermore, if $0\le x<\gamx$, then \refT{TBiggins1} can be used as in the
proof of \refT{Tfillup}, see \eqref{bt-}, to show that \as{} 
$\Nlechi{xt}(t)=\chNgechi{-xt}(t)=0$ for all large $t$, and thus 
\begin{equation}
  \label{kod-}
\frac{\log \Nlechi{xt}(t)}t 
\asto 
-\infty\le
\tgaxx(x).
\end{equation}
(Alternatively, we may use \cite[Theorem 5(i)]{Biggins95} directly.)
Finally, if $x=\gamx$, then for any $\eps>0$, by \eqref{kork},
\begin{equation}\label{korkz}
  \limsup_{\ttoo}\frac{\log\Nlechi{xt}(t)}t
\le
  \limsup_{\ttoo}\frac{\log\Nlechi{(x+\eps)t}(t)}t
\le\tgaxx(x+\eps).
\end{equation}
Since $\tgaxx$ is continuous 
by \refL{Ltgaxx}
and $\eps>0$ in \eqref{korkz}
is arbitrary, it follows from 
\eqref{kork}--\eqref{korkz} that for every $x\ge0$,
\begin{equation}\label{bra}
  \limsup_{\ttoo}\frac{\log\Nlechi{xt}(t)}t
\le\tgaxx(x).
\end{equation}
(See also \cite[Theorem 5 and the comment after it on the upper
  bound]{Biggins95}.) 
Note also that for $x_1=\gam+1$, say, \as{} for large $t$, $H(\ctt)\le x_1t$
and thus $\Nlechi{xt}(t)=\Nlechi{x_1t}(t)$ for all $x\ge x_1$.
Since $\tgaxx(x)$ is continuous, the argument in the proof of \refT{Tnge}
shows that \as, \eqref{bra} holds uniformly for all $x\ge0$.

In other words, for any $\gd>0$, there exists \as{} $t_0>0$ such that for all
$t\ge t_0$ and all $x\ge0$,
\begin{equation}\label{bru}
{\log\Nlechi{xt}(t)}
\le\bigpar{\tgaxx(x)+\gd}t.
\end{equation}
Hence, for all $s>0$ and $t\ge u\ge t_0$,
\begin{equation}\label{bruu}
  \begin{split}
{\log\Nlechi{s}(u)}
&\le\bigpar{\tgaxx(s/u)+\gd} u
=\frac{\tgaxx(s/u)}{s/u}s+\gd u
\\&
\le
\frac{\tgaxx(s/t)}{s/t}s+\gd t
=\bigpar{\tgaxx(s/t)+\gd} t,	
  \end{split}
\end{equation}
using the fact  that by \eqref{tgaxx},
$\tgaxx(x)/x$ is decreasing for $x>0$.
Furthermore, \eqref{bruu}  holds for $s=0$ too, again by \eqref{bru}.

We now return to the characteristic 1. By \eqref{nchi1}, for 
$t\ge t_0$,
\begin{equation}
  \begin{split}
  \Nlei{s}(t)
&=
\sum_{i=0}^{\floor{t-t_0}}\Nlechi{s}(t-i) + \Nlei{s}(t-\floor{t-t_0}-1)
\\&
\le \sum_{i=0}^{\floor{t-t_0}}\Nlechi{s}(t-i) + Z_{t_0},	
  \end{split}
\end{equation}
since $t-\floor{t-t_0}-1\le t_0$.
With $t_0$ as above, we thus have by \eqref{bruu} \as, for $t\ge t_0$ and
all $s\ge0$,
\begin{equation}
  \begin{split}
  \Nlei{s}(t)
&
\le \sum_{i=0}^{\floor{t-t_0}}e^{(\tgaxx(s/t)+\gd)t} + Z_{t_0}
\le (t+1)e^{(\tgaxx(s/t)+\gd)t} + Z_{t_0}
\end{split}
\end{equation}
and consequently, for all $x\ge0$,
\begin{equation}\label{brb}
\frac{\log\Nlei{xt}(t)}t
\le\tgaxx(x)+\gd+\frac{\log(t+1)}t+\frac{Z_{t_0}}t.
\end{equation}
Hence, \as{} there exists $t_1\ge t_0$ such that for all $t\ge t_1$
and all $x\ge0$,
\begin{equation}
\frac{\log\Nlei{xt}(t)}t
\le\tgaxx(x)+2\gd.
\end{equation}
Consequently, \as,
\begin{equation}\label{nlei+}
  \begin{split}
\limsup_{\ttoo}
\frac{\log \Nlei{xt}(t)}t	
\le
\tgaxx(x),
  \end{split}
\end{equation}
uniformly for all $x\ge0$.

For a lower bound, consider first $x>0$. For a given $y\ge x>0$ and any
$t>0$, let $t':=tx/y$. Then $t'\le t$ and
\begin{equation}
  \Nlei{xt}(t)
=
  \Nlei{yt'}(t)
\ge
  \Nlei{yt'}(t')
\ge
  \Nlechi{yt'}(t').
\end{equation}
Consequently, letting \ttoo{} and using \eqref{kork},
\begin{equation}
  \begin{split}
\frac{\log \Nlei{xt}(t)}t	
\ge
\frac{\log \Nlechi{yt'}(t')}t	
=
\frac{\log \Nlechi{yt'}(t')}{yt'/x}
\asto
x\frac{\chgaxx(-y)}{y}.	
  \end{split}
\end{equation}
Hence, \as, 
\begin{equation}
  \begin{split}
\liminf_{\ttoo}
\frac{\log \Nlei{xt}(t)}t	
\ge
x\frac{\chgaxx(-y)}{y},
  \end{split}
\end{equation}
for every $y\ge x$, and thus by \eqref{tgaxx}, \as,
\begin{equation}\label{nlei-}
  \begin{split}
\liminf_{\ttoo}
\frac{\log \Nlei{xt}(t)}t	
\ge
\tgaxx(x).
  \end{split}
\end{equation}
For $x=0$, we have $\Nlei0(t)=1$ for all $t\ge0$, and thus 
$\xfrac{\log \Nlei{0}(t)}t=0$. Hence, 
\eqref{nlei-} holds for $x=0$ too if
(and only if) $\tgaxx(0)=0$,
which by \eqref{AAAA} holds if and only if $\AAA\le0$.

Note that $\nle{xt}(\ctt)=\Nlei{xt}(t)$, \cf{} \eqref{nle} and \eqref{nchn}.
Hence, combining \eqref{nlei+} and \eqref{nlei-}
we obtain
\begin{equation}\label{nlei}
  \begin{split}
\frac{\log \nle{xt}(\ctt)}t	
=
\frac{\log \Nlei{xt}(t)}t	
\asto
\tgaxx(x)
  \end{split}
\end{equation}
for every $x>0$, and in the case $\tgaxx(0)=0$ for every $x\ge0$.
Again, by the argument in the proof of \refT{TBiggins2u}, \eqref{nlei} holds
\as{} uniformly on every interval $[x_0,\infty)$ with $x_0>0$, and in the case
$\tgaxx(0)=0$, uniformly on $[0,\infty)$.

Finally, the proof is completed as the proof of \refT{Tnge}.
\end{proof}

\begin{remark}\label{Repif}
  The lower part of the profile is thus described by the function $\tgaxx(x)$
  and not by $\chgaxx(-x)$. 
(By \refL{Ltgaxx}\ref{Ltgaxx->},
there is a difference only if $\AAA<0$, and only for $x<\barai$.)
However, if we count only nodes described by a
  \oichar{} $\chi$ for which \ref{BIGG2} holds, then  \eqref{kork} holds.
For example, this applies to non-full nodes in an \mst{}, using the
characteristic \eqref{chim}; 
we use this in the proof of \refT{Talla} below.
Similarly, it applies to external nodes in an
\mst, and to leaves in any tree generated by a branching process with
$N\ge1$
(using the argument in the proof of \refT{Tfillup} to verify \ref{BIGG2}).
\end{remark}

\begin{example}\label{Ebaraimary}  
In the \mary{} case $N=m$, we have $\hmu(0)=\E N=m$, so $\AAA\le0$.
Furthermore, \eqref{barai2} holds and shows that
\begin{equation}\label{baraimary}
\barai=\frac{m}{\sum_{i=1}^m \E\xi_i}.  
\end{equation}
In particular, 
with the usual ordering
$\xi_1\le\dots\le\xi_m$,
$\barai>0\iff \E\xi_m<\infty$.
(This holds for all our examples of \mary{} trees, including the \mst{s}.)

For the \mst{}, we can as well use a random labelling of the children as in
\refSS{SSmst}, which yields (with the notation there)
$\E\xi_i=\E (S+X_i)=H_{m-1}$, and thus 
\begin{equation}\label{baraimst}
\barai=1/H_{m-1}   
\end{equation}
by \eqref{baraimary}. In particular, for the \bst, $\barai=1$.

In the \mary{} case with $\barai>0$, by \refL{Ltgaxx},
$\tgaxx(x)=x\log m$ for $x\in[0,\barai]$. Hence, \refT{Tnle} yields, \as,
uniformly for $k\le(\barai/\ga)\log n$,
  \begin{equation}\label{tnlex}
\nle{k}(T_n) = n^{\tgaxx(\ga k/\log n)/\ga+o(1)}
= n^{ k\log m/\log n+o(1)}
=m^{k+o(\log n)}.
  \end{equation}
Equivalently, \as,
uniformly for $k/\log n\in[\gd,\barai/\ga]$ for any $\gd>0$,
  \begin{equation}\label{tnlex2}
\nle{k}(T_n) 
=m^{k+o(k)}.
  \end{equation}
We can interpret this as saying that, \as, a ``large part'' (in a
logarithmic sense) of all possible nodes of depth $\le k$ are filled for 
$k\le(\barai/\ga)\log n$, 
but not for (substantially) larger $k$ as a consequence of \eqref{tgaxxgb3}.
Compare with \refT{Tfillup}, which says that \emph{all} possible nodes are
filled up to depth $S(T_n)\sim\bgamx\log n$ \as,  
and note that $\bgamx<\barai/\ga$ by
\eqref{baraig}, except in the trivial case $\barai=0$.

A stronger result will be proved in \refT{Talla}.
\end{example}

\begin{remark}
  \label{Rdepth0}
In the case $\AAA>0$, when $\tgaxx(0)>0$ by \eqref{AAAA}, \eqref{tnle} is
evidently false for $x=0$. If we consider $x=x(n)>0$ with $x\to0$ as \ntoo,
the situation is more complicated. 
There are examples where \eqref{tnle} fails due to irregular growth of the
tree, see \refE{Ebadapple}, but in other, more regular cases,  
\eqref{tnle} holds uniformly for all $x>0$ such that $x\log n\ge1$,
see \refT{Tnotbad}. 
(The upper bound \eqref{nlei+} always holds; the problem is the lower bound.)
\end{remark}

So far we have considered the number of nodes of depth $\le s$ and $\ge s$,
for a given  $s$.
We are also interested in the number of nodes with depth exactly a given
value, and we thus define, for a tree $T$ and an integer $k\ge0$, 
\begin{equation}
  \label{nzk}
  n_{k}(T):=\bigabs{\set{v\in T:\dep(v)=k}}.
\end{equation}
We shall see that the asymptotics of these numbers are governed by the
function $\hgaxx(x)$ defined by
\begin{equation}\label{hgaxx}
  \hgaxx(x):=\min\bigset{\gaxx(x),\tgaxx(x)}.
\end{equation}
Note that $\hgaxx(x)=\tgaxx(x)$ for $0\le x\le\gb\qw$
and $\hgaxx(x)=\gaxx(x)$ for $x\ge\gb\qw$, 
see \eqref{gaxxgb} and \eqref{tgaxxgb}.

\begin{lemma}
\label{Lhgaxx}
The function $\hgaxx(x)$ is concave on $\ooo$; it is finite for $x<\bara$ 
and $-\infty$ for $x>\bara$, and it is continuous except possibly at $\bara$.
Moreover, $\hgaxx(x)$ has a global maximum at $x=1/\gb$, with
$\hgaxx(1/\gb)=\ga$ and $\hgaxx(x)<\ga$ for $x\neq1/\gb$.
Furthermore,
  \begin{equation}\label{lhgaxx}
  \hgaxx(x)
=
  \begin{cases}
x\log\hmu(0)
= x\log(\E N)
, & 0\le x<\barai,
\\
\gaxxx(x), & x\ge \barai.
  \end{cases}
\end{equation}
Hence, $\hgaxx(x)$ is analytic on $(0,\bara)$, except at $\barai$ 
(if $\barai>0$), where it is not analytic.
Also, 
\begin{equation}\label{hgaxx0}
  \hgaxx(0)=0\iff\AAA\le0.
\end{equation}

If $x>0$, or if $x\ge0$ and $\AAA\le0$, then also
\begin{equation}\label{hgaxxinf}
  \begin{split}
\hgaxx(x)
=
\inf_{0\le\gth<\infty}\bigset{x\log\hmu(\gth)+\gth}.
  \end{split}
\end{equation}
\end{lemma}

\begin{proof}
Simple consequences of 
the definition \eqref{hgaxx} and
Lemmas \ref{LX}\ref{lxgaxx} and \ref{Ltgaxx}
together with \eqref{gaxxx} and \refR{Rgaxxx},
with \eqref{hgaxxinf} following directly from \eqref{hgaxx}, \eqref{gaxx2}
and \eqref{tgaxxinf}.
\end{proof}

\begin{remark}
Compare the formulas 
\eqref{gaxx2}, \eqref{chgaxx2} (with $-x$), \eqref{gaxxx2}, \eqref{tgaxxinf} and
\eqref{hgaxxinf}, expressing various quantities as the infimum of the same
function over different ranges.
\end{remark}

\begin{theorem}\label{Tn=}
Assume \refBP{} and \refBPlastpsi.
If $0<x<\gam/\ga$, then \as,
as \ntoo,
  \begin{equation}\label{tn=}
\nz{\floor{x\log n}}(T_n) = n^{\hgaxx(\ga x)/\ga+o(1)}.
  \end{equation}
Moreover, this hold
uniformly for $x$ in any compact subset of $(0,\gam/\ga)$.
Furthermore, if further $\AAA\le0$, then \eqref{tn=} holds  uniformly for
$x$ in any compact subset of $[0,\gam/\ga)$.
\end{theorem}

\begin{proof}
We will show the result in the form
  \begin{equation}\label{tn=2}
\frac{\log\nzfloor{x t}(\ctt)}{t} \asto\hgaxx(x),
  \end{equation}
uniformly for $x$ in any compact subset of $(0,\gam)$,
and if  $\AAA\le0$, uniformly for
$x$ in any compact subset of $[0,\gam)$.
This implies \eqref{tn=} by the usual argument using \eqref{tex2tau}.
Furthermore, since $\hgaxx(x)$ is continuous on $[0,\gam)$, it suffices to
  consider the case when $xt$ is an integer, \ie,
  \begin{equation}\label{tn=3}
\frac{\log\nz{k}(\ctt)}{t} =\hgaxx(k/t)+o(1),
  \end{equation}
uniformly for $k/t$ in any compact subset of $(0,\gam)$
or $[0,\gam)$, respectively.

The upper bound is easy. Since $\nz{k}(\ctt)\le \nle{k}(\ctt)$
and $\nz{k}(\ctt)\le \nge{k}(\ctt)$, we get the upper bound 
  \begin{equation}\label{tn<3}
\limsup_{\ttoo}\frac{\log\nz{k}(\ctt)}{t} \le\hgaxx(k/t),
  \end{equation}
from Theorems \ref{Tnle} and
\ref{Tnge}, or more precisely from \eqref{nlei+}
and \eqref{kk+}, uniformly for $1\le k\le x' t$ for any $x'<\gam$.

For the corresponding lower bound, we first consider the special case when
the number of offspring $N$ in the branching process is bounded, $N\le M$ for
some constant $M$. (This implies $\hmu(0)=\E N\le M<\infty$, and thus
$\AAA\le 0$.) 
Let $\gd>0$. 
On the interval
$[0,\gb\qw]$, $\hgaxx(x)=\tgaxx(x)$ is strictly increasing, and thus,
by compactness,
there exists $\eps>0$ such that if $0\le x\le y \le \gb\qw$ with
$y-x\ge\gd$, then $\tgaxx(y)-\tgaxx(x)\ge3\eps$.
Consequently, 
\eqref{nlei} implies that
\as, for $t$ large and for all $x\in[0,\gb\qw-\gd]$,
\begin{equation}
  \begin{split}
\nle{(x+\gd)t}(\ctt)	
 - \nle{xt}(\ctt)
&
\ge e^{(\tgaxx(x+\gd)-\eps)t}-e^{(\tgaxx(x)+\eps)t}	
\\&
\ge e^{\tgaxx(x)t}\bigpar{e^{2\eps t}-e^{\eps t}}	
\ge e^{\tgaxx(x)t}.
  \end{split}
\end{equation}
Furthermore, if $x=k/t$ for an integer $k$, then each node $v$ in $\ctt$ with
depth $h(v)\in[xt,(x+\gd)t)$ has an ancestor $v'$ with depth $k=xt$, and at
most $\sum_{i=0}^{\gd t} M^i \le M^{2\gd t}$ such nodes $v$ have the same
  ancestor $v'$; hence
\begin{equation}
  \begin{split}
\nz{k}(\ctt)
\ge M^{-2\gd t}\bigpar{\nle{(x+\gd)t}(\ctt) - \nle{xt}(\ctt)}
\ge e^{\tgaxx(x)t-2\gd t\log M}
  \end{split}
\end{equation}
and thus
\begin{equation}\label{jd1}
  \begin{split}
\frac{\log\nz{k}(\ctt)}t
\ge \tgaxx(x)-2\gd \log M
= \hgaxx(k/t)-2\gd \log M .
  \end{split}
\end{equation}

Similarly,  
on the interval
$[\gb\qw,\gam)$, $\hgaxx(x)=\gaxx(x)$ is strictly decreasing, and thus,
by compactness, for any $x_1<\gam$,
there exists $\eps>0$ such that if $\gb\qw\le x\le y \le x_1$ with
$y-x\ge\gd$, then $\gaxx(x)-\gaxx(y)\ge3\eps$. We may also assume $\eps\le\gd$.
Consequently, 
\eqref{kk+} implies that
\as, for $t$ large and for all $x\in[\gb\qw,x_1-\gd]$,
\begin{equation}\label{johannes}
  \begin{split}
\nge{xt}(\ctt)	
 - \nge{(x+\gd)t}(\ctt)
&
\ge e^{(\gaxx(x)-\eps)t}-e^{(\gaxx(x+\gd)+\eps)t}	
\\&
\ge e^{\gaxx(x)t}\bigpar{e^{-\eps t}-e^{-2\eps t}}	
\ge e^{\gaxx(x)t-2\eps t}
\ge e^{\gaxx(x)t-2\gd t}.
  \end{split}
\end{equation}
Hence, arguing as above, \as{} for large $t$, if $x=k/t$ for an integer $k$
with $x\in[\gb\qw,x_1-\gd]$, then
\begin{equation}
  \begin{split}
\nz{k}(\ctt)
\ge M^{-2\gd t}\bigpar{\nge{xt}(\ctt) - \nge{(x+\gd)t}(\ctt)}
\ge e^{\gaxx(x)t-2\gd t-2\gd t\log M}
  \end{split}
\end{equation}
and thus
\begin{equation}\label{jd2}
  \begin{split}
\frac{\log\nz{k}(\ctt)}t
\ge \gaxx(x)-2\gd (\log M+1)
= \hgaxx(k/t)-2\gd (\log M+1) .
  \end{split}
\end{equation}

Finally, if $k/t\in[\gb\qw-\gd,\gb\qw]$, we similarly obtain from
\eqref{johannes} 
\begin{equation}
  \begin{split}
\nz{k}(\ctt)
\ge M^{-3\gd t}\bigpar{\nge{\gb\qw}(\ctt) - \nge{(\gb\qw+\gd)t}(\ctt)}
\ge e^{\gaxx(\gb\qw)t-2\gd t-3\gd t\log M}
  \end{split}
\end{equation}
and thus, using $\gaxx(\gb\qw)=\ga\ge\hgaxx(k/t)$,
\begin{equation}\label{jd3}
  \begin{split}
\frac{\log\nz{k}(\ctt)}t
\ge \hgaxx(k/t)-\gd (3\log M+2).
  \end{split}
\end{equation}
Combining \eqref{jd1}, \eqref{jd2} and \eqref{jd3}, we see that \as,
\eqref{jd3} holds for all  $k\in[0,(x_1-\gd)t]$ and all  $t\ge t_0$ (with $t_0$
random and depending on $\gd$ and $x_1$, but not on $k$).
Since $\gd$ and $x_1$ are arbitrary, this implies that for any $x_2<\gam$, 
\as,
\begin{equation}\label{tn=4}
  \begin{split}
\frac{\log\nz{k}(\ctt)}t
\ge \hgaxx(k/t)+o(1)
  \end{split}
\end{equation}
uniformly for $k/t\in[0,x_2]$, which together with the upper bound
proved above yields \eqref{tn=3} in the special case when $N$ is bounded.

The general case follows by pruning the branching process and the
corresponding tree $\ctt$. Let $M$ be a large integer and 
let each individual keep only at most $M$ children,
for example by discarding all
children after the $M$ first of every individual. We denote the pruned tree
by $\ctt\MM$, and write similarly $\mu\MM$ and so on for other
quantities for the pruned version.
Note that, by monotone convergence and \eqref{hmub},
\begin{equation}\label{hmuMM}
\hmu\MM(\gth)\upto\hmu(\gth)
\qquad\text{as \Mtoo}.
\end{equation}
for every $\gth$. It is easy to verify that
\refBP{} except \ref{BPnonlattice} 
hold for the pruned process too, provided $M$ is large enough.
Furthermore, also \ref{BPnonlattice} holds except in some exceptional
cases; in those cases we can
modify the
pruning by selecting the $M$ surviving children by a suitable random
procedure that preserves both \ref{BPnonlattice} and \eqref{hmuMM}; 
we omit the details.
(Alternatively, it seems probable
that \ref{BPnonlattice} is not really needed for
the proofs in the present section, but we have not verified this in detail.)
Moreover, $\ctt\MM$ is a subtree of $\ctt$, and thus, by \eqref{tn=4} for
the special case just treated,
\begin{equation}\label{jd4}
  \begin{split}
\frac{\log\nz{k}(\ctt)}t
\ge \frac{\log\nz{k}(\ctt\MM)}t
\ge \hgaxxMM(x)+o(1),
  \end{split}
\end{equation}
\as{} for large $t$, uniformly for $k/t\in[0,x_1]$.
The result \eqref{tn=3} now follows from the following lemma.
\end{proof}

\begin{lemma}
Let $\hgaxxMM(x)$ be $\hgaxx$ for the branching process pruned to at most
$M$ children for each individual, using a pruning such that \eqref{hmuMM} holds.
If $0< x<\gam$, then, as $M\to\infty$,
\begin{equation}\label{hgaxxMM}
\hgaxxMM(x)\upto\hgaxx(x),   
\end{equation}
and this holds uniformly for $x$ in any compact subset of $(0,\gam)$;
if $\AAA\le0$, it holds
uniformly in any compact subset of $[0,\gam)$.
\end{lemma}

\begin{proof}
By \eqref{hmuMM} and \eqref{hgaxxinf}, $\hgaxxMM(x)$ is increasing in $M$
and $\hgaxxMM(x)\le\hgaxx(x)$ for every $x>0$.
If $\AAA\le0$, then also $\hgaxx(0)=0=\hgaxxMM(0)$ for every $M$ by
\eqref{hgaxx0}. 
Let $\eps>0$ and let $x>0$, or $x=0$ and $\AAA\le0$. We shall show that
there exists a neighbourhood $U_x$ of $x$, and an $M'<\infty$, such that for
all $y\in U_x$ and all $M\ge M'$,
\begin{equation}\label{jordgubb}
  \hgaxx(y)-\hgaxxMM(y)<\eps.
\end{equation}
This implies the result by a standard compactness argument.

Let, for $x>0$ and $\gth\ge0$, 
\begin{equation}\label{fx}
f_x(\gth):=x\loghmu(\gth)+\gth.  
\end{equation}
Let $\AAAo:=\max\set{\AAA,0}$.
By \refL{LX},
$f_x(\gth)=\infty$ if $0\le\gth<\AAAo$
and $f_x(\gth)<\infty$ if $\gth>\AAAo$.
Thus, \eqref{hgaxxinf} can be written 
\begin{equation}
  \hgaxx(x)
=\inf_{\gth\ge0}f_x(\gth)
=\inf_{\gth\ge\AAAo}f_x(\gth).
\end{equation}
\refL{LX} also implies that $f_x(\gth)$ is continuous and strictly convex on
$\AAAoo$. (We allow $f_x(\AAAo)=\infty$.)
If $x<\gam$, then $x+\eta<\gam$ for
some small $\eta>0$, which by \eqref{gamma1} implies $\gaxx(x+\eta)\ge0$ and
thus by \eqref{gaxx2}, for every $\gth>\ga$,
$(x+\eta)\loghmu(\gth)+\gth\ge0$ or
$\loghmu(\gth)\ge-\gth/(x+\eta)$; this yields
\begin{equation}
  f_x(\gth)\ge -x\frac{\gth}{x+\eta}+\gth = \frac{\eta\gth}{x+\eta}
\to\infty
\quad\text{as}\quad
\gth\to\infty.
\end{equation}
It follows that if  $0 < x < \gam$, then the strictly convex function
$f_x(\gth)$ attains its minimum at a unique point $\gth(x)\in\AAAoo$.

Assume first $\gth(x)>\AAAo$. 
Choose $\gth_1$ and $\gth_2$ with $\AAAo<\gth_1<\gth(x)<\gth_2<\infty$.
Then $f_x(\gth_1)>f_x(\gth(x))<f_x(\gth_2)$. Since $f_x\MM(\gth)\to
f_x(\gth)$ for every $\gth$ by \eqref{hmuMM}, it follows that 
for some $\gd>0$, some $M_0<\infty$ and all $M\ge M_0$,
\begin{equation}\label{mid3}
 f_x\MM(\gth_1)>f_x\MM(\gth(x))+\gd<f_x\MM(\gth_2).  
\end{equation}
If $\gth\in[\gth_1,\gth_2]$ and $M\ge M_0$, then
\begin{equation}
\hmu\MMo(\gth_2)
\le \hmu\MM(\gth_2)
\le\hmu\MM(\gth)
\le\hmu\MM(\gth_1)\le\hmu(\gth_1) 
\end{equation}
and thus
\begin{equation}\label{midsC}
|\log\hmu\MM(\gth)|\le 
C:=\max\bigset{|\loghmu\MMo(\gth_2)|,|\loghmu(\gth_1)|}.  
\end{equation}
Let $\gd_1=\gd/(3C)$. 
By \eqref{fx} and \eqref{midsC}, 
if $|y-x|<\gd_1$, $M\ge M_0$ and $\gth\in[\gth_1,\gth_2]$,
then
$|f_x\MM(\gth)-f_y\MM(\gth)|\le C|x-y|\le\gd/3$. Hence, \eqref{mid3} implies
\begin{equation}\label{mid3y}
 f_y\MM(\gth_1)>f_y\MM(\gth(x))<f_y\MM(\gth_2).  
\end{equation}
Since $f_y\MM(\gth)$ is convex, this implies that the minimum of
$f_y\MM(\gth)$ is attained at some $\gth\MM(y)\in[\gth_1,\gth_2]$, and the
same holds for $f_y(\gth)$.
Consequently, if $|y-x|<\gd_1$ and $M\ge M_0$, then,
using \eqref{fx} again,
\begin{equation}\label{oberon}
  \begin{split}
  \hgaxx(y)-\hgaxxMM(y)
&=\inf_{\gth\in[\gth_1,\gth_2]}f_y\xpar{\gth}
-
\inf_{\gth\in[\gth_1,\gth_2]}f_y\MM\xpar{\gth}
\\&
\le
\sup_{\gth\in[\gth_1,\gth_2]}\bigpar{f_y(\gth)-f_y\MM\xpar{\gth}}
\\&
\le
(x+\gd_1)\sup_{\gth\in[\gth_1,\gth_2]}
\bigpar{\log\hmu(\gth)-\log\hmu\MM\xpar{\gth}}. 	
  \end{split}
\end{equation}
Since $\hmu(\gth_1)<\infty$, it follows easily by \eqref{hmub} and
dominated convergence that
$\hmu\MM(\gth)\to\hmu(\gth)$ uniformly on $[\gth_1,\gth_2]$ as \Mtoo, and 
thus it follows from \eqref{oberon} that 
\eqref{jordgubb} holds
if $M$ is large enough, 
for all $y>0$ with $|y-x|<\gd_1$.

Next, assume $x>0$ and $\gth(x)=\AAAo$. (This entails $f_x(\AAAo)<\infty$
and thus $\hmu(\AAAo)<\infty$.)
If $\AAAo=0$, we choose $\gth_1=0$ and some $\gth_2>\gth(x)=0$;
we then can argue as above with minor modifications. Again, for $y-x$ small
and $M$ large, the infima are
attained on $[\gth_1,\gth_2]$ and \eqref{oberon} holds.

If $\gth(x)=\AAAo=\AAA>0$, we  choose $\gth_1$ and $\gth_2$ with
$0<\gth_1<\gth(x)<\gth_2<\infty$.
Then 
the argument above leading to \eqref{mid3y},
but ignoring $\gth_1$, shows that
if $|y-x|$ is small and $M$ large, then 
\begin{equation}
  f_y\MM(\AAA)<f_y\MM(\gth_2).
\end{equation}
Furthermore, 
$\hmu\MM(\gth_1)\upto\hmu(\gth_1)=\infty$,
and it follows 
from \eqref{fx} that $f_{x/2}\MM(\gth_1)\to\infty$ as \Mtoo.
Consequently, $M$ is large enough and $|y-x|<x/2$, then
\begin{equation}\label{puck}
  f_y\MM(\gth_1)\ge f_{x/2}\MM(\gth_1)>f_{2x}(\AAA)\ge f_y(\AAA)
\ge f_y\MM(\AAA).
\end{equation}
Hence we see again that 
if $|y-x|$ is small and $M$ large,
then \eqref{mid3y} holds
and thus the minimum of $f_y\MM(\gth)$ is attained in 
$[\gth_1,\gth_2]$.
Similarly, in analogy with \eqref{puck}, $f_y(\AAA)<f_y(\gth_2)$ if $|y-x|$
is small, and then the minimum of $f_y(\gth)$ is attained in $[\AAA,\gth_2]$
(since $f_y(\gth)=\infty$ for $\gth<\AAA$).

Furthermore,
since $\hmu(\gth)$ is decreasing, we have for $\gth\le\AAA$ and any $y>0$
\begin{equation}
  f_y\MM(\gth)\ge y\loghmu\MM(\AAA)+\gth
=f_y\MM(\AAA)-(\AAA-\gth).
\end{equation}
Consequently, provided we have chosen $\gth_1>\AAA-\eps/2$,
\begin{equation}
\inf_{\gth\in[\gth_1,\gth_2]}  f_y\MM(\gth)
\ge 
\inf_{\gth\in[\AAA,\gth_2]}  f_y\MM(\gth)-\eps/2.
\end{equation}
It follows that if $|y-x|$ is small and $M$ large, then
\begin{equation}\label{titania}
  \begin{split}
  \hgaxx(y)-\hgaxxMM(y)
&=\inf_{\gth\in[\AAA,\gth_2]}f_y\xpar{\gth}
-
\inf_{\gth\in[\gth_1,\gth_2]}f_y\MM\xpar{\gth}
\\&
\le
\inf_{\gth\in[\AAA,\gth_2]}f_y\xpar{\gth}
-
\inf_{\gth\in[\AAA,\gth_2]}f_y\MM\xpar{\gth}+\eps/2
\\&
\le
\sup_{\gth\in[\AAA,\gth_2]}\bigpar{f_y(\gth)-f_y\MM\xpar{\gth}}+\eps/2
\\&
\le
(x+\gd_1)\sup_{\gth\in[\AAA,\gth_2]}
\bigpar{\log\hmu(\gth)-\log\hmu\MM\xpar{\gth}}
+\eps/2.
  \end{split}
\end{equation}
In the present case $\hmu(\AAA)<\infty$ and it follows
that $\hmu\MM(\gth)\to\hmu(\gth)$ uniformly on $[\AAA,\gth_2]$,
\cf{} the statement after \eqref{oberon}. Hence, the \rhs{} of
\eqref{titania}
is smaller than $\eps$ if $M$
is large enough, so \eqref{jordgubb} holds in this case too for $|y-x|$
small and $M$ large.

Finally, assume $x=0$ and $\AAA\le0$. Then $\hgaxx(0)=0$ by \eqref{hgaxx0} 
and since $\hgaxx(x)$ is continuous at 0 by \refL{Lhgaxx}, there exists
$\gd$ such that $\hgaxx(y)<\eps$ for $0\le y<\gd$. Consequently, for such $y$
and all $M$,
\begin{equation}
  \hgaxx(y)-\hgaxxMM(y)
\le \hgaxx(y)<\eps.  
\end{equation}
This completes the proof that \eqref{jordgubb} holds in all cases, for $y$
in a neighbourhood of $x$ and large $M$.
\end{proof}

\begin{example}
  \label{EPBST}
  For the binary search tree already treated in Examples \ref{EHBST} and
  \ref{EFBST} we have, by \eqref{baraimst},
$\barai=1$,  and thus by \eqref{lhgaxx}
  and \eqref{gaxxxbst}
\begin{align}
  \label{hgaxxbst}
\hgaxx(x)&=
\begin{cases}
x\log 2,&0\le x\le 1, \\
x\log 2-x\log x+x-1,& x\ge 1.
\end{cases}
\end{align}
The profile is asymptotically given by \eqref{tn=}, with $\ga=1$, uniformly
for $x$ in any compact subset of $[0,\gam)$.
In particular, we see as in \refE{Ebaraimary} that for $k\le \log n$, a
``large part'' of all $2^k$ possible nodes at depth $k$ are occupied;
see also \refT{Talla} and \refE{Eallamst} below.
More precise results (obtained by other methods) are given by
\cite{ChauvinDrmotaJ2001}, 
\cite{ChauvinKMR2005}, 
\cite{DrmotaHwang2005}, 
\cite{DevroyeHwang2006}, 
\cite{FuchsHN2006},
\cite{DrmotaJansonNeininger2008}
and \cite[Section 6.5]{Drmota}.
\end{example}

\begin{example}\label{EPRRT}
  For the \rrt{} in Example \ref{EHRRT} 
we have, by \eqref{hmuRRT},
$\hmu(\gth)=1/\gth$, $\gth>\AAA=0$.
Hence, $\hmu(0)=\infty$, and thus $\barai=0$.
Consequently, by \eqref{lhgaxx}
  and \eqref{gaxxx2},
\cf{} \eqref{gaxxRRT},
\begin{align}
  \label{hgaxxrrt}
\hgaxx(x)&=
\gaxxx(x)=
-x\log x+x,
\qquad x\ge0.
\end{align}
The profile is asymptotically given by \eqref{tn=}, with $\ga=1$, uniformly
for $x$ in any compact subset of $[0,\gam)$.
In this case, more precise results (obtained by other methods) are given by
\cite{DevroyeHwang2006}, 
\cite{FuchsHN2006} 
and \cite[Section 6.3]{Drmota}.
\end{example}

\begin{example}\label{EPPREF}
For a simple example with $\AAA>0$, 
consider the (positive) linear preferential attachment tree in \refE{Elinear+},
with a parameter $\rho>0$ and, by \eqref{hmulin},
$\hmu(\gth)=\rho/(\gth-1)$ for $\gth>\AAA=1$.
Recall that $\ga=\rho+1$ and $\gb=\rho\qw$ by \eqref{ga+} and \eqref{betalin}.
Since $\hmu(0)=\infty$, we have $\barai=0$.
A simple calculation shows that for $x>0$, the infima in \eqref{gaxxx2} and
\eqref{hgaxxinf} are attained at $\gth=1+x$, and thus
\begin{equation}\label{hgaxxpref}
  \hgaxx(x)=\gaxxx(x)=x\log\rho-x\log x+x+1,
\qquad x\ge0.
\end{equation}
We obtain also, by \eqref{gaxx2} and \eqref{chgaxx2}, or by \refR{Rgaxxx}
and \eqref{hgaxxpref},
see also \eqref{gaxxgb} and \eqref{chgaxxgb},
\begin{align}
  \gaxx(x)&=
  \begin{cases}
	\rho+1, & x\le\rho,
\\
x\log\rho-x\log x+x+1, & x\ge\rho,
  \end{cases}
\\
  \chgaxx(-x)&=
  \begin{cases}
-\infty, & x<0,
\\
x\log\rho-x\log x+x+1, &0\le x\le\rho,
\\
	\rho+1, & x\ge\rho.
  \end{cases}
\end{align}

The profile is asymptotically given by \eqref{tn=}, with $\ga=\rho+1$, uniformly
for $x$ in any compact subset of $(0,\gam)$; note that \eqref{tn=} does
\emph{not} hold for $x=0$, since
$\hgaxx(0)=1$  while  $n_0(T_n)=1$.
However, this is the only exception, and
\eqref{tn=} extends, uniformly, to all $x\le x_1<\gam/\ga$ satisfying the
obvious condition 
$x\log n\ge1$, see \refT{Tnotbad} and \refE{EPlin2} below.
(More precise results, obtained by other methods,
in the case $\rho=1$, \ie{} the plane oriented recursive tree in \refE{EPORT},
are given 
by \cite{Hwang2007} and \cite{Sulzbach2008}.)
\end{example}

The following example shows that if $\AAA>0$, the restriction in \refT{Tn=}
that $x$ lies in a compact subset of $(0,\gam)$, and thus stays away
from 0, is necessary; the theorem does not hold in general for $x=x(n)\to0$,
even if we assume $k=x\log n\ge1$ or $x\log n\to\infty$.

\begin{example}
  \label{Ebadapple}
Let $(t_j)$ be a rapidly increasing sequence with $t_1=1$ and
$t_{j+1}>t_j+1$, $j\ge1$.
Let each individual get $\floor{e^{t_j}-e^{t_{j-1}}}$ children born at 
uniformly random times in $[t_j-1,t_j]$ for each $j\ge1$. (Let here
$t_0:=-\infty$, say.)
Then
\begin{equation}
  \hmu(1)\ge\sumji\floor{e^{t_j}-e^{t_{j-1}}}e^{-t_j}=\infty
\end{equation}
while for any $\gth>1$,
\begin{equation}
  \hmu(\gth)\le\sumji\floor{e^{t_j}-e^{t_{j-1}}}e^{-\gth(t_j-1)}<\infty.
\end{equation}
Hence, $\AAA=1$. It is easy to verify \refBP.

At time $t=t_{j+1}-1$, each individual has at most $e^{t_j}$ children, and
thus $n_k(\ctt)\le e^{kt_j}$ for every $k\ge0$; hence
\begin{equation}\label{psk}
  \frac{\log n_k(\ctt)}{t} \le \frac{kt_j}{t_{j+1}-1}
\le \frac{2kt_j}{t_{j+1}}.
\end{equation}
For any given function $\go(t)$ with $\go(t)=o(t)$ as \ttoo, we can choose
$(t_j)$ such that $\go(t_{j+1})/t_{j+1}<1/(jt_j)$, and then \eqref{psk}
shows that
\begin{equation}
  \liminf_\ttoo\frac{\log n_k(\ctt)}{t}=0
\end{equation}
uniformly for $k\le C\go(t)$, for any fixed $C$.

As a consequence, using the weight $\psi=1$ in \refE{E1}
(and assuming as we may that $\go(t)/t$ is decreasing),
it follows that
\begin{equation}
  \liminf_\ttoo\frac{\log n_k(T_n)}{\log n}=0
\end{equation}
uniformly for $k\le\go(\log n)$.
In particular, \eqref{tn=} does not hold for all $x=x(n)\to0$,
so it does not hold uniformly for $x>0$
(even assuming $x\log n\ge1$).
(In this case, $\hgaxx(\ga x)\to\hgaxx(0)=\AAA>0$.)

The same argument applies also to $\nle{k}$ 
and shows that \eqref{tnle} does not hold uniformly for
all $x>0$ such that $x\log n\ge1$.
\end{example}

In \refE{Ebadapple},
 the birth times are distributed very irregularly, and \eqref{tn=} fails
 already for depth $x\log n=1$, the children of the root.
In more regular cases, \eqref{tn=} holds for depth 1; the following theorem
shows that then it holds for all depths $x\log n\ge1$ with $x\le x'$ for
some $x'<\gam/\ga$.

\begin{theorem}
  \label{Tnotbad}
Suppose that the point process $\Xi$ is such that \as{}
\begin{equation}\label{tn=Xi}
  \Xi([0,t])=e^{(\AAA+o(1))t}.
\end{equation}
as \ttoo.
Then \as,
for any $x_1<\gam$,
as \ttoo,
\begin{equation}\label{ew}
\frac{\log \nz{k}(\ctt)}{t}
= \hgaxx(k/t)+o(1),
\end{equation}
uniformly for $1\le k\le x_1t$.
Furthermore, \eqref{tnle} holds uniformly for all $x$
such that $x\log n\ge1$, and 
\eqref{tn=} holds uniformly for all $x$ 
in any compact subset of $[0,\gam/\ga)$
such that $x\log n\ge1$.
\end{theorem}

\begin{proof}
Recall first that by \refT{TBiggins1}, $H(\ctt)/t\asto\gam$. Hence, if we
let $b:=\gam/2$, then there exists $t_0$ such that with probability at least
$1/2$,
\begin{equation}\label{good}
  H(\ctt)> bt 
\qquad\text{for all } t\ge t_0.
\end{equation}

Say that an individual $x$ in the branching process is \emph{good} if
\eqref{good} holds for the subtree rooted at $x$ (with time measured from
the birth of $x$), and let $G(t)$ be the number of children of the root in
$\ctt$ that are good. Since the total number of children of the root is
$\Xi(\oot)$, and each is good with probability at least $1/2$ and these
events are independent of each other, the law of large numbers implies that
\as{} $G(t)\ge\frac{1}3\Xi\xoot$  for large $t$. 

Let $0<\gd<1$. We have seen that \as, for large $t$, the
number of good children of the root at time $(1-\gd)t$ is,
using the assumption \eqref{tn=Xi},
\begin{equation}\label{goody}
  G\bigpar{(1-\gd)t}\ge\frac{1}3\Xi([0,(1-\gd)t])
=e^{\AAA(1-\gd)t+o(t)}.
\end{equation}
Each of these good children sprouts a tree that at time $t$ has age at least
$\gd t$ and thus, by the definition of good, has at least one node in each
of the following $\ceil{b\gd t}$ generations (provided $t>t_0/\gd$).
Hence, for large $t$,
if $1\le k\le b\gd t$, then
$\nz k(\ctt) \ge  G\bigpar{(1-\gd)t}$ and thus, by \eqref{goody},
\begin{equation}\label{hw}
  \begin{split}
\frac{\log \nz{k}(\ctt)}{t}
&
\ge \AAA(1-\gd)-\gd
=\hgaxx(0)-(\AAA+1)\gd
.
  \end{split}
\end{equation}
Let $\eps>0$.
Since $\hgaxx(x)$ is continuous, 
we can find some small $\gd>0$ 
such that if $0\le x\le b\gd$, then $\hgaxx(x)<\hgaxx(0)+\eps$.
We may furthermore assume that $\gd<\eps/(\AAA+1)$, and then \eqref{hw}
implies that, with $x_0=b\gd$,
for any $k$ with $1\le k\le x_0t$,
\begin{equation}\label{birdie}
\frac{\log \nz{k}(\ctt)}{t}
\ge \hgaxx(k/t)-2\eps.
\end{equation}
For any $x_1<\gam$, 
the same inequality \eqref{birdie} holds, for large $t$, 
also for $k/t\in[x_0,x_1]$
by (the proof of) \refT{Tn=}, so \eqref{birdie} extends, for large $t$, 
to $1\le k\le x_1t$. 
Together with the upper bound \eqref{tn<3}, this yields \eqref{ew}.
The usual argument using the stopping times $\tau(n)$ yields \eqref{tn=} 
uniformly for $1\le x\log n\le x'\log n$, for any $x'<\gam/\ga$.
\end{proof}

\begin{example}\label{EPlin2}
For the linear preferential attachment in \refE{EPPREF}, 
the children are born according to a pure birth process with birth rates
$\gl_k=k+\rho$. From the birth of the first child, this
process 
stochastically dominates the Yule process $(\cY_t)$ in \refE{EYule1} (with
the standard
rate $\ga=1$), which
is a birth 
process with rates $\gl_k=k$.
Conversely, the process describing the births is dominated by a Yule process
started with $\ceil{\rho}$ individuals, \ie, the sum of $\ceil\rho$
independent Yule processes. It follows from \eqref{yule} that \as{}
\begin{equation}
  0<\liminf_{\ttoo} \bigpar{e^{-t}\Xi(\oot)} 
\le \limsup_{\ttoo} \bigpar{e^{-t}\Xi(\oot)}<\infty, 
\end{equation}
and in
particular \eqref{tn=Xi} holds.
Thus, \refT{Tnotbad} shows that if $x'<\gam/\ga$, then
\eqref{tn=} extends, uniformly, 
to all $x\le x'$ satisfying the obvious condition 
$x\log n\ge1$, as said in \refE{EPPREF} above.

The same arguments show that \eqref{tnle} extends
to all $x$ satisfying $x\log n\ge1$, now using the uniform 
upper bound \eqref{nlei+}. 
\end{example}

In the \mary{} case, we can give a sharper result for the lower part of the
profile, showing that in the range $k< (\barai/\ga)\log n$, the estimate
$m^{k+o(\log n)}$ in \eqref{tnlex} in \refE{Ebaraimary} can be improved to 
$m^k(1+o(1))$, at least in the case $\AAA<0$. (I.e., when the birth times
$\xi_i$ have some exponential moment.)

\begin{theorem}
  \label{Talla}
  In the \mary{} case $N=m$,
assume \refBP{} and \refBPlastpsi, 
and also
$\AAA<0$.
Then, $0<\gamx<\barai$ and for every $x_1\in(\gamx,\barai]$,
\as,
\begin{equation}\label{talla1}
  n_k(T_n)
=
m^k- n^{\chgaxx(-\ga k/\log n)/\ga+o(1)}
\end{equation}
uniformly for $\ga k/\log n\in[x_1,\barai]$.
Hence, for every $\eps>0$, \as
\begin{equation}\label{talla2}
  n_k(T_n)
=
m^k(1+o(1))
\end{equation}
uniformly for $k\le (\barai/\ga-\eps) \log n$.
\end{theorem}

\begin{proof}
For an \mary{} tree $T$, and $k\ge0$,
let $\nb_k(T):=m^k-n_k(T)$, \ie, the number of potential nodes of depth  $k$
that are  \emph{not} in $T$.
Furthermore, let $\vo_k(T)$ be the set of nodes in $T$ with depth $k$ that
have outdegree less than $m$, and let $\no_k(T):=|\vo_k(T)|$.
We similarly define $\no_{\le s}(T)$ for real $s\ge0$.

Every potential node that is not in $T$ has a unique ancestor in $T$ with
outdegree less than $m$, \ie, an ancestor in $\bigcup_{j=0}^\infty\vo_j(T)$,
and for every $j<k$, each node in $\vo_j(T)$ 
is the ancestor of between $m^{k-j-1}$ and $m^{k-j}$ potential nodes of
depth $k$ that are not in $T$.
Hence,
\begin{equation}\label{nb}
  \sum_{j=0}^{k-1} m^{k-j-1}\no_j(T)
\le \nb_k(T)
\le
  \sum_{j=0}^{k-1} m^{k-j}\no_j(T).
\end{equation}

 Since $\AAA<0$, we have $\gamx>0$ by \refL{Lgamx}, and
thus $\barai>\gamx$ by \eqref{baraig}.
Furthermore, since $\chgaxx(-x)$ is right-continuous by \refL{LXch},
\eqref{chgamma2} implies that $\chgaxx(\gamx)\ge0$.
Define, for this proof, 
\begin{equation}\label{aaa}
  \aaa(x):=
  \begin{cases}
	\chgaxx(-x), &   x\ge\gamx,
\\
\chgaxx(-\gamx)-(\gamx-x)\log m,& x\le\gamx.
  \end{cases}
\end{equation}
Recall that $\chgaxx(-x)$ is continuous on $[\baram,\infty)$ by \refL{LXch} and
  $\gamx\ge\baram$ by \eqref{gamxbaram};
hence 
$\aaa(x)$ is continuous on
  $(-\infty,\infty)$. 
Furthermore, $\aaa(x)$ is differentiable at every point except  
$\gamx$, $\gb\qw$ and, perhaps, $\barao$, see \refR{Rgaxxx}.
By \refL{Ltgaxx}\ref{Ltgaxx->} and \eqref{tgaxxgb},
$\aaa(x)=\chgaxx(-x)\le \tgaxx(x)=x\log m$ for $x\in[\gamx,\barai]$, with
equality for 
$x=\barai$, and thus the left derivative $(\aaa)'(\barai-)\ge\log m$.
Since $\chgaxx(-x)$ is convex, $(\aaa)'(x)\ge \log m$ for
$x\in(\gamx,\barai)$, and by the definition \eqref{aaa}, this trivially
holds for $x<\gamx$ too.
Consequently,
\begin{equation}\label{bbb}
    \aaa(y)\le \aaa(x)-(x-y)\log m,
\qquad y\le x\le \barai.
\end{equation}

We use \refT{TBiggins2u}\ref{tbiggins2u-} with characteristic $\chi$ given by
\eqref{chim}, noting that \ref{BIGG2} holds by the proof of \refT{Tfillup}.
This shows that,
for $x\in(\gamx,\infty)$,
  \begin{equation}\label{miserere}
\frac{\log \no_{\le xt}(\ctt)}t\asto\chgaxx(-x)
=\aaa(x).	
  \end{equation}
Since $\aaa(x)$ is continuous, we now can argue as in the proof of
\refT{Tnle}, see \eqref{kork}--\eqref{bra}, and obtain, \as,
  \begin{equation}
\limsup_{t\to\infty}\frac{\log \no_{\le xt}(\ctt)}t\le \aaa(x),
  \end{equation}
uniformly for $x\ge0$. (For $x<\gamx$, the limsup is actually 
$-\infty$,
since \as{} $\no_{\le xt}(\ctt)=0$ for large $t$.)
Consequently, by \eqref{tex2tau}, \as
  \begin{equation}
	\begin{split}
\limsup_{n\to\infty}\frac{\log \no_{\le x\log n}(T_n)}{\log n}
&=
\limsup_{n\to\infty}\frac{\log \no_{\le x(\log n/\tau(n))\tau(n)}(\cT_{\tau(n)})}{\tau(n)}.
\frac{\tau(n)}{\log n}	  
\\
&\le \aaa(\ga x)/\ga,
	\end{split}
\raisetag{0.8\baselineskip}
  \end{equation}
uniformly for $x\in[0,x_1]$, for any fixed $x_1$.

Consequently, for any $\gd>0$, \as, for large $n$
\begin{equation}
 \no_{\le x\log n}(T_n) \le n^{\aaa(\ga x)/\ga+\gd},
\qquad 0\le x\le \barai/\ga.
\end{equation}
For such $n$
and any $k\le(\barai/\ga)\log n$,
we then have, by \eqref{nb} and \eqref{bbb}, 
\begin{equation}
  \begin{split}
  \nb_k(T_n)&
\le \sum_{j=0}^{k-1} m^{k-j} n^{\aaa(\ga j/\log n)/\ga+\gd}
=
\sum_{j=0}^{k-1} n^{(k-j)\log m/\log n+\aaa(\ga j/\log n)/\ga+\gd}
\\&
\le
(\barai/\ga)\log n\cdot  n^{\aaa(\ga k/\log n)/\ga+\gd}.	
  \end{split}
\raisetag{\baselineskip}
\end{equation}
Since $\gd>0$ is arbitrary, this shows that
\as, for large $n$,
uniformly for $k\le(\barai/\ga)\log n$,
\begin{equation}\label{agnus}
  \nb_k(T_n)
\le  n^{\aaa(\ga k/\log n)/\ga+o(1)}.
\end{equation}

Similarly,
for any $x_1>\gamx/\ga$, \eqref{miserere} implies
  \begin{equation}
	\begin{split}
\lim_{n\to\infty}\frac{\log \no_{\le x\log n}(T_n)}{\log n}
&= \chgaxx(-\ga x)/\ga,
	\end{split}
\raisetag{0.8\baselineskip}
  \end{equation}
uniformly for $x\in[x_1,\barai/\ga]$, \as{} for any fixed $x_1$.
This yields, by \eqref{nb},
\as
\begin{equation}\label{agno}
 \nb_k(T_n)\ge \no_{\le k-1}(T_n) =n^{\chgaxx(-\ga k/\log n)/\ga+o(1)}, 
\end{equation}
uniformly for $\ga k/\log n\in[x_1',\barai]$, for any $x_1'>x_1$.
Combining \eqref{agnus} and \eqref{agno} (and changing $x_1$), we obtain 
\eqref{talla1}. 

Let $y\in(\gamx/\ga,\barai/\ga)$ and define $k:=\floor{y\log n}$.
Then $\ga k/\log n\to \ga y\in (\gamx,\barai)$, and thus,
using \refL{Ltgaxx}\ref{Ltgaxx->}\ref{Ltgaxx-gagb},
\begin{equation}
\chgaxx(-\ga k/\log n)\to\chgaxx(-\ga y)
<  \tgaxx(\ga y)
=\ga y \log m.
\end{equation}
Hence, for some $\gd=\gd(y)>0$ and large $n$,
\begin{equation}\label{vf}
  n^{\chgaxx(-\ga k/\log n)/\ga}
\le n^{ y\log m-\gd}
\le n^{ k\log m/\log n-\gd}
=m^k n^{-\gd}.
\end{equation}
For large $n$, \eqref{talla1} applies and yields by \eqref{vf}, \as,
\begin{equation}
  n_k(T_n)/m^k = 1 - O\bigpar{n^{-\gd+o(1)}}=1-o(1).
\end{equation}
Furthermore, for any \mary{} tree $T$, $n_k(T)/m^k$ is a decreasing function
of $k$, and thus \eqref{talla2} holds uniformly for all $k\le y\log n$.
\end{proof}

 Define, for an \mary{} tree $T$,
 \begin{equation}
\Sh(T):=\max\bigset{k:n_k(T)\ge \tfrac12  m^k}   
 \end{equation}
\ie, the last level that is at least half-full.
(We could consider any other fixed fraction.)

\begin{corollary}\label{Calla}
Assume
$N=m$,
\refBP{}, \refBPlastpsi{}
and $\AAA<0$. Then
$\Sh(T_n)/\log n\asto \barai/\ga$.
\end{corollary}

\begin{proof}
  By \eqref{talla2}, for every $\eps>0$, 
\as{} $\Sh(T_n)>(\barai/\ga-\eps)\log n$ for large $n$.

Conversely, if $x=\barai/\ga+\eps$, with $\eps>0$, then 
$\tgaxx(\ga x)<\ga x \log m$ by \eqref{tgaxxgb3}, and consequently, for some
$\gd>0$, \eqref{tnle} yields,
with $k:=\floor{x\log n}$, \as,
\begin{equation}
  n_k(T_n)
\le n_{\le x\log n}(T_n) =n^{x \log m-\gd+o(1)}
=m^k n^{-\gd+o(1)} = o\bigpar{m^k}.
\end{equation}
Consequently, \as, for large $n$,
$\Sh(T_n)<k\le x\log n$.
\end{proof}

\begin{example}\label{Eallamst}
  For the \mst, we have $\ga=1$ and
$\barai=1/H_{m-1}$, where $H_{m-1}$ as usual is the  harmonic number.
Thus, \refC{Calla} and, in more detail, \refT{Talla}, show that
\emph{most} possible nodes exist up to depth
$\approx \barai\log n=H_{m-1}\qw\log n$,
but not further. In contrast, by \refT{Tfillup},
\emph{all} possible nodes exist only up to depth $\approx \gamx\log n$,
where $\gamx<\barai$.
In the binary case, $\gamx\doteq0.37336$
by \refE{EFBST}, while $\barai=1$.
\end{example}

\begin{remark}\label{Rduhre}
  Since $n_k(T_n)\le|T_n|=O(n)$ \as{} by \refT{Tex}, it follows from
  \eqref{talla2} that in the \mary{} case,
$\barai/\ga\le1/\log m$. This can also easily be seen
  analytically, and with strict inequality.
In fact, since $\log\hmu(\gth)$ is strictly convex on $[0,\ga]$,
\begin{equation}
\barai\qw=- (\log\hmu)'(0+)
>
-\frac{\log\hmu(\ga)-\log\hmu(0)}{\ga}
=\frac{\log m}{\ga}
\end{equation}
Hence,
\begin{equation}
  \barai <\ga/\log m.
\end{equation}
\end{remark}

\subsection{Typical depth}\label{SSdepth}

Given a rooted tree $T$, let the random variable $\hx(T)$ denote the 
depth of a random node, \ie, $\hx(T):=\dep(v)$ where $v$ is chosen uniformly at
random among the nodes of $T$.
It follows from the results on the profile in \refSS{SSprofile}
that 
a typical node in $T_n$ has depth $\approx (\ga\gb)\qw\log n$.
More precisely, we can show the following.

\begin{theorem}
  \label{Ttypical}
Assume \refBP{} and \refBPlastpsi.
Then, for every $\eps>0$, there exists $\gd>0$ such that \as,
conditioning on the tree $T_n$,
\begin{equation}\label{typicalgd}
  \P\bigpar{\bigabs{\hx(T_n)/\log n -(\ga\gb)\qw}>\eps\mid T_n}
=o\bigpar{n^{-\gd}}.
\end{equation}
In particular,
\begin{equation}
  \label{typicalp}
\bigpar{\hx(T_n)/\log n \mid T_n}\pto (\ga\gb)\qw \qquad\text{a.s.}
\end{equation}

Furthermore, for any $r>0$, 
\begin{equation}\label{typicalEas}
 \E \bigpar{\hx(T_n)^r\mid T_n}/\log^r n \asto (\ga\gb)^{-r}
.
\end{equation}
 \end{theorem}

\begin{proof}
By Theorems \ref{Tnge} and \ref{Tnle},
assuming as we may $\eps<(\ga\gb)\qw$ and $(\ga\gb)\qw+\eps<\gam/\ga$,
and using also \eqref{tex2e}, \as{}
  \begin{equation}
	\begin{split}
&\P\bigpar{|\hx(T_n)/\log n-(\ga\gb)\qw|\ge \eps\mid T_n}
\\&\qquad
=\frac{\nge{((\ga\gb)\qw+\eps)\log n}(T_n)
+\nle{((\ga\gb)\qw-\eps)\log n}(T_n)}{|T_n|}
\\&\qquad
=\frac{n^{\gaxx(\gb\qw+\eps\ga)/\ga+o(1)} + n^{\tgaxx(\gb\qw-\eps\ga)/\ga+o(1)}}
{n^{1+o(1)}}
\\&\qquad
=n^{-\gd_1+o(1)},
	\end{split}
  \end{equation}
where 
\begin{equation}
\gd_1:=1-\max\bigset{\gaxx(\gb\qw+\eps\ga)/\ga,\tgaxx(\gb\qw-\eps\ga)/\ga}>0  
\end{equation}
by \eqref{gaxxgb} and \eqref{tgaxxgb2}. Hence \eqref{typicalgd} holds for every
$\gd<\gd_1$. 

Obviously, \eqref{typicalgd}  implies \eqref{typicalp}.

Furthermore, by \refT{Theight}, 
\as{} for all large $n$, 
\begin{equation}
\hx(T_n)/\log n\le H(T_n)/\log n\le \gam/\ga+1<\infty.  
\end{equation}
Consequently, the sequence of random variables
$\bigpar{\hx(T_n)^r/\log^r n \mid T_n}$
is \as{} bounded, and thus \eqref{typicalp} implies \eqref{typicalEas} by
dominated (or bounded) convergence.
\end{proof}

In \eqref{typicalEas}, we consider
the conditional distribution and conditional moments given $T_n$.
(This is the quenched version, see \refR{Rk}.)
We can also consider 
the unconditional distribution and moments (\ie, the annealed version); this 
means that we first sample a random $T_n$ and then a random node $v$ in $T_n$,
and consider the depth of $v$ in $T_n$.

\begin{theorem}  \label{TMtypical}
Assume \refBP{} and \refBPlastpsi.
Then
\begin{equation}\label{typicalpQ}
  \hx(T_n)/\log n \pto (\ga\gb)\qw.
\end{equation}

Moreover, if also \ref{BP+} holds,
then, for any $r>0$, 
\begin{equation}\label{typicalE}
 \E \bigpar{\hx(T_n)^r}/\log^r n \to (\ga\gb)^{-r}
.
\end{equation}
 \end{theorem}

\begin{proof}
The unconditional convergence \eqref{typicalpQ} follows from 
\eqref{typicalgd} (or \eqref{typicalp}) by taking the expectation.

Moreover,  we have $\hx(T_n)\le H(T_n)$. 
Hence, \refL{LM3} holds also for $\hx(T_n)$, and
\eqref{typicalE} follows from \eqref{typicalpQ} by the argument in the proof
of \refT{Ton}. 
\end{proof}

\begin{remark}\label{Rdepthkeys}
  For an \mst, we can also consider the depth of a uniformly random key.  
Since the number of keys per node is bounded, and \eqref{olle} (with
$\phi=\psi$) holds, it is easy to see that \refT{Ttypical} holds in this
setting too.

Similarly, one might consider \eg{} the depth of a random leaf in $T_n$.
\end{remark}

\begin{remark}
\citet[Section 8.5, in particular Corollary 8.5.4]{Jagers} 
has shown much more refined results for a fixed time $t$, showing that the 
generation number is asymptotically normally distributed under some
conditions.
(See also \cite{Martin-Lof} for a special case.)
However, as remarked there, the results proved in \cite{Jagers} do not show
\as{} 
convergence as \ttoo, and thus we cannot obtain results for $T_n=\cttaux$ by
our usual technique. We conjecture that results analoguous to 
the ones in \cite[Section 8.5]{Jagers} hold, so that the depth is
asymptotically normal with variance $\Theta(n)$ under very general
conditions, but we leave this as an open problem.  
\end{remark}

\begin{remark}
Consider the \mary{} case $N=m$.
  Since $T_n$ has about $n$ nodes by \refT{Tex}, and most of them have depth
  $\approx(\ga\gb)\qw\log n$, we necessarily have
$(\ga\gb)\qw\ge1/\log m$ in order to have room for all nodes.
This can also easily be seen  analytically,  with strict inequality.
Let $\bxi$ be the birth time $\xi_i$ of a randomly chosen child. (I.e, $i$
is chosen uniformly at random in \set{1,\dots,m}.)
Then, recalling \eqref{malthustau},
\begin{equation}
  \E e^{-\ga\bxi}=\frac{1}{m}\sum_{i=1}^m\E e^{-\ga\xi_i}=\frac{1}m.
\end{equation}
Jensen's inequality with the strictly convex function $x\log x$ yields,
with strict inequality since $\bxi$ is not concentrated at a single value
by \ref{BPnonlattice}, using \eqref{el3},
\begin{equation}
  \frac{1}{m}\log\frac{1}{m}
< \E \bigpar{-\ga\bxi e^{-\ga\bxi}} 
=-\ga \frac{1}{m}\sum_{i=1}^m \E\bigpar{\xi_i e^{-\ga\xi_i}}
=-\frac{\ga\gb}m.
\end{equation}
Thus $\ga\gb<\log m$, and $(\ga\gb)\qw>1/\log m$.
Compare \refR{Rduhre}.
\end{remark}

\subsection{Total path length}\label{SSpathlength}

The total path length $\TPL$ is
defined (for any rooted tree) as
the sum of the depths of all nodes:
\begin{equation}\label{TPL}
  \TPL(T):=\sum_{v\in T} \dep(v).
\end{equation}
The total path length is closely connected to the typical depth $\hx(T)$
studied in \refSS{SSdepth};
it follows from the definitions of $\hx(T)$ and $\TPL(T)$ that for any
fixed rooted tree $T$,
\begin{equation}
  \E\bigpar{\hx(T)}
=\frac{\sum_{v\in T}\dep(v)}{|T|}
=\frac{\TPL(T)}{|T|}.
\end{equation}
For a random tree, we thus obtain the same result for 
the conditional expectation:
\begin{equation}\label{tpl1}
  \E\bigpar{\hx(T)\mid T}
=\frac{\TPL(T)}{|T|}.
\end{equation}
In other words, $\TPL(T)$ equals the conditional expectation of $\hx(T)$
times the number of nodes.

\begin{theorem}
  \begin{thmenumerate}
  \item\label{tpli}%
Assume \refBP{} and \refBPlastpsi. Then
  \begin{equation}\label{tpla}
	\frac{\TPL(T_n)}{n\log n} \asto 
\frac{1}{\ga\gb \mpsi}.
  \end{equation}
\item \label{tplii}%
If furthermore $|T_n|\le Cn$ (deterministically) for some constant $C$,
then also
  \begin{equation}\label{tplb}
	\frac{\E\TPL(T_n)}{n\log n} \asto 
\frac{1}{\ga\gb \mpsi}.
  \end{equation}
  \end{thmenumerate}
\end{theorem}

\begin{proof}
\pfitemref{tpli}
Applying \eqref{tpl1} to $T_n$, we obtain
\begin{equation}
  \frac{\TPL(T_n)}{n\log n}
=\frac{|T_n|}{n}\cdot\frac{\E\bigpar{\hx(T_n)\mid T_n}}{\log n}
\end{equation}
and the result follows by \eqref{tex2e} and \eqref{typicalEas}.

\pfitemref{tplii}
By \eqref{TPL}, $L(T)\le |T| H(T)$. 
Hence, the assumption implies 
$L(T_n)/(n\log n)\le C H(T_n)/\log n$, and \eqref{tplb} follows from
\eqref{tpla} and \refL{LM3} by the argument in the proof of \refT{Ton}.
\end{proof}

\begin{example}
  For the \bst{} in \refE{EBST}, we use the weight $\psi=1$ and thus $\mpsi=1$;
  furthermore $\ga=1$ and $\gb=1/2$. Hence,
$\TPL(T_n)\sim 2n\log n$ a.s.
In this case, much more detailed results are known, see \eg{}
\cite{Regnier1989}, \cite{Rosler1991} and \cite{Rosler2001}.
\end{example}

\begin{remark}
  We can similarly study versions of the path length with the summation
  \eqref{TPL}
only over  a subset of all nodes $v$, for example 
just summing over the set of leaves $v$. We leave the details to the reader.
In an \emst, two standard examples are the internal
path length and the external path length. 
Another version for the \mst{} is the sum of the depths of all keys, \cf{}
\refR{Rdepthkeys}, see e.g., \cite{Baeza-Yates} and
\cite{Mahmoud:pathlength}. This is more natural in the case of $m$-ary
search trees since this is the natural measure for the efficiency of the
corresponding sorting algorithm. The total path length using the sum over
all nodes was considered in e.g., \cite{Holmgren2011} for studying cuttings
in split trees (there the $m$-ary search tree was given as one
example). Both versions of the path lengths i.e., the sum over all keys,
respectively the sum over all nodes were considered in e.g.,
\cite{BroutinHolmgren2012} (for the analysis of general split trees).  
\end{remark}

\section{Asymptotic normality?}\label{Sfurther}
As said in the introduction, it is natural to try to show
asymptotic normality of the number of fringe trees of a given type.
There are several previous results of this type for special cases.
Central limit
laws for fringe trees have been shown, by several different methods, 
for binary search trees in \eg{}
\cite{Devroye1991}, \cite{Devroye2002}, \cite{FlajoletGM1997}, \cite{Fuchs},
\cite{ChangF10}  and
\cite{HolmgrenJanson1},         
and
for random recursive trees in \eg{} 
\cite{Fuchs} and
\cite{HolmgrenJanson1}.
For $ m $-ary search trees, the situation is more complicated:
no results for general fringe trees have been published (this is work in
progress \cite{HolmgrenJansonSileikis}), but some special cases
(such as the degree distribution and 
the number of fringe trees of a given size)
and related quantities 
(the number of internal nodes)
have been treated,
and it turns out that central limit theorems hold for $m\le26$ but not for
$m\ge 27$, see \eg{}
\cite{MahmoudPittel1989}, \cite{Mahmoud:Evolution},
\cite{LewMahmoud1994}, \cite{ChernHwang2001},  \cite{Hwang2003},
\cite{ChauvinPouyanne}, 
\cite{FillKapur-mary}, 
\cite{SJ154}
and \cite{HolmgrenJanson2}. 
Further examples of asymptotic normality include the degree distribution of
plane oriented recursive trees (preferential attachment trees, see Example
\ref{EPORT})
\cite{BergeronFS92}, 
\cite{MS:trees}, 
\cite{SJ154}, \cite{SJ155}, 
and the number of internal nodes in \medianllbst{} for
$\ell\le58$, but not for  $\ell\ge59$, see
\cite{ChernHwang2001}, \cite{ChernHwangTsai2002}.
 
The examples of \mst{} and \medianllbst{s}
thus show that central limit theorems do not
always hold for fringe trees of the random trees generated by \CMJbp{es} 
as in the present paper.

\begin{problem}\label{PCLT}
Find a
characterization of the \CMJ{} processes that yield asymptotic normality for
the number of fringe trees of a given type.
\end{problem}

Using the methods of \refS{S:BP}, \refP{PCLT} can be seen as a special case
of the following problem for branching processes:
\begin{problem}\label{P2}
Find a
characterization of the \CMJbp{es}
such that for suitable
characteristics $\phi$ and $\psi$, and with $\tau(n)$ as in \refS{S:BP},
$\zgf_{\tau(n)}$ is asymptotically normal as \ntoo.  
\end{problem}

\refP{P2} considers a stopped branching process. It is closely related to
the following problem for fixed times:
\begin{problem}\label{P3}
Find a
characterization of the \CMJbp{es} 
such that for suitable
characteristics $\phi$,
$(\zgf_{t}-m_\gf Z_t)/\sqrt{Z_t}$ is asymptotically normal as \ttoo.  
\end{problem}

This problem has been studied, at least for some branching processes.
\citet[Theorems VIII.3.1 and VIII.12.1]{AsmussenH} give a central limit
theorem of this type for a somewhat different class of branching processes,
\viz{} multi-type Markov branching processes. In principle, 
as pointed out in \cite{AsmussenH},
this class
includes the \CMJbp{es} studied here (with the ``type'' taken as the
entire previous history of the individual), but the resulting type space is
typically so large that the technical conditions in \cite{AsmussenH} are not
satisfied. (In particular, ``Condition (M)''.) 
However, for the \CMJbp{es} used in the examples above, with life histories
that are composed of one or several independent waiting times, 
the process can be described using a finite dimensional type space. It seems
that the results in \cite{AsmussenH} then apply and can be translated to
conditions for these \CMJbp{es}. 
Presumably, the same conditions then apply to Problems \ref{PCLT} and
\ref{P2} too, 
but that remains an open problem.

Moreover, it seems likely that the same type of conditions apply to much
more general \CMJbp{es}. The conditions in \cite{AsmussenH} are stated in
terms of eigenvalues of a certain  operator $A$ defined by the process, and the
result says (under some technical assumptions) that 
if $\gl_1$ is the largest eigenvalue of $A$ (this eigenvalue is real), 
then we have asymptotic
normality if every other eigenvalue $\gl$ has $\Re\gl\le\frac12\gl_1$,
but (typically, at least) not otherwise.
The same condition also appears in the different but closely related context of
generalized \Polya{} urns, see \cite{SJ154}. We conjecture that this
condition (in a suitable form) applies to rather general \CMJbp{es}.
This has been proved in the discrete-time case \cite{SJbranching},
but the continuous-time case relevant here is more challenging.

\begin{remark}
In contrast, 
for conditioned Galton--Watson trees 
(see \refR{RcGWt}),
asymptotic normality for fringe trees holds in general, see \cite{SJ285}.
(Such trees are not treated in the present paper.)
\end{remark}

\section*{Acknowledgements}
We thank Peter Jagers and Olle Nerman for helpful comments
on branching processes
and 
Shankar Bhamidi
for suggesting the use of the limiting sin-tree to compute the constant in
\refT{TMC2}.

\appendix

\section{Birth processes} 
\label{AA}

Recall that a \emph{pure birth process} with birth rates $\gl_k\ge0$
is a continuous-time stochastic process $(\cX_t)_{t\ge0}$, taking values in
$\bbZgeo$ and  with some given
initial value $\cX_0=x_0$ (usually 0 or 1), which is Markov and such that
when $\cX_t=k$, the process jumps to $k+1$ with rate $\gl_k$; equivalently,
the process jumps from $k$ to $k+1$ after an exponential waiting time
$Y_k\sim\Exp(\gl_k)$, and all these waiting times are independent.
(We allow $\gl_k=0$; we then let $Y_k=\infty$.)
Thus, 
the time the process jumps to $k$ is
\begin{equation}\label{sk}
S_k:=\sum_{i=x_0}^{k-1} Y_i,
\qquad k\ge x_0+1,
\end{equation}
and, for $k\ge x_0$, with $S_{x_0}:=0$,
\begin{equation}
  \label{wc}
\cX(T)=k \iff S_k \le t<S_{k+1}.
\end{equation}

\begin{example}\label{EA1}
In the branching process corresponding to a general preferential
attachment tree in \refE{Epref},
the children of an individual 
are born at the jumps of a pure birth process (with birth rates $\gl_k=w_k$);
the birth process $\cX_t$ is thus the number of children at time $t$.
\end{example}

\begin{example}\label{EPoisson}
  As a special case (see \refE{ERRT}), the counting process $\Xi[0,t]$
  corresponding to a Poisson process $\Xi$ with intensity 1
is a pure birth process with constant intensity $\gl_k=1$, started at 0.
More generally, a pure birth process with constant birth rate $\gl_k=\gl$,
started at 0, is
a Poisson process with intensity $\gl$.
(We have earlier defined a Poisson process as a point process; the
corresponding pure birth process considered here
is also called a Poisson process. There is an obvious equivalence between
the two points of view, 
and hardly any risk of confusion.)
\end{example}

\begin{example}\label{EYule1}
The \emph{Yule process} in \eg{} \refE{ERRT} is a pure birth process with
birth rates $\gl_k=k$, started at 1.
More generally, for a \CMJ{} process where each individual gets children
according to a Poisson process with intensity $\ga>0$, 
the total size (number of individuals $Z_t$) 
is a pure birth process with birth rates $\gl_k=\ga k$, started at 1; 
we call this a \emph{Yule process with rate $\ga>0$}.
(It evidently differs from the standard case $\ga=1$ only by a simple change
of time.) 

If $(\cY_t)$ is a Yule process with rate $\ga$, it  thus follows from
\eqref{olle1} that 
\begin{equation}\label{yule}
e^{-\ga t}  \cY_t \to W
\end{equation}
for some random variable $W$. (Note that the intensity measure $\mu$ is 
$\ga\ddx t$, so \eqref{malthus} holds and the Malthusian parameter $\ga$
equals the rate $\ga$.) It is easy to verify \eqref{xlogx} and thus $W>0$
\as; in fact, it follows from \eqref{yulet} below that $W\sim\Exp(1)$. (This
is one of the few cases with a simple explicit distribution for the limit $W$.)
\end{example}

We state a general result on stopping a pure  birth process by an
exponential clock $\tau$.

\begin{theorem}\label{TA1}
  Let $(\cX_t)$ be a pure birth process with birth rates $\gl_k\ge0$, 
started at $\cX_0=0$.
Furthermore, let $\tau\sim\Exp(\ga)$ be independent of the birth process.
Then $X:=\cX_{\tau}$ has the distribution
\begin{equation}\label{anna}
  \P(X=k)=\frac{\ga}{\gl_k+\ga} \prod_{i=0}^{k-1} \frac{\gl_i}{\gl_i+\ga}
=\frac{\ga\prod_{i=0}^{k-1}\gl_i}{\prod_{i=0}^{k} (\gl_i+\ga)},
\qquad k\ge0.
\end{equation}
\end{theorem}

We give two different proofs (both simple) to illustrate different ways of
arguing with exponential random variables; the first proof is more direct
probabilistic and the second more analytic.
(The second proof is essentially the same as \eqref{prefD}--\eqref{prefD=}
given in \refE{Epref}; it was there given for a special case but the
argument is general, as is shown below.)

\begin{proof}[First proof of \refT{TA1}]
Regard $\tau$ as an exponential random clock that strikes and stops the process.
When $\cX_t=k$ and $\tau>t$, so the process has not yet stopped,
the next event that happens is either that
the clock strikes (rate $\ga$), and then $X=k$, 
or that $\cX_t$ jumps to $k+1$ (rate $\gl_k$), and then $X>k$.
Consequently, 
\begin{align}
  \P(X=k\mid X\ge k) &=\frac{\ga}{\ga+\gl_k}
,\label{anna1}
\\
  \P(X\ge k+1\mid X\ge k) &=\frac{\gl_k}{\ga+\gl_k}
.\label{anna2}
\end{align}
By \eqref{anna2} and induction,
\begin{equation}
  \P(X\ge k)= \prod_{i=0}^{k-1} \frac{\gl_i}{\gl_i+\ga},
\end{equation}
and the result follows by \eqref{anna1}.
\end{proof}

\begin{proof}[Second proof of \refT{TA1}]
With the notation above, we have by \eqref{wc}
\begin{equation}
\P(X\ge k)=
\P(\cX_\tau\ge k)=
\P(S_k\le\tau).
\end{equation}
Conditioning on $S_k$, we have 
$\P(S_k\le\tau\mid S_k)=e^{-\ga S_k}$, and taking the expectation we find,
using \eqref{sk} and independence of $Y_1,\dots,Y_{k-1}$,
\begin{equation}
\P(X\ge k)
=\E e^{-\ga S_k}
=\prod_{i=0}^{k-1}\E e^{-\ga Y_i}
=\prod_{i=0}^{k-1}\frac{\gl_i}{\gl_i+\ga}.
\end{equation}
The result follows by taking the difference $\P(X\ge k)-\P(X\ge k+1)$.
\end{proof}

We consider in particular the linear case, when
\begin{equation}\label{linearbirth}
  \gl_k=\chi k+\rho,
\end{equation}
for some constants $\chi$ and $\rho$.
(As in \refE{Elinear+-}, only the ration $\chi/\rho$ matters, up to a change
of time scale, so we might assume $\chi\in\set{1,0,-1}$, but we shall not
require this.)
Note that Examples \ref{EPoisson} and \ref{EYule1} both are of this type,
with $(\chi,\rho)=$ $(0,\rho)$ and $(1,0)$ (or $(\ga,0)$), respectively.

Note that $\rho=\gl_0>0$, while $\chi$ can be any real number.
As in \refE{Elinear+-}, if $\chi<0$, we have to assume that $\rho=m|\chi|$
for a non-negative integer $m$ (and $\cX_0\le m$); then $\gl_m=0$ and the
process stops when it reaches $m$, so the values $\gl_k$, $k>m$, can be
ignored.

\begin{theorem}\label{TAlin}
  Let $(\cX_t)$ be a pure birth process with birth rates $\gl_k=\chi k+\rho$
as in \eqref{linearbirth}, for some $\chi$ and $\rho$, started at $\cX_0=0$.
Furthermore, let $\tau\sim\Exp(\ga)$ be independent of the birth process.
Then $X:=\cX_{\tau}$ has the distribution, when $\chi\neq0$,
\begin{equation}\label{talin}
  \begin{split}
  \P(X=k)&
=\frac{\ga\prod_{i=0}^{k-1}(\chi i+\rho)}{\prod_{i=0}^{k} (\chi i+\rho+\ga)}
\\&=
\frac{\ga}{\rho+\ga}\cdot\frac{\rise{\rho/\chi}k}{\rise{(\rho+\ga)/\chi+1}k}
,\qquad k\ge0.    
  \end{split}
\end{equation}
Thus, using the notation in \refD{DHG}, 
$X$ has the 
hypergeometric distribution $\HG\bigpar{\rho/\chi,1;(\rho+\ga)/\chi+1}$.

In the special case $\chi=0$ (so $\gl_k=\rho$ is constant),
we have instead
\begin{equation}\label{talin0}
  \begin{split}
  \P(X=k)&
=
  \frac{\ga}{\rho+\ga}\parfrac{\rho}{\rho+\ga}^{k}
,\qquad k\ge0.    
  \end{split}
\end{equation}
Thus, in this case
$X$ has the geometric distribution $\Geo\bigpar{\ga/(\ga+\rho)}$.
\end{theorem}
\begin{proof}
  An immediate corollary of \refT{TA1}.
\end{proof}

In the linear case, it is also easy to find the distribution of $\cX_t$
for a fixed
$t$. We begin with the expectation.

\begin{theorem}\label{TAext}
 Let $(\cX_t)$ be a pure birth process with birth rates $\gl_k=\chi k+\rho$
as in \eqref{linearbirth}, for some $\chi$ and $\rho$, started at $\cX_0=x_0$.
Then, for every $t\ge0$,  
\begin{equation}\label{Ext}
  \E\cX_t=
  \begin{cases}
 \Bigpar{\frac{\rho}{\chi}+x_0} e^{\chi t}-\frac{\rho}{\chi} & \chi\neq0,
\\
\rho t+x_0, &\chi=0.
  \end{cases}
\end{equation}
Consequently, for every $\chi$ and every $t\ge0$,
\begin{equation}\label{ddExt}
  \frac{\ddx}{\ddx t} \E\cX_t = (\rho+\chi x_0) e^{\chi t}.
\end{equation}
\end{theorem}
\begin{proof}
Since $\cX_t$ grows with a rate that is a linear function $\chi k+\rho$ of
the current state $k$, 
its expectation $\E\cX_t$ grows at rate $\chi\E\cX_t+\rho$, \ie,
\begin{equation}\label{ddext}
  \frac{\ddx}{\ddx t} \E\cX_t = \chi\E\cX_t+\rho.
\end{equation}
This differential equation,
with the initial value $\E\cX_0=x_0$, has the solution \eqref{Ext}.
(The reader that finds this argument too informal may note that
$\cX_t-\int_0^t(\chi\cX_s+\rho)\dd s$ is a martingale, and take the
expectation to obtain \eqref{ddext}.) 
\end{proof}

We give the distribution of $\cX_t$
only for the case $x_0=0$, leaving the general case
to the reader.
\begin{theorem}\label{TApxt}
 Let $(\cX_t)$ be a pure birth process with birth rates $\gl_k=\chi k+\rho$
as in \eqref{linearbirth}, for some $\chi\neq0$ and $\rho$, started at
$\cX_0=0$. 
Then, for all $t\ge0$ and $k\ge0$,  
\begin{equation}\label{Pxt}
\P(\cX_t=k)=
\frac{\rise{\rho/\chi}k}{k!}(1-e^{-\chi t})^k e^{-\rho t}
=
(-1)^k\binom{-\rho/\chi}{k}(1-e^{-\chi t})^k e^{-\rho t}.
\end{equation}
Equivalently, 
the \pgf{} is given by
\begin{equation}\label{pgfExt}
\E z^{\cX_t}=e^{-\rho t}\bigpar{1-z(1-e^{-\chi t})}^{-\rho/\chi}
=\parfrac{e^{-\chi t}}{1-z(1-e^{-\chi t})}^{\rho/\chi}.
\end{equation}

In the case $\chi>0$, this says
$\cX_t\sim\NBin(\rho/\chi,e^{-\chi t})$,
a negative binomial distribution.

In the case $\chi<0$, when necessarily $\rho=m|\chi|$ for some positive
integer $m$, \eqref{Pxt} can be written 
\begin{equation}\label{Pxt-}
\P(\cX_t=k)=
\binom{m}{k}(1-e^{-|\chi| t})^k e^{-(m-k)|\chi| t},
\end{equation}
and thus
$\cX_t\sim\Bin(\rho/|\chi|,1-e^{-|\chi| t})$,
a binomial distribution.

In the case $\chi=0$, we have instead
\begin{equation}\label{Pxt0}
\P(\cX_t=k)=
\frac{(\rho t)^k}{k!}e^{-\rho t}
\end{equation}
and
\begin{equation}\label{pgfExt0}
\E z^{\cX_t}=e^{(z-1)\rho t}.
\end{equation}
Thus $\cX_t\sim\Po(\rho t)$, a Poisson distribution, when $\chi=0$.
\end{theorem}

Note the well-known fact 
for the Poisson process in \refE{EPoisson},
we thus have 
\eqref{Pxt0}, yielding
a Poisson distribution $\Po(\rho t)$.

\begin{proof}
  With $p_k(t):=\P(\cX_t=k)$, we have the initial values $p_0(0)=1$,
  $p_k(0)=0$, $k\ge1$, and the 
Kolmogorov forward equation (with $p_{-1}:=0$)
\begin{equation}
  p'_k(t)=\gl_{k-1}p_{k-1}(t)-\gl_kp_{k}(t)
=(\chi(k-1)+\rho)p_{k-1}(t)-(\chi k+\rho)p_{k}(t).
\end{equation}
This system of differential
equations determines $p_k(t)$ uniquely for $k=0,1,\dots$,
and it is easily verified that the functions in \eqref{Pxt} and \eqref{Pxt0}
give a solution. The probability generating functions \eqref{pgfExt} and
\eqref{pgfExt0} follow by summation. (Alternatively, one may derive and
solve a partial differential for $g(z,t):=\E z^{\cX_t}$.)

In the case $\chi<0$ and $\rho=m|\chi|=-m\chi$,
\eqref{Pxt} can be written as
\begin{equation}
\P(\cX_t=k)=
\binom{m}{k}(e^{|\chi| t}-1)^k e^{-m|\chi| t}
\end{equation}
and \eqref{Pxt-} follows.
\end{proof}

\begin{example}\label{EYuleA}
  Let $\cY_t$ be a Yule process, see \refE{EYule1}. This birth process starts
  at $\cY_0=1$, so we consider instead $\cY_t-1$ which is a pure birth process
  started at 0, with birth rates $\gl_k=k+1$, $k\ge0$.
\refT{TApxt} applies with $\chi=\rho=1$ and yields 
the well-known result
$\cY_t-1\sim\NBin(1,e^{-t})=\Geo(e^{-t})$, and thus 
\begin{equation}\label{yulet}
\cY_t\sim\Gei(e^{-t}),
\qquad t\ge0.  
\end{equation}
For the expectation we get from \eqref{yulet}, or directly from \refT{TAext}
with 
$\chi=1$, $\rho=0$ and $x_0=1$, 
$\E \cY_t=e^t$.

Similarly, \refT{TAlin} applied to $\cY_t-1$ shows that if $\tau\sim\Exp(\ga)$
is independent of the Yule process, then 
\begin{equation}\label{yuletau}
\cY_\tau-1\sim\HG(1,1;\ga+2).  
\end{equation}
\end{example}

\section{Hypergeometric functions and distributions}\label{ASShyper}

Recall that the hypergeometric function $\FF(a,b;c;z)$ 
(also denoted by $\FFx(a,b;c;z)$) 
is defined by the sum
\begin{equation}
  \label{hyper}
\FF(a,b;c;z)
:=\sum_{n=0}^\infty \frac{\rise{a}n\rise{b}n}{\rise{c}n}\cdot\frac{z^n}{n!},
\end{equation}
see \eg{} \cite[\S15.2]{NIST} or \cite[\S5.5]{CM}.
In general, the parameters $a,b,c$ can be arbitrary complex numbers (except
that $c=0,-1,-2,\dots$ is allowed only in special cases), and $z$ may be a
complex number, but we are here only interested in real $a,b,c$ and $z$.
If $a\in\bbZleo$ or $b\in\bbZleo$, then the hypergeometric terms in
\eqref{hyper} vanish for $n>|a|$ or $n>|b|$, respectively, so
$\FF(a,b;c;z)$ is a polynomial; otherwise
the series \eqref{hyper}
converges for $|z|<1$ and diverges for $|z|>1$.
(The hypergeometric
function $\FF(a,b;c;z)$ extends by analytic continuation to
$z\in\bbC\setminus[1,\infty)$, but we have no use for this extension here.)

The hypergeometric series \eqref{hyper} 
converges for $z=1$ if and only if
$a\in\bbZleo$, $b\in\bbZleo$ (in these cases the sum is finite, as said
above)
or $\Re(c-a-b)>0$, and then its sum is, 
as shown by \citet{Gauss},
see also 
\cite[(15.4.20), (15.4.24)]{NIST}, 
\begin{equation}\label{gauss}
  \FF(a,b;c;1)=\frac{\gG(c)\gG(c-a-b)}{\gG(c-a)\gG(c-b)}.
\end{equation}

We say that a random variable has a \emph{(general) hypergeometric distibution}
if its \pgf{} is, up to a normalization constant, a hypergeometric function
$\FF(a,b;c;z)$, for some $a,b,c$.
We denote such a distribution by $\HG(a,b;c)$. (There seems
to be no standard notation.)
Some such distributions appear above in the study of random trees, and we
give here some general properties and examples of such
distributions, as a background and for easy reference.
See further \eg{} \citet[Chapter 6]{Johnson} and the references given there.

We repeat the definition somewhat more formally:
\begin{definition}
  \label{DHG}
The \emph{general hypergeometric distribution} $\HG(a,b;c)$ is the
distribution of a non-negative integer-valued random variable $X$ such that
\begin{equation}
  \label{hg}
\P(X=k)=C\frac{\rise{a}k\rise{b}k}{\rise{c}k\,k!}, \qquad k\ge0,
\end{equation}
for some constant $C$. 
Equivalently, the \pgf{} is
\begin{equation}\label{hgpgf}
  \E z^X = C \FF(a,b;c;z).
\end{equation}
\end{definition}

By \eqref{hgpgf},
the normalizing constant $C$ is necessarily $1/\FF(a,b;c;1)$
and thus, by  \eqref{gauss}, 
\begin{equation}\label{hgC}
  C=\frac{1}{\FF(a,b;c;1)}=\frac{\gG(c-a)\gG(c-b)}{\gG(c)\gG(c-a-b)}.
\end{equation}
Furthermore, the \pgf{} \eqref{hgpgf} is 
\begin{equation}
  \E z^X=\frac{\FF(a,b;c;z)}{\FF(a,b;c;1)}.
\end{equation}
Note the symmetry $\HG(a,b;c)=\HG(b,a;c)$.

\begin{remark}\label{RHG}
The definition \eqref{hg} is equivalent to
  \begin{equation}
  \frac{\P(X=k+1)}{\P(X=k)}
=\frac{(k+a)(k+b)}{(k+c)(k+1)}, \qquad k\ge0,	
  \end{equation}
with a suitable interpretation if $\P(X=k)=0$.
\end{remark}

A hypergeometric distribution $\HG(a,b;c)$ does not exist for all real
parameters $a,b,c$. We see from \eqref{hg} and \eqref{hgC} that a necessary
and sufficient condition for the existence of $\HG(a,b;c)$ is that
$\rise{a}k\rise{b}k/\rise{c}k\ge 0$ and that $\FF(a,b;c;1)<\infty$.
We do not give precise necessary and sufficient conditions for this here,
see \eg{} \cite{Johnson},
but we note the following cases where $\HG(a,b;c)$ exists;
these comprise
all cases of interest to us (and to others as far as we know), if we recall
that $a$ and $b$ can be interchanged.
\begin{romenumerate}
\item \label{hg+++}
$a,b> 0$, $c>a+b$.
The distribution has support $\bbZgeo$.
(Examples are \eqref{pdlin} and \refT{TAlin} for $\chi>0$.)
\item \label{hg-+-}
$a\in\bbZleo$, $b>0$, $c<a+1$.
The distribution has support \set{0,\dots,|a|}.
(Examples are \eqref{pd-} and \refT{TAlin} for $\chi<0$.) 
\item \label{hg--+}
$a\in\bbZleo$, $b<a+1$,  $c>0$.
The distribution has support \set{0,\dots,|a|}.
(One example is \refE{EHGclassical} below. Typically, as there, we have both
$a,b\in\bbZleo$ and $c>0$; then, if we do not assume $b\le a$, the support is
\set{0,\dots,\min(|a|,|b|)}.)
\end{romenumerate}
 
\begin{remark}\label{RHG0}
  As said above, the hypergeometric function \eqref{hyper} in general
does not exist
  when $c\in \bbZleo$.
However, it is still possible to define $\HG(a,b;c)$ is some cases.
We assume that also $a\in\bbZleo$. (Of course, the case $b\in\bbZleo$ is
similar, by symmetry.)
If $c<a$ (and $b>0$, included in \ref{hg-+-}), there is no problem with the
definition 
\eqref{hg}, letting $X$ have support \set{0,\dots,|a|}.
Also the case $c=a$ is interpreted in this way; in particular, note that
$\HG(-n,1;-n)$ is the uniform distribution on \set{0,\dots,|a|}
\cite[\S 6.10.1]{Johnson}.

In the case $a,c\in\bbZleo$ and $a<c$,  \eqref{hg} does not make sense,
since it would yield infinite values.
We can extend the definition to this case by rewriting \eqref{hg} as
\begin{equation}
  \label{hg0}
\P(X=k)=C'\frac{\rise{a}k\rise{b}k}{\gG(c+k)\,k!}, \qquad k\ge0,
\end{equation}
\cf{} \eqref{rise}, where now 
$C'=\gG(c)C=\gG(c)/\FF(a,b;c;1)$, which exists by
continuity also for $c\in\bbZleo$. (The function $\FF(a,b;c;z)/\gG(c)$ is
denoted $\mathbf{F}(a,b;c;z)$ in \cite{NIST}.)
With this interpretation, we can define $\HG(a,b;c)$ also in this case,
for suitable $b$.
(See \refE{EHGclassical} for an example.)
Note that for this case, $\P(X=k)=0$ for $0\le k\le |c|$. 
It is easily verified that if $X\sim\HG(a,b;c)$ with $c=-n\in\bbZleo$,
then
$X=X'+n+1$ with $X'\sim \HG(a+n+1,b+n+1;n+2)$. (This makes it possible to
reduce to the case $c\notin\bbZleo$.)
\end{remark}

\begin{example}\label{EHGclassical}
  The classical hypergeometric distribution describes the number of red
  balls in a sample of $m$ balls drawn (without replacement) from a 
  population of $N$ balls, of which $n$ are red;
see \eg{} \cite[Section II.6]{FellerI}. 
(Here $n,m,N$ are integers with
  $0\le n\le N$ and $0\le m\le N$.)
The probability function is
  \begin{equation}\label{HGc}
	\begin{split}
\P(X=k)&=\frac{\binom{n}{k}\binom{N-n}{m-k}}{\binom{N}{m}}
\\
&=\frac{n!\,(N-n)!\,(N-m)!\,m!}{k!\,(n-k)!\,(m-k)!\,(N-m-n+k)!\,N!}
\\
&=\frac{(N-m)!\,(N-n)!}{N!\,(N-m-n)!}
\cdot\frac{\rise{-m}k\rise{-n}k}{\rise{N-m-n+1}k}.
	\end{split}
  \end{equation}
(In the case $m+n>N$, the final line in \eqref{HGc} is not valid, but it can
be interpreted as in \refR{RHG0}.)
This is thus the distribution $\HG(-m,-n;N-m-n+1)$.
\end{example}

\begin{remark}
  In the case $a=-n\in\bbZleo$, when $X\sim\HG(a,b;c)$ satisfies $0\le X\le n$,
it is easy to see that $n-X\sim \HG(a,a-c+1;a-b+1)=\HG(-n,1-n-c;1-n-b)$.
(Natural examples can be seen from \refE{EHGclassical}.)
\end{remark}

\begin{remark}
  Most hypergeometric distributions in this paper are of the special form
  $\HG(a,1;c)$, \ie, with $b=1$.
In this case, \eqref{hg} and \eqref{hgC} simplify and yield
\begin{equation}\label{waring}
  \P(X=k)=\frac{c-a-1}{c-1}\cdot\frac{\rise{a}k}{\rise{c}k}.
\end{equation}
Such distributions are also called \emph{Waring distributions} 
\cite[\S  6.10.4]{Johnson}. 
\end{remark}

As said above a hypergeometric variable $X\sim\HG(a,b;c)$ with $a\in\bbZleo$ or
$b\in\bbZleo$ is bounded, and thus has moments of all orders. 
If $a,b\notin\bbZleo$, the distribution has a power-law tail, and thus only
a finite number of moments. We give a precise 
asymptotic  formula for $\P(X=k)$ and then formulas for (factorial) moments.

\begin{theorem}\label{THGtail}
  Suppose that $X\sim\HG(a,b;c)$, and that $a,b\notin\bbZleo$. Then, as
  \ktoo,
  \begin{equation}
\P(X=k)\sim \frac{\gG(c-a)\gG(c-b)}{\gG(a)\gG(b)\gG(c-a-b)} k^{a+b-c-1}.	
  \end{equation}
\end{theorem}
\begin{proof}
  By \eqref{hg} (or \eqref{hg0}), \eqref{hgC} and \eqref{rise},
	  \begin{equation}
\P(X=k)=
\frac{\gG(c-a)\gG(c-b)}{\gG(a)\gG(b)\gG(c-a-b)}\cdot
\frac{\gG(k+a)\gG(k+b)}{\gG(k+c)\gG(k+1)}.
	  \end{equation}
The result follows since $\gG(x+s)/\gG(x)\sim x^s$ as $x\to\infty$ for
every fixed $s$.
\end{proof}

\begin{theorem}\label{THGmom}
Suppose $X\sim\HG(a,b;c)$ and $m\in\bbZgeo$.
  \begin{romenumerate}
  \item \label{THGmomNo}
If $a,b\notin\bbZleo$ and $m\ge c-a-b$, then the moment $\E X^m$ and the
factorial moment $\E\fall Xm$ are infinite.
\item \label{THGmomYes}
If $a\in\bbZleo$,  $b\in\bbZleo$ or $m<c-a-b$,
then
  \begin{equation}\label{EHGm}
	\E \fall{X}m= \frac{\rise{a}m\rise{b}m}{\rise{c-a-b-m}m}.
  \end{equation}
In particular, for $a\in\bbZleo$,  $b\in\bbZleo$ or $c>a+b+1$,
\begin{equation}\label{EHG}
  \E X =\frac{ab}{c-a-b-1}
\end{equation}
and, for $a\in\bbZleo$,  $b\in\bbZleo$ or $c>a+b+2$,
\begin{equation}\label{VHG}
  \Var X =\frac{ab(c-a-1)(c-b-1)}{(c-a-b-1)^2(c-a-b-2)}.
\end{equation}
\end{romenumerate}
\end{theorem}
\begin{proof}
\pfitemref{THGmomNo}
Follows from \refT{THGtail}.

\pfitemref{THGmomYes}
  It follows from \eqref{hyper}, using \eqref{gauss} and \eqref{rise}, that 
  \begin{equation*}
	\begin{split}
	E\fall{X}m
&=
\frac{\frac{\ddx^m}{\ddx z^m}\FF(a,b;c;z)|_{z=1}}{\FF(a,b;c;1)}
= \frac{\rise{a}m\rise{b}m}{\rise{c}m}\cdot\frac{\FF(a+m,b+m;c+m;1)}{\FF(a,b;c;1)}
\\&
= \frac{\rise{a}m\rise{b}m}{\rise{c-a-b-m}m}.	  
	\end{split}
  \end{equation*}
This proves \eqref{EHGm}. Taking $m=1$ we find \eqref{EHG}, and the cases
$m=1$ and $m=2$ yield \eqref{VHG} after a small calculation, using $\Var X
= \E \fall X2+\E X-(\E X)^2$.
\end{proof}

We see from Theorems \ref{TAlin} and \ref{TApxt} that a hypergeometric
distribution can arise as a mixture $\NBin(a,e^{-\chi \tau})$ 
of negative binomial distributions,
or a mixture $\Bin(m,1-e^{-|\chi| \tau})$ of binomial distributions. 
In this case, $\tau\sim\Exp(\ga)$, so the random
parameter $e^{-|\chi|\tau}$ in the negative binomial or binomial
distribution has the beta
distribution $B(\ga/|\chi|,1)$,
and $1-e^{-|\chi|\tau}$  has the beta
distribution $B(1,\ga/|\chi|)$.
This extends to mixing using arbitrary beta distributions.

\begin{theorem}\label{Tmix}
  \begin{thmenumerate}
  \item \label{TmixNBin}
Let $\gax,\gbx,r>0$ be real numbers. 
Let $X$ have a distribution that is a mixture of 
the negative binomial distribution $\NBin(r,p)$ with 
$p\sim B(\gax,\gbx)$. 
Then
$X\sim\HG(r,\gbx;r+\gax+\gbx)$. 
  \item \label{TmixBin}
Let $\gax,\gbx,m>0$, with $m$ an integer. 
Let $X$ have a distribution that is a mixture of 
the binomial distribution $\Bin(m,p)$ with 
$p\sim B(\gax,\gbx)$. 
Then
$X\sim\HG(-m,\gax;1-\gbx-m)$. 
  \end{thmenumerate}
\end{theorem}

\begin{proof}
The proofs of both parts are similar: we use the definitions of the negative
binomial, binomial and beta distributions, evaluate a beta integral and make
some manipulations using \eqref{rise} and \eqref{fall}.
It is not difficult to keep track of the constant factors during the
calculations (and, indeed, this is a useful check, which we leave to the
reader), but it is simpler to ignore them, since the final factor is
determined by \eqref{hgC} and thus does not need to be computed;
we thus just write $C_1,\dots$ for
various constants (depending on the parameters but not on $k$).
\CCreset
\pfitemref{TmixNBin}
  \begin{equation}
	\begin{split}
\P(X=k)&=
\intoi \frac{\rise{r}k}{k!}p^r(1-p)^k\cdot
\frac{p^{\gax-1}(1-p)^{\gbx-1}}{B(\gax,\gbx)}\dd p
\\&
=\frac{\rise{r}k}{k!}\frac{B(r+\gax,k+\gbx)}{B(\gax,\gbx)}
=\CC\frac{\rise{r}k\gG(\gbx+k)}
 {k!\,\gG(r+\gax+\gbx+k)}
\\&
=
\CC
\frac{\rise{r}k\rise{\gbx}k}{\rise{r+\gax+\gbx}kk!}.
	\end{split}
  \end{equation}

\pfitemref{TmixBin}
  \begin{equation}
	\begin{split}
\P(X=k)&=
\intoi \frac{\fall mk}{k!}p^k(1-p)^{m-k}\cdot
\frac{p^{\gax-1}(1-p)^{\gbx-1}}{B(\gax,\gbx)}\dd p
\\&
=\frac{\fall mk}{k!}\frac{B(k+\gax,m-k+\gbx)}{B(\gax,\gbx)}
=\CC\frac{\fall mk\gG(k+\gax)\gG(\gbx+m-k)}
 {k!}
\\&
=\CC\frac{\fall mk\gG(k+\gax)}{\fall{\gbx+m-1}{k}k!}
=\CC\frac{\rise{-m}k\rise{\gax}k}{\rise{1-\gbx-m}{k}k!}
. \end{split}
\raisetag{\baselineskip}
  \end{equation}
\end{proof}

\begin{remark}\label{Rbetabin}
In this context the resulting hypergeometric distributions are known as 
\emph{beta-negative binomial distributions}
and \emph{beta-binomial distributions}.
(Note that we can obtain any distribution of type \ref{hg+++} or \ref{hg-+-}
above.) 
There are also several other names used for various cases of the general
hypergeometric distribution,
see \cite{Johnson}; in particular, case \ref{hg-+-} is sometimes called
\emph{negative hypergeometric}.
\end{remark}

\begin{example}\label{EYule2}
  \citet{Yule} considered a simple model of evolution, where each existing
  species creates new species in the same genus with a constant rate $\gl_s$,
and also (independently) new species in new genera with another rate
$\gl_g$.
What is the limiting distribution of the number of species in a random
genus?

Note that
the evolution of all species, ignoring their genus, is a \CMJ{} branching
process, where each individual gets children according to a Poisson process
with intensity $\gl_s+\gl_g$.  
Hence, assuming that we start with a single species,
the total number of species forms a Yule process with rate
$\gll$, see \refE{EYule1}.
Similarly, the number of species in the same genus as the root (the original
species) forms a Yule process with rate $\gl_s$.

One way to treat this problem is to consider each genus as an
individual in a \CMJ{} process, where each individual has an internal
Yule process $\cY_t$ with rate $\gl_s$ (the number of species in the genus),
and new births occur with rate $\gl_g\cY_t$. Since $\E \cY_t=e^{\gl_s t}$, 
see \refT{TAext}, the offspring process has intensity 
$\mu(\ddx t)=\gl_ge^{\gl_s t}\dd t$, 
from which it follows that \eqref{malthus} holds with the
Malthusian parameter $\ga=\gll$. 
The assumptions \refBP{} hold, and \refT{TBP} shows that
 the number of species in a random genus converges in
distribution to $\cY_\taux$, the number of species in the root at time
$\taux\sim\Exp(\gll)$.
The Yule process $\cY_t$ starts as 1, but we may as in \refE{EYuleA} (which is
the case $\gl_s=1$) apply \refT{TAlin} 
to $\cY_t-1$, with $\chi=\rho=\gl_s$, and it follows that the 
asymptotic distribution of the number of species in a given genus is 
$1+\HG(1,1;3+\gl_g/\gl_s)$.
(One can, as said above, 
also use \refT{Tmix}\ref{TmixNBin}, since \refT{TApxt} implies that
$\cY_t-1$ has the geometric distribution
$\NBin(1,e^{-\gl_s t})=\Geo(e^{-\gl_st})$, and $e^{-\gl_s\taux}\sim
B((\gll)/\gl_s,1)$.) 

This result was found by \citet{Yule} (by a different method), and
a distribution of the form $1+\HG(1,1;c)$ is therefore called a 
\emph{Yule distribution}, 
see further \cite{Simon} and \cite[\S 6.10.3]{Johnson}.
(Here $c>2$. Often one writes $c=2+\rho$, with $\rho>0$; in the present
example, thus $\rho=(\gll)/\gl_s$.)
Note that the case $c=3$ appears in \eqref{rrtM} and (shifted to
$\HG(1,1;3)$) in \eqref{mstkeys}.

An alternative method to treat this example is to consider the Yule process
(with rate $\gll$) of all species.
Call the first species in each genus the \emph{progenitor} of the
genus, and give each progenitor a mark; then
each species (except the original one) is marked with probability
$p=\gl_g/(\gll)$, and these marks are independent of each other and of
everything else. 
Hence, we obtain the same asymptotic distribution of fringe trees (except
for the mark at the root)
if we
sample a random progenitor as if we sample a random species, \cf{}
\refS{Ssample}. 
Using \refT{TBP} it follows that
the number of species in a random genus converges in distribution 
to the number of descendants in the same genus of a fixed individual stopped
at a random age $\tau\sim\Exp(\gll)$. 
This yields the same result as above.

Note also that
if we erase the edges between different genera and only keep the edges
between species in the same genus in Yule's model, we obtain a growing
forest. If we let $F_n$ be this forest when it has reached $n$ nodes,  
we obtain a growing forest process which is the same as UGROW defined by 
\citet{DevroyeMcDR}. 
Results for the size of the subtree rooted at a given node in UGROW
are given by 
\citet{DevroyeMcDR} and \citet{Pakes}.
\end{example}

\begin{example}\label{EBose}
Let $n$ indistinguishable balls be distributed in $m\ge2$ labelled boxes, so
that 
each of the $\binom{n+m-1}{m-1}$ possibilities has the same probability. 
(This is called \emph{Bose--Einstein statistics},
see \eg{}
\cite[page 39]{FellerI} or \cite[Example 12.2]{SJ264}.)
Equivalently, we consider a uniformly random composition of $n$ into $m$
(possibly empty) parts.

The number $X_1$ of balls in a given box, say box 1, 
has the distribution
\begin{equation}
  \P(X_1=k)=\frac{\binom{n-k+m-2}{m-2}}{\binom{n+m-1}{m-1}}
=\frac{(m-1)\fall nk}{(n+m-1)\fall{n+m-2}k},
\qquad k=0,\dots,n,
\end{equation}
which shows that $X_1\sim\HG(-n,1;2-n-m)$.
See
\cite[Chapter 10.4.2]{Johnson}. 
\end{example}

\begin{example}
\Polya's urn contains balls of different colours. We draw a ball uniformly
at random and replace it together with $c$ new balls of the same colour.
This is repeated $n$ times. 
Let $W$ be the number of white balls drawn, assuming that the urn initially
contains $w$ white and $b$ black (or non-white) balls.

We assume $c\neq0$, to avoid the trivial case $c=0$ when $W$ has a binomial
distribution, but $c<0$ is allowed, meaning that balls are removed. 
In particular, $c=-1$ gives drawing without replacement, 
when $W$ has the classical hypergeometric distribution
in \refE{EHGclassical}.
(It is natural to let $c,b,w$ be integers, but the model has a natural
interpretation also for real values of these parameters, see \eg{}
\cite[Remark 4.2]{SJ154}.) 
In the case $c<0$, we
assume that $b,w$ and $n$ are such that we never can be required to remove a
ball of a colour that is no longer present, or draw a ball when the urn is
empty. 

It is easy to see that
$\P(X=k)$ is proportional to
\begin{multline}
  \binom nk \prod_{i=1}^k (w+(i-1)c) \prod_{j=1}^{n-k} (b+(j-1)c)
=
  \binom nk c^n  \rise{w/c}k \rise{b/c}{n-k}
\\
=c^n\rise{b/c}n \frac{\fall{n}{k}\rise{w/c}k}{k!\,\fall{b/c+n-1}{k}}  
\end{multline}
and thus
$W\sim\HG(-n,w/c;1-n-b/c)$, see \cite[\S 6.2.4]{Johnson}.

For a connection with \refE{EBose}, suppose instead that the urn starts with
one ball each of $m$ colours (including white), and that $c=1$.
The number $W$ of white balls drawn is the same as if we start with 1 white and
$m-1$ black balls, and thus $W\sim\HG(-n,1;2-n-m)$. On the other hand, it is
easy to see by induction, that for each $n$, the composition of the urn is
uniform over all possible colour combinations. 
Thus $W$ has the same distribution as $X_1$ in \refE{EBose} (with colours
corresponding to boxes).
\end{example}

\section{Order statistics of exponential variables}\label{AC}

Let, for $1\le k\le m$, $V_{m,k}$ be the $k$:th smallest of $m$ \iid{} 
$\Exp(1)$ random variables $E_1,\dots,E_m$; further, let $V_{m,0}=0$.

\begin{theorem}
  \label{TC}
Let $1\le k\le m$.
Then the following holds.
\begin{romenumerate}
\item \label{tca}
If $Y_i\sim\Exp(i)$ are independent, then
\begin{equation}
V_{m,k}\eqd \sum_{i=m-k+1}^m Y_i.  
\end{equation}

\item \label{tcb}
$V_{m,k}$ has the density function
  \begin{equation}\label{etcb}
	\frac{m!}{(m-k)!\,(k-1)!}e^{-(m-k+1)x}\xpar{1-e^{-x}}^{k-1},
\qquad x>0.
  \end{equation}

\item \label{tcc}
$e^{-V_{m,k}}$ has the beta distribution $B(m-k+1,k)$.

\item \label{tcd}
For any $\gth>-(m-k+1)$,
\begin{equation}
  \E e^{-\gth V_{m,k}} = \prod_{i=m-k+1}^m\frac{i}{i+\gth} 
= \frac{\fall{m}k}{\fall{m+\gth}k}.
\end{equation}
\end{romenumerate}
\end{theorem}

\begin{proof}
  \pfitemref{tca}
Consider $m$ independent exponential clocks that strike at
$E_1,\dots,\allowbreak E_m$. 
As is well-known, by the lack of memory for exponential
distributions, the waiting times $W_i:=V_{m,i}-V_{m,i-1}$ between the
strikes are independent and exponentially distributed, with
$W_i\sim\Exp(m-i+1)$ since there are $m-i+1$ clocks left. 
Let $Y_i:=W_{m-i+1}\sim\Exp(i)$ and note that $V_{m,k}=\sum_{i=1}^k W_i$.

\pfitemref{tcc}
Let $U_i:=e^{-E_i}$. Then $U_1,\dots,U_m\sim U(0,1)$ are $m$ \iid{} uniform
random variables, and $e^{-V_{m,k}}$ is the $k$:th largest of them; it is
  well-known that this has the beta distribution $B(m-1+k,k)$.

\pfitemref{tcb}
By \ref{tcc},
$e^{-V_{m,k}}$ has the density 
\begin{equation}
  \frac{\gG(m+1)}{\gG(m-k+1)\gG(k)} x^{m-k}(1-x)^{k-1}
=   \frac{m!}{(m-k)!\,(k-1)!} x^{m-k}(1-x)^{k-1},
\end{equation}
for $0<x<1$, and \eqref{etcb} follows.

\pfitemref{tcd}
A simple consequence of \ref{tca}, or alternatively of \ref{tcc}.
\end{proof}

\newcommand\AAP{\emph{Adv. Appl. Probab.} }
\newcommand\JAP{\emph{J. Appl. Probab.} }
\newcommand\JAMS{\emph{J. \AMS} }
\newcommand\MAMS{\emph{Memoirs \AMS} }
\newcommand\PAMS{\emph{Proc. \AMS} }
\newcommand\TAMS{\emph{Trans. \AMS} }
\newcommand\AnnMS{\emph{Ann. Math. Statist.} }
\newcommand\AnnPr{\emph{Ann. Probab.} }
\newcommand\CPC{\emph{Combin. Probab. Comput.} }
\newcommand\JMAA{\emph{J. Math. Anal. Appl.} }
\newcommand\RSA{\emph{Random Struct. Alg.} }
\newcommand\ZW{\emph{Z. Wahrsch. Verw. Gebiete} }
\newcommand\DMTCS{\jour{Discr. Math. Theor. Comput. Sci.} }

\newcommand\AMS{Amer. Math. Soc.}
\newcommand\Springer{Springer-Verlag}
\newcommand\Wiley{Wiley}

\newcommand\vol{\textbf}
\newcommand\jour{\emph}
\newcommand\book{\emph}
\newcommand\inbook{\emph}
\def\no#1#2,{\unskip#2, no. #1,} 
\newcommand\toappear{\unskip, to appear}

\newcommand\arxiv[1]{\texttt{arXiv:#1}}
\newcommand\arXiv{\arxiv}

\def\nobibitem#1\par{}


\begin{thebibliography}{999}

\bibitem[Aldous(1991)]{Aldous}
David Aldous,
Asymptotic fringe distributions for general families of random trees.
\emph{Ann. Appl. Probab.} \vol1 (1991), no. 2, 228--266. 

\bibitem[Asmussen and Hering(1983)]{AsmussenH}
S{\o}ren Asmussen and Heinrich Hering,
\emph{Branching processes}. 
Birkh{\"a}user, 
Boston, MA, 1983. 

\bibitem[Athreya(2007)]{Athreya} 
K. B. Athreya,
Preferential attachment random graphs with general weight function.
\emph{Internet Math.} \vol4 (2007), no. 4, 401--418. 

\bibitem[Athreya, Ghosh, and Sethuraman(2008)]{AthreyaGS}
Krishna B. Athreya, Arka P.  Ghosh and  Sunder Sethuraman,
Growth of preferential attachment random graphs via continuous-time branching
processes. 
\emph{Proc. Indian Acad. Sci. Math. Sci.} \vol{118} (2008), no. 3, 473--494.

\bibitem[Baeza-Yates(1987)]{Baeza-Yates}
Ricardo A. Baeza-Yates. Some average measures in 
$m$-ary search trees.
\emph{Inform. Process. Lett.} \vol{25} (1987), no. 6, 375--381.

\bibitem[Barab\'asi and Albert(1999)]{BarabasiA}
Albert-L\'aszl\'o Barab\'asi and R\'eka Albert,
Emergence of scaling in random networks.
\emph{Science} 
\vol{286} (1999), no. 5439,  509--512.

\bibitem[Bennies and Kersting(2000)]{BenniesK}
J\"urgen Bennies and G\"otz Kersting, 
A random walk approach to Galton--Watson trees. 
\emph{J. Theoret. Probab.} \vol{13} (2000), no. 3, 777--803. 

\bibitem{BergeronFS92}
Fran\c cois Bergeron, Philippe Flajolet and Bruno Salvy,
Varieties of increasing trees. 
\emph{CAAP '92 (Rennes, 1992)}, 24--48,
Lecture Notes in Comput. Sci. 581, Springer, Berlin, 1992

\bibitem[Bertoin(2006)]{Bertoin}
Jean Bertoin,
\book{Random Fragmentation and Coagulation Processes}.
Cambridge University Press, Cambridge, 2006.

\bibitem[Bhamidi(2007)]{Bhamidi}
Shankar Bhamidi,
Universal techniques to analyze preferential attachment tree and networks:
Global and local analysis.
Preprint, 2007. 
\url{http://www.unc.edu/~bhamidi/preferent.pdf}

\bibitem[Biggins(1976)]{Biggins76}
J. D. Biggins, 
The first- and last-birth problems for a multitype age-dependent branching
process. 
\emph{Advances in Appl. Probability} \textbf{8} (1976), no. 3, 446--459.

\bibitem[Biggins(1977)]{Biggins77martingale}
J. D. Biggins, 
Martingale convergence in the branching random walk. 
\emph{J. Appl. Probability} \vol{14} (1977), no. 1, 25--37.

\bibitem[Biggins(1977)]{Biggins77}
J. D. Biggins, 
Chernoff's theorem in the branching random walk. 
\emph{J. Appl. Probability} \vol{14} (1977), no. 3, 630--636.



\bibitem[Biggins(1995)]{Biggins95}
J. D. Biggins, 
The growth and spread of the general branching random walk. 
\emph{Ann. Appl. Probab.} \textbf{5} (1995), no. 4, 1008--1024.

\bibitem[Biggins(1997)]{Biggins97}
J. D. Biggins, 
How fast does a general branching random walk spread? 
\emph{Classical and Modern Branching Processes (Minneapolis, MN, 1994)}, 
19--39, 
Springer, New York, 1997.

\bibitem[Biggins and Grey(1996)]{BigginsGrey1996}
J. D.  Biggins and D. R.  Grey,
A note on the growth of random trees. 
\emph{Statist. Probab. Lett.} \vol{32} (1997), no. 4, 339--342.

\bibitem[Blum and  Fran\c{c}ois(2005)]{BlumF}
Michael G. B. Blum and  Olivier Fran\c{c}ois, 
Minimal clade size and external branch length under the neutral coalescent.
\emph{Adv. in Appl. Probab.} \vol{37} (2005), no. 3, 647--662. 

\nobibitem[Blum, Fran\c{c}ois and Janson(2006)]{SJ180}
Michael G. B. Blum,  Olivier Fran\c{c}ois and Svante Janson,
The mean, variance and limiting distribution of two statistics sensitive to
phylogenetic tree balance. 
\emph{Ann. Appl. Probab.} \vol{16} (2006), no. 4, 2195--2214. 

\bibitem[Bollob\'as,  Riordan,  Spencer and Tusn\'ady(2001)]{Bollobasetal}
B\'ela Bollob\'as, Oliver Riordan, Joel Spencer and G\'abor Tusn\'ady,
The degree sequence of a scale-free random graph process.
\emph{Random Structures Algorithms} \vol{18} (2001), no. 3, 279--290. 

\bibitem[B{\'o}na(2013)]{Bona}
Mikl\'os B{\'o}na,
$k$-protected vertices in binary search trees.
\emph{Adv. in Appl. Math.} \textbf{53} (2014), 1--11.

\bibitem[B{\'o}na and Pittel(2014)]{BonaPittel}
Mikl\'os B{\'o}na and Boris Pittel,
On a random search tree: asymptotic enumeration of vertices by distance from
leaves.
Preprint, 2014.
\arxiv{1412.2796}

\bibitem[Broutin and Devroye(2006)]{BroutinDevroye2006} 
Nicolas Broutin and Luc Devroye, 
Large deviations for the weighted height of an extended class of trees. 
\emph{Algorithmica} \textbf{46} (2006), no. 3-4, 271--297. 

\bibitem[Broutin, Devroye and McLeish(2008)]{BroutinDevroyeMcLeish2008}
Nicolas Broutin, Luc Devroye and Erin McLeish,  
Weighted height of random trees. 
\emph{Acta Inform.} \textbf{45} (2008), no. 4, 237--277. 

\bibitem[Broutin, Devroye, McLeish and de la Salle(2008)]{BroutinDevroyeEtal2008}
Nicolas Broutin, Luc Devroye, Erin McLeish, and Mikael de la Salle, 
The height of increasing trees.
\emph{Random Structures Algorithms} \vol{32} (2008), no. 4, 494--518. 

\bibitem{BroutinHolmgren2012}
Nicolas Broutin and Cecilia Holmgren,
The total path length of split trees. \emph{Ann. Appl. Probab.}
\textbf{22} (2012), no. 5, 1745--1777.

\bibitem[Chang and  Fuchs(2010)]{ChangF10}
Huilan Chang and Michael Fuchs,
Limit theorems for patterns in phylogenetic trees.
\emph{J. Math. Biol.} \vol{60} (2010), no. 4, 481--512.

\bibitem[Chauvin and Drmota(2006)]{ChauvinDrmota2006}
Brigitte Chauvin and  Michael Drmota, 
The random multisection problem, travelling waves and the distribution of
the height of $m$-ary search trees. 
\emph{Algorithmica} \textbf{46} (2006), no. 3-4, 299--327.

\bibitem{ChauvinDrmotaJ2001} 
 Brigitte Chauvin,  Michael Drmota and Jean Jabbour-Hattab,
 The profile of binary search trees. 
\emph{Ann. Appl. Probab.} \textbf{11} (2001), no. 4, 1042--1062. 

\bibitem{ChauvinKMR2005}
B. Chauvin, T. Klein, J.-F. Marckert and A. Rouault, 
Martingales and profile of binary search trees.
\emph{Electron. J. Probab.} \textbf{10} (2005), no. 12, 420--435.

\bibitem[Chauvin and  Pouyanne(2004)]{ChauvinPouyanne} 
Birgitte Chauvin and Nicolas Pouyanne, 
$ m $-ary search trees when $ m\geq 27 $: a strong asymptotics for the space
  requirement. 
\emph{Random Structures Algorithms} \textbf{24}, (2004),  133--154.

\bibitem[Chern and Hwang(2001)]{ChernHwang2001}
Hua-Huai Chern and Hsien-Kuei Hwang, Phase changes in random
$ m $-ary search trees and generalized quicksort.
\emph{Random Structures Algorithms} \vol{19} (2001), no. 3-4, 316--358.

\bibitem[Chern, Hwang and Tsai(2002)]{ChernHwangTsai2002}
Hua-Huai Chern, Hsien-Kuei Hwang and Tsung-Hsi Tsai,
An asymptotic theory for Cauchy-Euler differential equations with
applications to the analysis of algorithms. 
\emph{J. Algorithms} \textbf{44} (2002), no. 1, 177--225.

\bibitem[Devroye(1986)]{Devroye1986} 
Luc Devroye,  A note on the height of binary search trees. 
\emph{J. Assoc. Comput. Mach.} \textbf{33} (1986), 489--498.

\bibitem[Devroye(1987)]{Devroye1987} 
Luc Devroye, 
Branching processes in the analysis of the heights of trees. 
\emph{Acta Inform.} \textbf{24} (1987), 277--298.

\bibitem[Devroye(1990)]{Devroye1990} 
Luc Devroye,  On the height of random $m$-ary search trees. 
\emph{Random Structures Algorithms} \textbf{1} (1990), no. 2, 191--203. 


\bibitem[Devroye(1991)]{Devroye1991} 
Luc Devroye,  
Limit laws for local counters in random binary search trees. 
\emph{Random Structures Algorithms} \textbf{2} (1991), no. 3, 303--315.  

\bibitem[Devroye(1993)]{Devroye93} 
Luc Devroye, 
On the expected height of fringe-balanced trees. 
\emph{Acta Inform.} \textbf{30} (1993), 459--466. 

\bibitem[Devroye(1998)]{Devroye1998}
Luc Devroye, 
Branching processes and their applications in the analysis of tree
structures and tree algorithms. 
\emph{Probabilistic Methods for Algorithmic Discrete Mathematics}, 249--314,
eds. M. Habib, C. McDiarmid, J. Ramirez and B. Reed, 
Springer, Berlin, 1998.

\bibitem[Devroye(1998)]{Devroye98split}  
Luc Devroye,  
Universal limit laws for depth in random trees. 
\emph{SIAM J. Comput.} \textbf{28} (1998), 409--432.


\bibitem[Devroye(2002/3)]{Devroye2002} 
Luc Devroye, Limit laws for sums of functions of subtrees
of random binary search trees. 
\emph{SIAM J. Comput.} \textbf{32} (2002/03), no. 1, 152--171.


\bibitem{DevroyeHwang2006}
Luc Devroye and Hsien-Kuei Hwang, 
Width and mode of the profile for some random trees of logarithmic height.
\emph{Ann. Appl. Probab.} \textbf{16} (2006), no. 2, 886--918.

\bibitem[Devroye and Janson(2014)]{SJ283}
Luc Devroye and Svante Janson,
Protected nodes and fringe subtrees in some random trees. 
\emph{Electronic Communications Probability} \textbf{19} (2014), no. 6, 1--10. 

\bibitem[Devroye, McDiarmid and Reed(2002)]{DevroyeMcDR}
Luc Devroye, Colin McDiarmid and Bruce Reed,
Giant components for two expanding graph processes.  
\emph{Mathematics and Computer Science, II (Versailles, 2002)}, 
161--173,
Birkh\"auser, Basel, 2002.

\bibitem{Doney}
R. A. Doney,
A limit theorem for a class of supercritical branching processes.
\emph{Journal of Applied Probability} \vol9 (1972), no. 4, 707--724.


\bibitem[Drmota(2009)]{Drmota}
Michael Drmota,
\emph{Random Trees}.
Springer, Vienna, 2009.

\bibitem[Drmota(2003)]{DrmotaII2003}
Michael Drmota,
An analytic approach to the height of binary search trees II.
\emph{J. ACM} \textbf{50} (2003), no. 3, 333--374.

\bibitem[Drmota, Fuchs and Lee(2014)]{DFL}
Michael Drmota, Michael Fuchs and Yi-Wen Lee,
Limit laws for the number of groups formed by social animals under the extra
clustering model. (Extended abstract.)
\emph{Proceedings, 
25th International Conference on Probabilistic, Combinatorial and Asymptotic
Methods for the Analysis of Algorithms,
AofA'14
(Paris, 2014)},
73--84, 
\emph{DMTCS Proceedings}, 2014.

\bibitem[Drmota, Gittenberger, Panholzer, Prodinger and Ward(2009)]
{Drmota:shape}
Michael Drmota, Bernhard Gittenberger, Alois Panholzer, Helmut Prodinger and
Mark Daniel Ward,
On the shape of the fringe of various types of random trees.
\emph{Math. Methods Appl. Sci.} \vol{32} (2009), no. 10, 1207--1245. 


\bibitem{DrmotaHwang2005}  
Michael Drmota and   Hsien-Kuei Hwang, 
Bimodality and phase transitions in the profile variance of random binary
search trees. 
\emph{SIAM J. Discrete Math.} \textbf{19} (2005), no. 1, 19--45.

 
\bibitem[Drmota, Janson and  Neininger(2008)]{DrmotaJansonNeininger2008}
Michael Drmota, Svante Janson and Ralph Neininger,  
A functional limit theorem for the profile of search trees. 
\emph{Ann. Appl. Probab.} \textbf{18} (2008), no. 1, 288--333. 

\bibitem[Durand, Blum and  Fran\c{c}ois(2007)]{DBF07}
Eric Durand, Michael G. B. Blum and  Olivier Fran\c{c}ois, 
Prediction of group patterns in social mammals based on a coalescent model. 
\emph{J. Theoret. Biol.} \vol{249} (2007), no. 2, 262--270.

\bibitem[Durand and  Fran\c{c}ois(2010)]{DF10}
Eric Durand and  Olivier Fran\c{c}ois, 
Probabilistic analysis of a genealogical model of animal group patterns. 
\emph{J. Math. Biol.} \vol{60} (2010), no. 3, 451--468.

\bibitem{FellerI}
William Feller, 
\emph{An Introduction to Probability Theory and Its Application}, volume I, 
third edition, Wiley, 
New York, 1968.

\bibitem[Fill and Kapur(2005)]{FillKapur-mary}
James Allan Fill  and  Nevin Kapur,
Transfer theorems and asymptotic distributional results for $m$-ary search
trees. 
\emph{Random Structures Algorithms} \textbf{26} (2005), no. 4, 359--391.


\bibitem[Flajolet, Gourdon and Mart\'inez(1997)]{FlajoletGM1997}
Philippe Flajolet, Xavier Gourdon and Conrado Mart\'inez,
Patterns in random binary search trees.
\emph{Random Structures Algorithms} \textbf{11} (1997), no. 3, 223--244.

\bibitem[Fuchs(2008)]{Fuchs} 
Michael Fuchs, 
Subtree sizes in recursive trees and binary search trees: Berry--Esseen
bounds and Poisson approximations. 
\emph{Combin. Probab. Comput.}
\textbf{17}, (2008), no. 5, 661--680. 

\bibitem[Fuchs, Hwang and Neininger(2006)]{FuchsHN2006} 
Michael Fuchs, Hsien-Kuei Hwang, Ralph Neininger. 
Profiles of random trees: limit theorems for random recursive trees and
binary search trees. 
Algorithmica \textbf{46} (2006), no. 3-4, 367--407.

\bibitem[Gastwirth and Bhattacharya(1984)]{GastBhatt84}
J. L. Gastwirth and P. K. Bhattacharya, 
Two probability models of pyramid or chain letter schemes demonstrating that
their promotional claims are unreliable. 
\emph{Oper. Res.} \vol{32} (1984), no. 3, 527--536. 

\bibitem[Gauss(1813)]{Gauss}
Carlo Friderico [Carl Friedrich] Gauss,
Disquisitiones generales circa seriem infinitam 
$1+\frac{\ga\gb}{1.\gam}x+\frac{\ga(\ga+1)\gb(\gb+1)}{1.2.\gam(\gam+1)}xx+
\frac{\ga(\ga+1)(\ga+2)\gb(\gb+1)(\gb+2)}{1.2.3.\gam(\gam+1)(\gam+2)}x^3+
$ etc., pars prior.
\emph{Commentationes societatis regiae scientiarum Gottingensis recentiores}
{II} (1813).
Reprinted in 
\emph{Werke}, vol 3, 123--162, G\"ottingen, 1863.
\url{http://gdz.sub.uni-goettingen.de/dms/load/toc/?PPN=PPN235999628}

\bibitem{CM}
Ronald L. Graham, Donald E. Knuth and Oren Patashnik,
\emph{Concrete Mathematics}.
2nd ed., Addison-Wesley,  Reading, MA, 1994.

\bibitem{Gut}
Allan Gut, 
\emph{Probability: A Graduate Course},
2nd ed., Springer, New York, 2013. 

\bibitem[Harris(1963)]{Harris}
Theodore E. Harris,
\emph{The Theory of Branching Processes}. 
Springer-Verlag, Berlin;
Prentice-Hall, Inc., Englewood Cliffs, N.J., 1963.

\bibitem[Heimb\"urger(2014)]{Heimburger}
Axel Heimb\"urger,
Asymptotic distribution of two-protected nodes in $m$-ary search trees.
Master thesis, Stockholm University and KTH, 2014.
\url{diva-portal.org/smash/get/diva2:748258/FULLTEXT01.pdf}


\bibitem[Hennequin(1991)]{Hennequin91}
Pascal Hennequin,
Analyse en moyenne d'algorithmes, tri rapide et arbres de recherche.
Ph.D. dissertation, {\'E}cole Polytechnique, Palaiseau, 1991.
\url{http://algo.inria.fr/AofA/Research/src/Hennequin.PhD.html}

\bibitem[van der Hofstad(2014)]{Hofstad}
Remco van der Hofstad, 
\emph{Random Graphs and Complex Networks. Vol. I}.
Lecture notes. October 20, 2014 version.
\url{http://www.win.tue.nl/~rhofstad/NotesRGCN.html}

\bibitem{Holmgren2011} 
Cecilia Holmgren,
A weakly 1-stable distribution for the number of random records and cuttings
in split trees. 
\emph{Adv. in Appl. Probab.} \textbf{43} 151--177.


\bibitem[Holmgren and Janson(2015)]{HolmgrenJanson1} 
Cecilia Holmgren and Svante Janson,
Limit laws for functions of fringe trees for binary search trees and
  random recursive trees. 
\emph{Electron. J. Probab.} \textbf{20} (2015), no. 4, 1--51.

\bibitem[Holmgren and Janson(2015)]{HolmgrenJanson2} 
Cecilia Holmgren and Svante Janson,
Asymptotic distribution of two-protected nodes in
ternary search trees. 
\emph{Electron. J. Probab.} \textbf{20} (2015), no. 9, 1--20.

\bibitem{HolmgrenJansonSileikis}
Cecilia Holmgren, Svante Janson and Matas \v{S}ileikis, 
Multivariate normal limit laws for the numbers of fringe subtrees in 
$m$-ary search trees and preferential attachment trees. 
{In preparation}.

\bibitem[Hwang(2003)]{Hwang2003}
Hsien-Kuei Hwang. Second phase changes in random $ m $-ary search trees and
generalized quicksort: 
convergence rates. \emph{Ann. Probab.} \vol{31} (2003), no. 2, 609--629.

\bibitem[Hwang(2007)]{Hwang2007}
Hsien-Kuei Hwang,
Profiles of random trees: plane-oriented recursive trees. 
\emph{Random Structures Algorithms} \textbf{30} (2007), no. 3, 380--413.

\bibitem[Jagers(1975)]{Jagers}
Peter Jagers,
\emph{Branching Processes with Biological Applications.}
John Wiley \& Sons, London, 
1975. 

\bibitem[Jagers and Nerman(1984)]{JagersNerman1984}
Peter Jagers and  Olle Nerman,
The growth and composition of branching populations. 
\emph{Adv. in Appl. Probab.} \vol{16} (1984), no. 2, 221--259. 

\bibitem[Jagers and Nerman(1996)]{JagersNerman1996}
Peter Jagers and  Olle Nerman,
The asymptotic composition of supercritical multi-type branching
populations. 
\emph{S{\'e}minaire de Probabilit{\'e}s, XXX}, 40--54, 
Lecture Notes in Math. {1626}, 
Springer, Berlin, 1996. 

\bibitem[Janson(2004)]{SJ154}  
Svante Janson, 
Functional limit theorems for multitype branching processes and generalized
\Polya{} urns.  
\emph{Stoch. Process. Appl.} \textbf{110} (2004),  177--245.


\bibitem[Janson(2005)]{SJ155}
Svante Janson,
Asymptotic degree distribution in random recursive trees.
\emph{Random Structures Algorithms} \vol{26} (2005), no. 1-2, 69--83. 

\bibitem[Janson(2012)]{SJ264}
Svante Janson,
Simply generated trees, conditioned Galton-Watson trees,
 random allocations and condensation.
\emph{Probability Surveys} \vol9 (2012), 103--252.

\bibitem[Janson(2016)]{SJ285}
Svante Janson,
Asymptotic normality of fringe subtrees and additive functionals in
conditioned Galton--Watson trees. 
\emph{Random Structures Algorithms} 
\textbf{48} (2016), no. 1, 57--101.

\bibitem[Janson(2015)]{SJ299}
Svante Janson,
Maximal clades in random binary search trees.
\emph{Electron. J. Combin.} \textbf{22} (2015), no. 1, Paper P1.31.  

\bibitem{SJbranching}
Svante Janson.
Asymptotic normality in Crump--Mode--Jagers processes: the discrete time case.
In preparation.

\bibitem[Janson and Neininger(2008)]{Janson2008}
Svante Janson and Ralph Neininger, 
The size of random fragmentation trees. 
\emph{Probab. Theory Related Fields} \textbf{142} (2008), no. 3-4, 399--442. 

\bibitem[Johnson, Kemp and Kotz(2005)]{Johnson}
Norman L. Johnson,
Adrienne W. Kemp, 
and Samuel Kotz,
\emph{Univariate Discrete Distributions}.
3rd ed., 
John Wiley \& Sons, Hoboken, NJ, 2005.

\bibitem{Kallenberg}
Olav Kallenberg,
\book{Foundations of Modern Probability.}
2nd ed., Springer, New York, 2002. 

\bibitem[Kalpathy and Mahmoud(2014+)]{KalMahmoud:Degree}
Ravi Kalpathy and
Hosam  Mahmoud, 
Degree profile of $m$-ary search trees: A vehicle for data structure
compression. 
\emph{Probab. Engrg. Inform. Sci.} \vol{30} (2016), no. 1, 113--123. 

\bibitem[Kingman(1975)]{Kingman}
J. F. C. Kingman, 
The first birth problem for an age-dependent branching process.
\emph{Ann. Probability} \vol{3} (1975), no. 5, 790--801. 

\bibitem[Knuth(1973)]{Knuth1998} 
Donald E. Knuth,
\emph{The Art of Computer Programming. Vol. 3: Sorting and Searching.} 
2nd ed., Addison-Wesley, Reading, MA, 1998.

\bibitem[Kolmogorov(1941)]{Kolmogorov}
A. N. Kolmogoroff, 
\"Uber das logarithmisch normale Verteilungsgesetz der Dimensionen der
Teilchen bei Zerst\"uckelung. 
\jour{C. R. (Doklady) Acad. Sci. URSS (N. S.)}  \vol{31}  (1941), 99--101.

\bibitem[Krapivsky and Redner(2001)]{KrapivskyR}
P. L. Krapivsky and S. Redner,
Organization of growing random networks.
\emph{Phys. Rev. E} \vol{63} (2001), 066123.

\bibitem[Krapivsky, Redner and Leyvraz(2000)]{KrapivskyRL}
P. L. Krapivsky, S. Redner, and F. Leyvraz,
Connectivity of Growing Random Networks.
\emph{Phys. Rev. Lett.} \vol{85} (2000), 4629--4632.

\bibitem[Kuba and Panholzer(2008)]{KubaPanholzer}
Markus Kuba and Alois Panholzer, 
Isolating a leaf in rooted trees via random cuttings. 
\emph{Ann. Comb.} \vol{12} (2008), no. 1, 81--99.

\bibitem[Lew and Mahmoud(1994)]{LewMahmoud1994}
William Lew and Hosam M. Mahmoud, 
The joint distribution of elastic buckets in multiway search trees. 
\emph{SIAM J. Comput.} \vol{23} (1994), no. 5, 1050--1074.

\bibitem[Lu and Feng(1998)]{LuFeng}
Jiang Lu and Qilin Feng,
Strong consistency of the number of vertices of given degrees in nonuniform
random recursive trees.  
\emph{Yokohama Math. J.} \vol{45} (1998), no. 1, 61--69. 

\bibitem[Mahmoud(1986)]{Mahmoud:pathlength}
Hosam M. Mahmoud, 
On the average internal path length of $m$-ary search trees,
\emph{Acta Inform.} \textbf{23} (1986), 111--117.

\bibitem[Mahmoud(1992)]{Mahmoud:Evolution}
Hosam M. Mahmoud,
\emph{Evolution of Random Search Trees}. 
John Wiley \& Sons, 
New York, 1992. 

\bibitem[Mahmoud(1994)]{Mahmoud:pyramids}
Hosam M. Mahmoud, 
A strong law for the height of random binary pyramids.
\emph{Ann. Appl. Probab.} \vol4 (1994), no. 3, 923--932. 

\bibitem[Mahmoud and Pittel(1989)]{MahmoudPittel1989}
 Hosam M. Mahmoud and Boris Pittel,  
Analysis of the space of search trees under the random insertion algorithm. 
\emph{J. Algorithms} \vol{10} (1989), no. 1, 52--75. 

\bibitem{MS:trees}
Hosam M. Mahmoud  and R. T. Smythe,  
Asymptotic joint normality of outdegrees of nodes in random recursive trees. 
\emph{Random Structures Algorithms} \textbf3 (1992), no. 3, 255--266.

\bibitem[Mahmoud, Smythe and Szyma\'nski(1993)]{MahmoudSS}
Hosam M. Mahmoud,  R. T. Smythe and Jerzy Szyma\'nski,  
On the structure of random plane-oriented recursive trees and their branches.
\emph{Random Structures Algorithms} \vol4 (1993), no. 2, 151--176.

\bibitem[Mahmoud and Ward(2012)]{MahmoudWard-bst}
Hosam M. Mahmoud and Mark Daniel Ward,
Asymptotic distribution of two-protected nodes in random binary search
trees.
\emph{Appl. Math. Lett.}
\vol{25} (2012), no. 12,  2218--2222.

\bibitem[Mahmoud and Ward(2015)]{MahmoudWard-rrt}
Hosam M. Mahmoud and Mark D. Ward,
Asymptotic properties of protected nodes in random recursive trees.
\emph{J. Appl. Probab.} \textbf{52} (2015), no. 1, 290--297.

\bibitem{Martin-Lof}
Anders Martin-L\"of,
A limit theorem for the size of the $n$th generation of an age-dependent
branching process. 
\emph{J. Math. Anal. Appl.} \textbf{15} (1966), 273--279. 

\bibitem{McDiarmid1995}
Colin McDiarmid,
Minimal positions in a branching random walk.
\emph{Ann. Appl. Probab.} \textbf{5} (1995), no. 1, 128--139.

\bibitem[M\'ori(2002)]{Mori}
T. F. M\'ori,
On random trees. 
\emph{Studia Sci. Math. Hungar.} \vol{39} (2002), no. 1-2, 143--155.

\bibitem[Muntz and Uzgalis(1971)]{Muntz71}
Richard Muntz and Robert Uzgalis, 
Dynamic storage allocation for binary search trees in a two-level memory.
\emph{Proceedings of the Princeton Conference on Information Sciences and
Systems}
\vol{4} (1971), 345--349.

\bibitem[Nerman(1981)]{Nerman}
Olle Nerman, 
On the convergence of supercritical general (C-M-J) branching processes. 
\emph{Z. Wahrsch. Verw. Gebiete} \vol{57} (1981), no. 3, 365--395.

\bibitem[Nerman and Jagers(1984)]{NermanJagers1984}
Olle Nerman and Peter Jagers,
The stable double infinite pedigree process of supercritical branching
populations. 
\emph{Z. Wahrsch. Verw. Gebiete} \vol{65} (1984), no. 3, 445--460. 

\bibitem[Neveu(1986)]{Neveu}
Jacques Neveu,  
Arbres et processus de Galton--Watson. 
\emph{Ann. Inst. H. Poincar\'e Probab. Statist.} \vol{22} \no2 (1986),  
199--207.

\bibitem{NIST}
\emph{NIST Handbook of Mathematical Functions}. 
Edited by Frank W. J. Olver, Daniel W. Lozier, Ronald F. Boisvert and
Charles W. Clark. 
Cambridge Univ. Press, 2010. \\
Also available as 
\emph{NIST Digital Library of Mathematical Functions},
\url{http://dlmf.nist.gov/}


\bibitem[Oliveira and Spencer(2005)]{OSpencer}
Roberto Oliveira and Joel Spencer,
Connectivity transitions in networks with super-linear preferential attachment.
\emph{Internet Math.} \vol2 (2005), no. 2, 121--163.

\bibitem[Pakes(2013)]{Pakes}
Anthony G. Pakes,
Limit laws for UGROW random graphs.
\emph{Statist. Probab. Lett.} \vol{83} (2013), no. 12, 2607--2614.

\bibitem[Panholzer and Prodinger(2007)]{PanholzerProdinger}
Alois Panholzer and  Helmut Prodinger,
Level of nodes in increasing trees revisited.
\emph{Random Structures Algorithms} \vol{31} (2007), no. 2, 203--226. 


\bibitem[Pittel(1984)]{Pittel1984}
Boris Pittel,
On growing random binary trees. 
\emph{J. Math. Anal. Appl.} \vol{103} (1984), no. 2, 461--480.

\bibitem[Pittel(1994)]{Pittel1994}
Boris Pittel,
Note on the heights of random recursive trees and random $m$-ary search trees, 
\emph{Random Structures Algorithms} \textbf{5} (1994), 337--347.

\bibitem[Reed(2003)]{Reed}
Bruce Reed,
The height of a random binary search tree.
\emph{J. ACM} \textbf{50} (2003), no. 3, 306--332.

\bibitem{Regnier1989} 
Mireille R\'egnier, 
A limiting distribution for quicksort. 
\emph{RAIRO Inform. Th\'eor. Appl.} \textbf{23} (1989) 335--343.


\bibitem{Rosler1991} 
Uwe R\"osler, 
A limit theorem for "Quicksort". 
\emph{RAIRO Inform. Th\'eor. Appl.} \textbf{25} (1991), 85--100.

\bibitem{Rosler2001} 
Uwe R\"osler,  
On the analysis of stochastic divide and conquer algorithms. 
\emph{Algorithmica} \textbf{29} (2001) 238--261. 

\bibitem[Rudas and T{\'o}th(2009)]{RudasT}
Anna Rudas and B{\'a}lint T{\'o}th,
Random tree growth with branching processes -- a survey. 
\emph{Handbook of Large-Scale Random Networks}, 171--202, 
Bolyai Soc. Math. Stud. {18}, 
Springer, Berlin, 2009.

\bibitem[Rudas,  T{\'o}th and Valk{\'o}(2007)]{RudasTV}
Anna Rudas, B{\'a}lint T{\'o}th and Benedek Valk{\'o},
 Random trees and general branching processes. 
\emph{Random Structures Algorithms} \vol{31} (2007), no. 2, 186--202.

\bibitem[Simon(1955)]{Simon}
Herbert A. Simon,
On a class of skew distribution functions.
\emph{Biometrika} \vol{42} (1955), 425--440.

\bibitem{Sulzbach2008} 
Henning Sulzbach, 
A functional limit law for the profile of plane-oriented recursive trees.
\emph{Fifth Colloquium on Mathematics and Computer Science}, 
\emph{Discrete Math. Theor. Comput. Sci. Proc.} \textbf{AI} (2008),
339--350.


\bibitem[Szyma{\'n}ski(1987)]{Szymanski}
Jerzy Szyma{\'n}ski,
On a nonuniform random recursive tree. 
\emph{Annals of Discrete Math.} \vol{33} (1987), 
297--306.

\bibitem[Yule(1925)]{Yule}
G. Udny Yule,
A mathematical theory of evolution, based on the conclusions of
Dr. J. C. Willis, F.R.S. 
\emph{Philos. Trans. Roy. Soc. B} \textbf{213} (1925), 21--87.

\end{thebibliography}
\end{document}